\documentclass[leqno, 10pt, a4paper]{amsart}
\usepackage{amsmath}
\usepackage{amssymb, mathrsfs, mathabx, a4wide, bbm, color}
\usepackage{tikz}
\usetikzlibrary{arrows.meta,decorations.markings,calc}
\tikzset{
  cpfig/.style={x=0.1\columnwidth,y=0.1\columnwidth,line width=0.25pt,
    >={Straight Barb[length=1.8pt,width=2.8pt]}},
  cpshade/.style={fill=black!11},
  arrowat/.style={postaction={decorate},
    decoration={markings,mark=at position #1 with {\arrow{>}}}},
  arrowsat/.style 2 args={postaction={decorate},
    decoration={markings,mark=at position #1 with {\arrow{>}},mark=at position #2 with {\arrow{>}}}},
}
\newcommand{\cpaxes}{%
  \draw[line width=0.2pt,-{Stealth[length=2.5pt,width=2pt]}] (-4.55,0)--(4.4,0);
  \draw[line width=0.2pt,-{Stealth[length=2.5pt,width=2pt]}] (0,-4.45)--(0,4.45);}
\newcommand{\cpbox}{\useasboundingbox (-5,-5.125) rectangle (5,4.875);}
\usepackage{caption}

\usepackage[utf8]{inputenc}

\theoremstyle{definition}
\newtheorem{definition}{Definition}[section]

\theoremstyle{theorem}
 \newtheorem{theorem}[definition]{Theorem}
 \newtheorem{lemma}[definition]{Lemma}
 \newtheorem{proposition}[definition]{Proposition}
 \newtheorem{corollary}[definition]{Corollary}

 \theoremstyle{remark}
 \newtheorem{example}[definition]{Example}
 \newtheorem{remark}[definition]{Remark}
 
\theoremstyle{theorem}

\newtheorem{theoremalph}{Theorem}
\newtheorem{propositionalph}[theoremalph]{Proposition}

\newcommand{\op}[1]{\operatorname{#1}}

\newcommand{\norm}[1]{\ensuremath{\|{#1}\|}}

\newcommand{\acou}[2]{\ensuremath{\left\langle #1 , #2 \right\rangle}} 
\newcommand{\brak}[1]{\ensuremath{\langle #1\rangle}}
\newcommand{\acoup}[2]{\ensuremath{\left(#1|#2\right)}}

\newcommand{\Tr}{\ensuremath{\op{Tr}}}

\newcommand{\Hol}{\op{Hol}}

\newcommand{\B}{\ensuremath{\mathbb{B}}} 
\newcommand{\C}{\ensuremath{\mathbb{C}}} 
\newcommand{\bH}{\ensuremath{\mathbf{H}}} 
\newcommand{\N}{\ensuremath{\mathbb{N}}} 
\newcommand{\R}{\ensuremath{\mathbb{R}}} 
\newcommand{\bS}{\ensuremath{\mathbb{S}}} 
\newcommand{\T}{\ensuremath{\mathbb{T}}} 
\newcommand{\U}{\ensuremath{\mathbb{U}}} 
\newcommand{\Z}{\ensuremath{\mathbb{Z}}} 

\newcommand{\Rn}{\ensuremath{\R^{n}}}

\newcommand{\fp}{\ensuremath{\mathfrak{p}}}

\newcommand{\cK}{\ensuremath{\mathscr{K}}}
\newcommand{\cL}{\ensuremath{\mathscr{L}}}
\newcommand{\cS}{\ensuremath{\mathscr{S}}}

\newcommand{\scD}{\mathscr{D}}
\newcommand{\sE}{\mathscr{E}}
\newcommand{\sF}{\mathscr{F}}
\newcommand{\sH}{\mathscr{H}}
\newcommand{\sJ}{\mathscr{J}}
\newcommand{\sK}{\mathscr{K}}
\newcommand{\sL}{\mathscr{L}}
\newcommand{\sR}{\mathscr{R}}

\newcommand{\bQ}{\mathbf{Q}}

\newcommand{\psido}{$\Psi$DO} 
\newcommand{\psidos}{$\Psi$DOs} 

\newcommand{\stS}{\mathbb{S}}

\newcommand{\supp}{\op{supp}}

\newcommand{\dist}{\op{dist}}

\newcommand{\dom}{\op{dom}}

\newcommand{\rk}{\op{rk}}

\newcommand{\bt}{\ensuremath{\bullet}}

\newcommand{\ran}{\op{ran}}
   
\newcommand{\Sp}{\op{Sp}}

\newcommand{\subsubset}{\subseteq\!\subseteq}

\newcommand{\loc}{\textup{loc}}

\newcommand{\dl}{\partial} 
\newcommand{\OP}{\op{Op}} 
\newcommand{\car}{\mathbbm{1}}

\numberwithin{equation}{section}
\begin{document}

\title{Functional Calculus on Noncommutative Tori, II.\\
Complex Powers, Logarithms, and Sectorial Projections}

\author{Gihyun Lee}
 \address{Institute for Mathematics, University of Potsdam, Germany}
 \email{gihyun.lee@uni-potsdam.de}

\author{Rapha\"el Ponge}
 \address{Department of Mathematics and Statistics, University of Ottawa, Canada}
  \curraddr{Department of Mathematics and Statistics, University of Wyoming, Laramie, WY, USA}
 \email{ponge.math@icloud.com}

\begin{abstract}
This paper develops a systematic theory of complex powers, logarithms, and sectorial projections of elliptic pseudodifferential operators on noncommutative tori, extending to this setting the classical constructions of Seeley and others. Building on the parametric pseudodifferential calculus of the prequel~\cite{LP:Part1}, we construct the complex powers associated with a given ray, show that they form a holomorphic family of pseudodifferential operators with the semigroup property, and compute their symbols. We further establish exponential growth bounds on vertical strips in the operator, Schatten, and trace-class topologies by means of a new holomorphic calculus for pseudodifferential families. The logarithm is identified as a pseudodifferential operator whose symbol is determined by the resolvent symbol, and the associated trace formula is derived. Sectorial projections are constructed as contour integrals and shown to be of order zero. This yields analogues for noncommutative tori of results of Wodzicki, Okikiolu, and Gaarde--Grubb, and provides the analytic foundations for the spectral-geometric applications developed in subsequent papers.
\end{abstract}

\maketitle 

\tableofcontents

\section{Introduction}
Noncommutative (NC) tori are arguably the most ubiquitous examples of noncommutative spaces. Given a real anti-symmetric $n\times n$ matrix $\theta=(\theta_{jk})$, the $C^*$-algebra $C(\T^n_\theta)$ is generated by unitaries $U_1,\ldots,U_n$ subject to the relations,
\begin{equation*}
 U_lU_j = e^{2i\pi \theta_{jl}} U_jU_l, \qquad j,l=1,\ldots, n.
\end{equation*}
For $\theta=0$ we recover the algebra of continuous functions on the ordinary torus $\T^n=\R^n\slash 2\pi \Z^n$. The pseudodifferential calculus on NC tori was introduced by Connes~\cite{Co:CRAS80} and Baaj~\cite{Ba:CRAS88} and was developed in detail in~\cite{HLP:Part1,HLP:Part2}. This calculus has played a central role in recent work on the differential geometry of NC tori (see, e.g.,~\cite{CF:MJM19, CM:JAMS14, CT:Baltimore11, DS:SIGMA15, FK:JNCG12, FK:JNCG15, FW:JPDOA11, FGK:JNCG19, LL:JNCG25, LM:GAFA16, LNP:TAMS16, Liu:arXiv18a, Liu:arXiv18b, MP:AIM23, SXZ:JFA23, vNSZ:JFA25} and the references therein).  

This paper is a continuation of~\cite{LP:Part1}, where we built a pseudodifferential calculus with parameter on NC tori and used it to study resolvents of elliptic operators. Building on those results, we present in this paper a systematic construction of complex powers, logarithms, and sectorial projections of elliptic pseudodifferential operators (\psidos) on NC tori. The extension of these constructions to smooth parametric families of elliptic operators will be taken up in a subsequent paper~\cite{LP:SG3}. Together, these results pave the way for various spectral-geometric applications, including heat trace asymptotics, zeta-function analysis, and Weyl laws for NC tori (see~\cite{LP:SG1, LP:SG2}). 

\subsection*{Novelty of the paper}
The approach of the paper is novel in several respects. First, the approach to complex powers of elliptic operators applies to \psidos, not just differential operators (compare with Shubin~\cite{Sh:Springer01}). Moreover, it relies on the parametric \psido\ calculus of~\cite{LP:Part1}, which is a low-tech alternative to the technically far more involved weakly parametric calculus of Grubb--Seeley~\cite{GS:IM95} (see~\cite{LL:JNCG25} for an extension of that calculus to noncommutative tori). 

Second, the approach of this paper provides exponential decay along vertical lines for families $\Gamma(z)P^{-z}$, $z\in \C\setminus \Z_{-}$, if the spectrum of the principal symbol of $P$ is contained in the half-plane $\{\Re \lambda>0\}$ and $P$ has no nonzero eigenvalue in the closed half-plane $\{\Re\lambda\leq 0\}$. These results are instrumental in establishing short-time heat trace asymptotics for such operators in the forthcoming articles~\cite{LP:SG1, LP:SG2}. In particular, we may allow the principal symbol of $P$ to be non-scalar (compare with Duistermaat--Guillemin~\cite{DG:IM75}), and this bypasses the use of the weakly parametric calculus of~\cite{GS:IM95} to get such results. 

Third, it has been pointed out in~\cite{BCLZ:JPDOA12} that the approach of Seeley~\cite{Se:PSPM67} to complex powers does not extend to showing that sectorial projections are \psidos. In contrast, the approach of this paper extends \emph{verbatim} to showing that sectorial projections are \psidos\ (compare with Grubb~\cite{Gr:MS12}). Moreover, it makes it simple to define logarithms of elliptic \psidos\ and to show that they are \psidos\ (compare with Okikiolu~\cite{Ok:Duke95}).  

\subsection*{Pseudodifferential operators on NC tori}
We refer to Section~\ref{sec:NCtori}, and the references therein, for the main background on NC tori, including the definition of the $L_2$-space $L_2(\T^n_\theta)$. The Sobolev spaces $W^s_2(\T^n_\theta)$, $s\in \R$, are defined in Section~\ref{sec:PsiDOs}. The \emph{smooth noncommutative torus} is 
\begin{equation*}
 C^\infty(\T^n_\theta):=\bigg\{ \sum_{k\in \Z^n}  u_kU^k;\ (u_k)_{k\in \Z^n}\in \cS(\Z^n)\bigg\}, \qquad U^k:=U_1^{k_1}\cdots U_n^{k_n},
\end{equation*}
where $\cS(\Z^n)$ is the space of rapid-decay sequences labeled by $\Z^n$. The algebra $C^\infty(\T^n_\theta)$ is a nuclear Fr\'echet $*$-algebra closed under holomorphic functional calculus, and for $\theta=0$ it recovers the algebra of smooth functions on the ordinary torus $\T^n$. Equivalently, $C^\infty(\T^n_\theta)$ is the space of smooth vectors for the action $\alpha_s(U^k)=e^{is\cdot k}U^k$, $s\in\R^n$, whose infinitesimal generators are the canonical derivations $\delta_1,\ldots,\delta_n$ of $\T^n_\theta$. 

The main background on pseudodifferential operators on NC tori is presented in Section~\ref{sec:PsiDOs}. A standard symbol of order $m\in \R$ is a map $\rho(\xi)\in C^\infty(\T^n_\theta\times\R^n)$ such that $\|\delta^\alpha\partial_\xi^\beta \rho(\xi)\|=O((1+|\xi|)^{m-|\beta|})$ for all multi-orders $\alpha$ and $\beta$, where $\delta^\alpha=\delta_1^{\alpha_1}\cdots\delta_n^{\alpha_n}$. The space of these symbols is denoted $\stS^m(\T^n_\theta\times\R^n)$, and quantization associates to every $\rho(\xi)\in\stS^m(\T^n_\theta\times\R^n)$ the operator 
\begin{equation}
 P_\rho u = (2\pi)^{-n} \iint e^{is\cdot\xi} \rho(\xi) \alpha_{-s}(u)\, ds\, d\xi, \qquad u \in C^\infty(\T^n_\theta),
 \label{eq:Intro.quantization}
\end{equation}
defined as an oscillating integral. Such an operator extends continuously to $P_\rho: W_2^{s+m}(\T^n_\theta)\rightarrow W_2^s(\T^n_\theta)$ for every $s\in\R$. If $m<0$, it lies in the weak Schatten class $\sL_{n|m|^{-1},\infty}$, and it is trace-class if $m<-n$.\footnote{We refer to Appendix~\ref{sec:quasi+Schatten} for background on quasi-Banach spaces and on the Schatten classes $\sL_p$ and their weak versions $\sL_{p,\infty}$.}

A \emph{classical symbol} of order $q\in\C$ admits an asymptotic expansion $\rho(\xi)\sim\sum_{j\geq 0}\rho_{q-j}(\xi)$ into terms homogeneous of degree $q-j$. The homogeneous components depend only on the operator $P=P_\rho$, and $\rho_q(\xi)$ is called the \emph{principal symbol}. The class $\Psi^q(\T^n_\theta)$ consists of the operators $P_\rho$ for such symbols $\rho$. Such an operator is \emph{elliptic} if $\rho_q(\xi)$ is invertible for all $\xi\neq 0$, equivalently if $P$ admits a parametrix in $\Psi^{-q}(\T^n_\theta)$. In this case, $P$ satisfies elliptic regularity ($Pu\in W_2^s\Rightarrow u\in W_2^{s+m}$, $m=\Re q$), and if $m>0$ it is a closed operator on $L_2(\T^n_\theta)$ with domain $W_2^m(\T^n_\theta)$, whose spectrum is either all of $\C$ or a discrete set of isolated eigenvalues of finite multiplicity. 

\subsection*{Pseudodifferential operators with parameter}
The parametric \psido\ calculus of~\cite{LP:Part1} is reviewed in Section~\ref{sec:Parametric-PsiDOs}. Here the parameter sets are \emph{pseudo-cones}, i.e., connected sets $\Lambda\subseteq\C$ that agree with a cone $\Theta\subseteq\C^*$ outside a disk about the origin, and a $\Hol^d(\Lambda)$-family in a locally convex space $\sE$ is a holomorphic map $\Lambda\to\sE$ satisfying $O((1+|\lambda|)^d)$ bounds uniformly on closed sub-pseudo-cones. \emph{Standard parametric symbols} are $\stS^{m,d}(\T^n_\theta\times\R^n\times\Lambda):=\Hol^d(\Lambda;\stS^m(\T^n_\theta\times\R^n))$, while \emph{classical parametric symbols} $\rho(\xi;\lambda)\in S^{m,d}(\T^n_\theta\times\R^n\times\Lambda)$ are defined by asymptotic expansions $\rho(\xi;\lambda)\sim\sum_{j\geq 0}\rho_{m-j}(\xi;\lambda)$ whose terms are jointly homogeneous in $(\xi,\lambda)$ with respect to a fixed weight $w>0$, i.e., $\rho_{m-j}(t\xi;t^w\lambda)=t^{m-j}\rho_{m-j}(\xi;\lambda)$ for all $t>0$. The associated \emph{parametric \psidos} $\Psi^{m,d}(\T^n_\theta;\Lambda)$ are the families $P_\rho(\lambda)$ obtained by quantizing classical parametric symbols as in~(\ref{eq:Intro.quantization}). 

Given an elliptic \psido\ $P$ of order $w>0$, its \emph{elliptic parameter set} is
\begin{equation*}
 \Theta(P) = \C^*\setminus\Bigl[\bigcup_{\xi\neq 0}\Sp(\rho_w(\xi))\Bigr]. 
\end{equation*}
This is an open cone in $\C^*$. If $\Theta(P)\neq\emptyset$, then $P-\lambda$ admits a parametrix in $\Psi^{-w,-1}(\T^n_\theta;\Lambda)$ for any pseudo-cone $\Lambda$ obtained by adjoining a disk to $\Theta(P)$, the spectrum of $P$ is a discrete set of eigenvalues with finite multiplicity, any cone $\Theta'$ with $\overline{\Theta'}\setminus\{0\}\subseteq\Theta(P)$ contains at most finitely many of them, and $\|(P-\lambda)^{-1}\|=O(|\lambda|^{-1})$ as $\lambda\to\infty$ in $\Theta'$. Removing from $\Theta(P)$ every ray through an eigenvalue yields the open cone $\check\Theta(P)$, every ray of which is a ray of minimal growth. One of the main theorems of~\cite{LP:Part1} states that the resolvent $(P-\lambda)^{-1}$ belongs to $\Psi^{-w,-1}(\T^n_\theta;\Lambda(P))$, where $\Lambda(P)$ is a specific pseudo-cone built from $\check\Theta(P)$ (see  Theorem~\ref{thm:Resolvent.resolvent-is-psido-with-parameter} for the precise statement). 

\subsection*{Spectral theory of non-selfadjoint elliptic \psidos}
In Section~\ref{sec:Nonselfadjoint}, we look in more detail at the spectral theory of non-selfadjoint elliptic operators. This is needed for the 
construction of complex powers and sectorial projections. Let $P\in\Psi^q(\T^n_\theta)$ be elliptic with $m=\Re q>0$ and $\Sp(P)\neq\C$. For each $\lambda\in\Sp(P)$, the generalized eigenspace $E_\lambda(P)=\bigcup_{\ell\geq 1}\ker(P-\lambda)^\ell$ is a finite-dimensional subspace of $C^\infty(\T^n_\theta)$ (see Proposition~\ref{prop:Spectral.projection-lambda}). Moreover, the \emph{Riesz projection},
\begin{equation*}
 \Pi_\lambda(P)= \frac{1}{2i\pi} \int_{|\zeta-\lambda|=r} (\zeta-P)^{-1}\,d\zeta,
\end{equation*}
projects onto $E_\lambda(P)$ and has $E_{\bar\lambda}(P^*)^\perp$ as nullspace (\emph{loc.\ cit.}). In addition, this is a smoothing operator (Corollary~\ref{cor:Spectral.Riesz-projection-smoothing}). 

The \emph{partial inverse} $P^{-1}$  is defined as the operator inverting $P$ on $E_0(P^*)^\perp$ and annihilating $E_0(P)$. In particular, 
\begin{equation*}
 PP^{-1}=1-\Pi_0(P) \qquad \text{and} \qquad P^{-1}P=1-\Pi_0(P). 
\end{equation*}
We show that $P^{-1}$ is a classical \psido\ of order $-q$ (Proposition~\ref{prop:Spectral.partial-inverse-basic-properties}). Moreover, the partial inverse of $P^*$ is the adjoint $(P^{-1})^*$, and, for every $j\geq 2$, the partial inverse of $P^j$ is equal to $(P^{-1})^j$ (see Proposition~\ref{prop:Spectral.partial-inverse-adjoint-exponentiation}). 

In addition, if $0\in \Sp(P)$, we give a precise description of the singularity near $\lambda=0$ of the resolvent $(P-\lambda)^{-1}$ (see Proposition~\ref{prop:Resolvent.Meromorphic}).

\subsection*{Holomorphic families of \psidos}
In Section~\ref{sec:hol-PsiDOs}, we give a systematic account of holomorphic families of \psidos\ on NC tori, also considered in~\cite{LNP:TAMS16, Po:SIGMA20}. Given an open set $\Omega\subseteq\C$, a family of (classical) symbols $\rho(z)_{z\in \Omega}$ is \emph{holomorphic} if its order $w(z)$ is holomorphic, $\rho(z)$ is a holomorphic family in $C^\infty(\T^n_\theta\times\R^n)$, and the bounds in the asymptotic expansion $\rho(z)\sim \sum \rho_{w(z)-j}(z)$ are uniform on compact subsets of $\Omega$. A holomorphic family of \psidos\ is one arising by quantization from such a family of symbols. The composition of two such families is again a holomorphic family of \psidos\ (Corollary~\ref{cor:Hol.composition-of-psidos}), and every $P\in \Psi^m(\T^n_\theta)$ embeds into a holomorphic family $P(z)\in \Psi^\bt(\T^n_\theta)$ of order $m+z$ with $P(0)=P$ (Proposition~\ref{prop:Hol.gauging-PsiDos}), called a \emph{holomorphic gauging} of $P$ in the terminology of Guillemin~\cite{Gu:AIM93, Gu:JFA93}.  

We also show that various embeddings of \psidos\ into locally convex spaces or quasi-Banach spaces of operators give rise to holomorphic families. In particular, we get the following result. 

\begin{propositionalph}\label{prop:Intro.hol-fam-properties}
 Let $P(z)$, $z\in \Omega$, be a holomorphic family of \psidos\ of order $w(z)$. 
 \begin{enumerate}
 \item $P(z)$, $z\in \Omega$, is a holomorphic family in  $\sL(C^\infty(\T^n_\theta))$. 
 
 \item If $\Re w(z)\leq m$ on $\Omega$, then  $P(z)$, $z\in \Omega$, is a holomorphic family in $\sL(W_2^{s+m}(\T^n_\theta),W_2^s(\T^n_\theta))$ for all $s\in\R$. 
 
 \item If $\Re w(z)<m$ (resp., $\Re w(z)\leq m$) on $\Omega$ with $m<0$, then $P(z)$, ${z\in \Omega}$,  is a holomorphic family in $\sL_{p}$ (resp., $\sL_{p,\infty}$) with 
$p=n|m|^{-1}$.
\end{enumerate}
\end{propositionalph}

In~\cite{LP:SG2} we will establish full heat trace asymptotics for elliptic \psidos\ (not just differential operators) on NC tori, building on a precise analysis, carried out in~\cite{LP:SG1}, of the singularities of their zeta functions. This requires rapid decay along vertical lines of these zeta functions multiplied by the Gamma function (see, e.g., \cite{GS:JGA96}), for which the key technical tool is a new notion of $\Hol_\infty(\Omega)$-families of \psidos\ introduced in this article: a holomorphic family, of symbols or of \psidos, with values in a locally convex TVS or quasi-Banach space $\sE$, that is uniformly bounded in $\sE$ on every closed vertical half-strip $\Sigma\subseteq\Omega$ (assuming every point of $\Omega$ lies in the interior of some such half-strip contained in $\Omega$). 

All the properties of holomorphic families of \psidos\ above extend to $\Hol_\infty$-families. In particular,  their composition is again a $\Hol_\infty$-family of \psidos\ (Proposition~\ref{prop:unif.composition}), and there are $\Hol_\infty$ analogues of Proposition~\ref{prop:Intro.hol-fam-properties} (see Propositions~\ref{prop:unif.propertiesP(z)-W2s} and~\ref{prop:unif.propertiesP(z)-sLp} for the precise statements). 

\subsection*{Complex powers of elliptic \psidos}
Section~\ref{sec:powers} is devoted to the construction of complex powers of elliptic \psidos\ and the study of their properties. As mentioned above, the construction works for elliptic \psidos, not just differential operators (compare with Shubin~\cite{Sh:Springer01}), and it enables us to get exponential decay along vertical strips for the associated zeta functions multiplied by the Gamma function. 

Let $P\in\Psi^w(\T^n_\theta)$, $w>0$, be an elliptic \psido\ with $\Theta(P)\neq\emptyset$ and $L_\phi=\{\arg\lambda=\phi\}$ a ray in $\check\Theta(P)$. For 
$\lambda\in \C^*\setminus L_\phi$ and $z\in \C$ we define $\lambda_\phi^z$ to be $|\lambda|^ze^{iz\arg_\phi\lambda}$, where 
$\arg_\phi:\C^*\setminus L_\phi\to(\phi-2\pi,\phi)$ is the continuous determination of the argument associated with $\phi$. As in~\cite{Se:PSPM67}, we define 
\begin{equation*}
 P_\phi^z:= \frac{i}{2\pi}\int_{\Gamma} \lambda^z_\phi (P-\lambda)^{-1}d\lambda, \qquad \Re z<0,
\end{equation*}
where $\Gamma$ runs from $\infty$ along $L_\phi$ inward, turns clockwise around a small circle about $0$, and then goes back to $\infty$ along $L_{\phi}$ (see Fig.~\ref{Fig:contour2}).  The integral converges in $\sL(L_2(\T^n_\theta))$, as well as in $\sL(C^\infty(\T^n_\theta))$ thanks to the resolvent estimates of~\cite{LP:Part1}, and $P_\phi^z$, $\Re z<0$, is a holomorphic family of \psidos\ of order $wz$ (Proposition~\ref{prop:powers.param-psido-hol-psidos}) satisfying the semigroup properties,
\begin{equation*}
 P_\phi^{z_1+z_2}=P_\phi^{z_1}P_\phi^{z_2}, \quad \Re z_i<0, \qquad
 P_\phi^{-k}=P^{-k}, \quad k=1,2,\ldots. 
\end{equation*}
The definition of $P_\phi^z$ is extended to any $z\in \C$ by letting $P_\phi^z=P^k P_\phi^{z-k}$, for any non-negative integer $k>\Re z$. This does not depend on $k$, and it defines $P_\phi^z$ directly as a \psido. 

\begin{theoremalph}\label{thm:Intro.complex-powers}
The following hold. 
\begin{enumerate}
 \item $P_\phi^z$, $z\in \C$, is a holomorphic family of \psidos\ of order $wz$. 
 
 \item The principal symbol of $P_\phi^z$ is $\rho_w(\xi)_\phi^z$. 
 
 \item We have the group properties, 
 \begin{gather*}
 P_\phi^0=1-\Pi_0(P), \qquad P_\phi^{z_1+z_2}=P_\phi^{z_1}P_\phi^{z_2}, \quad z_i\in \C,\\ 
 P_\phi^k=\left(1-\Pi_0(P)\right)P^k, \qquad  P_\phi^{-k}=P^{-k}, \qquad k=1,2,\ldots.  
\end{gather*}
\end{enumerate}
\end{theoremalph}

We have explicit formulas for all the homogeneous components of the symbol of $P_\phi^z$ (Eq.~(\ref{eq:Powers.homogeneous-symbols})), and $P_\phi^z$ inherits the Sobolev and Schatten-class mapping properties of Proposition~\ref{prop:Intro.hol-fam-properties} with $w(z)=wz$ (Corollaries~\ref{cor:powers.powers-W2s} and~\ref{cor:powers.powers-sLp}). In particular, they apply to the Laplace--Beltrami operators of~\cite{HP:JGP20} (see Propositions~\ref{prop:Powers.Laplace-Beltrami1} and~\ref{prop:Powers.Laplace-Beltrami2}). They also hold for the absolute values $|P|:=\sqrt{P^*P}$ of elliptic \psidos\ (see Proposition~\ref{prop:powers.absolute-value}). 

Suppose now that $\check\Theta(P)\supseteq\{\alpha\leq\arg\lambda\leq 2\pi-\alpha\}$ for some $\alpha\in(0,\pi)$, and denote by $\alpha_0\in[0,\pi)$ the infimum of all such $\alpha$. In this case, we may take $\phi=\pi$, and to ease notation we set $P^z:=P_\pi^z$. We then establish uniform exponential bounds along vertical strips for the complex powers $P^z$ (this applies in particular to the Laplace--Beltrami operators and to the absolute values above, for which $\alpha_0=0$). We also denote by $\bH_\pm$ the half-planes $\{\pm \Im z>0\}$. 

\begin{theoremalph}
 Given any $\alpha \in (\alpha_0,\pi)$, the following hold. 
 \begin{enumerate}
 \item $e^{\pm i \alpha z}P^z$, $z\in \bH_\pm$, is a $\Hol_\infty$-family of \psidos.
 
 \item Let $\Sigma$ be any closed vertical strip in $\C$, and set $\sigma=\max\{\Re z; z\in \Sigma\}$. 
\begin{enumerate}
 \item 
$P^z\in e^{\alpha |\Im z|}L^\infty\left(\Sigma; \sL\left(W_2^{s+w\sigma}(\T^n_\theta), W_2^{s}(\T^n_\theta)\right)\right)$ for all $s\in \R$.

\item If $\sigma<0$ and $p:=n(w|\sigma|)^{-1}$, then $P^z\in e^{\alpha |\Im z|}L^\infty\left(\Sigma; \sL_{p,\infty}\right)$.
\end{enumerate}
\end{enumerate}
\end{theoremalph}

For applications to the heat equation, we are actually interested in the family $\Gamma(z)P^{-z}$, $\Re z>0$ (see~\cite{LP:SG2}).  We have the following exponential decay results for this family. 

\begin{theoremalph}\label{thm:Intro.rapid-decay-zeta}
 Assume that $\alpha_0<\pi/2$, i.e., $\check{\Theta}(P)$ contains $\{\Re \lambda\leq 0\}\setminus\{0\}$. Let $\Sigma \subseteq \C\setminus \Z_{-}$ be a closed vertical half-strip or strip, and set  $\sigma=\min\{\Re z;\ z\in \Sigma\}$. For all $\epsilon\in (0,\pi/2-\alpha_0)$, the following hold. 
 \begin{enumerate}
 \item $\Gamma(z)P^{-z}\in e^{-\epsilon |\Im z|}   L^\infty\left(\Sigma; \sL\left(W_2^{s}(\T^n_\theta), W_2^{s+w\sigma}(\T^n_\theta)\right)\right)$ for all $s\in \R$. 

\item If $\sigma>0$ and $p:=n(w\sigma)^{-1}$, then $\Gamma(z)P^{-z}$ is in $e^{-\epsilon |\Im z|} L^\infty\left(\Sigma; \sL_{p,\infty}\right)$.

\item If $\sigma>w^{-1}n$, then $\Gamma(z)P^{-z} \in e^{-\epsilon |\Im z|}   L^\infty(\Sigma; \sL_1)$, and so there is $C_{\Sigma\epsilon}>0$ such that 
\begin{equation*}
   \big|\Gamma(z)\Tr\left[P^{-z}\right]\big| \leq C_{\Sigma \epsilon} e^{-\epsilon |\Im z|}\qquad \forall z\in \Sigma.
\end{equation*}
 \end{enumerate}
\end{theoremalph}

More generally, the above results hold \emph{mutatis mutandis} for the families $AP^{z}$ and $\Gamma(z)AP^{-z}$ with $A\in\Psi^q(\T^n_\theta)$, $q\in \R$ (see Proposition~\ref{prop:Powers.APz-hol-infty}, Remark~\ref{rmk:unif.powers-Ws-APz}, and Theorem~\ref{thm:unif.rapid-decay-zeta}).

\subsection*{Logarithms of elliptic \psidos}
As above, let $L_\phi=\{\arg\lambda=\phi\}$ be a ray in $\check\Theta(P)$. By Theorem~\ref{thm:Intro.complex-powers} and Proposition~\ref{prop:Intro.hol-fam-properties}, the family $P_\phi^z$, $z\in \C$, is a holomorphic one-parameter group in $\sL(C^\infty(\T^n_\theta))$. We define the logarithm $\log_\phi(P)$ as the infinitesimal generator of this one-parameter group, i.e., 
\begin{equation*}
 \log_\phi(P):= \left.\frac{d}{dz}\right|_{z=0} P_\phi^z \qquad \text{in}\ \sL(C^\infty(\T^n_\theta)),
\end{equation*}
equivalently,
\begin{equation*}
 \log_\phi(P) = \frac{i}{2\pi}\int_{\Gamma} \log_\phi(\lambda)\,\lambda^{-1}P(P-\lambda)^{-1}\,d\lambda,
\end{equation*}
where $\log_\phi(\lambda)$ is the branch of the logarithm associated with $L_\phi$. This integral converges in $\sL(C^\infty(\T^n_\theta))$, and in fact in $\sL(W^s_2(\T^n_\theta), W^t_2(\T^n_\theta))$ for all $s>t$ (Proposition~\ref{prop:Log.properties} and Remark~\ref{rmk:Log.convergenceWs}). 

If $P$ is the flat Laplacian $\Delta$ and $\phi\in (0,2\pi)$, then $\log_\phi(P)$ is just the logarithm $\log(\Delta)$ obtained by functional calculus for $\car_{(0,\infty)}(t)\log(t)$, i.e., the \psido\ with symbol $(1-\chi(\xi))\log(|\xi|^2)\sim 2\log|\xi|$, where $\chi(\xi)$ is a suitable cut-off function around $\xi=0$. In general, we have the following result. 

\begin{theoremalph}
 We may write
 \begin{equation}\label{eq:Intro.logP-PsiDO}
 \log_\phi (P)= \frac{1}{2}w \log(\Delta) + Q_\phi, 
\end{equation}
where $Q_\phi$ is a \psido\ of order $0$ whose principal symbol is $\omega_0(\xi)=\log_\phi[\rho_w(\xi|\xi|^{-1})]$.
\end{theoremalph}

We have explicit formulas for the other homogeneous components $\omega_{-j}(\xi)$, $j\geq 1$, of the symbol of $Q_\phi$ (Eq.~(\ref{eq:Log.symbols2})), and $\log_\phi(P)$ is a \psido\ with symbol $\tilde{\omega}(\xi)\sim w\log|\xi| + \omega_0(\xi)+\omega_{-1}(\xi)+\cdots$. It follows that $\log_\phi(P)$ is closable as an operator on $L_2(\T^n_\theta)$, and its closure has the same domain as $\log(\Delta)$ (Remark~\ref{rmk:Log.domain}). 

Finally, since $P_\phi^z$, $\Re z<-w^{-1}n$, is a holomorphic family of trace-class operators by Theorem~\ref{thm:Intro.complex-powers} and Proposition~\ref{prop:Intro.hol-fam-properties}, and the operators $\log_\phi(P)P_\phi^z$ are trace-class as well, we obtain the differentiation formula,
\begin{equation}\label{eq:Intro.trace-formula-log}
 \frac{d}{dz}\Tr\left[P_\phi^z\right]=\Tr\left[\log_\phi(P)P_\phi^z\right], \qquad \Re z<-w^{-1}n.
\end{equation}

\subsection*{Sectorial projections}
To study how the powers $P_\phi^z$ and the logarithm $\log_\phi(P)$ depend on the spectral cut $L_\phi$, we introduce, as in previous approaches (see, e.g.,~\cite{GG:MS08,Ok:Duke95,Wo:IM82, Wo:PhD, Wo:Weyl}), sectorial projections, first considered by Burak~\cite{Bu:Pisa70}.

Let  $L_{\phi'}=\{\arg \lambda=\phi'\}$ be another ray in $\check\Theta(P)$ with $\phi<\phi'\leq\phi+2\pi$, and let $S_{\phi,\phi'}:=\{\phi<\arg \lambda <\phi'\}$ be the angular sector between $L_\phi$ and $L_{\phi'}$. Define
\begin{equation}\label{eq:Intro.sectorial-projection}
 \Pi_{\phi,\phi'}(P):= \frac{i}{2\pi}\int_{\Gamma'}\lambda^{-1}P(P-\lambda)^{-1}d\lambda, 
\end{equation}
where $\Gamma'$ runs from $\infty$ along $L_{\phi'}$ inward, turns clockwise along an arc of a small circle about $0$ inside the sector $S_{\phi,\phi'}$, and goes back to $\infty$ along $L_{\phi}$ (see Fig.~\ref{fig:sectorial-contour}). Here $P(P-\lambda)^{-1}$ is a family in $\Psi^{0,-1}(\T^n_\theta; \Lambda(P))$, and hence in $\Hol^{-1}(\Lambda(P);\sL(C^\infty(\T^n_\theta)))$, so the integral converges to a continuous operator on $C^\infty(\T^n_\theta)$. Moreover, we have the following result. 

\begin{theoremalph}
 The following hold. 
 \begin{enumerate}
 \item $\Pi_{\phi,\phi'}(P)$ is a \psido\ of order $0$ whose principal symbol is $\Pi_{\phi,\phi'}(\rho_w(\xi))$.\footnote{For every $\xi\neq 0$ this is the sectorial projection of $\rho_w(\xi)$ onto the sector $S_{\phi,\phi'}$. See Remark~\ref{rmk:sectorial.principal-symbol}.}

 \item $\Pi_{\phi,\phi'}(P)$ is a smoothing operator if and only if $S_{\phi,\phi'}\subseteq \Theta(P)$. 
\end{enumerate}
\end{theoremalph}

We refer to~Eq.~(\ref{eq:Powers.homogeneous-symbols-sectorial}) for explicit formulas for the negative-order homogeneous components of the symbol of  $\Pi_{\phi,\phi'}(P)$. 

The above result implies that  $\Pi_{\phi,\phi'}(P)$ is a bounded operator on $L_2(\T^n_\theta)$. Standard arguments further show that 
$\Pi_{\phi,\phi'}(P)$ is a projection such that
\begin{itemize}
 \item its range is a closed subspace containing all the generalized eigenspaces $E_\lambda(P)$ associated with eigenvalues in the sector $S_{\phi,\phi'}$,
 
 \item its nullspace is a closed subspace containing all the generalized eigenspaces $E_\lambda(P)$ associated with eigenvalues that are not contained in $S_{\phi,\phi'}$.
\end{itemize}
Because of this, $\Pi_{\phi,\phi'}(P)$ is called the \emph{sectorial projection} of $P$ onto the sector $S_{\phi,\phi'}$. 

In addition, we have the following asymmetry formulas, which show how the complex powers $P_\phi^z$, $z\in \C$, and the logarithm $\log_\phi(P)$ depend on the spectral cut $L_\phi$. These formulas extend to NC tori the asymmetry formulas of Wodzicki~\cite{Wo:IM82, Wo:PhD, Wo:Weyl} for complex powers (see also~\cite{Po:IJM06}), and the asymmetry formulas for logarithms of Okikiolu~\cite{Ok:Duke95} and Gaarde--Grubb~\cite{GG:MS08}. 

\begin{theoremalph}
 We have 
\begin{gather} \label{eq:Intro.powers.asymmetry} 
 P_\phi^z-P_{\phi'}^z = \left(1-e^{2i\pi z}\right) \Pi_{\phi,\phi'}(P)P_\phi^z,  \quad z\in \C,\\
 \log_{\phi'}(P) -  \log_{\phi}(P)= (2i\pi) \Pi_{\phi,\phi'}(P). 
 \label{eq:Intro.Log.difference}  
\end{gather}
\end{theoremalph}

\subsection*{Prospective applications}
The results of this paper provide the analytic foundation for several forthcoming developments, to be taken up in~\cite{LP:SG1,LP:SG2,LP:SG3} and elsewhere: a comprehensive account of zeta functions and $\zeta$-regularized determinants of elliptic \psidos\ on NC tori, together with local and microlocal Weyl laws (\cite{LP:SG1}), short-time asymptotics for heat semigroup traces $\Tr[Ae^{-tP}]$, with $A\in\Psi^q(\T^n_\theta)$, $q\in\R$, and $P$ elliptic of order $w>0$ with $\Theta(P)\supseteq\{\Re\lambda\leq 0\}\setminus\{0\}$, together with variational formulas for zeta functions, determinants, and heat coefficients of families of elliptic \psidos\ (\cite{LP:SG2,LP:SG3}), an extension of the Gauss--Bonnet theorem of Connes--Tretkoff~\cite{CT:Baltimore11} to NC tori of arbitrary dimension and Riemannian metric (\cite{LP:SG3}), and Weyl laws for negative-order \psidos\ on NC tori extending those of Birman--Solomyak~\cite{BS:VLU77,BS:SMJ79} and conjectured in~\cite{MP:AIM23}, together with the semiclassical Weyl laws for curved NC tori also conjectured there, and a reinterpretation, in these terms, of the spectral asymptotics for quantized derivatives on NC tori of Sukochev--Xiong--Zanin~\cite{SXZ:JFA23}. 

\subsection*{Organization of the paper}
Section~\ref{sec:NCtori} reviews noncommutative tori, and Section~\ref{sec:PsiDOs} the pseudodifferential calculus on them~\cite{Ba:CRAS88,Co:CRAS80,HLP:Part1,HLP:Part2}. Section~\ref{sec:Parametric-PsiDOs} recalls the parametric \psido\ calculus of~\cite{LP:Part1}, including the resolvent of an elliptic \psido\ as an element of $\Psi^{-w,-1}(\T^n_\theta;\Lambda(P))$. Section~\ref{sec:Nonselfadjoint} develops the spectral theory of non-selfadjoint elliptic \psidos\ needed throughout the paper: generalized eigenspaces, Riesz projections, partial inverses, and the singularity structure of the resolvent at $\lambda=0$. Section~\ref{sec:hol-PsiDOs} introduces holomorphic and $\Hol_\infty$-families of symbols and \psidos\ on NC tori, the technical backbone for the constructions that follow. Sections~\ref{sec:powers}--\ref{sec:sectorial} then carry out the paper's main constructions, in the order described above: complex powers, logarithms, and sectorial projections of elliptic \psidos\ on NC tori, culminating in the trace formula~(\ref{eq:Intro.trace-formula-log}) and the asymmetry formulas~(\ref{eq:Intro.powers.asymmetry})--(\ref{eq:Intro.Log.difference}). Finally, Appendix~\ref{sec:quasi+Schatten} collects the background on quasi-Banach spaces and Schatten classes used throughout the paper.

\subsection*{Acknowledgements} The first-named author acknowledges the support of BK21 PLUS SNU, Mathematical Sciences Division (South Korea). The research of the second-named author was partially supported by NSFC grant No.~11971328 (China). 

\section{Noncommutative Tori} \label{sec:NCtori}
We review the main definitions and properties of noncommutative tori. See~\cite{Co:NCG, HLP:Part1, Ri:PJM81, Ri:CM90} for fuller accounts.

Throughout this paper, $\theta =(\theta_{jk})$ denotes a real anti-symmetric $n\times n$ matrix ($n\geq 2$). The \emph{noncommutative torus} $C(\T^n_\theta)$ is the $C^*$-algebra generated by the unitary operators $U_1, \ldots, U_n$ subject to the relations, 
\begin{equation*}
 U_kU_j = e^{2i\pi \theta_{jk}} U_jU_k, \qquad j,k=1, \ldots, n. 
\end{equation*}
For $\theta=0$ we obtain the $C^*$-algebra $C(\T^n)$ of continuous functions on the ordinary torus $\T^n=\R^n\slash 2\pi \Z^n$. Throughout, setting $\theta=0$ likewise recovers the corresponding classical objects on $\T^n$ (the $L_2$-space and its Fourier decomposition, the smooth subalgebra, and the derivations $\delta_j$, which reduce to $D_{x_j}=\frac1i\frac{\partial}{\partial x_j}$). For $k=(k_1,\ldots, k_n)\in \Z^n$, set
 \begin{equation*}
 U^k:=U_1^{k_1} \cdots U_n^{k_n}.
\end{equation*}
These unitaries span a dense $*$-subalgebra of $C(\T^n_\theta)$.

Let $\tau:C(\T^n_\theta)\rightarrow \C$ be the standard trace defined by 
 \begin{equation*} 
 \tau\left( U^k\right) = \left\{
\begin{array}{ll}
 1 &  \text{if $k=0$},  \\
 0 &  \text{otherwise.}
 \end{array}\right. 
\end{equation*}
This is a faithful positive linear trace on $C(\T^n_\theta)$. We thus define an inner product on $C(\T^n_\theta)$ by
\begin{equation}
 \acoup{u}{v} = \tau\left( uv^* \right), \qquad u,v\in C(\T_\theta^n). 
 \label{eq:NCtori.cAtheta-innerproduct}
\end{equation}
By definition, $L_2(\T^n_\theta)$ is the Hilbert space obtained as the completion of $C(\T^n_\theta)$ with respect to the inner product~(\ref{eq:NCtori.cAtheta-innerproduct}). 

The family $\{ U^k; k \in \Z^n\}$  is an orthonormal basis of $L_2(\T^n_\theta)$. Thus, every $u\in L_2(\T^n_\theta)$ can be uniquely written as 
\begin{equation} \label{eq:NCtori.Fourier-series-u}
 u =\sum_{k \in \Z^n} u_k U^k, \qquad u_k=\acoup{u}{U^k}, 
\end{equation}
where the series converges in $L_2(\T^n_\theta)$. By analogy with the case $\theta=0$, we call the series $\sum_{k \in \Z^n} u_k U^k$ in~(\ref{eq:NCtori.Fourier-series-u}) the Fourier series of $u\in L_2(\T^n_\theta)$. In what follows, we denote by $\|\cdot\|_0$ the norm of $L_2(\T^n_\theta)$, to distinguish it from the norm $\|\cdot\|$ of $C(\T^n_\theta)$.

The multiplication of $C(\T^n_\theta)$ uniquely extends to a continuous bilinear map $C(\T_\theta^n)\times L_2(\T^n_\theta) \rightarrow L_2(\T^n_\theta)$. This provides us with a unital $*$-representation of $C(\T^n_\theta)$ (this is actually the GNS representation defined by $\tau$). As $\tau$ is faithful, this representation is injective, and hence isometric, i.e.,
          \begin{equation*} 
                      \left\| u \right\| = \sup_{\|v\|_0=1} \|uv\|_0 \qquad \forall u \in C(\T^n_\theta). 
           \end{equation*}
This allows us to identify any $u \in C(\T^n_\theta)$ with the sum of its Fourier series in $L_2(\T^n_\theta)$. In general, this Fourier series need not converge in $C(\T^n_\theta)$.

There is a natural strongly continuous action $(s,u)\rightarrow \alpha_s(u)$ of $\R^n$ on $C(\T^n_\theta)$ by $*$-automorphisms such that
\begin{equation*}
\alpha_s(U^k)= e^{is\cdot k} U^k \qquad  \text{for all $k\in \Z^n$ and $s\in \R^n$}. 
\end{equation*}
The \emph{smooth noncommutative torus} is the algebra of smooth vectors for this action, 
\begin{equation*}
 C^\infty(\T^n_\theta):=\left\{ u \in C(\T^n_\theta); \ \alpha_s(u) \in C^\infty\left(\R^n; C(\T^n_\theta)\right)\right\}. 
\end{equation*}
Let $\cS(\Z^n)$ be the space of rapid-decay sequences with complex entries. In terms of the Fourier series decomposition~(\ref{eq:NCtori.Fourier-series-u}) we have
\begin{equation*}
 C^\infty(\T^n_\theta)=\bigg\{ u=\sum_{k\in \Z^n} u_k U^k;\ (u_k)_{k\in \Z^n}\in \cS(\Z^n)\bigg\}. 
\end{equation*}

For $j=1,\ldots, n$, let $\delta_j:C^\infty(\T^n_\theta) \rightarrow C^\infty(\T^n_\theta) $ be the  derivation defined by 
\begin{equation*}
 \delta_j(u) = D_{s_j} \alpha_s(u)|_{s=0}, \qquad u\in C^\infty(\T^n_\theta) , 
\end{equation*}
where we have set $D_{s_j}=\frac{1}{i}\partial_{s_j}$. In particular, for $j,l=1,\ldots, n$, we have
\begin{equation*}
 \delta_j(U_l) = \left\{ 
 \begin{array}{ll}
 U_j & \text{if $l=j$},\\
 0 & \text{if $l\neq j$}. 
\end{array}\right.
\end{equation*}
More generally, for $\beta \in \N_0^n$ we set
\begin{equation*}
 \delta^\beta(u) = D_s^\beta \alpha_s(u)|_{s=0} = \delta_1^{\beta_1} \cdots \delta_n^{\beta_n}(u). 
\end{equation*}
We endow $C^\infty(\T^n_\theta)$ with the locally convex topology defined by the semi-norms
\begin{equation*}
 C^\infty(\T^n_\theta) \ni u \longrightarrow \left\|\delta^\beta (u)\right\| ,  \qquad \beta\in \N_0^n. 
\end{equation*}
With the involution inherited from $C(\T^n_\theta)$, this topology turns $C^\infty(\T^n_\theta)$ into a unital Fr\'echet $*$-algebra. It is in fact a nuclear Fr\'echet--Montel space~\cite{HLP:Part1}. Moreover, it is closed under holomorphic functional calculus (see, e.g., \cite{Co:AdvM81, HLP:Part1}). This implies that an element $u\in C^\infty(\T^n_\theta)$ is invertible in $C^\infty(\T^n_\theta)$ if and only if it is invertible in $C(\T^n_\theta)$. 

Let $\scD'(\T^n_\theta)$ be the topological dual of $C^\infty(\T^n_\theta)$, equipped with its strong topology. We regard its elements as analogues of distributions on $\T^n$.
 
Any $u \in C^\infty(\T^n_\theta) $ defines a continuous linear form on $C^\infty(\T^n_\theta)$ by $\acou{u}{v} =\tau(uv)$. Note that for all $u,v \in C^\infty(\T^n_\theta) $,
 \begin{equation} \label{eq:NCtori.distrb-innerproduct-eq}
  \acou{u}{v} =\acoup{v}{u^*}=\acoup{u}{v^*}.
\end{equation}
 The map $u \rightarrow \acou{u}{\cdot}$ embeds $C^\infty(\T^n_\theta)$ continuously into $\scD'(\T^n_\theta)$, and by~(\ref{eq:NCtori.distrb-innerproduct-eq}) this embedding extends to a continuous embedding of $L_2(\T^n_\theta)$ into $\scD'(\T^n_\theta)$. Given any $u\in \scD'(\T^n_\theta) $, its Fourier series is
\begin{equation*}
 \sum_{k\in \Z^n} u_k U^k, \qquad \text{where}\ u_k:= \acou{u}{(U^k)^*}. 
\end{equation*}
Here the unitaries $U^k$, $k\in \Z^n$, are regarded as elements of $\scD'(\T^n_\theta)$. It can be shown that every $u\in \scD'(\T^n_\theta) $ is the sum of its Fourier series in $\scD'(\T^n_\theta)$ (see~\cite{HLP:Part1}). 

\section{Pseudodifferential Operators on Noncommutative Tori} \label{sec:PsiDOs}
We recall the main definitions and properties of the pseudodifferential calculus on noncommutative tori~\cite{Ba:CRAS88, Co:CRAS80, HLP:Part1, HLP:Part2}, following~\cite{HLP:Part1, HLP:Part2}.

\subsection{Symbols}
For an open set $U\subseteq \R^{n'}$ and a locally convex space $\sE$, we define $C^\infty(U;\sE)$ as in~\cite{HLP:Part1}. This is a Fr\'echet space whenever $\sE$ is Fr\'echet. For $\sE=C^\infty(\T^n_\theta)$ we write $C^\infty(\T^n_\theta\times U)$, which is a Fr\'echet algebra.

\begin{definition}[Standard Symbols, see~\cite{Ba:CRAS88, Co:CRAS80}]
$\stS^m (\T^n_\theta\times\Rn)$, $m\in\R$, consists of maps $\rho(\xi)\in C^\infty (\T^n_\theta\times\Rn)$ such that, for all multi-orders $\alpha$ and $\beta$, there exists $C_{\alpha \beta} > 0$ such that
\begin{equation*} 
\norm{\delta^\alpha \partial_\xi^\beta \rho(\xi)} \leq C_{\alpha \beta} \left( 1 + | \xi | \right)^{m - | \beta |} \qquad \forall \xi \in \R^n .
\end{equation*}
Equipped with the semi-norms, 
\begin{equation*}
p_N^{(m)}(\rho):=\sup_{|\alpha|+|\beta|\leq N} \sup_{\xi\in\Rn}(1+|\xi|)^{-m+|\beta|}\norm{\delta^\alpha\partial_\xi^\beta\rho(\xi)}, \qquad N\in\N_0 ,
\end{equation*}
$\stS^m (\T^n_\theta\times\Rn)$ is a Fr\'echet space~\cite{HLP:Part1}. 
\end{definition}

\begin{remark}
 Let $(m_j)_{j\geq 0}\subseteq \R$ be a decreasing sequence such that $m_j\rightarrow -\infty$. Given $\rho(\xi)$ in $C^\infty(\T^n_\theta\times\R^n)$ and $\rho_j(\xi)\in \stS^{m_j}(\T^n_\theta\times\R^n)$, $j\geq 0$, we shall write $\rho(\xi) \sim \sum_{j \geq 0} \rho_j(\xi)$ when 
\begin{equation}
 \rho(\xi)-\sum_{j<N} \rho_j(\xi)\in \stS^{m_N}(\T^n_\theta\times\R^n) \qquad \forall N\geq 0.
 \label{eq:PsiDOs.asymptotic-expansion-standard} 
\end{equation}
Note that this implies that $\rho(\xi)\in \stS^{m_0}(\T^n_\theta\times\R^n)$. 
\end{remark}
\begin{definition} 
  $\stS^{-\infty}(\T^n_\theta\times\R^n)$ consists of maps $\rho(\xi)\in C^\infty (\T^n_\theta\times\Rn)$ such that, for all  $N\geq 0$ and multi-orders $\alpha$, $\beta$, there exists $C_{N\alpha \beta} > 0$ such that
\begin{equation*} 
\norm{\delta^\alpha \partial_\xi^\beta \rho(\xi)} \leq C_{N\alpha \beta} \left( 1 + | \xi | \right)^{-N} \qquad \forall \xi \in \R^n .
\end{equation*}
Equivalently, $\stS^{-\infty}(\T^n_\theta\times\R^n)=\bigcap_{m\in\R}\stS^m(\T^n_\theta\times\R^n)$.
\end{definition}

\begin{definition}[Homogeneous Symbols] 
$S_q (\T^n_\theta\times\R^n)$, $q \in \C$, consists of  $\rho(\xi) \in 
C^\infty(\T_\theta^n\times(\R^n\setminus 0))$ that are homogeneous of degree $q$, i.e., 
$\rho( t \xi ) =t^q \rho(\xi)$ for all $\xi \in \R^n \setminus 0$ and $t > 0$. If $\chi(\xi)\in C^\infty_c(\R^n)$ is such that $\chi(\xi)=1$ near $\xi=0$, then $(1-\chi(\xi))\rho(\xi)\in \stS^{\Re q}(\T^n_\theta\times\R^n)$. 
\end{definition}

\begin{definition}[Classical Symbols, see \cite{Ba:CRAS88}]\label{def:Symbols.classicalsymbols}
$S^q (\T^n_\theta\times\R^n)$, $q \in \C$, consists of  $\rho(\xi)\in C^\infty(\T^n_\theta\times\R^n)$ that admit an asymptotic expansion,
\begin{equation*}
\rho(\xi) \sim \sum_{j \geq 0} \rho_{q-j} (\xi),  \qquad \rho_{q-j} \in S_{q-j} (\T^n_\theta\times\R^n), 
\end{equation*}
where $\sim$ means that, for all $N\geq 0$ and multi-orders $\alpha$, $\beta$, there exists $C_{N\alpha\beta} >0$ such that, for all $\xi \in \R^n$ with $| \xi | \geq 1$, we have
\begin{equation} \label{eq:Symbols.classical-estimates}
\Big\| \delta^\alpha \partial_\xi^\beta \big( \rho - \sum_{j<N} \rho_{q-j} \big)(\xi)\Big\| \leq C_{N\alpha\beta} | \xi |^{\Re{q}-N-| \beta |} .
\end{equation}
\end{definition}

\begin{remark} \label{rmk:Symbols.classical-inclusion}
The conditions~(\ref{eq:Symbols.classical-estimates}) are equivalent to requiring that, for all $N$, $\alpha$, $\beta$, as soon as $J$ is large enough there is $C_{NJ\alpha\beta}>0$ such that, for all $\xi\in \R^n$, $|\xi|\geq 1$, we have
\begin{equation} \label{eq:Symbols.classical-estimates-qualitative}
\Big\| \delta^\alpha \partial_\xi^\beta \big( \rho - \sum_{j<J} \rho_{q-j} \big)(\xi)\Big\| \leq C_{NJ\alpha\beta} | \xi |^{-N}.
\end{equation}
Moreover, if $\chi(\xi)\in C^\infty_c(\R^n)$ is such that $\chi=1$ near $\xi=0$, then $\rho(\xi)\sim \sum_{j\geq 0} \rho_{q-j}(\xi)$ in the sense of~(\ref{eq:Symbols.classical-estimates}) if and only if $\rho(\xi)\sim \sum_{j\geq 0} (1-\chi(\xi))\rho_{q-j}(\xi)$ in the sense of~(\ref{eq:PsiDOs.asymptotic-expansion-standard}). In particular, $S^q(\T^n_\theta\times\R^n)\subseteq \stS^{\Re{q}}(\T^n_\theta\times\R^n)$. 
\end{remark}

\begin{example}
 Any polynomial map $\rho(\xi)=\sum_{|\alpha|\leq m} a_\alpha \xi^\alpha$, $a_\alpha\in C^\infty(\T^n_\theta) $, $m\in \N_0$, is in $S^m(\T^n_\theta\times\R^n)$. Moreover, setting $\brak{\xi}=(1+|\xi|^2)^{\frac12}$, $\xi\in \R^n$, we have $\brak{\xi}^q \in S^q(\T^n_\theta\times\R^n)$ for all $q\in \C$ (see, e.g., \cite{HLP:Part1}). 
\end{example}

\subsection{Pseudodifferential operators}
Let $\rho(\xi)\in \stS^m(\T^n_\theta\times\R^n)$, $m\in \R$. For any $u\in C^\infty(\T^n_\theta) $, the map $\rho(\xi)\alpha_{-s}(u)$ is an amplitude in $A^{m_+}(\T^n_\theta\times\R^{n}\times\R^{n})$, $m_+=\op{max}(m,0)$~\cite{HLP:Part1}, so we may define
\begin{equation*}
P_\rho u = (2\pi)^{-n}\iint e^{is\cdot\xi}\rho(\xi)\alpha_{-s}(u)\,ds\,d\xi,
\end{equation*}
where the above integral is meant as an oscillating integral (see~\cite{HLP:Part1}). This defines a continuous linear operator $P_\rho:C^\infty(\T^n_\theta)  \rightarrow C^\infty(\T^n_\theta) $~\cite{HLP:Part1}. Equivalently~\cite{CT:Baltimore11, HLP:Part1}, in terms of the Fourier decomposition~(\ref{eq:NCtori.Fourier-series-u}),
\begin{equation}
 P_\rho u = \sum_{k\in \Z^n} u_k\rho(k)U^k \qquad \text{for all $u=\sum_{k\in \Z^n}u_kU^k\in C^\infty(\T^n_\theta) $}. 
 \label{eq:PsiDOs.toroidal-def} 
\end{equation}
We call $\rho(\xi)$ a \emph{symbol} of $P_\rho$. 

\begin{remark}\label{rmk:PsiDOs.uniqueness-symbol}
 The symbol is not unique. However, if $P=P_\rho$ with $\rho(\xi)\in \stS^m(\T^n_\theta\times \R^n)$, then~(\ref{eq:PsiDOs.toroidal-def}) gives $\rho(k)=P\left(U^k\right) \left(U^k\right)^{-1}$ for all $k\in \Z^n$, so the restriction of $\rho$ to $\Z^n$ is uniquely determined by $P$. Moreover, any two symbols of $P$ in $\stS^m(\T^n_\theta\times \R^n)$ differ by an element of $\stS^{-\infty}(\T^n_\theta\times \R^n)$~\cite{HLP:Part1}.
\end{remark}

\begin{definition}
$\Psi^q(\T^n_\theta)$,  $q\in \C$, consists of all linear operators $P_\rho:C^\infty(\T^n_\theta) \rightarrow C^\infty(\T^n_\theta) $ with $\rho(\xi)$ in $S^q(\T^n_\theta\times\R^n)$.
\end{definition}

\begin{remark} \label{rem:PsiDOs.symbol-uniqueness}
If $P=P_\rho\in\Psi^q(\T^n_\theta)$, $\rho(\xi)\sim \sum \rho_{q-j}(\xi)$, then by Remark~\ref{rmk:PsiDOs.uniqueness-symbol} the homogeneous components $\rho_{q-j}(\xi)$ are uniquely determined by $P$. The leading term $\rho_q(\xi)$ is the \emph{principal symbol} of $P$. 
\end{remark}

\begin{example}
 A differential operator on $C^\infty(\T^n_\theta)$ is of the form $P=\sum_{|\alpha|\leq m}a_\alpha\delta^\alpha$, $a_\alpha\in C^\infty(\T^n_\theta)$ (see~\cite{Co:CRAS80, Co:NCG}). This is a \psido\ of order $m$ with symbol $\rho(\xi)= \sum a_\alpha \xi^\alpha$ (see~\cite{HLP:Part1}).
\end{example}

\begin{example} \label{eq:PsiDOs.Lambda-to-the-s-PsiDO}
 The  Laplacian $\Delta:= \delta_1^2 + \cdots + \delta_n^2$ is a selfadjoint unbounded operator on $L_2(\T^n_\theta)$ with domain 
 $\op{Dom}(\Delta):=\{ u=\sum_{k\in \Z^n} u_kU^k\in L_2(\T^n_\theta); \ \sum_{k\in \Z^n} |k|^4 |u_k|^2<\infty\}$. It is isospectral to the Laplacian on $\T^n$. Set $\Lambda^s=(1+\Delta)^{\frac{s}2}$, $s\in \C$. Then $\Lambda^s$ is the \psido\ with symbol $\brak{\xi}^s$, and so $\Lambda^s\in \Psi^s(\T^n_\theta)$ (see~\cite{HLP:Part1}). In fact, the family $\Lambda^s$, $s\in \C$, forms a one-parameter group of \psidos. 
\end{example}

We refer to~\cite{HLP:Part1, LNP:TAMS16} for an equivalent description of \psidos\ on $\T^n_\theta$ in terms of discrete (also called toroidal) symbols, defined solely on $\Z^n$. 

\begin{definition}
 $\Psi^{-\infty}(\T^n_\theta)$ consists of all linear operators $P_\rho:C^\infty(\T^n_\theta) \rightarrow C^\infty(\T^n_\theta) $ with $\rho(\xi)$ in $\bS^{-\infty}(\T^n_\theta\times\R^n)$. Equivalently, $\Psi^{-\infty}(\T^n_\theta)  = \bigcap_{q\in \C} \Psi^q(\T^n_\theta)$. 
\end{definition}

\subsection{Composition of \psidos}
Suppose we are given symbols $\rho_1(\xi)\in\stS^{m_1}(\T^n_\theta\times\Rn)$, $m_1\in\R$, and $\rho_2(\xi)\in\stS^{m_2}(\T^n_\theta\times\Rn)$, $m_2\in\R$. As $P_{\rho_1}$ and $P_{\rho_2}$ are linear operators on $C^\infty(\T^n_\theta)$, the composition $P_{\rho_1}P_{\rho_2}$ makes sense as such an operator.  
In addition, we define the map $\rho_1\sharp\rho_2:\Rn \rightarrow C^\infty(\T^n_\theta) $ by
\begin{equation} \label{eq:Composition.symbol-sharp}
\rho_1\sharp\rho_2(\xi) = (2\pi)^{-n}\iint e^{it\cdot\eta}\rho_1(\xi+\eta)\alpha_{-t}[\rho_2(\xi)]\,dt\,d\eta , \qquad \xi\in\Rn ,
\end{equation}
where the integral is meant as an oscillating integral (see~\cite{Ba:CRAS88, HLP:Part2}). 

\begin{lemma}[see \cite{Ba:CRAS88, Co:CRAS80, HLP:Part2}] \label{prop:Composition.sharp-continuity-standard-symbol}
Let $\rho_1(\xi)\in \stS^{m_1}(\T^n_\theta\times\R^n)$ and  $\rho_2(\xi)\in \stS^{m_2}(\T^n_\theta\times\R^n)$, $m_j\in \R$. 
\begin{enumerate}
 \item $\rho_1\sharp\rho_2(\xi)\in \stS^{m_1+m_2}(\T^n_\theta\times\Rn)$, and we have $ \rho_1\sharp\rho_2(\xi) \sim \sum\frac{1}{\alpha !}\partial_\xi^\alpha\rho_1(\xi)\delta^\alpha\rho_2(\xi)$. 

 \item We have $P_{\rho_1}P_{\rho_2}=P_{\rho_1\sharp \rho_2}$.  
 \end{enumerate}
\end{lemma}

For classical symbols this refines as follows. 

\begin{proposition}[see \cite{Ba:CRAS88, Co:CRAS80, HLP:Part2}] \label{prop:Composition.composition-PsiDOs}
Let  $P_1\in \Psi^{q_1}(\T^n_\theta)$, $q_1\in \C$, have symbol $\rho_1(\xi)\sim\sum_{j\geq 0}\rho_{1,q_1-j}(\xi)$, and let  $P_2\in \Psi^{q_2}(\T^n_\theta)$, $q_2\in \C$, have symbol $\rho_2(\xi)\sim\sum_{j\geq 0}\rho_{2,q_2-j}(\xi)$.
\begin{enumerate}
 \item $\rho_1 \sharp \rho_2(\xi)\in S^{q_1+q_2}(\T^n_\theta\times\R^n)$ with $\rho_1\sharp\rho_2(\xi) \sim\sum (\rho_1\sharp\rho_2)_{q_1+q_2-j}(\xi)$, where 
 \begin{equation*}
(\rho_1\sharp\rho_2)_{q_1+q_2-j}(\xi)=\sum_{k+l+|\alpha|=j}\frac{1}{\alpha !}\partial_\xi^\alpha \rho_{1,q_1-k}(\xi)\delta^\alpha \rho_{2,q_2-l}(\xi), \qquad j\geq 0.
\end{equation*}

\item The composition $P_1P_2=P_{\rho_1\sharp \rho_2}$ is contained in $\Psi^{q_1+q_2}(\T^n_\theta)$.  
\end{enumerate}
\end{proposition}

\subsection{Action on $\scD'(\T^n_\theta)$}
Given a linear operator $P: C^\infty(\T^n_\theta)  \rightarrow C^\infty(\T^n_\theta) $, a formal adjoint is any linear operator $P^*: C^\infty(\T^n_\theta)  \rightarrow C^\infty(\T^n_\theta) $ such that
\begin{equation*} 
\acoup{Pu}{v} = \acoup{u}{P^* v} \qquad \forall u,v \in C^\infty(\T^n_\theta) ,
\end{equation*}
where $\acoup{\cdot}{\cdot}$ is the inner product~(\ref{eq:NCtori.cAtheta-innerproduct}). A formal adjoint, when it exists, is unique.

It can be shown (see~\cite{Ba:CRAS88, HLP:Part2}) that if $\rho\in \stS^m(\T^n_\theta\times\Rn)$, $m\in\R$, then $P_\rho$ has a formal adjoint and there is an explicit symbol $\rho^\star(\xi)\in  \stS^m(\T^n_\theta\times\Rn)$ such that $P_\rho^*=P_{\rho^\star}$, and so $P_\rho^*$ is a \psido. In particular, if $P\in \Psi^q(\T^n_\theta)$, $q\in \C$, has principal symbol $\rho_q(\xi)$, then $P^*$ is in $\Psi^{\bar{q}}(\T^n_\theta)$ and its principal symbol is $\rho_q(\xi)^*$. 

As an application, we get the following extension result. 

\begin{proposition}[see \cite{HLP:Part2}] \label{Adjoints.PsiDOs-extension}
 Let $\rho(\xi)\in\stS^m(\T^n_\theta\times\Rn)$, $m\in\R$. Then $P_\rho$ uniquely extends to a continuous linear operator $P_\rho:\scD'(\T^n_\theta) \rightarrow \scD'(\T^n_\theta) $.
\end{proposition}

\subsection{Sobolev spaces}
 Let $\Delta=\delta_1^2+\cdots+\delta_n^2$ be the Laplacian on $C^\infty(\T^n_\theta)$. Given any $s\in \R$, the operator $\Lambda^s=(1+\Delta)^{\frac{s}2}$ is a (classical) \psido\ of order $s$ (\emph{cf}.\ Example~\ref{eq:PsiDOs.Lambda-to-the-s-PsiDO}). Therefore, by Proposition~\ref{Adjoints.PsiDOs-extension}, it uniquely extends to a continuous linear operator $\Lambda^s:\scD'(\T^n_\theta) \rightarrow \scD'(\T^n_\theta) $.
 
\begin{definition}[see \cite{HLP:Part2, Sp:Padova92, XXY:MAMS18}] 
The Sobolev space $W_2^s(\T^n_\theta)$, $s\in \R$, is the Hilbert space consisting of all $u\in \scD'(\T^n_\theta) $ such that $\Lambda^su \in L_2(\T^n_\theta)$. It is equipped with the Hilbert norm, 
\begin{equation*}
 \| u\|_s= \|\Lambda^s u\|_0, \qquad u \in W^{s}_2(\T^n_{\theta}).
\end{equation*}
\end{definition}

In terms of the Fourier decomposition $u = \sum u_k U^k$ in $\scD'(\T^n_\theta)$ we have 
 \begin{gather*}
u\in W^{s}_2(\T^n_{\theta}) \Longleftrightarrow \sum_{k\in \Z^n}(1+|k|^2)^{s} |u_k|^2<\infty,\\
\| u\|_s^2= \sum_{k\in \Z^n} (1+|k|^2)^{s}|u_k|^2, \qquad u\in W_2^s(\T^n_\theta).
\end{gather*}

If  $t>s$, then the inclusion of $W^{t}_2(\T^n_{\theta})$ into $W^{s}_2(\T^n_{\theta})$ is compact (see \cite{HLP:Part2, XXY:MAMS18}). Moreover, 
the Sobolev spaces $W^s_2(\T^n_\theta)$, $s\in \R$,  provide us with a natural scale of Hilbert spaces interpolating between $C^\infty(\T^n_\theta)$ and $\scD'(\T^n_\theta)$. In 
particular, we have
           \begin{equation*}
                   C^\infty(\T^n_\theta)= \bigcap_{s\in\R} W^{s}_2(\T^n_{\theta}) \qquad \text{and} \qquad \bigcup_{s\in\R} W^{s}_2(\T^n_{\theta})=  \scD'(\T^n_\theta) . 
           \end{equation*}

We have the following Sobolev mapping properties of \psidos\ on noncommutative tori.

\begin{proposition}[see \cite{HLP:Part2}] \label{prop:Sob-Mapping.rho-on-Hs}
Let $\rho(\xi)\in \stS^m(\T^n_\theta\times\R^n)$, $m\in\R$. 
\begin{enumerate}
 \item $P_\rho$ uniquely extends to a continuous linear operator
$P_\rho:W^{s+m}_2(\T^n_{\theta})\rightarrow W^{s}_2(\T^n_{\theta})$ for every $s\in \R$. 

\item If $m\leq 0$, then $P_\rho$ uniquely extends to a bounded operator $P_\rho: L_2(\T^n_\theta) \rightarrow L_2(\T^n_\theta)$. 
\end{enumerate}
\end{proposition}

We refer to Appendix~\ref{sec:quasi+Schatten} and~\cite{GK:AMS69, Si:AMS05} for the background on Schatten and weak Schatten classes. We have the following result. 

\begin{proposition}[see~\cite{HLP:Part2}] \label{prop:PsiDOs.Schatten} 
 Let $\rho(\xi)\in \bS^m(\T^n_\theta\times\R^n)$, $m< 0$, and set $p=n|m|^{-1}$. Then the operator $P_\rho$ is in the weak Schatten class $\sL_{p,\infty}$. In particular, it is trace-class if $m<-n$. 
\end{proposition}

\subsection{Smoothing operators}
A linear operator $R:C^\infty(\T^n_\theta) \rightarrow \scD'(\T^n_\theta) $ is \emph{smoothing} when it extends to a continuous linear operator $R:\scD'(\T^n_\theta)  \rightarrow C^\infty(\T^n_\theta) $ (the extension is unique since $C^\infty(\T^n_\theta)$ is dense in $\scD'(\T^n_\theta)$). It can be shown that smoothing operators are exactly the pseudodifferential operators with symbol in  $\stS^{-\infty}(\T^n_\theta\times\R^n)$ (see~\cite{HLP:Part1, HLP:Part2}).

\subsection{Symbol map} 
For $m\in \R$, set $\OP^m(\T_\theta^n)=\{P_\rho;\ \rho(\xi) \in \stS^m(\T^n_\theta\times \R^n)\}\subseteq \sL(C^\infty(\T^n_\theta))$, and put $\OP^\bt(\T_\theta^n)=\bigcup_{m\in \R}\OP^m(\T_\theta^n)$ and $\stS^\bt(\T^n_\theta\times \R^n)=\bigcup_{m\in \R}\stS^m(\T^n_\theta\times \R^n)$. Although the symbol of an operator $P\in \OP^\bt(\T_\theta^n)$ is not unique, we can still construct a symbol map which is a right inverse of the quantization map $\rho \rightarrow P_\rho$. Namely, we have the following result. 
\begin{proposition}[\cite{HLP:Part1, LP:Part1}] \label{prop:PsiDOs.rho-to-rhotilde-map-continuity}
There exists a map $\OP^\bt(\T_\theta^n)\ni P \rightarrow \rho_{\!{}_{P}}\in \stS^\bt(\T^n_\theta\times \R^n)$ with the following properties: 
\begin{enumerate}
\item[(i)] $P=P_{\rho_{\!{}_{P}}}$ for all $P\in  \OP^\bt(\T_\theta^n)$. 

\item[(ii)] For every $m\in \R$, the map $\rho \rightarrow \rho - \rho_{\!{}_{P_\rho}}$ is continuous from $\stS^m(\T^n_\theta\times \R^n)$ to $\stS^{-\infty}(\T^n_\theta
\times \R^n)$. 

\item[(iii)] If $P\in \Psi^q(\T^n_\theta)$, $q\in \C$, then $\rho_{\!{}_{P}}(\xi)\in S^q(\T^n_\theta\times \R^n)$. 
\end{enumerate}
\end{proposition}

\begin{remark}
 Part~(i) and Remark~\ref{rmk:PsiDOs.uniqueness-symbol} imply that $\rho(k)=\rho_{\!{}_{P_\rho}}(k)$ for all $k\in \Z^n$, and Part~(ii) implies that, for every $m\in \R$, the map $\rho\rightarrow \rho_{\!{}_{P_\rho}}$ is continuous from $\stS^m(\T^n_\theta\times \R^n)$ to itself. We may take the symbol map $P\rightarrow \rho_{\!{}_{P}}$ to be of the form 
 \begin{equation} \label{eq:PsiDOs.symbol-map-definition}
\rho_{\!{}_{P}}(\xi) = \sum_{k\in\Z^n} \phi(\xi-k)P(U^k)(U^k)^{-1} , \qquad \xi\in\Rn ,
\end{equation}
 where $\phi$ is a suitable function in $\cS(\R^n)$ (see~\cite[Lemma~4.5.1]{RT:Birkhauser10}).  
\end{remark}

\subsection{Ellipticity and parametrices}

\begin{definition}
An operator $P\in\Psi^q(\T^n_\theta)$, $q\in \C$, is \emph{elliptic} when its principal symbol $\rho_{q}(\xi)$ is invertible for all $\xi \in \R^n\setminus 0$. In this case, $\rho_{q}(\xi)^{-1} \in S_{-q}(\T^n_\theta\times\R^n)$~\cite{Ba:CRAS88, HLP:Part2}. 
\end{definition} 

\begin{example}
 Suppose that $q\in \R$ and the principal symbol $\rho_{q}(\xi)$ of $P$ is such that, for some $c>0$,
 \begin{equation} \label{eq:Elliptic.positivity-criterion}
 \acoup{\rho_q(\xi)\eta}{\eta}\geq c|\xi|^q \|\eta\|_0^2 \qquad \text{for all $\eta \in L_2(\T^n_\theta)$ and $\xi \in \R^n\setminus 0$}.
\end{equation}
 Then $P$ is elliptic (see~\cite{HLP:Part2}). This condition is satisfied by the (flat) Laplacian $\Delta=\delta_1^2+\cdots+\delta_n^2$, by the conformal deformations $k\Delta k$, $k\in C^\infty(\T^n_\theta) $, $k>0$, introduced by Connes--Tretkoff~\cite{CT:Baltimore11} and considered by various other authors since, and by the Laplace--Beltrami operators of~\cite{HP:JGP20}. 
\end{example}

\begin{proposition}[see \cite{Ba:CRAS88, Co:CRAS80, HLP:Part2}] \label{prop:Elliptic.existence-parametrix} \label{prop:Elliptic.regularity}
Let $P\in\Psi^q(\T^n_\theta)$, $q\in\C$, be elliptic with symbol  $\rho(\xi)\sim\sum_{j\geq 0} \rho_{q-j}(\xi)$, and set $m=\Re q$. 
\begin{enumerate}
 \item $P$ admits a parametrix $Q\in \Psi^{-q}(\T^n_\theta)$, i.e., $ PQ=QP=1 \bmod \Psi^{-\infty}(\T^n_\theta)$. 

\item Any parametrix $Q\in \Psi^{-q}(\T^n_\theta)$ has symbol $\sigma(\xi)\sim\sum_{j\geq 0} \sigma_{-q-j}(\xi)$, where 
\begin{gather} \label{eq:PsiDOs.symbol-parametrix1}
 \sigma_{-q}(\xi)=\rho_q(\xi)^{-1},\\
\sigma_{-q-j}(\xi)=-\!\!\!\sum_{\substack{k+l+|\alpha|=j \\ l<j}}\!\! \frac{1}{\alpha !}\rho_q(\xi)^{-1}\partial_\xi^\alpha\rho_{q-k}(\xi)\delta^\alpha\sigma_{-q-l}(\xi), \qquad j\geq 1 .
\label{eq:PsiDOs.symbol-parametrix2}
\end{gather}

\item For any $s\in \R$ and $u\in{\scD'(\T^n_\theta)}$, we have 
\begin{equation*}
Pu\in W^{s}_2(\T^n_{\theta})\Longleftrightarrow u\in W^{s+m}_2(\T^n_{\theta}) .
\end{equation*}

\item The operator $P$ is hypoelliptic, i.e.,  for any $u\in \scD'(\T^n_\theta) $, we have
\begin{equation}\label{eq:PsiDOs.hypoellipticity}
 Pu \in C^\infty(\T^n_\theta)  \Longleftrightarrow u \in C^\infty(\T^n_\theta) . 
\end{equation}
\end{enumerate}
\end{proposition}

If $m>0$, it can further be shown that $P-\lambda$ is hypoelliptic in the sense of~(\ref{eq:PsiDOs.hypoellipticity}) for every $\lambda \in \C$.  This has the following consequence.

\begin{corollary}[\cite{HLP:Part2}] \label{cor:Elliptic.P-lambda-hypoell}
 Let $P\in\Psi^q(\T^n_\theta)$, $q\in \C$, be elliptic. Then 
   \begin{equation*}
 \ker P :=\left\{u\in \scD'(\T^n_\theta) ; \ Pu=0\right\}\subseteq  C^\infty(\T^n_\theta). 
\end{equation*}
If $\Re q>0$, then, for every $\lambda \in \C$, we also have 
\begin{equation*}
 \ker (P-\lambda) :=\left\{u\in \scD'(\T^n_\theta) ; \ (P-\lambda )u=0\right\}\subseteq  C^\infty(\T^n_\theta). 
\end{equation*}
 \end{corollary}

\subsection{Spectral theory of elliptic \psidos} \label{subsec:spectra-of-psidos}
Let $P\in \Psi^q(\T^n_\theta)$ be elliptic with $m:=\Re q>0$. We regard $P$ as an unbounded operator on $L_2(\T^n_\theta)$ with domain $W_2^m(\T^n_\theta)$. It can then be shown that $P$ is closed and Fredholm. Its adjoint agrees with its formal adjoint $P^*\in  \Psi^{\bar{q}}(\T^n_\theta)$ with domain $W_2^m(\T^n_\theta)$. In particular, if $P$ is formally selfadjoint (resp., normal), then it is selfadjoint (resp., normal). 

We denote by $\Sp(P)$ the spectrum of $P$, i.e., the set of $\lambda \in \C$ such that $P-\lambda$ is not a bijection from its domain $W_2^m(\T^n_\theta) $ onto $L_2(\T^n_\theta)$. For all $\lambda \in \C\setminus \Sp(P)$, the operator $P-\lambda:W_2^m(\T^n_\theta)  \rightarrow L_2(\T^n_\theta)$ is actually a topological isomorphism (see~\cite{HLP:Part2}). Moreover, $\Sp(P^*)=\{\overline{\lambda}; \ \lambda \in \Sp(P)\}$.

\begin{proposition}[see \cite{HLP:Part2}] \label{prop:Spectral.spectrum-P}
 There are only two possibilities for the spectrum of $P$. Either $\Sp(P)$ is all of $\C$, or it is an unbounded discrete set consisting of isolated eigenvalues with finite multiplicity. 
 \end{proposition}
 
 For normal operators, we have the following stronger result. 
 
 \begin{proposition}[see \cite{HLP:Part2}] \label{prop:Spectral.ortho-splitting}
If $P$ is normal and $\Sp (P)\neq \C$ (e.g., $P$ is selfadjoint), then we have an orthogonal decomposition, 
 \begin{equation*}
 L_2(\T^n_\theta) = \bigoplus_{\lambda \in \Sp(P)} \ker (P-\lambda). 
\end{equation*}
In particular, $P$ admits an orthonormal eigenbasis. 
\end{proposition}

 \section{Parametric Pseudodifferential Calculus} \label{sec:Parametric-PsiDOs}
We survey the main definitions and properties of pseudodifferential operators with parameter from~\cite{LP:Part1}. Throughout, $\C^*=\C\setminus \{0\}$, and we refer to~\cite[Appendix~C]{HLP:Part1} for background on smooth maps valued in locally convex spaces.

\subsection{Pseudo-cones and $\Hol^d(\Lambda)$-families} 
By a \emph{cone} $\Theta \subseteq \C^*$ we always mean one with vertex at the origin: $\lambda \in \Theta \Rightarrow t\lambda \in \Theta$ for all $t>0$.

 \begin{definition}
A  connected set $\Lambda\subseteq \C$ is called a \emph{pseudo-cone} when there are a cone $\Theta \subseteq \C^*$ and a disk $D$ about the origin such that 
$\Lambda\setminus D=\Theta\setminus D$. The cone $\Theta$, which is unique, is called the \emph{conical part} of $\Lambda$. If $\Lambda$ is a pseudo-cone, then so is $\overline{\Lambda}$, and so is $\op{Int}(\Lambda)$ whenever it is connected.
 \end{definition}

\begin{example}
 Any angular sector $\Theta=\{ \phi <\arg \lambda <\phi'\}$ is a cone, and hence is a pseudo-cone. Chopping off a disk from $\Theta$ or gluing to it a disk or an annulus  provides us with pseudo-cones with conical part $\Theta$. 
 \end{example}
 
\begin{figure}[h]
\begin{minipage}{0.20\linewidth}
\centering
\begin{tikzpicture}[cpfig]\cpbox
 \begin{scope}\clip (-4.55,-4.425) rectangle (4.25,4.425);
  \fill[cpshade] (0,0)--(135:9)--(225:9)--cycle;
  \fill[cpshade] (0,0)--(-15:9)--(30:9)--cycle;
  \fill[cpshade] (0,0) circle (1.7);
  \draw (135:1.7)--(135:9) (225:1.7)--(225:9) (30:1.7)--(30:9) (-15:1.7)--(-15:9);
  \draw (30:1.7) arc[start angle=30,end angle=135,radius=1.7];
  \draw (225:1.7) arc[start angle=225,end angle=345,radius=1.7];
 \end{scope}\cpaxes
\end{tikzpicture}
\end{minipage}
\begin{minipage}{0.20\linewidth}
\centering
\begin{tikzpicture}[cpfig]\cpbox
 \begin{scope}\clip (-4.45,-4.3) rectangle (4.35,4.35);
  \fill[cpshade] (0,0)--(135:9)--(225:9)--cycle;
  \fill[cpshade] (0,0)--(-45:9)--(45:9)--cycle;
  \fill[cpshade] (0,0)--(45:1.5) arc[start angle=45,end angle=135,radius=1.5]--cycle;
  \draw (45:1.5)--(45:9) (135:1.5)--(135:9) (0,0)--(-45:9) (0,0)--(-135:9);
  \draw (45:1.5) arc[start angle=45,end angle=135,radius=1.5];
 \end{scope}\cpaxes
\end{tikzpicture}
\end{minipage}
\begin{minipage}{0.20\linewidth}
\centering
\begin{tikzpicture}[cpfig]\cpbox
 \begin{scope}\clip (-4.55,-4.425) rectangle (4.25,4.425);
  \fill[cpshade] (-136:0.575) arc[start angle=-136,end angle=150,radius=0.575]
     -- (150:1.025) arc[start angle=150,end angle=-136,radius=1.025] -- cycle;
  \fill[cpshade] (105:1.025)--(105:9)--(150:9)--(150:0.575)--(150:1.025) arc[start angle=150,end angle=105,radius=1.025]--cycle;
  \fill[cpshade] (-105:1.025)--(-105:9)--(-136:9)--(-136:0.575)--(-136:1.025) arc[start angle=-136,end angle=-105,radius=1.025]--cycle;
  \draw (150:0.575)--(150:9) (105:1.025)--(105:9) (-136:0.575)--(-136:9) (-105:1.025)--(-105:9);
  \draw (150:0.575) arc[start angle=150,end angle=-136,radius=0.575];
  \draw (105:1.025) arc[start angle=105,end angle=-105,radius=1.025];
 \end{scope}\cpaxes
\end{tikzpicture}
\end{minipage}
\caption{Examples of pseudo-cones.}
\end{figure}
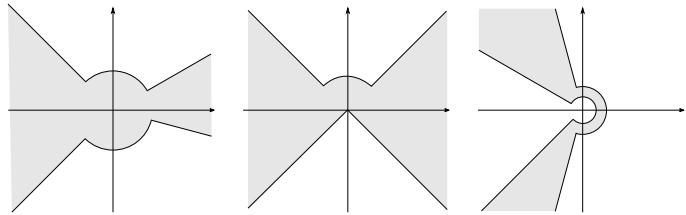

\begin{definition}
 Given pseudo-cones  $\Lambda$ and $\Lambda'$, we shall write $\Lambda'\subsubset\Lambda$ to mean that $\overline{\Lambda'}\subseteq \op{Int}(\Lambda)$. If $\Theta$, $\Theta'$ are the conical parts of $\Lambda$, $\Lambda'$ respectively, this implies that $\Theta'\cap\{|\lambda|=1\}$ is relatively compact in $\Theta\cap\{|\lambda|=1\}$.
\end{definition}
 
In what follows, we let $\sE$ be a topological vector space (TVS).  

\begin{definition}\label{def:Holp.holomorphic-maps} 
Given any open $\Omega \subseteq \C$, a map $u:\Omega\rightarrow \sE$ is \emph{holomorphic at a point} $z_0\in \Omega$ if there is $u'(z_0)\in \sE$ such that
\begin{equation}
 \lim_{z\rightarrow z_0} \frac{u(z)-u(z_0)}{z-z_0} = u'(z_0) \quad \text{in $\sE$}. 
\end{equation}
 We say that the map $u(z)$ is \emph{holomorphic on} $\Omega$ if it is holomorphic at every point of $\Omega$. We denote by $\Hol(\Omega;\sE)$ the space of holomorphic maps from $\Omega$ to $\sE$. 
 \end{definition}

\begin{remark}
If $\Phi:\sE\rightarrow \sE_1$ is a continuous linear map to another TVS $\sE_1$, then $\Phi[u(z)]$ is holomorphic with $(\Phi[u(z)])'= \Phi[u'(z)]$. Moreover, if $\sE$ is locally convex, then a map $u:\Omega\rightarrow \sE$ is holomorphic if and only if it is $C^\infty$ and satisfies $\partial_{\overline{z}}u=0$ (see, e.g., \cite[Remark~4.10]{LP:Part1}). 
\end{remark}

\begin{definition}
 Suppose $\Omega$ is open in $\R^N\times \C$ for some $N\geq 1$. Then $C^{\infty,\omega}(\Omega; \sE)$ consists of $C^\infty$ maps $(\eta,\lambda) \rightarrow f(\eta;\lambda)$ from $\Omega$ to $\sE$ that are holomorphic in $\lambda$. For $\sE=C^\infty(\T^n_\theta)$ we write $C^{\infty,\omega}(\T^n_\theta\times \Omega)$.
\end{definition}

\begin{remark}\label{rmk:Symbols.identification-Cinftyw}
If $\sE$ is locally convex, then $C^{\infty,\omega}(\Omega; \sE)$ can be identified with the closed subspace of $C^\infty(\Omega;\sE)$ of solutions of $\partial_{\widebar{\lambda}} f=0$. For $\Omega=V\times \Lambda$ with $V\subseteq \R^N$ and $\Lambda\subseteq \C$ open, this gives an isomorphism $C^{\infty,\omega}(\Omega; \sE)\simeq \Hol(\Lambda; C^\infty(V;\sE))$. For $\sE=C^\infty(\T^n_\theta)$ in particular, $C^{\infty,\omega}(\Omega; \sE)\simeq \Hol(\Lambda; C^\infty(\T^n_\theta\times V))$.
\end{remark}

Throughout the remainder of this section, we let $\Lambda$ be an \emph{open} pseudo-cone. 

\begin{definition}\label{def:symbols.Hold-sE}
Suppose that $\sE$ is a locally convex space. Then $\Hol^d(\Lambda;\sE)$, $d\in\R$, consists of holomorphic maps $f: \Lambda \rightarrow \sE$ such that, for every continuous semi-norm $p$ on $\sE$ and every pseudo-cone $\Lambda'\subsubset\Lambda$, there is $C_{p\Lambda'}>0$ such that
\begin{equation} \label{eq:Parameter.vector-with-parameter-estimate}
p\left[f(\lambda)\right]\leq C_{p\Lambda'}(1+|\lambda|)^d \qquad \forall\lambda\in\Lambda' .
\end{equation}
This definition continues to make sense if $\sE$ is a quasi-Banach space by requiring the estimate~(\ref{eq:Parameter.vector-with-parameter-estimate}) to hold for the quasi-norm of $\sE$.  
\end{definition}

\subsection{Standard parametric symbols} 
\begin{definition}
 $\stS^{m,d}(\T^n_\theta\times\R^n\times \Lambda)$, $m,d\in \R$, consists of maps $\rho(\xi;\lambda)\in C^{\infty,\omega} (\T^n_\theta\times\R^n\times \Lambda)$ such that, for all pseudo-cones $\Lambda'\subsubset \Lambda$ and multi-orders $\alpha$ and $\beta$, there is $C_{\Lambda'\alpha\beta}>0$ such that
\begin{equation*}
 \big\|\delta^\alpha \partial_\xi^\beta \rho(\xi;\lambda)\big\| \leq  C_{\Lambda'\alpha\beta}(1+|\lambda|)^d(1+|\xi|)^{m-|\beta|} \qquad \forall (\xi, \lambda)\in \R^n\times \Lambda'. 
\end{equation*}
\end{definition}

\begin{example}
 Any symbol $\rho(\xi)\in \stS^m(\T^n_\theta\times\R^n)$ can be regarded as an element of $\stS^{m,0}(\T^n_\theta\times\R^n\times \Lambda)$ that does not depend on $\lambda$. 
\end{example}

\begin{remark} \label{rmk:Parameter.standard-symbols-with-parameter-identification}
Let $\rho(\xi;\lambda)\in \stS^{m,d}(\T^n_\theta\times\R^n\times \Lambda)$. For every $\lambda\in \Lambda$, we get a symbol $\rho(\cdot;\lambda)\in \stS^m(\T^n_\theta\times\R^n)$, and so we get a family in $\stS^m(\T^n_\theta\times\R^n)$ parametrized by $\Lambda$. This is actually a $\Hol^d(\Lambda)$-family. This provides us with a  one-to-one correspondence,  
\begin{equation}
 \stS^{m,d}(\T^n_\theta\times\R^n\times \Lambda)\simeq \Hol^d(\Lambda;\stS^m(\T^n_\theta\times\R^n)).
 \label{eq:Parameter.parametric-symbols-identification} 
\end{equation}
\end{remark}

\begin{definition}
 $\stS^{-\infty,d}(\T^n_\theta\times\R^n\times \Lambda)$, $d\in \R$, consists of maps $\rho(\xi;\lambda)\in C^{\infty,\omega}(\T^n_\theta\times\R^n\times \Lambda)$ such that, given any $N\geq 0$, for all pseudo-cones $\Lambda'\subsubset \Lambda$ and multi-orders $\alpha$ and $\beta$, there is $C_{\Lambda'N\alpha\beta}>0$ such that
\begin{equation*}
 \big\|\delta^\alpha \partial_\xi^\beta \rho(\xi;\lambda)\big\| \leq  C_{\Lambda'N\alpha\beta}(1+|\xi|)^{-N}(1+|\lambda|)^d \qquad \forall (\xi, \lambda)\in \R^n\times \Lambda'. 
\end{equation*}
Equivalently, $\stS^{-\infty,d}(\T^n_\theta\times\R^n\times \Lambda)=\bigcap_{m\in \R} \stS^{m,d}(\T^n_\theta\times\R^n\times \Lambda)$. 
\end{definition}

\begin{remark}
 The one-to-one correspondence~(\ref{eq:Parameter.parametric-symbols-identification}) induces a vector space isomorphism,  
\begin{equation*}
\stS^{-\infty,d}(\T^n_\theta\times\R^n\times \Lambda)\simeq \Hol^d(\Lambda;\stS^{-\infty}(\T^n_\theta\times\R^n)).
\end{equation*}
\end{remark}

In what follows, we let $(m_j)_{j \geq 0}$ be a (strictly) decreasing sequence of real numbers converging to $-\infty$. 

\begin{definition} \label{def:Parameter.Standard-asymptotic}
Given $\rho(\xi;\lambda)\in C^{\infty,\omega}(\T^n_\theta\times\R^n\times \Lambda)$ and $\rho_{j}(\xi;\lambda)\in\stS^{m_j,d}(\T^n_\theta\times\R^n\times \Lambda)$, $j\geq 0$, we shall write $\rho(\xi;\lambda)\sim\sum_{j\geq 0}\rho_{j}(\xi;\lambda)$ when
\begin{equation} \label{eq:Parameter.Standard-asymptotic}
\rho(\xi;\lambda)-\sum_{j<N}\rho_{j}(\xi;\lambda)\in\stS^{m_N,d}(\T^n_\theta\times\R^n\times \Lambda) \qquad \text{for all $N\geq 0$} .
\end{equation}
\end{definition}

\begin{remark}
 It can be shown (see~\cite{LP:Part1}) that~(\ref{eq:Parameter.Standard-asymptotic}) is equivalent to requiring that, given any $N\geq 0$, for all pseudo-cones $\Lambda'\subsubset \Lambda$ and multi-orders $\alpha$ and $\beta$, as soon as $J$ is large enough there is a constant $C>0$ such that, for all $(\xi,\lambda)\in \R^n\times \Lambda'$, we have 
\begin{equation} \label{eq:Parameter.Standard-asymptotic-qualitative}
 \big\|\delta^\alpha \partial_\xi^\beta \big(\rho-\sum_{j<J}\rho_{j}\big)(\xi;\lambda)\big\| \leq C(1+|\lambda|)^d(1+|\xi|)^{-N}. 
\end{equation}
\end{remark}

\subsection{Homogeneous parametric symbols}
Throughout this paper $w>0$ is a fixed positive real number, and $\Theta$ denotes the conical part of $\Lambda$ (an open cone in $\C^*$). For $c\geq 0$, we set
\begin{equation} \label{eq:Parameter.parameter-set-for-homogeneous-symbols}
\Omega_c(\Theta) = \big\{ (\xi,\lambda)\in(\Rn\setminus 0)\times\C; \ \text{$\lambda\in\Theta$ or $|\lambda|<c|\xi|^w$} \big\} .
\end{equation}
In particular, $\Omega_0(\Theta) = (\Rn\setminus 0)\times\Theta$. 

\begin{definition} \label{def:Parameter.homogeneous-symbol}
$S_m^d(\T^n_\theta\times\Omega_c(\Theta))$, $m,d\in\R$, consists of maps $\rho(\xi;\lambda)\in C^{\infty,\omega}(\T^n_\theta\times\Omega_c(\Theta))$ that satisfy the following two properties:
\begin{enumerate}
\item[(i)] $\rho(t\xi;t^w\lambda) = t^m\rho(\xi;\lambda)$ for all $(\xi,\lambda)\in\Omega_c(\Theta)$ and $t>0$.

\item[(ii)] Given any cone $\Theta'$ such that $\overline{\Theta'}\setminus \{0\} \subseteq \Theta$, for all compact sets $K\subseteq \R^n\setminus 0$ and multi-orders $\alpha$, $\beta$, there is $C_{\Theta'K\alpha\beta}>0$ such that 
\begin{equation*}
 \norm{\delta^\alpha \partial_\xi^\beta \rho(\xi;\lambda)} \leq C_{\Theta'K\alpha\beta} \big(1+|\lambda|\big)^d \qquad \forall (\xi,\lambda)\in K\times \Theta'. 
\end{equation*}
\end{enumerate}
\end{definition}

\begin{example} \label{ex:Parameter.homogeous-symbol-can-be-seen-as-homogeneous-symbol-with-parameter}
 Any homogeneous symbol $\rho(\xi)\in S_m(\T^n_\theta\times\R^n)$ can be regarded as an element of $S^{0}_m(\T^n_\theta\times\Omega_c(\Theta))$ for every $c>0$. 
\end{example}

The following lemma shows how to cut off homogeneous parametric symbols in order to get standard parametric symbols. 

\begin{lemma}[see~\cite{LP:Part1}] \label{lem:Parameter.homogeneous-symbol-estimate}
Suppose that $\rho(\xi;\lambda)\in S_m^d(\T^n_\theta\times\Omega_c(\Theta))$, $c>0$, $m,d\in\R$. Let $\chi(\xi)\in C_c^\infty(\Rn)$ be such that $\chi(\xi) = 1$ for 
$|\xi|\leq (c^{-1}R)^{\frac{1}{w}}$, where $R>0$ is such that $\Lambda \subseteq \Theta \cup D(0,R)$. 
\begin{enumerate}
 \item[(i)] If $d\geq 0$, then $(1-\chi(\xi))\rho(\xi;\lambda)\in \stS^{m,d}(\T^n_\theta\times\R^n\times \Lambda)$.
 
 \item[(ii)] If $d<0$, then  
 \begin{equation*}
 (1-\chi(\xi))\rho(\xi;\lambda)\in  \bigcap_{d\leq d'\leq 0} \stS^{m+w|d'|,d'}(\T^n_\theta\times\R^n\times \Lambda). 
 \end{equation*}
\end{enumerate}
\end{lemma}

\subsection{Classical parametric symbols} 
Note that if $c>0$ and $R>0$ is such that $\Lambda \subseteq \Theta \cup D(0,R)$, then $\Omega_c(\Theta) \supseteq ( \R^n \setminus B(r) ) \times \Lambda$ for all $r> (c^{-1}R)^{\frac1{w}}$.

\begin{definition} \label{def:Parameter.classical-symbol}
$S^{m,d}(\T^n_\theta\times\R^n\times \Lambda)$, $m, d\in\R$, consists of maps $\rho(\xi;\lambda)\in C^{\infty,\omega}(\T^n_\theta\times\R^n\times \Lambda)$ for which there are $c>0$ and 
$\rho_{m-j}(\xi;\lambda)\in S_{m-j}^d(\T^n_\theta\times\Omega_c(\Theta))$, $j\geq 0$, such that 
\begin{equation*}
\rho(\xi;\lambda)\sim\sum_{j\geq 0}\rho_{m-j}(\xi;\lambda), 
\end{equation*}
in the sense that, for all $N\geq 0$ and multi-orders $\alpha$, $\beta$, given any pseudo-cone $\Lambda'\subsubset \Lambda$ and $r>0$ such that 
$ \Omega_c(\Theta) \supseteq \big( \R^n \setminus B(r) \big) \times \Lambda'$, as soon as $J$ is large enough there is $C_{\Lambda'NJr\alpha \beta}>0$ such that, for all $(\xi,\lambda)\in (\R^n\setminus B(r)) \times\Lambda'$, we have 
\begin{equation} \label{eq:Parameter.classical-symbol-asymptotics}
 \big\| \delta^\alpha \partial_\xi^\beta \big(\rho-\sum_{j<J}\rho_{m-j}\big)(\xi;\lambda)\big\| \leq C_{\Lambda'NJr\alpha \beta}|\xi|^{-N}(1+|\lambda|)^{d} . 
\end{equation}
\end{definition}

\begin{example} \label{ex:Parameter.classical-symbol-can-be-seen-as-classical-symbol-with-parameter}
 Let $\rho(\xi)\in S^m(\T^n_\theta\times\R^n)$, $\rho(\xi)\sim \sum \rho_{m-j}(\xi)$. As mentioned in Example~\ref{ex:Parameter.homogeous-symbol-can-be-seen-as-homogeneous-symbol-with-parameter}, each homogeneous symbol $\rho_{m-j}(\xi)$ can be regarded as an element of $S^{0}_{m-j}(\T^n_\theta\times\Omega_c(\Theta))$ for any $c>0$. Note also that the asymptotic expansion $\rho(\xi)\sim \sum \rho_{m-j}(\xi)$ in the sense of~(\ref{eq:Symbols.classical-estimates}) implies that we have an asymptotic expansion in the sense of~(\ref{eq:Parameter.classical-symbol-asymptotics}) with $d=0$. It then follows that $\rho(\xi)\in S^{m,0}(\T^n_\theta\times\R^n\times \Lambda)$. 
\end{example}

\begin{remark} \label{rmk:Parameter.classical-symbol-asymptotics-equivalent-conditions}
Suppose we are given $\rho_{m-j}(\xi;\lambda)\in S_{m-j}^{d}(\T^n_\theta\times\Omega_c(\Theta))$, $j\geq 0$. Let $\chi(\xi)\in C^\infty_c(\R^n)$ be as in Lemma~\ref{lem:Parameter.homogeneous-symbol-estimate}. Recall that, by Lemma~\ref{lem:Parameter.homogeneous-symbol-estimate}, $(1-\chi(\xi))\rho_{m-j}(\xi;\lambda)\in \stS^{m-j+wd_{-},d}(\T^n_\theta\times\R^n\times \Lambda)$, where $d_{-}=\max(0,-d)$.  
Then, the following are equivalent: 
\begin{enumerate}
 \item[(i)] $\rho(\xi;\lambda)\sim \sum_{j\geq 0} \rho_{m-j}(\xi;\lambda)$ in the sense of~(\ref{eq:Parameter.classical-symbol-asymptotics}). 
 
 \item[(ii)] $\rho(\xi;\lambda)\sim \sum_{j\geq 0} (1-\chi(\xi))\rho_{m-j}(\xi;\lambda)$ in the sense of~(\ref{eq:Parameter.Standard-asymptotic}) or~(\ref{eq:Parameter.Standard-asymptotic-qualitative}). 
\end{enumerate}
Note that here~(\ref{eq:Parameter.Standard-asymptotic}) means that,  for all $N\geq 0$, as soon as $J\geq N+wd_{-}$, we have 
\begin{equation*}
\rho(\xi;\lambda) - \sum_{j<J}  (1-\chi(\xi))\rho_{m-j}(\xi;\lambda) \in \stS^{m-N,d}(\T^n_\theta\times\R^n\times \Lambda). 
\end{equation*}
This provides us with a quantitative version of the estimates~(\ref{eq:Parameter.classical-symbol-asymptotics}). 
 \end{remark}
 
Combining the above remark with Lemma~\ref{lem:Parameter.homogeneous-symbol-estimate} yields the following result. 
\begin{proposition}[see~\cite{LP:Part1}]\label{symbols:inclusion-classical-standard}
 Let $m,d\in \R$. 
\begin{enumerate}
 \item If $d\geq 0$, then $S^{m,d}(\T^n_\theta\times\R^n\times \Lambda)\subseteq\stS^{m,d}(\T^n_\theta\times\R^n\times \Lambda)$. 
 
 \item If $d<0$, then  
        \begin{equation*}
              S^{m,d}(\T^n_\theta\times\R^n\times \Lambda)\subseteq  \bigcap_{d\leq d'\leq 0} \stS^{m+w|d'|,d'}(\T^n_\theta\times\R^n\times \Lambda). 
         \end{equation*}
\end{enumerate}
 \end{proposition}

\subsection{\psidos\ with parameter} 
Let $\rho(\xi;\lambda)\in \stS^{m,d}(\T^n_\theta\times\R^n\times \Lambda)$. For each $\lambda \in \Lambda$, $P_\rho(\lambda)$ denotes the \psido\ with symbol $\rho(\cdot;\lambda)$, acting continuously on $C^\infty(\T^n_\theta)$ via the oscillating integral formula of Section~\ref{sec:PsiDOs}. We equip $\cL(C^\infty(\T^n_\theta))$ and $\cL(\scD'(\T^n_\theta))$ with their strong topologies.

\begin{proposition}[see~\cite{LP:Part1}] \label{prop:PsiDOs-parameter.PsiDO-gives-rise-to-family-of-continuous-operators-on-cAtheta}
 For any $\rho(\xi;\lambda)\in \stS^{m,d}(\T^n_\theta\times\R^n\times \Lambda)$, $m,d\in \R$, the family $P_\rho(\lambda)$ belongs to $\Hol^d(\Lambda; \cL(C^\infty(\T^n_\theta)))$ and uniquely extends to a family $P_\rho(\lambda)\in \Hol^d(\Lambda; \cL(\scD'(\T^n_\theta) ))$. 
\end{proposition}

\begin{definition}
$\Psi^{m,d}(\T^n_\theta;\Lambda)$, $m,d\in \R$, consists of families of operators $P_{\rho}(\lambda):C^\infty(\T^n_\theta)  \rightarrow C^\infty(\T^n_\theta) $ with $\rho(\xi;\lambda)$ in $S^{m,d}(\T^n_\theta\times\R^n\times \Lambda)$.
\end{definition}

\begin{example} \label{ex:PsiDOs-parameter.usual-PsiDO-can-be-seen-as-PsiDO-with-parameter}
 It follows from Example~\ref{ex:Parameter.classical-symbol-can-be-seen-as-classical-symbol-with-parameter} that any operator $P\in \Psi^m(\T^n_\theta)$ can be regarded as an element of $\Psi^{m,0}(\T^n_\theta;\Lambda)$. If $m=w$, then combining this with the obvious inclusion $\Psi^{w,0}(\T^n_\theta;\Lambda)\subseteq \Psi^{w,1}(\T^n_\theta;\Lambda)$ and the fact that $\lambda\in S_w^1(\T^n_\theta\times\Omega_c(\Theta))$ allows us to regard $P-\lambda$ as an element of $\Psi^{w,1}(\T^n_\theta;\Lambda)$. 
\end{example}

\begin{remark}\label{rmk:Parametric.parametric-Psidos-families-Cinfty-sDprime}
 It follows from Proposition~\ref{symbols:inclusion-classical-standard} and Proposition~\ref{prop:PsiDOs-parameter.PsiDO-gives-rise-to-family-of-continuous-operators-on-cAtheta} that every $P(\lambda)\in \Psi^{m,d}(\T^n_\theta;\Lambda)$, $m,d\in \R$, gives rise to $\Hol^d(\Lambda)$-families in $\cL(C^\infty(\T^n_\theta))$ and $\cL(\scD'(\T^n_\theta))$. 
\end{remark}

We also define \psidos\ with parameter of order~$-\infty$ as follows. 
\begin{definition}\label{rmk:PsiDOs-parameter.order-minus-infty-to-intersection-inclusion}
$\Psi^{-\infty,d}(\T^n_\theta;\Lambda)$, $d\in \R$,  consists of families of operators $P_{\rho}(\lambda):C^\infty(\T^n_\theta)  \rightarrow C^\infty(\T^n_\theta) $ with $\rho(\xi;\lambda)$ in $\stS^{-\infty,d}(\T^n_\theta\times\R^n\times \Lambda)$. Equivalently, $\Psi^{-\infty,d}(\T^n_\theta;\Lambda)= \bigcap_{m\in \R}\Psi^{m,d}(\T^n_\theta;\Lambda)$ (see~\cite{LP:Part1}). 
\end{definition}

Let $\rho_1(\xi;\lambda) \in \stS^{m_1,d_1}(\T^n_\theta\times\R^n\times \Lambda)$ and $\rho_2(\xi;\lambda)\in \stS^{m_2,d_2}(\T^n_\theta\times\R^n\times \Lambda)$. For each $\lambda \in \Lambda$, we denote by $\rho_1\sharp \rho_2(\xi;\lambda)$ the symbol given by~(\ref{eq:Composition.symbol-sharp}). By Proposition~\ref{prop:Composition.sharp-continuity-standard-symbol}, the composition $P_{\rho_1}(\lambda)P_{\rho_2}(\lambda)$ is the \psido\ with symbol $\rho_1\sharp \rho_2(\xi;\lambda)$.

\begin{proposition}[see~\cite{LP:Part1}] \label{prop:Parameter.composition-PsiDOs}
Let $P_1(\lambda)\in\Psi^{m_1,d_1}(\T^n_\theta;\Lambda)$ have symbol $\rho_1(\xi;\lambda)\sim\sum \rho_{1,m_1-j}(\xi;\lambda)$, and 
let $P_2(\lambda)\in\Psi^{m_2,d_2}(\T^n_\theta;\Lambda)$ have symbol $\rho_2(\xi;\lambda)\sim\sum \rho_{2,m_2-j}(\xi;\lambda)$. 
\begin{enumerate}
 \item $\rho_1\sharp \rho_2(\xi;\lambda)\in S^{m_1+m_2,d_1+d_2}(\T^n_\theta\times\R^n\times \Lambda)$ with 
 $\rho_1\sharp\rho_2(\xi;\lambda)\sim\sum (\rho_1\sharp\rho_2)_{m_1+m_2-j}(\xi;\lambda)$, where
\begin{equation*}
(\rho_1\sharp\rho_2)_{m_1+m_2-j}(\xi;\lambda) = \sum_{k+l+|\alpha|=j}\frac{1}{\alpha!}\partial_\xi^\alpha\rho_{1,m_1-k}(\xi;\lambda)\delta^\alpha\rho_{2,m_2-l}(\xi;\lambda),\qquad j\geq 0 .
\end{equation*}

\item The composition $P_1(\lambda)P_2(\lambda)=P_{\rho_1\sharp \rho_2}(\lambda)$ is in $\Psi^{m_1+m_2,d_1+d_2}(\T^n_\theta;\Lambda)$. 
\end{enumerate}
\end{proposition}

In addition, we have the following parametric version of Proposition~\ref{prop:Sob-Mapping.rho-on-Hs}. 

\begin{proposition}[see~\cite{LP:Part1}] \label{prop:PsiDOs-parameter.Sobolev-mapping-properties}
 Let $\rho(\xi;\lambda)\in \stS^{m,d}(\T^n_\theta\times\R^n\times \Lambda)$, $m,d\in \R$. 
 \begin{enumerate}
 \item The family $P_\rho(\lambda)$ belongs to $\Hol^d(\Lambda; \cL(W_2^{s+m}(\T^n_\theta), W_2^s(\T^n_\theta)))$ for every $s\in \R$.
 
 \item If $m\leq 0$, then $P_\rho(\lambda)\in \Hol^d(\Lambda; \cL(L_2(\T^n_\theta) ))$. 
\end{enumerate}
\end{proposition}

Combining Proposition~\ref{prop:PsiDOs-parameter.Sobolev-mapping-properties} with Proposition~\ref{symbols:inclusion-classical-standard} yields the following result. 
 
\begin{proposition}[see~\cite{LP:Part1}] \label{prop:PsiDOs-parameter.classical-PsiDO-Sobolev-mapping-properties}
 Let $P(\lambda)\in \Psi^{m,d}(\T^n_\theta;\Lambda)$, $m,d\in \R$. 
\begin{enumerate}
 \item If $d\geq 0$, then $P(\lambda) \in \Hol^d(\Lambda; \cL(W_2^{s+m}(\T^n_\theta), W_2^s(\T^n_\theta)))$ for every $s\in \R$. 
 
 \item If $d<0$, then, for every $s\in \R$, we have 
          \begin{equation*}
             P(\lambda) \in \bigcap_{d\leq d'\leq 0}\Hol^{d'}\big(\Lambda; \cL\big(W_2^{s+m+w|d'|}(\T^n_\theta), W_2^s(\T^n_\theta)\big)\big).  
          \end{equation*}
          
 \item If $m\leq 0$, then $P(\lambda)\in\Hol^{\bar{d}}(\Lambda; \cL(L_2(\T^n_\theta) ))$ with $\bar{d}:=\max(d,mw^{-1})$. 
 \end{enumerate}
\end{proposition}

\subsection{The resolvent of an elliptic \psido} \label{sec:Resolvent}
From now on, let $P:C^\infty(\T^n_\theta)\rightarrow C^\infty(\T^n_\theta)$ be an elliptic \psido\ of order $w>0$ with symbol $\rho(\xi)\sim\sum_{j\geq 0}\rho_{w-j}(\xi)$.

\begin{definition}
The \emph{elliptic parameter set} of $P$ is 
\begin{align} 
\Theta(P) &= \C^*\setminus \biggl[ \bigcup_{\xi\in\Rn\setminus 0} \Sp(\rho_w(\xi))\biggr] \nonumber\\
& = \big\{\lambda \in \C^*; \ \text{$\rho_w(\xi)-\lambda$ is invertible for all $\xi\in \R^n\setminus 0$}\big\}. 
\label{eq:Resolvent.elliptic-parameter-set-definition}
\end{align}
Since $\Sp(t^w\rho_w(\xi)) = t^w\Sp(\rho_w(\xi))$ for all $t>0$, we see that $\Theta(P)$ is a cone in $\C^*$. This is also an open set in $\C^*$ (see~\cite{LP:Part1}). 
\end{definition}

\begin{example}
 Suppose that $\rho_w(\xi)$ is selfadjoint for all $\xi\in \R^n\setminus 0$ (e.g., $P$ is selfadjoint). Then $\Sp (\rho_w(\xi))\subseteq \R$ for all  $\xi\in \R^n\setminus 0$, and so $\Theta(P) \supseteq \C\setminus \R$. Alternatively, if $\rho_w(\xi)$ is positive in the sense of~(\ref{eq:Elliptic.positivity-criterion}), then $\Sp (\rho_w(\xi))\subseteq (0,\infty)$, so that $\Theta(P)$ contains $\C\setminus [0,\infty)$. As $\Theta (P)$ is a cone and $\Sp (\rho_w(\xi))$ cannot be empty, this in fact gives $\Theta(P)=\C\setminus [0,\infty)$. 
\end{example}

Throughout the rest of this section, we assume $\Theta(P)\neq \emptyset$. In the terminology of~\cite{LP:Part1}, this means that $P$ is \emph{elliptic with parameter}. In addition, we set $c:=\inf\{\norm{\rho_w(\xi)^{-1}}^{-1};\ |\xi|=1\}$, and define 
\begin{equation*}
\Omega_c(P) = \{ (\xi,\lambda)\in(\Rn\setminus 0)\times\C ; \ \text{$\lambda\in\Theta(P)$ or $|\lambda|<c|\xi|^w$} \} = \Omega_c(\Theta(P)).
\end{equation*}

It can be shown that $\rho_w(\xi)-\lambda$ is invertible for all $(\xi,\lambda)\in \Omega_c(P)$ and $(\rho_w(\xi)-\lambda)^{-1}\in S_{-w}^{-1}(\T^n_\theta\times\Omega_c(P))$ 
(see~\cite{LP:Part1}).

For any $R>0$, set $\Lambda_R = \Theta(P)\cup D(0,R)$.

\begin{proposition}[see~\cite{LP:Part1}]  \label{thm:Resolvent.parametrix-with-parameter}
Suppose that $P$ is elliptic with parameter. Then, for every $R>0$, the following hold. 
\begin{enumerate}
 \item $P-\lambda$ admits a parametrix $Q(\lambda)\in\Psi^{-w,-1}(\T^n_\theta;\Lambda_R)$ in the sense that
\begin{equation*}
(P-\lambda)Q(\lambda) = Q(\lambda)(P-\lambda)=1  \quad \bmod \quad \Psi^{-\infty,0}(\T^n_\theta;\Lambda_R) .
\end{equation*}

\item Any parametrix $Q(\lambda)\in \Psi^{-w,-1}(\T^n_\theta;\Lambda_R)$ as above has symbol $\sigma(\xi;\lambda)\sim \sum_{j\geq 0}  \sigma_{-w-j}(\xi;\lambda)$, where 
$ \sigma_{-w-j}(\xi;\lambda)\in S^{-1}_{-w-j}(\T^n_\theta\times\Omega_c(P)) $ is given by
\begin{gather}
 \sigma_{-w}(\xi;\lambda) = \big(\rho_w(\xi)-\lambda\big)^{-1} , 
\label{eq:Resolvent.symbol-resolvent1} \\
 \sigma_{-w-j}(\xi;\lambda) = -\sum_{\substack{k+l+|\alpha|=j \\ l<j}}\frac{1}{\alpha !}\big(\rho_w(\xi)-\lambda\big)^{-1} \partial_\xi^\alpha \rho_{w-k}(\xi) \delta^\alpha\sigma_{-w-l}(\xi;\lambda), \qquad j\geq 1 .
 \label{eq:Resolvent.symbol-resolvent2}
\end{gather}
\end{enumerate}
\end{proposition}

\begin{theorem}[see~\cite{LP:Part1}]  \label{thm:Resolvent.P-has-discrete-spectrum-resolvent-estimate}
Suppose that $P$ is elliptic with parameter.
\begin{enumerate}
\item The spectrum of $P$ is an unbounded discrete subset of $\C$ consisting of eigenvalues with finite multiplicity.

\item For any cone  $\Theta'$ such that $\overline{\Theta'}\setminus \{0\} \subseteq \Theta(P)$, the following hold. 
\begin{enumerate}
 \item $\Theta'$ contains at most finitely many eigenvalues of $P$. 
 
 \item There are $R_0>0$ and $C>0$ such that 
\begin{equation*}
\big\|(P-\lambda)^{-1} \big\| \leq C|\lambda|^{-1} \qquad \forall \lambda \in \Theta'\setminus D(0,R_0). 
\end{equation*}
\end{enumerate}
\end{enumerate}
\end{theorem}

\begin{definition}
 A ray $L\subseteq \C^*$ is called a \emph{ray of minimal growth} for $P$ when the following two conditions are satisfied:
 \begin{enumerate}
\item[(i)] $L$ does not contain any eigenvalue of $P$. 

\item[(ii)] $\|(P-\lambda)^{-1}\|=O(|\lambda|^{-1})$ as $\lambda$ goes to $\infty$ along $L$.  
\end{enumerate}
\end{definition}

\begin{example}
 If $P$ is selfadjoint, then every ray $L\subseteq \C\setminus \R$ is a ray of minimal growth.
\end{example}

\begin{corollary}[see~\cite{LP:Part1}] \label{cor:Resolvent.ray-with-no-eigenvalue-is-a-ray-of-minimal-growth}
The following hold. 
\begin{enumerate}
 \item For any cone  $\Theta'$ such that $\overline{\Theta'}\setminus \{0\} \subseteq \Theta(P)$, all but finitely many rays contained in $\Theta'$ are rays of minimal growth for $P$. 

 \item Any ray $L\subseteq \Theta(P)$ that does not contain any eigenvalue of $P$ is a ray of minimal growth. 
\end{enumerate}
\end{corollary}

We set
\begin{equation*}
\check{\Theta}(P) = \Theta(P)\setminus \{ t\lambda; \ t>0, \ \lambda\in\Sp(P) \} ,
\end{equation*}
i.e., $\Theta(P)$ with all rays through eigenvalues of $P$ removed. By Corollary~\ref{cor:Resolvent.ray-with-no-eigenvalue-is-a-ray-of-minimal-growth}, every ray in $\check{\Theta}(P)$ is a ray of minimal growth. It can be shown that $\check{\Theta}(P)$ is a non-empty open cone in $\C^*$ (see~\cite{LP:Part1}).

\begin{definition} \label{def:Resolvent.Agmon-pseudo-cone}
Set $r_0=\op{dist}(0,\Sp(P)\cap \C^*)$. The pseudo-cone $\Lambda(P)$ is defined by
\begin{equation}\label{eq:Resolvent.Agmon-pseudo-cone}
 \Lambda(P) = \left\{
\begin{array}{cl}
 \check{\Theta}(P)\cup D(0,r_0) & \text{if $0\not\in \Sp(P)$},\medskip \\
 \check{\Theta}(P)\cup \big[D(0,r_0)\setminus \{0\}\big]  & \text{if $0\in \Sp(P)$}.
\end{array}\right. 
\end{equation}
\end{definition}

\begin{theorem}[see~\cite{LP:Part1}] \label{thm:Resolvent.resolvent-is-psido-with-parameter}
Suppose that $P$ is elliptic with parameter. Then the following hold. 
\begin{enumerate}
 \item The resolvent $(P-\lambda)^{-1}$ belongs to $\Psi^{-w,-1}(\T^n_\theta;\Lambda(P))$. 

\item $(P-\lambda)^{-1}$ has symbol $\sigma(\xi;\lambda)\sim \sum \sigma_{-w-j}(\xi;\lambda)$, where $\sigma_{-w-j}(\xi;\lambda)$ is given by~(\ref{eq:Resolvent.symbol-resolvent1})--(\ref{eq:Resolvent.symbol-resolvent2}). 

\item\label{prop:Resolvent.Sobolev} For every $s\in \R$, we have
\begin{equation*}
 (P-\lambda)^{-1} \in \bigcap_{0\leq a \leq 1} \Hol^{-1+a}\left(\Lambda(P); \cL\big(W_2^s(\T^n_\theta), W_2^{s+aw}(\T^n_\theta)\big)\right). 
\end{equation*}
\end{enumerate}
\end{theorem}

\begin{remark}
 We refer to~\cite{LL:JNCG25} for a description of the resolvent in a version for NC tori of the weakly parametric pseudodifferential calculus of~\cite{GS:IM95}. 
\end{remark}
\section{Spectral Theory of Non-Selfadjoint Elliptic \psidos}\label{sec:Nonselfadjoint}
In this section, we study the spectral projections and partial inverses of non-selfadjoint elliptic \psidos. In addition to the results of~\cite{LP:Part1}, we also give a precise description of the singularity at the origin of the resolvent of an elliptic \psido. 

Throughout this section, we let $P\in \Psi^q(\T^n_\theta)$ be an elliptic \psido\ with $m:=\Re q>0$. We further assume that $\Sp(P)\neq \C$. Thanks to Proposition~\ref{prop:Spectral.spectrum-P}, we then know that $\Sp(P)$ is a discrete set consisting of isolated eigenvalues with finite multiplicity.
In addition, we denote by $\cK$ the Banach ideal of compact operators on $L_2(\T^n_\theta)$.  

\subsection{Generalized eigenspaces and spectral projections}

\begin{lemma} \label{lem:Spectral.resol-holomorphic}
 Set $\Omega=\C\setminus \Sp(P)$. Then the resolvent $\lambda \rightarrow (P-\lambda)^{-1}$ is a holomorphic map from $\Omega$ to $\cL(L_2(\T^n_\theta), W_2^m(\T^n_\theta) )$. In particular, it gives rise to a holomorphic map from $\Omega$ to $\cK$. 
\end{lemma}
 \begin{proof}
For $\lambda \in \Omega$, the operator $P-\lambda:W_2^m(\T^n_\theta)  \rightarrow L_2(\T^n_\theta)$ is an isomorphism, and so $\lambda \rightarrow (P-\lambda)^{-1}$ is a holomorphic map from $\Omega$ to $\cL(L_2(\T^n_\theta), W_2^m(\T^n_\theta) )$. Composing with the compact inclusion of $W_2^m(\T^n_\theta)$ into $L_2(\T^n_\theta)$ then yields a holomorphic family of compact operators. 
 \end{proof}

 Given any $\lambda \in \Sp(P)$, the \emph{generalized  eigenspace} $E_\lambda(P)$ and the \emph{spectral projection} (or \emph{Riesz projection}) $\Pi_\lambda(P)$ are defined by
 \begin{equation}\label{eq:ST.Riesz-root} 
 E_\lambda(P)= \bigcup_{\ell \geq 1}\ker (P-\lambda)^\ell, \qquad \Pi_\lambda(P)= \frac{1}{2i\pi} \int_{|\zeta -\lambda|=r} (\zeta-P)^{-1} d\zeta.  
\end{equation}
 Here $r>0$ is chosen so that $\lambda$ is the only eigenvalue of $P$ contained in the disk $|\zeta-\lambda|\leq r$. Thanks to Lemma~\ref{lem:Spectral.resol-holomorphic}, this contour integral is well defined in $\cL(L_2(\T^n_\theta), W_2^m(\T^n_\theta))$, independently of the choice of $r$, and via the compact inclusion of $W_2^m(\T^n_\theta)$ into $L_2(\T^n_\theta)$ we may also regard $\Pi_\lambda(P)$ as a compact operator on $L_2(\T^n_\theta)$.  
 
 It is well known (see, e.g., \cite{GK:AMS69, RN:AK52}) that we have 
\begin{equation*}
 \Pi_\lambda(P)^2= \Pi_\lambda(P), \qquad \Pi_\lambda(P)\Pi_\mu(P)= 0 \quad \text{if $\mu\neq \lambda$}. 
\end{equation*}
These show that the spectral projections $\Pi_\lambda(P)$, $\lambda \in \Sp(P)$, form a family of mutually disjoint projections. 

Recall that $\Sp (P^*)=\left\{ \overline{\lambda}; \ \lambda \in \Sp(P)\right\}$. Therefore, taking into account that complex conjugation reverses orientation, we get
 \begin{equation} \label{eq:Spectral.Rieszproj-adjoint}
 \Pi_\lambda(P)^*=  \frac{-1}{2i\pi} \int_{|\zeta -\lambda|=r} [(\zeta-P)^{-1}]^* d\overline{\zeta}= 
 \frac{1}{2i\pi} \int_{|\zeta -\overline{\lambda}|=r} (\zeta-P^*)^{-1} d\zeta = \Pi_{\overline{\lambda}}(P^*). 
\end{equation}

In what follows, we denote by $\dotplus$ the algebraic direct sum of linear subspaces. 
 
 \begin{proposition}\label{prop:Spectral.projection-lambda}
 Let $\lambda \in \Sp(P)$. The following hold.
 \begin{enumerate}
 \item $E_\lambda(P)$ is a finite-dimensional subspace of $C^\infty(\T^n_\theta)$. In particular, there is an integer $\nu\geq 1$ such that $E_\lambda(P)=\ker (P-\lambda)^\nu$. 
  
\item $\Pi_\lambda(P)$ is the projection onto $E_\lambda(P)$ with kernel $E_{\overline{\lambda}}(P^*)^\perp$. In particular, we have
 \begin{equation*}
 L_2(\T^n_\theta) = E_\lambda(P) \dotplus E_{\overline{\lambda}}(P^*)^\perp.
\end{equation*}

\item We have
\begin{equation}
 E_{\overline{\lambda}}(P^*)^\perp\supseteq\bigplus^{\cdot}_{\mu \in \Sp(P)\setminus \{\lambda\}} E_\mu(P). 
 \label{eq:Spectral.Eolambdaperp-Emu}
\end{equation}
\end{enumerate}
\end{proposition}
 \begin{proof}
 Let us first show by induction that $\ker (P-\lambda)^\ell \subseteq C^\infty(\T^n_\theta) $ for all $\ell \in \N$. By Corollary~\ref{cor:Elliptic.P-lambda-hypoell} this is true for $\ell=1$. Moreover, if $\ker(P-\lambda)^\ell \subseteq C^\infty(\T^n_\theta) $ and $u\in \ker(P-\lambda)^{\ell+1}$, then $(P-\lambda)u$ is contained in $\ker(P-\lambda)^\ell$, and so $(P-\lambda)u \in C^\infty(\T^n_\theta) $. As $P-\lambda$ is hypoelliptic in the sense of~(\ref{eq:PsiDOs.hypoellipticity}) (see the paragraph preceding Corollary~\ref{cor:Elliptic.P-lambda-hypoell}), we then deduce that $u\in C^\infty(\T^n_\theta) $. This proves our claim. 
 
 Let $\ell \in \N$. We observe that, for all $u\in C^\infty(\T^n_\theta) $ and $\zeta \in \C\setminus \Sp(P)$, we have 
 \begin{equation*}
 (\zeta -P)^{-1}u= \sum_{0\leq j<\ell} (\zeta-\lambda)^{-(j+1)} (P-\lambda)^{j}u + (\zeta-\lambda)^{-\ell}(\zeta-P)^{-1} (P-\lambda)^\ell u.
\end{equation*}
 Indeed, it can be checked that the right-hand side is a solution in $W_2^m(\T^n_\theta) $ of the equation $(\zeta -P)v=u$. In particular, if $u\in \ker (P-\lambda)^\ell$, then $ (\zeta -P)^{-1}u= \sum_{j<\ell} (\zeta-\lambda)^{-(j+1)} (P-\lambda)^{j}u$, and so we get 
  \begin{equation*}
 \Pi_\lambda(P)u =  \sum_{0\leq j<\ell} \frac{1}{2i\pi} \int_{|\zeta -\lambda|=r} (\zeta-\lambda)^{-(j+1)} (P-\lambda)^{j}u d\zeta =u.    
\end{equation*}
This shows that $\ker (P-\lambda)^\ell\subseteq \ran \Pi_\lambda(P)$ for all $\ell \in \N$, and hence $E_\lambda(P)\subseteq  \ran \Pi_\lambda(P)$. 

As $\Pi_\lambda(P)$ is a projection, its range agrees with the eigenspace $\ker(1-\Pi_\lambda(P))$. As mentioned above, $\Pi_\lambda(P)$ is a compact operator on $L_2(\T^n_\theta)$, and so the above eigenspace must have finite dimension. It then follows that $\ran\Pi_\lambda(P)$ has finite dimension, and so its subspace $E_\lambda(P)$ has finite dimension as well. As $E_\lambda(P)$ is the union of the non-decreasing sequence of the (finite-dimensional) subspaces $\ker(P-\lambda)^\ell$, it then follows that there must be some $\nu\in \N$ such that  $E_\lambda(P)=\ker (P-\lambda)^\nu$. 

As $(\zeta -P)^{-1}$, $\zeta \in \C\setminus \Sp(P)$, is a holomorphic family in $\cL(L_2(\T^n_\theta),W_2^m(\T^n_\theta))$, we may regard $P(\zeta -P)^{-1}$ and $(\zeta -P)^{-1}P$ as holomorphic families in $\cL(W_2^m(\T^n_\theta),L_2(\T^n_\theta) )$. These two families agree, since, writing $P=\zeta -(\zeta-P)$, both equal $\zeta(\zeta -P)^{-1}-1$ on $W_2^m(\T^n_\theta)$. Thus, 
\begin{equation}\label{eq:Spectral.PPil-PilP}
 P\Pi_\lambda(P)=  \frac{1}{2i\pi}   \!\!\int_{|\zeta -\lambda|=r} \!\!\! P(\zeta-P)^{-1}d\zeta=  \frac{1}{2i\pi}  \!\! \int_{|\zeta -\lambda|=r}\!\!\!  (\zeta-P)^{-1}P d\zeta=\Pi_\lambda(P)P. 
\end{equation}
It then follows that $P(\ran \Pi_\lambda(P))\subseteq \ran \Pi_\lambda(P)$, and so $P$ induces a linear operator,
\begin{equation*}
 \tilde{P}: \ran \Pi_\lambda(P)\longrightarrow \ran \Pi_\lambda(P).
\end{equation*}

Note that $\Sp(\tilde{P})\subseteq \Sp(P)$. In addition, $\ker (\tilde{P}-\lambda)=\ker (P-\lambda)$, since $\ker (P-\lambda)\subseteq  \ran \Pi_\lambda(P)$. Let $\mu \in \Sp(P)\setminus \{\lambda\}$. The fact that $\Pi_\lambda(P)$ is a projection and $\Pi_\lambda(P)\Pi_\mu(P)=0$ implies that $\ran \Pi_\mu(P) \cap\ran \Pi_\lambda(P)=\{0\}$.  As $\ker (P-\mu)\subseteq \ran \Pi_\mu(P)$, we see that $\ran \Pi_\lambda(P)\cap \ker (P-\mu)=\{0\}$, and so $\ker (\tilde{P}-\mu)= \ker (P-\mu) \cap \ran \Pi_\lambda(P)=\{0\}$. This shows  that $\lambda$ is the only eigenvalue of $\tilde{P}$. As $\ran \Pi_\lambda(P)$ has finite dimension, this implies that there is $\ell \in \N$ such that $(\tilde{P}-\lambda)^\ell=0$. This means that $\ran \Pi_\lambda(P)\subseteq \ker(P-\lambda)^\ell\subseteq E_\lambda(P)$. It then follows that $\ran \Pi_\lambda(P)=E_\lambda(P)$, and so $\Pi_\lambda(P)$ projects onto $E_\lambda(P)$. 

Applying the above considerations to the pair $(P^*,\overline{\lambda})$ shows that $\ran \Pi_{\overline{\lambda}}(P^*)=E_{\overline{\lambda}}(P^*)$. As we know by~(\ref{eq:Spectral.Rieszproj-adjoint}) that $\Pi_{\lambda}(P)^*= \Pi_{\overline{\lambda}}(P^*)$, we deduce that 
\begin{equation*}
 \ker \Pi_{\lambda}(P) = \left[\ran \Pi_{\lambda}(P)^*\right]^\perp =  \left[\ran  \Pi_{\overline{\lambda}}(P^*)\right]^\perp = E_{\overline{\lambda}}(P^*)^\perp.
\end{equation*}
 Therefore, $\Pi_\lambda(P)$ is precisely the projection onto $E_\lambda(P)$ with kernel $E_{\overline{\lambda}}(P^*)^\perp$. In particular, we have  the direct sum splitting $ L_2(\T^n_\theta) = E_\lambda(P) \dotplus E_{\overline{\lambda}}(P^*)^\perp$. 
 
Finally, given any $\mu \in \Sp(P)\setminus \{\lambda\}$, the equality $\Pi_\lambda(P) \Pi_\mu(P)=0$ also means that $\ran \Pi_\mu(P) \subseteq \ker \Pi_\lambda(P)$. As $\ran \Pi_\mu(P)=E_\mu(P)$ and $ \ker \Pi_\lambda(P)=E_{\overline{\lambda}}(P^*)^\perp$, this shows that $E_\mu(P)\subseteq E_{\overline{\lambda}}(P^*)^\perp$. This gives~(\ref{eq:Spectral.Eolambdaperp-Emu}). 
 \end{proof}

The proof yields the following. 

\begin{corollary}\label{cor:Spectral.nilpotent-lambda}
 Let $\lambda \in \Sp(P)$. The operator $P-\lambda$ induces on $E_\lambda(P)$ a nilpotent operator,
\begin{equation*}
 N_\lambda: E_\lambda(P) \longrightarrow E_\lambda(P). 
\end{equation*}
\end{corollary}
\begin{proof}
 It is shown in the proof of Proposition~\ref{prop:Spectral.projection-lambda} that $P$ induces on $\ran \Pi_\lambda(P)$ a linear operator $\tilde{P}$ such that $\tilde{P}-\lambda$ is nilpotent. As $\ran \Pi_\lambda(P)=E_\lambda(P)$, this exactly means that $P-\lambda$ induces on $E_\lambda(P)$ a nilpotent linear operator $N_\lambda: E_\lambda(P) \rightarrow E_\lambda(P)$.  
\end{proof}
In what follows, we shall denote by $\nu_\lambda$ the nilpotency index of $N_\lambda$, i.e., the smallest integer $\nu \geq 1$ such that $N_\lambda^\nu=0$.

Proposition~\ref{prop:Spectral.projection-lambda} also gives the following. 

\begin{corollary} \label{cor:Spectral.Riesz-projection-smoothing}
 All the spectral projections $\Pi_\lambda(P)$, $\lambda \in \Sp(P)$, are smoothing operators. 
\end{corollary}
\begin{proof}
Let $\lambda \in \Sp(P)$. Applying Proposition~\ref{prop:Spectral.projection-lambda} to $P^*$ shows that $E_{\overline{\lambda}}(P^*)$ is a finite-dimensional subspace of $C^\infty(\T^n_\theta)$. Thus, if we set $k=\dim E_{\overline{\lambda}}(P^*)$, then we can find elements $v_j \in C^\infty(\T^n_\theta) $, $j=1,\ldots, k$, such that $\{v_1, \ldots, v_k\}$ is an orthonormal basis of $E_{\overline{\lambda}}(P^*)$.  In particular, if we denote by $\Pi$ the orthogonal projection onto $E_{\overline{\lambda}}(P^*)$, then 
 $\Pi u= \sum_{j=1}^k \acoup{u}{v_j} v_j$ for all $u\in L_2(\T^n_\theta)$.  As $\ran (1-\Pi)= E_{\overline{\lambda}}(P^*)^\perp =\ker \Pi_\lambda(P)$, 
 we see that, for all $u\in L_2(\T^n_\theta)$, we have 
 \begin{equation} \label{eq:Spectral.Rieszproj-decomp}
 \Pi_\lambda(P)u=  \Pi_\lambda(P)\Pi u =  \sum_{j=1}^k \acoup{u}{v_j}\Pi_\lambda(P)v_j = \sum_{j=1}^k \acou{u}{v_j^*} \Pi_\lambda(P)v_j. 
\end{equation}
Note that $ \Pi_\lambda(P)v_j \in C^\infty(\T^n_\theta) $, since $\ran \Pi_\lambda(P) =E_\lambda(P)\subseteq C^\infty(\T^n_\theta) $. Moreover, $u\rightarrow \acou{u}{v_j^*}$ is a continuous linear form on $\scD'(\T^n_\theta)$. Therefore, $u\rightarrow \acou{u}{v_j^*} \Pi_\lambda(P)v_j$ maps continuously $\scD'(\T^n_\theta)$ to $C^\infty(\T^n_\theta)$, i.e., this is a smoothing operator. Combining this with~(\ref{eq:Spectral.Rieszproj-decomp}) shows that  $\Pi_\lambda(P)$ is a smoothing operator. 
 \end{proof}

\begin{remark} \label{rmk:Spectral.P-normal-root-space-Riesz-projection}
If $P$ is normal and  $\Sp(P)\neq \C$ (e.g., $P$ is selfadjoint), then, for all $\lambda \in \Sp(P)$,  we have 
\begin{equation*}
 E_\lambda(P)=E_{\overline{\lambda}}(P^*)=\ker (P-\lambda). 
\end{equation*}
Moreover, the Riesz projection $\Pi_\lambda(P)$ is the orthogonal projection onto $\ker (P-\lambda)$.
\end{remark}

\subsection{Partial inverses of elliptic \psidos} \label{subsec:PartialInverse}
In what follows, we let $\rho_q(\xi)\in S_q(\T^n_\theta\times\R^n)$ be the principal symbol of $P$. The ellipticity of $P$ precisely means that $\rho_q(\xi)$ is invertible for every $\xi\neq 0$.

If $0\not \in \Sp(P)$, then $P^{-1}$ is an operator in $\Psi^{-q}(\T^n_\theta)$ with principal symbol  $\rho_q(\xi)^{-1}$ (see~\cite[Proposition~11.7 and Corollary~11.8]{HLP:Part2}). We would like to explain how to define a ``partial inverse'' for $P$ when $0\in \Sp(P)$ in such a way as to obtain a \psido-parametrix for $P$. 

From now on, we assume that $0\in \Sp(P)$. 

\begin{lemma}\label{lem:Spectral.Po}
With respect to the splittings $W_2^m(\T^n_\theta) = E_0(P)\dotplus [E_{0}(P^*)^\perp \cap W_2^m(\T^n_\theta)]$ and $L_2(\T^n_\theta) =E_0(P)\dotplus E_0(P^*)^\perp$, the operator $P$ takes the form,
\begin{equation}\label{eq:Spectral.P-matrix0}
 P= 
\begin{pmatrix}
 N_0 & 0\\
 0 & P_0
\end{pmatrix},
\end{equation}
where $N_0:E_0(P)\rightarrow E_0(P)$ is the nilpotent operator of Corollary~\ref{cor:Spectral.nilpotent-lambda} and 
 \begin{equation*}
 P_0: E_0(P^*)^\perp \cap W_2^m(\T^n_\theta) \longrightarrow E_0(P^*)^\perp,
\end{equation*}
is a continuous bijection with continuous inverse. Moreover, $\Sp(P_0)=\Sp(P)\setminus \{0\}$, and $E_\lambda(P_0)=E_\lambda(P)$ for all $\lambda \in \Sp(P)\setminus \{0\}$.
\end{lemma}
\begin{proof}
 Recall that $1-\Pi_0(P)$ is the projection onto $ E_0(P^*)^\perp$ that annihilates $E_0(P)$. Moreover, this is a \psido\ of order~0, since $\Pi_0(P)$ is a smoothing operator. In particular, $1-\Pi_0(P)$ maps continuously $W_2^m(\T^n_\theta)$ to itself. Thus, if $u\in E_0(P^*)^\perp \cap W_2^m(\T^n_\theta)$, then $u=(1-\Pi_0(P))u$ is contained in $(1-\Pi_0(P))(W_2^m(\T^n_\theta))$. Conversely, if $u\in (1-\Pi_0(P))(W_2^m(\T^n_\theta))$, then $u$ is contained in $E_0(P^*)^\perp \cap W_2^m(\T^n_\theta)$. Therefore, we see that
 \begin{equation*}
  E_0(P^*)^\perp \cap W_2^m(\T^n_\theta) = \left(1-\Pi_0(P)\right)\!\left(W_2^m(\T^n_\theta)\right). 
\end{equation*}
As $\Pi_0(P)$ commutes with $P$, we then get
\begin{equation*}
 P\left(E_0(P^*)^\perp \cap W_2^m(\T^n_\theta)\right)=P\left(1-\Pi_0(P)\right)\!\left(W_2^m(\T^n_\theta)\right)=\left(1-\Pi_0(P)\right)P\!\left(W_2^m(\T^n_\theta)\right)\subseteq E_0(P^*)^\perp. 
\end{equation*}
This shows that $P$ induces a linear operator, 
 \begin{equation*}
 P_0: E_0(P^*)^\perp \cap W_2^m(\T^n_\theta) \longrightarrow E_0(P^*)^\perp. 
\end{equation*}

With respect to the splittings $W_2^m(\T^n_\theta) = E_0(P)\dotplus [E_{0}(P^*)^\perp \cap W_2^m(\T^n_\theta)]$ and $L_2(\T^n_\theta) =E_0(P)\dotplus E_0(P^*)^\perp$, we can write $P$ as the $2\times 2$ matrix~\eqref{eq:Spectral.P-matrix0}, where $N_0$ is the nilpotent operator induced by $P$ on $E_0(P)$ (\emph{cf}.~Corollary~\ref{cor:Spectral.nilpotent-lambda}). In terms of graphs this means that $G(P)=G(N_0)\dotplus G(P_0)$. As $G(P_0)=G(P)\cap\big(E_0(P^*)^\perp\times L_2(\T^n_\theta)\big)$ and $G(P)$ is closed, $G(P_0)$ is closed, i.e., $P_0$ is a closed operator. 

For the range and nullspace of $P_0$, the decomposition~(\ref{eq:Spectral.P-matrix0}) implies that 
\begin{equation*}
 \ker(P_0)=\ker (P)\cap E_0(P^*)^\perp \quad \textup{and} \quad  \ran(P_0)=\ran (P)\cap E_0(P^*)^\perp. 
\end{equation*}
As $\ker (P)\subseteq E_0(P)$ and $E_0(P)\cap E_0(P^*)^\perp=\{0\}$, we deduce that $\ker(P_0)=\{0\}$, i.e., $P_0$ is one-to-one. Moreover, as $\ran(P)$ is closed, we have $\ran(P)=\ker(P^*)^\perp\supseteq E_0(P^*)^\perp$, and hence $ \ran(P_0)=\ran (P)\cap E_0(P^*)^\perp=E_0(P^*)^\perp$, which shows that $P_0$ is onto. Therefore, we see that $P_0$ is a bijection. 

As $P_0$ agrees with the restriction of $P$ to $W_2^m(\T^n_\theta)\cap E_0(P^*)^\perp$, it is continuous with respect to the $W_2^m(\T^n_\theta)$-topology. Moreover, $W_2^m(\T^n_\theta)\cap E_0(P^*)^\perp$ is a closed subspace of $W_2^m(\T^n_\theta)$, since it is the range of the bounded projection induced by $1-\Pi_0(P)$ on $W_2^m(\T^n_\theta)$. In particular, this is a  Hilbert space with respect to the $W_2^m(\T^n_\theta)$-inner product. Therefore, $P_0$ is a continuous bijection between Hilbert spaces, and so this is an invertible operator with continuous inverse.

Finally, in view of~(\ref{eq:Spectral.P-matrix0}), given any $\lambda \in \C^*$ and $j\geq 1$, we have
\begin{equation*}
 (P-\lambda)^{j}= 
\begin{pmatrix}
 (N_0-\lambda)^j & 0\\
 0 & (P_0-\lambda)^j
\end{pmatrix}.
 \end{equation*}
 The operator $(N_0-\lambda)^j$ is always invertible, since $N_0$ is nilpotent and $\lambda\neq 0$. Therefore, we see that $(P-\lambda)^{j}$ is invertible if and only if  $(P_0-\lambda)^{j}$ is invertible. For $j=1$ this ensures that $\Sp(P)\setminus \{0\}=\Sp(P_0)\setminus \{0\}$. As $P_0$ is invertible, and so $0\not\in \Sp(P_0)$,  it follows that 
 $\Sp(P_0)=\Sp(P)\setminus \{0\}$. Moreover, it also follows from~(\ref{eq:Spectral.P-matrix0}) that, if $\lambda \in\Sp(P)\setminus \{0\}$, then  $\ker(P_0-\lambda)^j=\ker(P-\lambda)^j$ for all $j\geq 1$. This implies that $E_\lambda(P_0)=E_\lambda(P)$. The proof is complete. 
 \end{proof}

 \begin{definition}[Partial Inverse]
The \emph{partial inverse} of $P$ is the operator $P^{-1}: L_2(\T^n_\theta) \rightarrow  W_2^m(\T^n_\theta)$ given by 
\begin{equation}\label{eq:Spectral.partial-inverse-matrix}
 P^{-1}= 
\begin{pmatrix}
 0 & 0\\
 0 & P_0^{-1}
\end{pmatrix}.
\end{equation}
That is, $P^{-1}=0$ on $E_0(P)$ and $P^{-1}u=P_0^{-1}u$ for all $u \in E_{0}(P^*)^\perp$. 
  \end{definition}
 
\begin{remark}
If $P$ is normal and $\Sp(P)\neq \C$, then we know from Remark~\ref{rmk:Spectral.P-normal-root-space-Riesz-projection} that $E_0(P)=E_0(P^*)=\ker P$. Therefore, in this case, the partial inverse vanishes on $\ker P$ and inverts $P$ on $(\ker P)^\perp$, and so we recover the usual notion of partial inverse. 
\end{remark}
  
It follows from the definition of $P^{-1}$ that 
\begin{equation*}
 \ker \left(P^{-1}\right) =E_0(P) \qquad \text{and} \qquad \ran \left(P^{-1}\right)= \ran \left(P_0^{-1}\right)=E_0(P^*)^\perp\cap W_2^m(\T^n_\theta). 
\end{equation*}
Moreover, on $L_2(\T^n_\theta)=E_0(P)\dotplus E_0(P^*)^\perp$, we have
\begin{equation}\label{eq:Spectral.PP-1-P-1P1}
 PP^{-1}= \begin{pmatrix}
 N_0 & 0\\
 0 & P_0
\end{pmatrix} 
\begin{pmatrix}
 0 & 0\\
 0 & P_0^{-1}
\end{pmatrix}
= \begin{pmatrix}
 0 & 0\\
 0 & 1
\end{pmatrix} =1-\Pi_0(P). 
\end{equation}
The same computation on $W_2^m(\T^n_\theta)= E_0(P)\dotplus (W_2^m(\T^n_\theta)\cap E_0(P^*)^\perp)$ gives
\begin{equation}\label{eq:Spectral.PP-1-P-1P2}
 P^{-1}P=1-\Pi_0(P). 
\end{equation}

 These properties characterize $P^{-1}$. More precisely, we have the following result. 
 
\begin{lemma}\label{lem:Spectral.partial-inverse-characterization} 
 The partial inverse $P^{-1}$ is the unique linear operator $Q:L_2(\T^n_\theta) \rightarrow W_2^m(\T^n_\theta)$ such that
\begin{equation}
 E_0(P) \subseteq \ker Q  \qquad \text{and} \qquad QP=1-\Pi_0(P)\quad \text{on $W_2^m(\T^n_\theta)$}.
 \label{eq:Spectral.partial-inverse-characterization}  
\end{equation}
\end{lemma}
\begin{proof}
These properties are satisfied by $P^{-1}$. Conversely, let  $Q:L_2(\T^n_\theta) \rightarrow W_2^m(\T^n_\theta)$ be a linear operator satisfying~(\ref{eq:Spectral.partial-inverse-characterization}). The fact that $ E_0(P)\subseteq \ker Q$ ensures that $Q\Pi_0(P)=0$, and so $Q(1-\Pi_0(P))=Q$. Combining this with the equality $PP^{-1}=1-\Pi_0(P)$ shows that, on $L_2(\T^n_\theta)$, we have
\begin{equation}
 QPP^{-1}=Q\big(1-\Pi_0(P)\big)=Q. 
 \label{eq:Spectral.QPP-1}
\end{equation}

As $\ran P^{-1}=E_0(P^*)^\perp\cap W_2^m(\T^n_\theta)\subseteq \ker \Pi_0(P)$, we have $\Pi_0(P)P^{-1}=0$, and so $(1-\Pi_0(P))P^{-1}=P^{-1}$. Combining this with the equality $QP=1-\Pi_0(P)$ then gives
\begin{equation*}
 QPP^{-1}=\big(1-\Pi_0(P)\big)P^{-1}=P^{-1}.
\end{equation*}
Comparing this to~(\ref{eq:Spectral.QPP-1}) then shows that $Q=P^{-1}$. The proof is complete. 
\end{proof}

We have the following description of the partial inverse. 
 
\begin{proposition} \label{prop:Spectral.partial-inverse-basic-properties} The partial inverse
$P^{-1}$ is an operator in $\Psi^{-q}(\T^n_\theta)$. This is a parametrix of $P$, and so its symbol is given by~(\ref{eq:PsiDOs.symbol-parametrix1})--(\ref{eq:PsiDOs.symbol-parametrix2}). In particular, its  principal symbol is $\rho_q(\xi)^{-1}$. 
\end{proposition}
 \begin{proof}
 As $\Pi_0(P)$ is a smoothing operator, the equalities $PP^{-1}=1-\Pi_0(P)$ and $P^{-1}P=1-\Pi_0(P)$ mean that $P^{-1}$ is a parametrix for $P$. Therefore, by Proposition~\ref{prop:Elliptic.regularity}, we only need to show that $P^{-1}\in \Psi^{-q}(\T^n_\theta)$. 
 
 Let $Q$ be a parametrix in $\Psi^{-q}(\T^n_\theta)$, so that $QP=1-R_1$ and $PQ=1-R_2$, where  $R_1$ and $R_2$ are smoothing operators. As 
 $Q(1-\Pi_0(P))=QPP^{-1}=(1-R_1)P^{-1}$ on $L_2(\T^n_\theta)$, we get 
 \begin{equation*}
 P^{-1}=Q(1-\Pi_0(P))+R_1P^{-1}. 
\end{equation*}
As $R_1$ is a smoothing operator, the operator $R_1P^{-1}$ maps continuously $L_2(\T^n_\theta)$ to $C^\infty(\T^n_\theta)$, and so we see that $P^{-1}$ induces a continuous operator from $C^\infty(\T^n_\theta)$ to itself. 

On $C^\infty(\T^n_\theta)$ we also have $(1-\Pi_0(P))Q=P^{-1}PQ=P^{-1}(1-R_2)$, and so we get
\begin{equation*}
 P^{-1}=Q-\Pi_0(P)Q+P^{-1}R_2. 
\end{equation*}
 As the operators $\Pi_0(P)Q$ and $P^{-1}R_2$ are smoothing, this shows that $P^{-1}$ agrees with $Q$ up to a smoothing operator, and so this is a \psido\ of order~$-q$. The proof is complete. 
\end{proof}

Finally, we have the following properties of the partial inverse. 

\begin{proposition}\label{prop:Spectral.partial-inverse-adjoint-exponentiation} The following hold.
 \begin{enumerate}
 \item The partial inverse of $P^*$ is $(P^{-1})^*$. 

\item For every integer $k\geq 2$, the partial inverse of $P^k$ is $(P^{-1})^k$. 
\end{enumerate}
\end{proposition}
\begin{proof}
As $P^{-1}$ is an operator in $\Psi^{-q}(\T^n_\theta)$, its adjoint $(P^{-1})^*$ is an operator in $\Psi^{-\overline{q}}(\T^n_\theta)$. In particular, we see that $(P^{-1})^*$ maps $L_2(\T^n_\theta)$ to $W_2^m(\T^n_\theta)$. Moreover, as $\ran P^{-1}=E_0(P^*)^\perp\cap W_2^m(\T^n_\theta)$, we see that $\ran P^{-1}$ is a dense subspace of $E_0(P^*)^\perp$, and so $\ker (P^{-1})^*= (\ran P^{-1})^\perp=E_0(P^*)$. In addition, it follows from~(\ref{eq:Spectral.Rieszproj-adjoint}) that $\Pi_0(P^*)=\Pi_0(P)^*$. Combining this with the equality $PP^{-1}=1-\Pi_0(P)$ gives
\begin{equation*}
 (P^{-1})^*P^*=\big(PP^{-1}\big)^*=\big(1-\Pi_0(P)\big)^*=1-\Pi_0(P^*). 
\end{equation*}
 All this shows that $(P^{-1})^*$ satisfies the assumptions of~Lemma~\ref{lem:Spectral.partial-inverse-characterization}, and so this is the partial inverse of $P^*$. 
 
Let $k$ be an integer~$\geq 2$. Note that $(P^{-1})^k$ is an operator in $\Psi^{-kq}(\T^n_\theta)$, and so it maps $L_2(\T^n_\theta)$ to $W_2^{km}(\T^n_\theta)$. Furthermore, we have $E_0(P^k)=\bigcup_{j\geq 1} \ker P^{kj}=\bigcup_{j\geq 1}\ker P^j=E_0(P)$. Thus,
\begin{equation}
 E_0(P^k)=E_0(P)=\ker P^{-1}\subseteq \ker \big(P^{-1}\big)^k. 
 \label{eq:Spectral.E_0Pk-kernel-P-k}
\end{equation}
This also ensures that $\ran \Pi_0(P^k)=\ran \Pi_0(P)$. Likewise, we have $E_0((P^k)^*)=E_0((P^*)^k)=E_0(P^*)$, and so $\ker   \Pi_0(P^k)=E_0((P^k)^*)^\perp= E_0(P^*)^\perp=\ker \Pi_0(P)$. This shows that $\Pi_0(P^k)$ and $\Pi_0(P)$ are projections with the same range and the same kernel, and so they agree. Combining this with the equalities $P^{-1}P=PP^{-1}=1-\Pi_0(P)$, we get
\begin{equation*}
 \big(P^{-1}\big)^kP^k=\big(P^{-1}P\big)^k=\big(1-\Pi_0(P)\big)^k=1-\Pi_0(P)=1-\Pi_0\big(P^k\big).
\end{equation*}
Together with~(\ref{eq:Spectral.E_0Pk-kernel-P-k}) and Lemma~\ref{lem:Spectral.partial-inverse-characterization}, this shows that $(P^{-1})^k$ is the partial inverse of $P^k$. 
 \end{proof}

\begin{remark}
 Given any integer $k\geq 2$, since the partial inverse of $P^k$ and the power $(P^{-1})^k$ agree, we shall denote them by $P^{-k}$. 
\end{remark}

\subsection{Singularity of $(P-\lambda)^{-1}$ near $\lambda=0$} Assume now that $P$ has order $w>0$ and $\Theta(P)\neq \emptyset$, where $\Theta(P)$ is defined in~(\ref{eq:Resolvent.elliptic-parameter-set-definition}), and keep the notation of Section~\ref{sec:Parametric-PsiDOs}. We now describe the singularity at the origin of the resolvent $(P-\lambda)^{-1}$ in terms of our parametric \psido-calculus. 

Let $\Lambda(P)$ be the pseudo-cone defined in~(\ref{eq:Resolvent.Agmon-pseudo-cone}). By Theorem~\ref{thm:Resolvent.resolvent-is-psido-with-parameter}, we have $(P-\lambda)^{-1}\in \Psi^{-w,-1}(\T^n_\theta;\Lambda(P))$. If $0\not\in \Sp(P)$, then $\Lambda(P)$ contains the origin, and so $(P-\lambda)^{-1}$ is a holomorphic family of parametric \psidos\ near the origin. Suppose instead that $0\in \Sp(P)$. Then the origin is an isolated point of $\C\setminus \Lambda(P)$, and so $\Lambda_0(P):=\Lambda(P)\cup\{0\}$ is an open pseudo-cone containing the origin with conical part $\check{\Theta}(P)$. By Corollary~\ref{cor:Spectral.nilpotent-lambda}, $P$ induces a nilpotent operator $N_0:E_0(P)\rightarrow E_0(P)$. As above, we denote by $\nu_0$ its nilpotency index. We then have the following description of the resolvent $(P-\lambda)^{-1}$ near $\lambda=0$. 

\begin{proposition}\label{prop:Resolvent.Meromorphic}
 Assume that $0\in \Sp(P)$, and set $\Lambda_0(P)=\Lambda(P)\cup\{0\}$. The following hold. 
\begin{enumerate}
 \item $(1-\Pi_0(P))(P-\lambda)^{-1}\in \Psi^{-w,-1}(\T^n_\theta;\Lambda_0(P))$. 
 
 \item We have
 \begin{align}
 (P-\lambda)^{-1} &= -\sum_{1\leq j \leq \nu_0} \lambda^{-j} \Pi_0(P)P^{j-1} + \big(1-\Pi_0(P)\big)(P-\lambda)^{-1} \nonumber\\
 & = - \sum_{1\leq j \leq \nu_0} \lambda^{-j} \Pi_0(P)P^{j-1} \quad \bmod  \Psi^{-w,-1}\big(\T^n_\theta;\Lambda_0(P)\big).
 \label{eq:Resolvent.singularity} 
\end{align}
\end{enumerate}
\end{proposition}
\begin{proof}
As $N_0$ is a nilpotent operator of nilpotency index $\nu_0$, for all $\lambda\in \C^*$, we have
\begin{equation*}
 (N_0 -\lambda)^{-1}  =-\lambda^{-1}\big(1-\lambda^{-1}N_0\big)^{-1} 
  = - \sum_{1\leq j \leq \nu_0} \lambda^{-j}N_0^{j-1} . 
 \end{equation*}
Combining this with~(\ref{eq:Spectral.P-matrix0}), we get that, for all $\lambda \in \C\setminus \Sp(P)$, we have
 \begin{align}
 (P-\lambda)^{-1} & = 
 \begin{pmatrix}
 (N_0 -\lambda)^{-1}  & 0\\
 0 & \big(P_0-\lambda\big)^{-1}
\end{pmatrix} 
\nonumber\\ 
 & = - \sum_{1\leq j \leq \nu_0}  \lambda^{-j}
 \begin{pmatrix}
N_0^{j-1} & 0 \\
 0 & 0
\end{pmatrix}
 +
 \begin{pmatrix}
0 & 0 \\
 0 & (P_0-\lambda)^{-1}
\end{pmatrix}
\label{eq:Spectral.resolvent-splitting}\\ 
& = -\sum_{1\leq j \leq \nu_0} \lambda^{-j} \Pi_0(P)P^{j-1} + \big(1-\Pi_0(P)\big)(P-\lambda)^{-1}. 
\nonumber \end{align}
Therefore, to complete the proof we just need to show that $(1-\Pi_0(P))(P-\lambda)^{-1}\in \Psi^{-w,-1}(\T^n_\theta;\Lambda_0(P))$.

Let $\lambda_0$ be in $\Sp(P_0)=\Sp(P)\setminus \{0\}$, and set $\tilde{P}=P+\lambda_0\Pi_0(P)$. As $\tilde{P}$ and $P$ agree up to a smoothing operator, $\tilde{P}$ is an operator in $\Psi^w(\T^n_\theta)$ with the same homogeneous symbols as $P$. In particular, it has the same principal symbol, and so $\tilde{P}$ is elliptic with parameter and $\Theta(\tilde{P})=\Theta(P)$. Moreover, Theorem~\ref{thm:Resolvent.resolvent-is-psido-with-parameter} applies, and so $(\tilde{P}-\lambda)^{-1}\in \Psi^{-w,-1}(\T^n_\theta;\Lambda(\tilde{P}))$. Regarding $1-\Pi_0(P)$ as an element of $\Psi^{0,0}(\T^n_\theta;\Lambda(\tilde{P}))$, we then deduce that 
 $(1-\Pi_0(P))(\tilde{P}-\lambda)^{-1}\in \Psi^{-w,-1}(\T^n_\theta;\Lambda(\tilde{P}))$.
 
 In view of~(\ref{eq:Spectral.P-matrix0}), we have 
  \begin{equation*}
 \tilde{P}=P+\lambda_0\Pi_0(P)= 
\begin{pmatrix}
 N_0+\lambda_0 & 0\\
 0 & P_0
\end{pmatrix}.
\end{equation*}
It follows that $\Sp(\tilde{P})=\Sp(N_0+\lambda_0)\cup \Sp(P_0)$. As $ \Sp(N_0+\lambda_0)=\{\lambda_0\}\subseteq \Sp(P_0)$, we see that 
$\Sp(\tilde{P})=\Sp(P_0)=\Sp(P)\setminus \{0\}$. As $\Theta(\tilde{P})= \Theta(P)$, this ensures that $\Lambda(\tilde{P})=\Lambda_0(P)$, and so we see that 
 $(1-\Pi_0(P))(\tilde{P}-\lambda)^{-1}\in \Psi^{-w,-1}(\T^n_\theta;\Lambda_0(P))$. 
 
 In the same way as in~(\ref{eq:Spectral.resolvent-splitting}), we have 
 \begin{equation*}
 \big(\tilde{P}-\lambda\big)^{-1}  = 
 \begin{pmatrix}
 (N_0+\lambda_0 -\lambda)^{-1}  & 0\\
 0 & \big(P_0-\lambda\big)^{-1}
\end{pmatrix}. 
\end{equation*}
 Thus, 
 \begin{equation*}
\left(1-\Pi_0(P)\right) \big(\tilde{P}-\lambda\big)^{-1}= 
\begin{pmatrix}
0 & 0\\
 0 & \big(P_0-\lambda\big)^{-1}
\end{pmatrix}= \left(1-\Pi_0(P)\right) \big(P-\lambda\big)^{-1}. 
\end{equation*}
It then follows that $(1-\Pi_0(P))(P-\lambda)^{-1}\in \Psi^{-w,-1}(\T^n_\theta;\Lambda_0(P))$. This completes the proof. 
\end{proof}

\begin{remark}
 If we further assume that $P$ is normal, then~(\ref{eq:Resolvent.singularity}) reduces to 
\begin{equation*}
  (P-\lambda)^{-1}= -\lambda^{-1} \Pi_0(P) \quad \bmod  \Psi^{-w,-1}\big(\T^n_\theta;\Lambda_0(P)\big).
\end{equation*}
\end{remark}

\section{Holomorphic Families of \psidos}\label{sec:hol-PsiDOs}  
In this section, we give a comprehensive account of holomorphic families of \psidos. They have also been considered in~\cite{LNP:TAMS16, Po:SIGMA20}. We also introduce a new notion of $\Hol_\infty$-families of \psidos. These are families that are uniformly bounded along half-strips. This notion will be instrumental in establishing exponential growth bounds along vertical strips for complex powers of elliptic \psidos\ in the next section (sub-exponential growth for positive elliptic \psidos). 

\subsection{Preliminaries}
Suppose that $X$ is a locally compact space and $\sE$ is a locally convex TVS or a quasi-Banach space. We denote by $C(X; \sE)$ the space of continuous maps $u:X\rightarrow \sE$. We also denote by $L^\infty(X; \sE)$ the space of bounded measurable maps $u:X\rightarrow \sE$. If $\sE$ is locally convex, we equip  $L^\infty(X; \sE)$ with the locally convex TVS topology generated by the semi-norms, 
\begin{equation*}
  u \longrightarrow \sup_{z\in X} \fp\circ u(z), 
\end{equation*}
where $\fp$ ranges over continuous semi-norms on $\sE$. If $\sE$ is a quasi-Banach space with quasi-norm $\|\cdot\|_\sE$, then we 
equip  $L^\infty(X; \sE)$ with the quasi-norm, 
\begin{equation*}
  u \longrightarrow \sup_{z\in X} \left\|u(z)\right\|_\sE. 
\end{equation*}

We further let $L^\infty_\loc(X;\sE)$ be the space of measurable maps $u:X \rightarrow \sE$ that are bounded on compact sets $K\subseteq X$. If $\sE$ is locally convex, then we equip $L^\infty_\loc(X; \sE)$, and its subspace $C(X;\sE)$, with the locally convex topology generated by the semi-norms, 
\begin{equation*}
 u \longrightarrow \sup_{z\in K} \fp\circ u(z), \qquad K\subseteq X \ \text{compact}, 
\end{equation*}
where $\fp$ ranges over continuous semi-norms on $\sE$. Note that $C(X;\sE)$ is a closed subspace of $L^\infty_\loc(X; \sE)$. If $\Omega \subseteq \C$ is an open set, then $\Hol(\Omega;\sE)$ is a subspace of $L^\infty_\loc(\Omega;\sE)$ and  $C(\Omega;\sE)$, and we equip it with the topology defined by the same semi-norms.  

From now on, we assume that $\sE$ is locally convex and quasi-complete (i.e., every bounded Cauchy net is convergent), and denote by $\sE'$ its topological dual. Given any continuous map $f:[a,b]\rightarrow \sE$, its Riemann integral $\int_a^bf(t)dt$ is defined in the same way as the classical Riemann integral of a scalar function (see, e.g., \cite{FJ:JDE68, Ha:BAMS82}, see also~\cite[Appendix~B]{HLP:Part1}). The main properties of the Riemann integrals of continuous scalar functions extend \emph{verbatim} to continuous $\sE$-valued maps. In addition, we have the following properties. 

\begin{proposition}\label{prop:Hol.integral}
Let  $f:[a,b]\rightarrow \sE$ be a continuous map. The following hold.
\begin{enumerate}
 \item The integral $\int_a^b f(t)dt$ is the unique element of $\sE$ such that 
\begin{equation*}
 \varphi\bigg(\int_a^b f(t)dt\bigg) = \int_a^b \varphi\circ f(t)dt\qquad \forall \varphi \in \sE'. 
\end{equation*}

\item For any continuous semi-norm $\fp$ on $\sE$, we have
\begin{equation}
  \fp\bigg(\int_a^b f(t)dt\bigg) \leq \int_a^b \fp\circ f(t)dt. 
  \label{eq:Hol.integral-semi-norm}
\end{equation}

\item Given any continuous linear map $\Phi:\sE\rightarrow \sE_1$, where  $\sE_1$ is some locally convex space,  we have 
\begin{equation*}
 \int_a^b \Phi\circ f(t)dt =  \Phi\bigg(\int_a^b f(t)dt\bigg) . 
\end{equation*}
\end{enumerate}
\end{proposition}

\subsection{Vector-valued holomorphic maps} 
 Let $\Omega$ be an open subset of the complex plane $\C$.  We continue to assume that $\sE$ is locally convex and quasi-complete. 
 
 If $u:\Omega \rightarrow \sE$ is a continuous map and $\gamma:[a,b]\rightarrow \Omega$ is a $C^1$-curve, we define the contour integral $\int_\gamma u(\lambda)d\lambda$ in the usual way, i.e., 
\begin{equation*}
 \int_\gamma u(\lambda)d\lambda = \int_a^b \gamma'(t) u\circ \gamma(t) dt. 
\end{equation*}
If $\gamma$ is a regular curve, then the above contour integral is invariant under orientation-preserving  reparameterizations. Note also that~(\ref{eq:Hol.integral-semi-norm}) implies that, for any continuous semi-norm $\fp$ on $\sE$, we have
\begin{equation}
 \fp\bigg( \int_\gamma u(\lambda)d\lambda\bigg) \leq \sup_{\lambda \in \ran \gamma} \fp\big[u(\lambda)\big] \cdot \int_a^b |\gamma'(t)|dt. 
 \label{eq:Hol.integral-Taylor-ineq} 
\end{equation}

\begin{proposition}\label{prop:Hol.unique-analytic-continuation}
Suppose that $\Omega$ is connected, and let $u:\Omega \rightarrow \sE$ be a holomorphic map that vanishes on a subset $S$ that has an accumulation point in $\Omega$. Then $u(z)=0$ on all of $\Omega$. 
\end{proposition}
\begin{proof}
 For every $\varphi \in \sE'$, the function $\varphi[u(z)]$ is analytic on $\Omega$ and vanishes on $S$, and so it vanishes identically on $\Omega$. Thus, given any $z\in \Omega$, we have that $\varphi[u(z)]=0$ for all $\varphi \in \sE'$. As $\sE$ is locally convex, this implies that $u(z)=0$ for all $z\in \Omega$. 
\end{proof}

\begin{proposition}\label{prop:Hol.weakly-strong}
 Given a map $u:\Omega \rightarrow \sE$ the following are equivalent: 
\begin{enumerate}
 \item[(i)] The map $u(z)$ is weakly holomorphic on $\Omega$, i.e., for every  $\varphi \in \sE'$ the function $\varphi\circ u$ is holomorphic on $\Omega$.  
 
 \item[(ii)] The map $u(z)$ is continuous on $\Omega$ and, for any open disk $D$ such that $\overline{D}\subseteq \Omega$, we have
\begin{equation}
 u(z)= \frac{1}{2i \pi} \int_{\partial D} (\zeta-z)^{-1}u(\zeta) d\zeta \qquad \forall z \in D. 
 \label{eq:Hol.Cauchy-formula}
\end{equation}

\item[(iii)] The map $u(z)$ is holomorphic on $\Omega$. 
\end{enumerate}
\end{proposition}
\begin{proof}
 It is immediate that (iii) implies (i). The proof of~\cite[Theorem~3.31]{Ru:FA} shows that (ii) implies (iii). It also shows that if $u(z)$ is weakly holomorphic, then $u(z)$ is continuous. Note that, although~\cite[Theorem~3.31]{Ru:FA} is stated for Fr\'echet spaces, these parts of the proof only use the local convexity and quasi-completeness of the space, and so the arguments hold \emph{verbatim} for quasi-complete locally convex spaces.
 
In addition, if $u(z)$ is weakly holomorphic and continuous, then, for every $\varphi\in \sE'$, the function $\varphi\circ u$ satisfies Cauchy's formula. Thus, given any open disk $D$ such that $\overline{D}\subseteq \Omega$, for all $z\in D$, we have 
\begin{equation}\label{eq:Hol.Cauchy-formula-varphiu}
 \varphi\circ u(z) =  \frac{1}{2i \pi} \int_{\partial D} (\zeta-z)^{-1}\varphi\circ u(\zeta) d\zeta= \varphi\bigg( \frac{1}{2i \pi} \int_{\partial D} (\zeta-z)^{-1}u(\zeta) d\zeta\bigg). 
\end{equation}
It then follows that $u(z)$ satisfies Cauchy's formula~(\ref{eq:Hol.Cauchy-formula}). This shows that (i) implies (ii). 
\end{proof}

In what follows, we say that a family $\sF\subseteq \sE'$ is \emph{separating} if, given any distinct elements $x,y\in \sE$, we can find $\varphi\in \sF$ such that $\varphi(x)\neq \varphi(y)$. 

\begin{corollary}\label{cor:Hol.hol-separating} 
 Let $u:\Omega \rightarrow \sE$ be a continuous map. Suppose there is a separating family $\sF \subseteq \sE'$ such that, for every $\varphi \in \sF$, the function $\varphi\circ u(z)$ is holomorphic on $\Omega$. Then the map $u(z)$ is holomorphic on $\Omega$. 
\end{corollary}
\begin{proof}
This is a generalization of the second half of the proof of Proposition~\ref{prop:Hol.weakly-strong}. Let $D$ be an open disk such that $\overline{D}\subseteq \Omega$. Given any $\varphi \in \sF$, the function $\varphi\circ u$ satisfies Cauchy's formula. Thus, in the same way as in~(\ref{eq:Hol.Cauchy-formula-varphiu}), we have 
 \begin{equation*}
  \varphi\circ u(z) = \varphi\bigg( \frac{1}{2i \pi} \int_{\partial D} (\zeta-z)^{-1}u(\zeta) d\zeta\bigg). 
\end{equation*}
As the family $\sF$ is separating, it follows that $u(z)$ satisfies Cauchy's formula~(\ref{eq:Hol.Cauchy-formula}), and hence is holomorphic on $\Omega$ by Proposition~\ref{prop:Hol.weakly-strong}. 
\end{proof}

\begin{corollary}\label{cor:Hol.continuous-hol}
Let $u:\Omega \rightarrow \sE$ be a continuous map. Suppose there is a  continuous embedding of locally convex spaces $\iota:\sE\rightarrow \sE_1$ such that the map $\iota\circ u(z)$ is holomorphic on $\Omega$. Then the map $u(z)$ is holomorphic on $\Omega$. 
\end{corollary}
\begin{proof}
As $\iota:\sE\rightarrow \sE_1$ is a continuous embedding, $\iota^*\sE_1':=\{\varphi\circ \iota; \ \varphi\in \sE_1'\}$ is a separating family in $\sE'$. By assumption, the map $\iota\circ u(z)$ is holomorphic on $\Omega$, so for all $\varphi \in \sE_1'$, the function $\varphi \circ \iota\circ u(z)$ is holomorphic on $\Omega$. 
That is, $\psi \circ u(z)$ is holomorphic on $\Omega$ for all $\psi \in \iota^*\sE_1'$. It then follows from Corollary~\ref{cor:Hol.hol-separating} that the map $u(z)$ is holomorphic on $\Omega$. 
\end{proof}

\begin{corollary}
 $\Hol(\Omega;\sE)$ is a closed subspace of $C(\Omega;\sE)$ (and hence is closed in $L^\infty_\loc(\Omega;\sE)$). 
\end{corollary}
\begin{proof}
 Let $(u_\alpha)\subseteq \Hol(\Omega;\sE)$ be a net converging to $u(z)$ in $C(\Omega;\sE)$. Then, for every $\varphi\in \sE'$, the net $(\varphi\circ u_\alpha)$ converges to $\varphi \circ u(z)$ in $C(\Omega)$, and hence $\varphi \circ u(z)$ is holomorphic on $\Omega$. Thus, the map $u(z)$ is weakly holomorphic on $\Omega$, and hence is holomorphic on $\Omega$ by Proposition~\ref{prop:Hol.weakly-strong}. This proves the result. 
\end{proof}

\subsection{Holomorphic families of standard symbols} 
We shall now look at holomorphic families of symbols parametrized by $\Omega$. In what follows, we denote by $\B^n$ the open unit ball in $\R^n$ and set $\U^n=\R^n\setminus \B^n$. 

For families of standard symbols we have the following refinement of Corollary~\ref{cor:Hol.continuous-hol}. 

\begin{lemma}\label{lem:Hol.loc-bdd-hol}
 Let $\rho(z)$, ${z\in \Omega}$,  be a locally bounded family in $\stS^m(\T^n_\theta\times\R^{n})$, $m\in \R$, which is holomorphic as a family in $C^\infty(\T^n_\theta\times\R^{n})$. Then $\rho(z)$, ${z\in \Omega}$,  is a holomorphic family in $\stS^m(\T^n_\theta\times\R^{n})$. 
\end{lemma}
\begin{proof}
 We have a continuous inclusion of $\stS^m(\T^n_\theta\times\R^{n})$ into $C^\infty(\T^n_\theta\times\R^{n})$. Thus, by Corollary~\ref{cor:Hol.continuous-hol}, it is enough to show that $\rho(z)$, ${z\in \Omega}$,  is a continuous family in $\stS^m(\T^n_\theta\times\R^{n})$. 
 
 Let $z_0\in \Omega$, and let $D$ be an open disk about $z_0$ of radius $r>0$ such that $\overline{D}\subseteq \Omega$. As $\rho(z)$, ${z\in \Omega}$,  is a holomorphic family in $C^\infty(\T^n_\theta\times\R^{n})$, by using Proposition~\ref{prop:Hol.integral} and Proposition~\ref{prop:Hol.weakly-strong}, we see that, given any multi-orders $\alpha, \beta$, we have
\begin{equation*}
 \delta^\alpha \partial_\xi^\beta\rho(z)(\xi)= \frac{1}{2i\pi} \int_{\partial D} (\zeta-z)^{-1}\delta^\alpha \partial_\xi^\beta\rho(\zeta)(\xi) d\zeta \qquad  \forall z\in D\ \forall \xi \in \R^n. 
\end{equation*}
Thus, if $|h|<r$, then, for all $\xi\in \R^n$, we have
\begin{equation*}
 \delta^\alpha \partial_\xi^\beta\big[\rho(z_0+h)(\xi)-\rho(z_0)(\xi)\big]= 
 \frac{1}{2i\pi} \int_{\partial D}\big[ (\zeta-z_0-h)^{-1}-(\zeta-z_0)^{-1}\big]\delta^\alpha \partial_\xi^\beta\rho(\zeta)(\xi) d\zeta. 
\end{equation*}
As $\rho(z)$, ${z\in \Omega}$,  is locally bounded in $\stS^m(\T^n_\theta\times\R^{n})$, we have 
\begin{equation*}
 \big\|\delta^\alpha \partial_\xi^\beta\rho(\zeta)(\xi)\big\| \leq C_{D\alpha\beta}\big(1+|\xi|\big)^{m-|\beta|} \qquad \forall \zeta \in \partial D \ \forall \xi \in \R^n. 
\end{equation*}
Combining all this with~(\ref{eq:Hol.integral-Taylor-ineq}), we deduce that, for $|h|<r/2$, we have 
\begin{align*}
 \big\| \delta^\alpha \partial_\xi^\beta[\rho(z_0+h)(\xi)-\rho(z_0)(\xi)]\big\| &\leq C_{D\alpha\beta}\big(1+|\xi|\big)^{m-|\beta|}
 \sup_{\zeta \in  \partial D} \big|(\zeta -z_0-h)^{-1}-(\zeta -z_0)^{-1}\big| \\
 &\leq  C_{D\alpha\beta}|h|\big(1+|\xi|\big)^{m-|\beta|}. 
\end{align*}
This implies that $\rho(z_0+h)\rightarrow \rho(z_0)$ in $\stS^m(\T^n_\theta\times\R^{n})$ as $h\rightarrow 0$ for all $z_0\in \Omega$. That is, $\rho(z)$, ${z\in \Omega}$,  is a continuous family in $\stS^m(\T^n_\theta\times\R^{n})$. 
\end{proof}

\begin{remark}
The following basic facts about holomorphic families in $\stS^m(\T^n_\theta\times\R^{n})$ are worth recording. 
\begin{enumerate}
\item[(a)]\label{rmk:Hol.hol-standard-symb-U} If  we already know that $\rho(z)$, ${z\in \Omega}$,  is a locally bounded family in $C^\infty(\T^n_\theta\times\R^{n})$, then in order to check that this is a locally bounded family in $\stS^m(\T^n_\theta\times\R^{n})$ it is enough to show that, for any compact $K\subseteq \Omega$ and any multi-orders $\alpha$ and $\beta$, we have
 \begin{equation}
 \big\| \delta^\alpha \partial_\xi^\beta \rho(z)(\xi)\big\| \leq C_{K\alpha\beta}|\xi|^{m-|\beta|} \qquad \forall (z,\xi)\in K\times \U^n.
 \label{eq:Hol.hol-standard-symb-U} 
\end{equation}

\item[(b)] If $\rho(z)$, ${z\in \Omega}$,  is a holomorphic family in $\stS^{m}(\T^n_\theta\times\R^{n})$, then $\delta^\alpha\partial_\xi^\beta \rho(z)$, ${z\in \Omega}$,  is a holomorphic family in $\stS^{m-|\beta|}(\T^n_\theta\times\R^{n})$ for all multi-orders $\alpha$ and $\beta$. 

\item[(c)]\label{rmk:Hol.product-hol} The product of $C(\T^n_\theta)$ gives rise to continuous bilinear maps $\stS^{m_1}(\T^n_\theta\times\R^{n})\times \stS^{m_2}(\T^n_\theta\times\R^{n})\rightarrow \stS^{m_1+m_2}(\T^n_\theta\times\R^{n})$ (see~\cite{HLP:Part1}). It then follows that if $\rho_{j}(z)$, ${z\in \Omega}$, $j=1,2$, are holomorphic families in $\stS^{m_j}(\T^n_\theta\times\R^{n})$, then $\rho_1(z)\rho_2(z)$, ${z\in \Omega}$,  is a holomorphic family in $\stS^{m_1+m_2}(\T^n_\theta\times\R^{n})$.
\end{enumerate}
\end{remark}

\begin{definition}
 Given holomorphic families $\rho(z)$, $z\in \Omega$, in $C^\infty(\T^n_\theta\times\R^{n})$ and  $\rho_{m-j}(z)$, $z\in \Omega$, in $\stS^{m-j}(\T^n_\theta\times\R^{n})$, $j\geq 0$, we shall write $\rho(z)\sim \sum \rho_{m-j}(z)$ if 
\begin{equation}
 \rho(z) - \sum_{j<N} \rho_{m-j}(z)\in L^\infty_\loc\big(\Omega; \stS^{m-N}(\T^n_\theta\times\R^{n})\big) \qquad \forall N\geq 0.
 \label{eq:Hol.asymptotic-standard} 
\end{equation}
\end{definition}

\begin{remark}\label{rmk:Hol.asymptotic-standard-stS} 
 In view of Lemma~\ref{lem:Hol.loc-bdd-hol}, the condition~(\ref{eq:Hol.asymptotic-standard}) implies that
 \begin{equation*}
 \rho(z) - \sum_{j<N} \rho_{m-j}(z)\in \Hol\big(\Omega; \stS^{m-N}(\T^n_\theta\times\R^{n})\big) \qquad \forall N\geq 0. 
\end{equation*}
In particular, $\rho(z)$, ${z\in \Omega}$,  is a holomorphic family in $\stS^{m}(\T^n_\theta\times\R^{n})$. 
\end{remark}

\begin{lemma}\label{lem:Hol.asymptotic-standard-NJ}
Suppose we are given holomorphic families $\rho(z)$, $z\in \Omega$, in $C^\infty(\T^n_\theta\times\R^{n})$ and  $\rho_{m-j}(z)$, $z\in \Omega$, in $\stS^{m-j}(\T^n_\theta\times\R^{n})$, $j\geq 0$. The following are equivalent:
\begin{enumerate}
\item[(i)] $\rho(z) \sim \sum \rho_{m-j}(z)$ in the sense of~(\ref{eq:Hol.asymptotic-standard}). 

\item[(ii)] For every $N\geq 0$, there is $J_0\geq 0$ such that 
\begin{equation}
 \rho(z) - \sum_{j<J} \rho_{m-j}(z)\in L^\infty_\loc\big(\Omega; \stS^{-N}(\T^n_\theta\times\R^{n})\big) \qquad \forall J\geq J_0.
  \label{eq:Hol.asymptotic-standardNJ}
\end{equation}
 \end{enumerate}
 \end{lemma}
\begin{proof}
 It is immediate that (i) implies (ii), so we only have to prove the converse. This is a standard argument. Assume that~(\ref{eq:Hol.asymptotic-standardNJ}) is satisfied. Let $N\geq 0$. There is $J_0$ such that
 \begin{equation*}
 \rho(z) - \sum_{j<J} \rho_{m-j}(z)\in L^\infty_\loc\big(\Omega; \stS^{m-N}(\T^n_\theta\times\R^{n})\big) \qquad \forall J\geq J_0. 
\end{equation*}
Set $J=\max(N,J_0)$. If  $N\leq j<J$, then $ \rho_{m-j}(z)\in L^\infty_\loc(\Omega; \stS^{m-N}(\T^n_\theta\times\R^{n}))$, since we have a continuous inclusion of 
$\stS^{m-j}(\T^n_\theta\times\R^{n})$ into $\stS^{m-N}(\T^n_\theta\times\R^{n})$. Thus, 
\begin{equation*}
  \rho(z) - \sum_{j<N} \rho_{m-j}(z)=  \rho(z) - \sum_{j<J} \rho_{m-j}(z) +   \sum_{N\leq j<J} \rho_{m-j}(z)\in  L^\infty_\loc\big(\Omega; \stS^{m-N}(\T^n_\theta\times\R^{n})\big).
\end{equation*}
This shows that $\rho(z) \sim \sum \rho_{m-j}(z)$ in the sense of~(\ref{eq:Hol.asymptotic-standard}), and hence (ii) implies (i). 
\end{proof}

\subsection{Holomorphic families of classical symbols} 
In what follows, we set 
\begin{equation*}
 S^\bt(\T^n_\theta\times\R^{n})=\bigcup_{q\in \C} S^q(\T^n_\theta\times\R^{n}). 
\end{equation*}
This is the class of all classical symbols on $\T^n_\theta$. 

\begin{definition}[see also~\cite{LNP:TAMS16, Po:SIGMA20}] \label{def:Hol.hol-family-classical-symbols}
We say that a family $\rho(z)(\xi)$, $z\in\Omega$, in $S^\bt(\T^n_\theta\times\Rn)$ is \emph{holomorphic} when the following conditions are satisfied:
\begin{enumerate}
\item[(i)] There is a holomorphic function $w(z)$ on $\Omega$ such that $\rho(z)(\xi) \in S^{w(z)}(\T^n_\theta\times\R^{n})$ for all $z \in \Omega$.

\item[(ii)] $\rho(z)$, ${z\in \Omega}$,  is a holomorphic family in $C^\infty(\T^n_\theta\times\Rn)$.

\item[(iii)] We have $\rho(z)\sim\sum_{j\geq 0}\rho_j(z)$ with $\rho_j(z)\in S_{w(z)-j}(\T^n_\theta\times\Rn)$, in the sense that, for any $N\geq 0$, any compact $K\subseteq \Omega$ and any multi-orders $\alpha$, $\beta$, we have 
\begin{equation} \label{eq:Hol.hol-family-classical-symbol-estimate}
\Big\| \delta^\alpha\partial_\xi^\beta \Big(\rho(z)-\sum_{j<N}\rho_j(z) \Big)(\xi) \Big\| \leq C_{NK\alpha\beta}|\xi|^{\Re{w(z)}-N-|\beta|}  \quad \text{for all}\ (z,\xi)\in K\times \U^n.
\end{equation}
\end{enumerate}
We denote by $\Hol(\Omega;S^\bt(\T^n_\theta \times \R^n))$ the class of holomorphic families in $S^\bt(\T^n_\theta \times \R^n)$. We call $w(z)$ the order (function) of the family $\rho(z)$, ${z\in \Omega}$. It is not unique, and may be replaced by $w(z)+\ell$ for any integer $\ell\geq 1$. 
\end{definition}

\begin{remark}\label{rmk:Hol.derivatives-hol}
 If $\rho(z)$, ${z\in \Omega}$,  is a holomorphic family in $S^\bt(\T^n_\theta\times\R^{n})$, then $\delta^\alpha\dl_\xi^\beta\rho(z)$, ${z\in \Omega}$,  is a holomorphic family in $S^\bt(\T^n_\theta\times\R^{n})$  for all multi-orders $\alpha$ and $\beta$. 
\end{remark}

\begin{remark}\label{rmk:Hol.holfamily-classical-standard}
Let $\rho(z)$, ${z\in \Omega}$,  be a holomorphic family in $S^\bt(\T^n_\theta\times\R^{n})$. If $\Re w(z)\leq m$ on $\Omega$, then (iii) with $N=0$ implies that $\rho(z)$, ${z\in \Omega}$,  satisfies estimates of the form~(\ref{eq:Hol.hol-standard-symb-U}). As $\rho(z)$, ${z\in \Omega}$,  is a holomorphic family in $C^\infty( \T^n_\theta\times\Rn)$, it then follows from Remark~\ref{rmk:Hol.hol-standard-symb-U} and Lemma~\ref{lem:Hol.loc-bdd-hol} that $\rho(z)$, ${z\in \Omega}$,  is a holomorphic family in $\stS^m(\T^n_\theta\times\R^{n})$. 
\end{remark}

\begin{remark}\label{rmk:Hol.holfamily-bounded-symbols}
It also follows from (iii) that, if $K\subseteq \Omega$ is compact and we set $w_K=\sup\{\Re w(z); \ z\in K\}$, then $\rho(z)$, $z\in K$, is a bounded family in $\bS^{w_K}(\T^n_\theta\times \R^n)$. 
\end{remark}

\begin{remark}\label{rmk:Hol.reduction-precompact}
 A family $\rho(z)$, $z\in \Omega$, in $S^\bt(\T^n_\theta\times\Rn)$ is holomorphic if and only if it is holomorphic near any point $z_0\in \Omega$. Thus,  restricting ourselves to precompact neighborhoods of points in $\Omega$ allows us to reduce to the case where $\Re w(z)\leq m$ for some $m\in \R$. 
\end{remark}

\begin{definition}
 A \emph{holomorphic family} in $S_\bt(\T^n_\theta\times \R^n)$ of degree $w(z)$ is a  holomorphic family $\rho(z)$, $z\in\Omega$, in 
 $C^\infty(\T^n_\theta\times (\Rn\setminus 0))$ such that $\rho(z)\in S_{w(z)}(\T^n_\theta\times \R^n)$ for all $z\in \Omega$, where $w(z)$ is a holomorphic function on $\Omega$. 
\end{definition}

\begin{lemma} \label{lem:Hol.homogeneous-symbol-holomorphic}
If $\rho(z)$, ${z\in \Omega}$,  is a holomorphic family in $S^\bt(\T^n_\theta\times \R^n)$ of order $w(z)$ and $\rho(z)\sim\sum \rho_j(z)$ in the sense of~(\ref{eq:Hol.hol-family-classical-symbol-estimate}), then $\rho_{j}(z)$, $z\in\Omega$, is a holomorphic family in  $S_\bt(\T^n_\theta\times \R^n)$ of degree $w(z)-j$ for all $j\geq 0$. 
\end{lemma}
\begin{proof} We only need to show that $\rho_{j}(z)$, $z\in\Omega$, is a holomorphic family in $C^\infty(\T^n_\theta\times (\Rn\setminus 0))$ for all $j\geq 0$. The estimate~(\ref{eq:Hol.hol-family-classical-symbol-estimate}) for $N=1$ and $\alpha=\beta=0$ means that, for any compact $K\subseteq \Omega$, we have
\begin{equation*}
 \big\| \rho(z)(\xi)-\rho_0(z)(\xi) \big\| \leq C_{K}|\xi|^{\Re{w(z)}-1}  \quad \forall (z,\xi)\in K\times \U^n.
\end{equation*}
Let $L\subseteq \R^n\setminus 0$ be compact. If $\lambda>0$ is large enough, then $\lambda L\subseteq \U^n$, and so we get 
\begin{equation*}
  \big\| \rho(z)(\lambda\xi)-\rho_0(z)(\lambda \xi) \big\| \leq C_{K}|\lambda\xi|^{\Re{w(z)}-1}  \quad \forall  (z,\xi)\in K\times L.
\end{equation*}
 As $\rho_0(z)(\lambda \xi)= \lambda^{w(z)}\rho_0(z)(\xi)$ we obtain
 \begin{equation*}
  \big\| \lambda^{-w(z)}\rho(z)(\lambda\xi)-\rho_0(z)( \xi) \big\| \leq C_{KL}\lambda^{-1}  \quad \forall (z,\xi)\in K\times L.
\end{equation*}
Similarly, given any multi-orders $\alpha$ and $\beta$, we have 
\begin{align*}
  \big\| \delta^\alpha \partial_\xi^\beta \big[ \lambda^{-w(z)}\rho(z)(\lambda\xi)-\rho_0(z)(\xi)\big] \big\| & =
   \big\|  \lambda^{-w(z)+|\beta|}\delta^\alpha \partial_\xi^\beta\rho(z)(\lambda\xi)-\delta^\alpha \partial_\xi^\beta\rho_0(z)(\xi)\big\| \\
  & \leq C_{KL\alpha\beta}\lambda^{-1} \qquad \forall (z,\xi)\in K\times L.
\end{align*}
This shows that $\lambda^{-w(z)}\rho(z)(\lambda\xi)$ converges to $\rho_0(z)$ in $C^\infty(\T^n_\theta\times(\R^n\setminus 0))$ uniformly on compact subsets of $\Omega$ as $\lambda \rightarrow \infty$. That is, it converges in $L^\infty_\loc(\Omega; C^\infty(\T^n_\theta\times(\R^n\setminus 0)))$. 
As $\Hol(\Omega;C^\infty(\T^n_\theta\times(\R^n\setminus 0)))$ is a closed subspace of $L^\infty_\loc(\Omega; C^\infty(\T^n_\theta\times(\R^n\setminus 0)))$, it follows that $\rho_0(z)$, ${z\in \Omega}$,  is a holomorphic family in $C^\infty(\T^n_\theta\times(\R^n\setminus 0))$. 

It can be similarly shown that, for all $j\geq 1$, uniformly on compact subsets of $\Omega$, 
\begin{equation*}
 \lambda^{-w(z)+j}\Big( \rho(z)(\lambda \xi) -\sum_{\ell <j} \rho_{\ell}(z)(\lambda \xi) \Big) \longrightarrow \rho_j(z)(\xi) \qquad \text{in}\  C^\infty\big(\T^n_\theta\times(\R^n\setminus 0)\big). 
\end{equation*}
An induction on $j$ then shows that $\rho_j(z)$, ${z\in \Omega}$,  is a holomorphic family in $C^\infty(\T^n_\theta\times(\R^n\setminus 0))$ for all $j\geq 0$. 
\end{proof}

\begin{remark} \label{rem:Hol.holfamily-homogeneous-cutoff-classical}
Let $\rho(z)$, $z\in\Omega$, be a holomorphic family in $C^\infty(\T^n_\theta\times(\R^n\setminus 0))$. Suppose that $\rho(z)(\xi)$ is homogeneous of degree $w(z)$, where $w(z)$ is a holomorphic function on $\Omega$. In addition, let $\chi(\xi)\in C_c^\infty(\Rn)$ be such that $\chi(\xi)=1$ near $\xi=0$ and $\supp \chi \subseteq \B^n$. Set
\begin{equation*}
 \tilde{\rho}(z)(\xi)=(1-\chi(\xi)) \rho(z)(\xi), \qquad z\in \Omega, \ \xi\in \R^n. 
\end{equation*}
Then $ \tilde{\rho}(z)$ is a holomorphic family  in $C^\infty(\T^n_\theta\times\R^{n})$ and $ \tilde{\rho}(z)\sim \rho(z)$ in the sense of~(\ref{eq:Hol.hol-family-classical-symbol-estimate}). Thus, $ \tilde{\rho}(z)$ is a holomorphic family in $S^\bt(\T^n_\theta\times\R^{n})$. In particular, this is a holomorphic family in $\stS^m(\T^n_\theta\times\R^{n})$ if $\Re w(z)\leq m$ on $\Omega$. 
\end{remark}

\begin{example}\label{ex:Hol.powers}
 Suppose that $\Omega=\C$ and let $\chi(\xi)\in C_c^\infty(\R^n)$ be as in Remark~\ref{rem:Hol.holfamily-homogeneous-cutoff-classical}. Set 
 $\rho(z)(\xi)=(1-\chi(\xi))|\xi|^{z}$, $\xi\in \R^n$, $z\in \C$. Then $\rho(z)$, ${z\in \Omega}$,  is a holomorphic family in $S^\bt(\T^n_\theta\times \R^n)$ such that $\rho(z)(\xi)\sim |\xi|^{z}$ in the sense of~(\ref{eq:Hol.hol-family-classical-symbol-estimate}). In particular, $\rho(z)(\xi)$ is a classical symbol of order $z$ for all $z\in \C$. 
\end{example}

\begin{lemma} \label{lem:Hol.asymptotic-classical-standard}
Let $w(z)$ be a holomorphic function on $\Omega$ such that $\Re w(z)\leq m$, and let $\rho(z)$, ${z\in \Omega}$,  be a holomorphic family in $C^\infty(\T^n_\theta\times \R^n)$. In addition, let $\chi\in C^\infty_c(\R^n)$ be as in Remark~\ref{rem:Hol.holfamily-homogeneous-cutoff-classical}. The following are equivalent: 
\begin{enumerate}
\item[(i)] $\rho(z)(\xi)\sim\sum_{j\geq 0}\rho_j(z)(\xi)$ in the sense of~(\ref{eq:Hol.hol-family-classical-symbol-estimate}).

\item[(ii)] Each family $\rho_j(z)$, ${z\in \Omega}$, $j\geq 0$, is a holomorphic family in $C^\infty(\T^n_\theta\times(\R^n\setminus 0))$, and  $\rho(z)(\xi)\sim\sum_{j\geq 0}(1-\chi(\xi))\rho_j(z)(\xi)$ in the sense of~(\ref{eq:Hol.asymptotic-standard}).
\end{enumerate}
\end{lemma}
\begin{proof}
Assume that $\rho(z)(\xi)\sim\sum_{j\geq 0}\rho_j(z)(\xi)$ in the sense of~(\ref{eq:Hol.hol-family-classical-symbol-estimate}), where $\rho_j(z)\in S_{w(z)-j}(\T^n_\theta\times\R^{n})$. For $j=0,1,\ldots $ set $\tilde{\rho}_j(z) =(1-\chi(\xi))\rho_j(z)$. By Lemma~\ref{lem:Hol.homogeneous-symbol-holomorphic}, each $\rho_j(z)$, $z\in \Omega$, is a holomorphic family in $C^\infty(\T^n_\theta\times(\R^n\setminus 0))$, and so we know by Remark~\ref{rem:Hol.holfamily-homogeneous-cutoff-classical} that $\tilde{\rho}_j(z)$, ${z\in \Omega}$,  is a holomorphic family in $\stS^{m-j}(\T^n_\theta\times\R^{n})$. In particular, given any $N\geq 1$, the difference $\rho(z)(\xi)-\sum_{j<N} \tilde{\rho}_j(z)(\xi)$ is a holomorphic family in $C^\infty(\T^n_\theta\times\R^{n})$, and hence is locally bounded. 
Moreover, the estimates~(\ref{eq:Hol.hol-family-classical-symbol-estimate}) imply that, for any compact $K\subseteq \Omega$ and any multi-orders $\alpha$ and $\beta$, for all $z\in K$ and $\xi\in \U^n$, we have
\begin{align*}
 \bigg\| \delta^\alpha \partial_\xi^\beta\big( \rho(z)(\xi)-\sum_{j<N} \tilde{\rho}_j(z)(\xi)\big)\bigg\|  =  \bigg\| \delta^\alpha \partial_\xi^\beta\big( \rho(z)(\xi)-\sum_{j<N} {\rho}_j(z)(\xi)\big)\bigg\| 
  \leq 
C_{KN\alpha\beta} |\xi|^{m-N-|\beta|}.  
\end{align*}
It then follows from Remark~\ref{rmk:Hol.hol-standard-symb-U} that $\rho(z)(\xi)-\sum_{j<N} \tilde{\rho}_j(z)(\xi)$ is a locally bounded family in $\stS^{m-N}(\T^n_\theta\times\R^{n})$ for all $N\geq 1$, and hence $\rho(z)(\xi)\sim\sum_{j\geq 0}(1-\chi(\xi))\rho_j(z)(\xi)$ in the sense of~(\ref{eq:Hol.asymptotic-standard}). 

Conversely, suppose that each family $\rho_j(z)$, ${z\in \Omega}$, $j\geq 0$, is a holomorphic family in $C^\infty(\T^n_\theta\times(\R^n\setminus 0))$, and $\rho(z)(\xi)\sim\sum_{j\geq 0}(1-\chi(\xi))\rho_j(z)(\xi)$ in the sense of~(\ref{eq:Hol.asymptotic-standard}). Let $K\subseteq \Omega$ be compact and set $w_K=\min \{\Re w(z); \ z\in K\}$. Let $N\geq 0$. By Lemma~\ref{lem:Hol.asymptotic-standard-NJ}, there is $J_0\geq 0$ such that 
\begin{equation*}
 \rho(z)-\sum_{j<J} \tilde{\rho}_j(z)\in L^\infty\big(K;\stS^{w_K-N}(\T^n_\theta\times\R^{n}) \big)\qquad \forall J\geq J_0. 
\end{equation*}
Thus, given any $J\geq J_0$ and any multi-orders $\alpha$ and $\beta$, for all $z\in K$ and $\xi\in \U^n$, we have
\begin{align}
 \bigg\| \delta^\alpha \partial_\xi^\beta\big( \rho(z)(\xi)-\sum_{j<J} {\rho}_j(z)(\xi)\big)\bigg\| &=  \bigg\| \delta^\alpha \partial_\xi^\beta\big( \rho(z)(\xi)-\sum_{j<J} \tilde{\rho}_j(z)(\xi)\big)\bigg\| \nonumber \\  
 &\leq C_{KNJ\alpha\beta} \big(1+|\xi|\big)^{w_K-N-|\beta|}
 \label{eq:Hol.asymptotic-clas-stand}\\
 & \leq C_{KNJ\alpha\beta}|\xi|^{\Re w(z)-N-|\beta|}. \nonumber
\end{align}

Set $J=\max(J_0,N)$, and suppose that $N\leq j \leq J$. As $\rho_j(z)$, ${z\in \Omega}$, $j\geq 0$, is a locally bounded family in $C^\infty(\T^n_\theta\times(\R^n\setminus 0))$ and is homogeneous of degree $w(z)-j$, for all $z\in K$ and $\xi\in \U^n$, we have
\begin{equation*}
 \big\|  \delta^\alpha \partial_\xi^\beta \rho_j(z)(\xi)\big\| = |\xi|^{\Re w(z)-j-|\beta|} \big\|  \delta^\alpha \partial_\xi^\beta \rho_j(z)\left(|\xi|^{-1}\xi\right)\big\| 
 \leq C_{Kj\alpha\beta} |\xi|^{\Re w(z)-N-|\beta|}. 
\end{equation*}
Combining this with~(\ref{eq:Hol.asymptotic-clas-stand}), we obtain
\begin{equation*}
 \bigg\| \delta^\alpha \partial_\xi^\beta\big( \rho(z)(\xi)-\sum_{j<N} {\rho}_j(z)(\xi)\big)\bigg\| \leq C_{KN\alpha\beta} |\xi|^{\Re w(z)-N-|\beta|} \qquad \forall (z,\xi)\in K\times 
 \U^n.  
\end{equation*}
This shows that $\rho(z)(\xi)\sim\sum_{j\geq 0}\rho_j(z)(\xi)$ in the sense of~(\ref{eq:Hol.hol-family-classical-symbol-estimate}). 
\end{proof}

\begin{lemma} \label{lem:Hol.asymptotic-classical-family}
Assume that $\Re w(z)\leq m$ on $\Omega$.  Suppose that $\rho(z)$ is a holomorphic family in $C^\infty(\T^n_\theta\times\R^{n})$ such that $\rho(z)\sim \sum_{\ell\geq 0} \rho^{(\ell)}(z)$ in the sense of~(\ref{eq:Hol.asymptotic-standard}), where $\rho^{(\ell)}(z)$, ${z\in \Omega}$,  is a holomorphic family in $S^\bt(\T^n_\theta\times\R^{n})$ of order $w(z)-\ell$. 
Then $\rho(z)$, ${z\in \Omega}$,  is a holomorphic family in $S^\bt(\T^n_\theta\times\R^{n})$ of order $w(z)$. 
\end{lemma}
\begin{proof}
By assumption  $\rho^{(\ell)}(z)\sim \sum_{j\geq 0} \rho^{(\ell)}_j(z)$ in the sense of~(\ref{eq:Hol.hol-family-classical-symbol-estimate}), where $ \rho^{(\ell)}_j(z)$ is homogeneous of degree $w(z)-\ell-j$. 
For $j=0,1, \ldots$ set $\rho_j(z)= \sum_{\ell=0}^j \rho^{(\ell)}_{j-\ell}(z)$. Lemma~\ref{lem:Hol.homogeneous-symbol-holomorphic} implies that $\rho_j(z)$, ${z\in \Omega}$,  is a holomorphic family in $C^\infty(\T^n_\theta\times(\R^n\setminus 0))$. Moreover, $\rho_j(z)(\xi)$ is homogeneous of degree $w(z)-j$. Let $\chi(\xi) \in C^\infty_c(\R^n)$ be as in Remark~\ref{rem:Hol.holfamily-homogeneous-cutoff-classical}. Lemma~\ref{lem:Hol.asymptotic-classical-standard} ensures that 
$\rho^{(\ell)}(z)\sim \sum_{j\geq 0} (1-\chi(\xi))\rho^{(\ell)}_j(z)$ in the sense of~(\ref{eq:Hol.asymptotic-standard}). Thus, still in this sense, we have
\begin{equation*}
 \rho(z) \sim \sum_{\ell \geq 0} \rho^{(\ell)}(z) \sim \sum_{\ell \geq 0} \sum_{j\geq 0} \big(1-\chi(\xi)\big)\rho^{(\ell)}_j(z)
 \sim \sum_{j\geq 0} \big(1-\chi(\xi)\big)\rho_j(z). 
\end{equation*}
It then follows from Lemma~\ref{lem:Hol.asymptotic-classical-standard} that $\rho(z) \sim \sum_{j \geq 0} \rho_j(z)$ in the sense of~(\ref{eq:Hol.hol-family-classical-symbol-estimate}), and hence $\rho(z)$, ${z\in \Omega}$,  is a holomorphic family in $S^\bt(\T^n_\theta\times\R^{n})$ of order $w(z)$. 
\end{proof}

\begin{lemma} \label{lem:Hol.holfamily-classical-product}
Let $\rho(z)$, ${z\in \Omega}$,  and $\sigma(z)$, ${z\in \Omega}$,  be holomorphic families in $S^\bt(\T^n_\theta\times\R^{n})$ with respective orders $w(z)$ and $u(z)$. Then $\rho(z)\sigma(z)$, ${z\in \Omega}$,  is a holomorphic  family  in $S^\bt(\T^n_\theta\times\R^{n})$ of order $w(z)+u(z)$.  
\end{lemma}
\begin{proof}
Suppose that $\rho(z)\sim \sum \rho_j(z)$ and $\sigma(z) \sim \sum \sigma_j(z)$ in the sense of~(\ref{eq:Hol.hol-family-classical-symbol-estimate}), where $ \rho_j(z)$ (resp., $ \sigma_j(z)$) is homogeneous of degree $w(z)-j$ (resp., $u(z)-j$). Without any loss of generality we may assume that $\Re w(z)\leq m$ and $\Re u(z)\leq p$ on $\Omega$.

 Let $\chi\in C^\infty_c(\R^n)$ be such that $\chi(\xi)=1$ near $\xi=0$ and $\supp \chi\subseteq \B^n$. Lemma~\ref{lem:Hol.homogeneous-symbol-holomorphic} ensures that 
$\rho_j(z)$, ${z\in \Omega}$,   and  $\sigma_j(z)$, ${z\in \Omega}$,  are holomorphic families in $C^\infty(\T^n_\theta\times(\R^n\setminus 0))$. Moreover, by 
 Lemma~\ref{lem:Hol.asymptotic-classical-standard}, $\rho(z)\sim \sum (1-\chi(\xi))\rho_j(z)$ and $\sigma(z) \sim \sum (1-\chi(\xi))\sigma_j(z)$ in the sense of~(\ref{eq:Hol.asymptotic-standard}). In particular, by using Remark~\ref{rmk:Hol.asymptotic-standard-stS}, we see that, given any $N\geq 1$, we have 
\begin{gather*}
 \rho(z)(\xi)= \sum_{j<N} (1-\chi(\xi))\rho_j(z)(\xi) \qquad \bmod \Hol\big(\Omega; \stS^{m-N}(\T^n_\theta\times\R^{n})\big), \\
  \sigma(z)(\xi)= \sum_{k<N} (1-\chi(\xi))\sigma_k(z)(\xi) \qquad \bmod \Hol\big(\Omega; \stS^{p-N}(\T^n_\theta\times\R^{n})\big).
\end{gather*}
Combining this with Remark~\ref{rmk:Hol.product-hol}, we deduce that
\begin{align}
 \rho(z)(\xi)\sigma(z)(\xi) & = \sum_{j,k<N}\left(1-\chi(\xi)\right)^2\rho_j(z)(\xi)\sigma_k(z)(\xi)  & \bmod \Hol\big(\Omega; \stS^{m+p-N}(\T^n_\theta\times\R^{n})\big), \nonumber\\ 
 & = \left(1-\chi(\xi)\right)^2 \sum_{j<N} \bigg(\sum_{k+\ell=j}\rho_k(z)(\xi)\sigma_\ell(z)(\xi)\bigg)  & \bmod \Hol\big(\Omega; \stS^{m+p-N}(\T^n_\theta\times\R^{n})\big). 
 \label{eq:Hol.product-hol-symbolsN}
\end{align}

For $j=0,1, \ldots $ set 
\begin{equation*}
 (\rho\sigma)_j(z)(\xi)= \sum_{k+\ell=j} \rho_k(z)(\xi)\sigma_\ell(z)(\xi), \qquad z\in \Omega, \ \xi\in \R^n\setminus 0. 
\end{equation*}
Then $(\rho\sigma)_j(z)$, ${z\in \Omega}$,  is a holomorphic family  in $C^\infty(\T^n_\theta\times(\R^n\setminus 0))$ and is homogeneous of degree $w(z)+u(z)-j$. In addition, set 
$\tilde{\chi}(\xi)=1-(1-\chi(\xi))^2 \in C^\infty_c(\R^n)$. Note that $\tilde{\chi}(\xi)=1$ near $\xi=0$ and $\supp \tilde{\chi}\subseteq \B^n$. Then~(\ref{eq:Hol.product-hol-symbolsN}) means that, in the sense of~(\ref{eq:Hol.asymptotic-standard}), we have
\begin{equation*}
 \rho(z)(\xi)\sigma(z)(\xi) \sim \sum_{j\geq 0} (1-\tilde{\chi}(\xi))(\rho\sigma)_j(z)(\xi). 
\end{equation*}
Lemma~\ref{lem:Hol.asymptotic-classical-standard} then ensures that  $\rho(z)\sigma(z)  \sim \sum_{j\geq 0} (\rho\sigma)_j(z)$ in the sense of~(\ref{eq:Hol.hol-family-classical-symbol-estimate}), and so $\rho(z)\sigma(z)$, ${z\in \Omega}$,   is a holomorphic  family  in $S^\bt(\T^n_\theta\times\R^{n})$ of order $w(z) +u(z)$.\end{proof}

\subsection{Holomorphic families of \psidos} 
In what follows, if $\rho(z)$, ${z\in \Omega}$,  is a holomorphic family of standard or classical symbols, we set
\begin{equation*}
 P_\rho(z)=P_{\rho(z)},  \qquad z\in \Omega. 
\end{equation*}
We also set
\begin{equation*}
 \Psi^\bt(\T^n_\theta)= \bigcup_{q\in \C} \Psi^q(\T^n_\theta). 
\end{equation*}
This is the class of all \emph{classical} \psidos\ on $\T^n_\theta$. 

\begin{definition}[see also~\cite{LNP:TAMS16, Po:SIGMA20}] \label{def:Hol.hol-family-PsiDOs}
A family $P(z)$, ${z\in \Omega}$,  in  $\Psi^\bt(\T^n_\theta)$ is called a \emph{holomorphic family of order} $w(z)$ if there is a family 
$\rho(z)$, $z\in \Omega$, in $\Hol(\Omega;S^\bt(\T^n_\theta\times\R^{n}))$ of order $w(z)$ such that
\begin{equation*}
 P(z)=P_\rho(z) \qquad \forall z \in \Omega. 
\end{equation*}
We denote by $\Hol(\Omega;\Psi^\bt(\T^n_\theta))$ the class of holomorphic families in $\Psi^\bt(\T^n_\theta)$ parametrized by $\Omega$. 
\end{definition}

\begin{example}\label{ex:Hol.Delta-powers}
 Suppose that $\Omega=\C$. Let $\Delta=\delta_1^2+\cdots + \delta_n^2$ be the flat Laplacian on $\T^n_\theta$. We define $\Delta^{z}$, $z\in \C$, by Borel functional calculus, with the convention $0^z=0$. For $\Re z\leq 0$ we get a bounded operator on $L_2(\T^n_\theta)$. For $\Re z>0$ we obtain an unbounded operator with domain $W^s_2(\T^n_\theta)$, where $s=2\Re z$. In any case, we have 
\begin{equation}\label{eq:Hol.Delta-powers}
 \Delta^z\big(U^k\big) = \left\{ 
 \begin{array}{cl} 
 |k|^{2z}U^k & \text{if $k \neq 0$},\\
 0 & \text{if $k = 0$}. 
\end{array}\right.
\end{equation}
Let $\chi\in C^\infty_c(\R^n)$ be such that $\chi(\xi)=1$ near $\xi=0$ and $\supp \chi\subseteq \B^n$. Set $\rho(z)(\xi)=(1-\chi(\xi))|\xi|^{2z}$, $z\in \C$, $\xi\in \R^n$. Then $\Delta^{z}(U^k)=\rho(z)(k)U^k$ for all $k\in \Z^n$, and so $\Delta^z=P_{\rho}(z)$. From Example~\ref{ex:Hol.powers} we know that $\rho(z)$, ${z\in \Omega}$,  is a holomorphic family in $S^\bt(\T^n_\theta\times\R^{n})$. It follows that  $\Delta^z$, $z\in \C$,  is a holomorphic family in $\Psi^\bt(\T^n_\theta)$. Note also that~(\ref{eq:Hol.Delta-powers}) for $z=0$ gives 
\begin{equation*}
 \Delta^z\big|_{z=0} =1-\Pi_0,
\end{equation*}
where $\Pi_0$ is the orthogonal projection onto $\ker \Delta=\C\cdot 1$. 
\end{example}

As mentioned in Section~\ref{sec:PsiDOs}, we have a well-defined symbol map $\Psi^q(\T^n_\theta)\ni  P\rightarrow \rho_{\!{}_{P}}\in S^q(\T^n_\theta\times \R^n)$ such that $P_{\rho_{\!{}_{P}}}=P$ for all $P\in  \Psi^q(\T^n_\theta)$. See Eq.~(\ref{eq:PsiDOs.symbol-map-definition}). 

\begin{lemma}\label{lem:hol.local-global}
 Let $P(z)$, ${z\in \Omega}$,  be a family in $\Psi^\bt(\T^n_\theta)$ such that $P(z)$ has order $w(z)$, where $w(z)$ is a holomorphic function on $\Omega$.  Then the following are equivalent:
\begin{enumerate}
\item[(i)] $\rho_{\!{}_{P(z)}}$, ${z\in \Omega}$,  is a holomorphic family in $S^\bt(\T^n_\theta\times \R^n)$ of order $w(z)$. 

 \item[(ii)] $P(z)$, $z\in \Omega$, is a holomorphic family in $\Psi^\bt(\T^n_\theta)$ of order $w(z)$. 
 
 \item[(iii)] $P(z)$ is a holomorphic family in $\Psi^\bt(\T^n_\theta)$ of order $w(z)$ near any point $z_0\in \Omega$. 
\end{enumerate}
\end{lemma}
\begin{proof}
It is immediate that (i) implies (ii) and (ii) implies (iii). Therefore, we only need to show that (iii) implies (i). 

Assume that (iii) is satisfied. To establish (i) it is enough to establish it near any point $z_0\in \Omega$. Thus, possibly by shrinking $\Omega$, we may assume that $P(z)$, ${z\in \Omega}$,  is a holomorphic family in $\Psi^\bt(\T^n_\theta)$ of order $w(z)$ and there is $m\in \R$ such that $\Re w(z)\leq m$. Thus, we may write $P(z) = P_\rho(z)$, $z\in \Omega$, where $\rho(z)$, ${z\in \Omega}$,  is a holomorphic family in $S^\bt(\T^n_\theta\times \R^n)$ of order $w(z)$. As $\Re w(z)\leq m$, we know from Remark~\ref{rmk:Hol.holfamily-classical-standard} that  $\rho(z)$, ${z\in \Omega}$,  is a holomorphic family in $\bS^m(\T^n_\theta\times \R^n)$.

Note that $ \rho_{\!{}_{P(z)}}=\rho_{\!{}_{P_\rho(z)}}$, $z\in \Omega$. By Part~(ii) of Proposition~\ref{prop:PsiDOs.rho-to-rhotilde-map-continuity}, we know that 
$\rho \rightarrow\rho_{\!{}_{P_\rho}} -\rho$ is a continuous linear map from $\bS^m(\T^n_\theta\times \R^n)$ to $\bS^{-\infty}(\T^n_\theta\times \R^n)$. It then follows that $ \rho_{\!{}_{P(z)}}-\rho(z)$, $z\in \Omega$, is a  holomorphic family in $\bS^{-\infty}(\T^n_\theta\times \R^n)$, and so $\rho_{\!{}_{P(z)}}$, ${z\in \Omega}$,  is a holomorphic family in $S^\bt(\T^n_\theta\times \R^n)$ of order $w(z)$. This shows that~(iii) implies~(i). 
\end{proof}

\begin{remark}
The implication $\textup{(ii)} \Rightarrow \textup{(i)}$ is the content of~\cite[Lemma~4.18]{Po:SIGMA20}.
\end{remark}

\begin{remark} \label{rem:Hol.symbol-uniqueness}
The proof of Lemma~\ref{lem:hol.local-global} shows that if $\rho(z)$, ${z\in \Omega}$,  is any holomorphic family in $S^\bt(\T_\theta^n\times\Rn)$ of order $w(z)$ such that $P(z) = P_\rho(z)$, then $\rho_{\!{}_{P(z)}}-\rho(z)$, $z\in \Omega$, is a holomorphic family in $\bS^{-\infty}(\T^n_\theta\times \R^n)$. Thus, if $\rho(z)\sim\sum_{j\geq 0}\rho_j(z)$, where $\rho_j(z)\in S_{w(z)-j}(\T_\theta^n\times\Rn)$ and $\sim$ is meant in the sense of~(\ref{eq:Hol.hol-family-classical-symbol-estimate}), then $\rho_{\!{}_{P(z)}}\sim\sum_{j\geq 0}\rho_j(z)$ in the sense of~(\ref{eq:Hol.hol-family-classical-symbol-estimate}). 
\end{remark}

\begin{proposition}\label{prop:Hol.propertiesP(z)-W2s} 
Let $P(z)$, ${z\in \Omega}$,  be a holomorphic family in $\Psi^\bt(\T^n_\theta)$ of order $w(z)$. 
\begin{enumerate}
 \item $P(z)$, ${z\in \Omega}$,  is a holomorphic family in $\cL(C^\infty(\T^n_\theta))$ and (uniquely extends to) a holomorphic family in $\cL(\scD'(\T^n_\theta))$. 
 
 \item If $\Re w(z)\leq m$ on $\Omega$, then $P(z)$, ${z\in \Omega}$,  is a holomorphic family in $\cL(W_2^{s+m}(\T^n_\theta), W_2^s(\T^n_\theta))$ for every $s\in \R$.
 
 \item Let $K\subseteq \Omega$ be compact, and set $w_K=\max\{\Re w(z); \ z\in K\}$. Then $P(z)$, $z\in K$, is a strongly continuous bounded family in $\sL(W_2^{s+w_K}(\T^n_\theta), W_2^s(\T^n_\theta))$ for every $s\in \R$. 
\end{enumerate}
\end{proposition}

\begin{proof}
In the first part, $\sL(C^\infty(\T^n_\theta))$ and $\cL(\scD'(\T^n_\theta))$ are endowed with their respective strong topologies, i.e., the topologies of uniform convergence on bounded sets. Write $P(z)=P_\rho(z)$, $z\in \Omega$, where $\rho(z)$, ${z\in \Omega}$,  is a holomorphic family in $S^\bt(\T^n_\theta\times \R^n)$. Without any loss of generality we may assume $\Re w(z)\leq m$ on $\Omega$ (\emph{cf}.~Remark~\ref{rmk:Hol.reduction-precompact}). In this case, $\rho(z)$, ${z\in \Omega}$,  is a holomorphic family in $\bS^m(\T^n_\theta\times \R^n)$ (\emph{cf}.\ Remark~\ref{rmk:Hol.holfamily-classical-standard}). The assignment $\rho \rightarrow P_\rho$ gives rise to continuous linear maps from $\bS^m(\T^n_\theta \times \R^n)$ to $\sL(C^\infty(\T^n_\theta))$, $\sL(\scD'(\T^n_\theta))$ and $\sL(W_2^{s+m}(\T^n_\theta), W_2^s(\T^n_\theta))$, $s\in \R$ (see~\cite{HLP:Part2}). It then follows that $P_\rho(z) = P(z)$, ${z\in \Omega}$,  is a holomorphic family in any of those locally convex spaces. This gives the first two parts. 

It remains to prove the last part. Let $K\subseteq \Omega$ be compact, and set $w_K=\max\{\Re w(z); \ z\in K\}$.  By Remark~\ref{rmk:Hol.holfamily-bounded-symbols} the family $\rho(z)$, $z\in K$, is bounded in $\bS^{w_K}(\T^n_\theta\times \R^n)$. Given any $s\in \R$, the continuity of the linear map $\rho \rightarrow P_\rho$ from $\bS^{w_K}(\T^n_\theta\times\R^n)$ to  $\sL(W_2^{s+w_K}(\T^n_\theta), W_2^s(\T^n_\theta))$ then ensures that  $P(z)$, $z\in K$, is a bounded family in $\sL(W_2^{s+w_K}(\T^n_\theta), W_2^s(\T^n_\theta))$. 

Let $\Omega'$ be a bounded open set such that $K\subseteq \Omega'\subseteq \overline{\Omega}'\subseteq \Omega$. Thus, there is $w'\geq w_K$ such that $\Re w(z)\leq w'$ on $\Omega'$. By the first part of the proof, $P(z)$, $z\in \Omega'$, is a holomorphic family in 
$\sL(W_2^{s+w'}(\T^n_\theta), W_2^s(\T^n_\theta))$, and so $P(z)$, $z\in K$, is a norm-continuous family in $\sL(W_2^{s+w'}(\T^n_\theta), W_2^s(\T^n_\theta))$. As $W_2^{s+w'}(\T^n_\theta)$ is dense in $W_2^{s+w_K}(\T^n_\theta)$ and $P(z)$, $z\in K$, is bounded in $\sL(W_2^{s+w_K}(\T^n_\theta), W_2^s(\T^n_\theta))$, it follows that $P(z)$, $z\in K$, is a strongly continuous family in 
$\sL(W_2^{s+w_K}(\T^n_\theta), W_2^s(\T^n_\theta))$. This proves the last part. 
\end{proof}

\begin{proposition}\label{prop:Hol.propertiesP(z)-sLp} 
Let $P(z)$, ${z\in \Omega}$,  be a holomorphic family in $\Psi^\bt(\T^n_\theta)$ of order $w(z)$. 
\begin{enumerate}
 \item If $\Re w(z)<m$ (resp., $\Re w(z)\leq m$) on $\Omega$ with $m<0$, then $P(z)$, ${z\in \Omega}$,  is a holomorphic family in $\sL_{p}$ (resp., $\sL_{p,\infty}$) with 
$p=n|m|^{-1}$. 
 
 \item   Let $K\subseteq \Omega$ be compact such that $w_K:=\max\{\Re w(z); \ z\in K\}<0$, and set $p=n|w_K|^{-1}$. Then the family $P(z)$, $z\in K$, is bounded in $\sL_{p,\infty}$, and hence is bounded in $\sL_r$ for all $r>p$. 
\end{enumerate}
\end{proposition}
\begin{proof}
Write $P(z)=P_\rho(z)$, $z\in \Omega$, where $\rho(z)$, ${z\in \Omega}$,  is a holomorphic family in $S^\bt(\T^n_\theta\times \R^n)$ of order $w(z)$. Suppose that $\Re w(z)\leq m$ on $\Omega$.  As in the proof of Proposition~\ref{prop:Hol.propertiesP(z)-W2s},  in this case $\rho(z)$, ${z\in \Omega}$,  is a holomorphic family in $\bS^m(\T^n_\theta\times \R^n)$.  Set $p=n|m|^{-1}$. As the quantization map $\rho \rightarrow P_\rho$ is continuous from $\bS^m(\T^n_\theta\times \R^n)$ to $\sL_{p,\infty}$ (\emph{cf}.~\cite{HLP:Part2}), we deduce that $P(z)$, ${z\in \Omega}$,  is a holomorphic family in $\sL_{p,\infty}$. 

Suppose now that $\Re w(z)<m$ on $\Omega$. Near any $z_0\in \Omega$, we can find a bounded open neighborhood $\Omega'$ of $z_0$ such that $\Re w(z)\leq m'<m$ on $\Omega'$. Set $q=n|m'|^{-1}$. Note that $q<p$, and so $\sL_{q,\infty}$ embeds continuously into $\sL_p$. By the first part of the proof, 
 $P(z)$, $z\in \Omega'$, is a holomorphic family in $\sL_{q,\infty}$. Thus, $P(z)$ is a holomorphic family in $\sL_p$ near every point $z_0\in \Omega$, and hence $P(z)$, ${z\in \Omega}$,  is a holomorphic family in $\sL_{p}$. 
 
Finally,  let $K\subseteq \Omega$ be a compact set such that $w_K:=\max\{\Re w(z); \ z\in K\}<0$,  and set  $p=n|w_K|^{-1}$. Then $\rho(z)$, $z\in K$, is a bounded family in $\bS^{w_K}(\T^n_\theta\times \R^n)$, and so $P(z)$, $z\in K$, is a bounded  family in $\sL_{p,\infty}$.  
\end{proof}

In what follows, we equip $\sL(\scD'(\T^n_\theta), C^\infty(\T^n_\theta))$ with its strong topology, i.e., the topology of uniform convergence on bounded subsets of $\scD'(\T^n_\theta)$. 

\begin{proposition}
 Let $\rho(z)$, ${z\in \Omega}$,  be a holomorphic family in $S^\bt(\T^n_\theta \times \R^n)$. The following are equivalent:
\begin{enumerate}
 \item[(i)] $P_\rho(z)$, ${z\in \Omega}$,  is a holomorphic family in $\sL(\scD'(\T^n_\theta), C^\infty(\T^n_\theta))$. 
 
 \item[(ii)] $\rho(z)$, ${z\in \Omega}$,  is a holomorphic family in $\stS^{-\infty}(\T^n_\theta\times\R^{n})$. 
 
 \item[(iii)] $P_\rho(z)\in \Psi^{-\infty}(\T^n_\theta)$ for all $z\in \Omega$. 
\end{enumerate}
\end{proposition}
\begin{proof}
 It is immediate that (i) implies (iii). Furthermore, as the quantization map $\rho \rightarrow P_\rho$ is continuous from $\bS^{-\infty}(\T^n_\theta\times\R^{n})$ to $\sL(\scD'(\T^n_\theta), C^\infty(\T^n_\theta))$ (see~\cite[Proposition~5.30]{LP:Part1}), it is clear that (ii) implies~(i). 
 
It remains to show that (iii) implies (ii). Suppose that $P_\rho(z)\in \Psi^{-\infty}(\T^n_\theta)$ for all $z\in \Omega$. We have $\rho(z)\sim \sum_{j\geq 0} \rho_j(z)$ in the sense of~(\ref{eq:Hol.hol-family-classical-symbol-estimate}), where $\rho_j(z)\in S_{w(z)-j}(\T^n_\theta\times\R^{n})$ and $w(z)$ is a holomorphic function on $\Omega$. Given any $z\in \Omega$, the fact that  $P_{\rho}(z)$ is a smoothing operator ensures that $\rho(z)\in \bS^{-\infty}(\T^n_\theta\times\R^{n})$, and so $\rho_j(z)=0$ for all $j\geq 0$. Thus, $\rho(z)\sim 0$ in the sense of~(\ref{eq:Hol.hol-family-classical-symbol-estimate}). This implies that $\rho(z)$, ${z\in \Omega}$,  is a holomorphic family in $\stS^{-\infty}(\T^n_\theta\times\R^{n})$. This shows that (iii) implies (ii). 
\end{proof}
In particular, if $\rho_1(z)$, ${z\in \Omega}$, and $\rho_2(z)$, ${z\in \Omega}$, are holomorphic families in $S^\bt(\T^n_\theta\times\R^{n})$ such that $P_{\rho_1}(z)=P_{\rho_2}(z)$ for all $z\in \Omega$, then $\rho_1(z)-\rho_2(z)$, ${z\in \Omega}$, is a holomorphic family in $\bS^{-\infty}(\T^n_\theta\times\R^{n})$. 

In what follows, given holomorphic families of standard or classical symbols $\rho_1(z)$, ${z\in \Omega}$,  and $\rho_2(z)$, ${z\in \Omega}$,  we set
\begin{equation*}
 \rho_1\sharp\rho_2(z)=\rho_1(z)\sharp \rho_2(z), \qquad z\in \Omega. 
\end{equation*}
By arguing along the same lines as the proof of~\cite[Proposition~5.6]{LP:Part1}, we obtain the following result. 

\begin{proposition}\label{prop:Hol.hol-sharp-product}
 Given holomorphic families $\rho_j(z)$, $z\in \Omega$, in $\stS^{m_j}(\T^n_\theta\times\R^{n})$, $j=1,2$, the composition
 $\rho_1\sharp\rho_2(z)$, ${z\in \Omega}$,  is a holomorphic family in $\stS^{m_1+m_2}(\T^n_\theta\times\R^{n})$. Moreover, in the sense of~(\ref{eq:Hol.asymptotic-standard}), we have
$\rho_1\sharp\rho_2(z) \sim \sum\frac{1}{\alpha !}\partial_\xi^\alpha\rho_1(z)\delta^\alpha\rho_2(z)$. 
\end{proposition}

We have the following version of this result for holomorphic families of classical symbols. 

\begin{proposition} \label{prop:Hol.composition-of-symbols}
  Let $\rho(z)$, ${z\in \Omega}$,  and $\sigma(z)$, ${z\in \Omega}$,  be holomorphic families in $S^\bt(\T^n_\theta\times\R^{n})$ with respective order functions $w(z)$ and $u(z)$. Then $\rho\sharp \sigma(z)$, ${z\in \Omega}$,  is a holomorphic family in $S^\bt(\T^n_\theta\times\R^{n})$ of order $w(z)+u(z)$. 
 \end{proposition}
\begin{proof}
  Without any loss of generality we may assume that $\Re w(z)\leq m$ and $\Re u(z) \leq p$ (\emph{cf}.~Remark~\ref{rmk:Hol.reduction-precompact}). In this case, $\rho(z)$, ${z\in \Omega}$,  and $\sigma(z)$, ${z\in \Omega}$,  are holomorphic families in $\stS^m(\T^n_\theta\times\R^{n})$ and $\bS^p(\T^n_\theta\times\R^{n})$, respectively. Proposition~\ref{prop:Hol.hol-sharp-product} then ensures that $\rho\sharp\sigma(z)$, ${z\in \Omega}$,  is a holomorphic family in $\stS^{m+p}(\T^n_\theta\times\R^{n})$ and we have $\rho\sharp\sigma(z) \sim \sum\frac{1}{\alpha !}\partial_\xi^\alpha\rho(z)\delta^\alpha\sigma(z)$ in the sense of~(\ref{eq:Hol.asymptotic-standard}). By Remark~\ref{rmk:Hol.derivatives-hol} and Lemma~\ref{lem:Hol.holfamily-classical-product}, each summand $\partial_\xi^\alpha\rho(z)\delta^\alpha\sigma(z)$, ${z\in \Omega}$,  is a holomorphic family in $S^\bt(\T^n_\theta\times\R^{n})$ of order $w(z)+u(z)-|\alpha|$. It then follows from Lemma~\ref{lem:Hol.asymptotic-classical-family} that $\rho\sharp\sigma(z)$, ${z\in \Omega}$,  is a holomorphic family in $S^\bt(\T^n_\theta\times\R^{n})$ of order $w(z)+u(z)$. 
\end{proof}

\begin{corollary} \label{cor:Hol.composition-of-psidos}
Let $P(z)$, ${z\in \Omega}$,  and $Q(z)$, ${z\in \Omega}$,  be holomorphic families in $\Psi^\bt(\T^n_\theta)$ of respective orders $w(z)$ and $u(z)$. Then $P(z)Q(z)$, ${z\in \Omega}$,  is a holomorphic family in $\Psi^\bt(\T^n_\theta)$ of order $w(z)+u(z)$. 
\end{corollary}

\subsection{Holomorphic gaugings} 
Holomorphic gaugings of symbols and \psidos\ are important ingredients in the construction of the canonical trace on non-integer-order \psidos\ (see~\cite{LNP:TAMS16, Po:SIGMA20}). The terminology \emph{holomorphic gauging} was coined by Guillemin~\cite{Gu:AIM93, Gu:JFA93}. 

\begin{definition}
 If $\rho(\xi)\in S^q(\T^n_\theta\times\R^{n})$, $q\in \C$,  then a \emph{holomorphic gauging} of $\rho(\xi)$ is any family $\rho(z)(\xi)\in \Hol(\C;S^\bt(\T^n_\theta\times\R^{n}))$ 
 of order $z+q$ such that $\rho(0)(\xi)=\rho(\xi)$. 
\end{definition}

\begin{lemma}\label{lem:Hol.gauging-symbols} 
 Any symbol $\rho(\xi)\in S^q(\T^n_\theta\times\R^{n})$, $q\in \C$, admits a holomorphic gauging. 
\end{lemma}
\begin{proof}
 Let $\chi(\xi)\in C_c^\infty(\R^n)$ be such that $\chi(\xi)=1$ near $\xi=0$ and $\supp \chi\subseteq \B^n$. Set $\sigma(z)(\xi)=\chi(\xi)+(1-\chi(\xi))|\xi|^z$, $z\in \C$, $\xi\in \R^n$. We know from Example~\ref{ex:Hol.powers} that $(1-\chi(\xi))|\xi|^z$, $z\in \C$, is a holomorphic family in $S^\bt(\T^n_\theta\times\R^{n})$ of order $z$, and so $\sigma(z)$, $z\in \C$, is a holomorphic family in $S^\bt(\T^n_\theta\times\R^{n})$ of order $z$. Note that $\sigma(0)(\xi)=\chi(\xi)+(1-\chi(\xi))=1$. Thus, $\sigma(z)$, $z\in \C$, is a holomorphic gauging of the constant symbol $1$. 
 
Set $\rho(z)(\xi)=\rho(\xi)\sigma(z)(\xi)$, $z\in \C$, $\xi \in \R^n$. Lemma~\ref{lem:Hol.holfamily-classical-product} ensures that  $\rho(z)$, $z\in \C$,  is a holomorphic family in $S^\bt(\T^n_\theta\times\R^{n})$ of order $q+z$. Moreover, we have $\rho(0)(\xi)=\rho(\xi)\sigma(0)(\xi)=\rho(\xi)$. Thus,  $\rho(z)$, $z\in \C$,  is a holomorphic gauging of $\rho(\xi)$, proving the result. 
\end{proof}

\begin{definition}\label{def:Hol.gauging-PsiDOs}
 If $P\in \Psi^q(\T^n_\theta)$, then a \emph{holomorphic gauging} of $P$ is any family $P(z)\in \Hol(\C;\Psi^{\bt}(\T^n_\theta))$ of order $z+q$ such that $P(0)=P$. 
 \end{definition}

\begin{remark}
 By Definition~\ref{def:Hol.hol-family-PsiDOs}, a holomorphic gauging $P(z)$ of $P$ amounts to a holomorphic family  $\rho(z)$, $z\in \C$,  in $S^\bt(\T^n_\theta\times\R^{n})$ of order $q+z$ such that $P(z)=P_\rho(z)$ for all $z\in \C$. 
\end{remark}

\begin{proposition} \label{prop:Hol.gauging-PsiDos}
Any $P\in \Psi^q(\T^n_\theta)$, $q\in \C$, admits a holomorphic gauging. 
\end{proposition}
\begin{proof}
 Let $\rho\in S^q(\T^n_\theta\times\R^{n})$ be such that $P=P_\rho$. Let  $\rho(z)$, $z\in \C$,  be a holomorphic gauging of $\rho$. The existence of such a family is guaranteed by Lemma~\ref{lem:Hol.gauging-symbols}. Then $P_\rho(z)$, $z\in \C$, is a holomorphic gauging of $P$. The result is proved. 
 \end{proof}

\subsection{$\Hol_\infty$-families of \psidos} \label{subsec:uniform}
In order to study the behavior of complex powers along vertical strips, it is convenient to introduce the notion of $\Hol_\infty$-family of \psidos.

In what follows, by a closed vertical strip we mean any subset of $\C$ of the form $\{a\leq\Re z \leq b\}$, $a\leq b$. By a closed vertical half-strip we shall mean any set of the form,
\begin{equation}
 \Sigma_\pm(a,b,c):=\{a\leq \Re z\leq b\}\cap \{\pm \Im z\geq c\}, \qquad a\leq b, \quad c\in \R.
 \label{eq:Uniform.half-strips}  
\end{equation}
By an open vertical strip (resp., half-strip) we shall mean the interior of a strip (resp., half-strip) as above with $a<b$.  

\begin{definition} \label{def:powers.admissible}
 An open $\Omega\subseteq \C$ is called \emph{admissible} if, for every $z_0\in \Omega$, there is an open 
vertical half-strip which contains $z_0$ and is contained in $\Omega$. 
\end{definition}

Examples of admissible open subsets of $\C$ include the following: 
\begin{enumerate}
 \item[(i)] The whole complex plane $\C$.  
 
 \item[(ii)] $\C\setminus A$, where $A$ is any closed subset of $\R$.  

\item[(iii)] Any left half-plane $\{\Re z<a\}$ or right half-plane $\{\Re z>a\}$. 
 
 \item[(iv)] Any lower half-plane $\{\Im z<a\}$ or upper half-plane $\{\Im z>a\}$. 
 
 \item[(v)] Any quadrant given by the intersection of a left/right half-plane as in~(iii) with a lower/upper half-plane as in~(iv). 
 
\end{enumerate}
Instances of open sets as in (iv) and (v) are the upper/lower half-planes and upper/lower left quadrants,
\begin{equation*}
 \bH_\pm=\{\pm \Im z>0\}, \qquad \bQ_{\pm}:= \{\Re z<0\} \cap \{\pm \Im z>0\}. 
\end{equation*}

Throughout the remainder of this section, we let $\Omega$ be an admissible open subset of $\C$.  

\begin{definition}
 $\Hol_\infty(\Omega)$ consists of holomorphic functions $f(z)$ on $\Omega$ that are uniformly bounded on every closed vertical half-strip $\Sigma \subseteq \Omega$. 
\end{definition}

\begin{example}\label{ex:Uniform.Gamma-function}
 Let $\Gamma(z)$ be Euler's gamma function. This is a holomorphic function on $\C\setminus \Z_{-}$\footnote{Here $\Z_{-}=\{0,-1,-2, \ldots\}$ consists of all non-positive integers.}. In fact, given any $\alpha \in [0,\pi/2)$, we have 
\begin{equation*}
 e^{\mp i\alpha z}\Gamma(z)\in \Hol_\infty(\bH_\pm). 
\end{equation*}
This is a standard result. Thanks to the functional equation $\Gamma(z+1)=z\Gamma(z)$, it is enough to show the uniform boundedness of $e^{\mp i\alpha z}\Gamma(z)$ on vertical half-strips $\Sigma_\pm(a,b,0)$ of the form~(\ref{eq:Uniform.half-strips}) with $0<a<b$. Note that, for $\Re z>0$, we have 
\begin{equation*}
 \Gamma(z)= \int_0^\infty t^{z-1} e^{-t}dt = \int_{e^{\pm i\alpha}[0,\infty)} t^{z-1} e^{-t}dt. 
\end{equation*}
Thus, if $z=x+iy$ and $\pm y>0$, then 
\begin{equation*}
 \big|\Gamma(z)\big| \leq \int_0^\infty t^{x-1} e^{\mp \alpha y} e^{-(\cos \alpha)t}dt = e^{-\alpha |y|}  \int_0^\infty t^{x-1} e^{-(\cos \alpha)t}dt, 
\end{equation*}
from which the result follows. 
\end{example}

In what follows, we let $\sE$ be a locally convex TVS or a quasi-Banach space.  

\begin{definition}
 A $\Hol_\infty(\Omega)$-\emph{family} in $\sE$ is any holomorphic family in $\sE$ which is uniformly bounded on any closed vertical half-strip $\Sigma \subseteq \Omega$. 
 \end{definition}

 \begin{remark}
 If $\sE$ is locally convex, then a map $u(z)\in \Hol(\Omega;\sE)$ is $\Hol_\infty$ if and only if, for any continuous semi-norm $\fp$ on $\sE$ and any closed vertical half-strip $\Sigma \subseteq \Omega$, we have 
\begin{equation}\label{eq:Hol.Holinf-seminorms}
  \sup_{z\in \Sigma}\fp\big[u(z)\big]<\infty. 
\end{equation}
If $\sE$ is a quasi-Banach space with quasi-norm $\|\cdot\|_{\sE}$, then the above condition should be replaced by 
\begin{equation*}
  \sup_{z\in \Sigma}\| u(z)\|_{\sE}<\infty. 
\end{equation*}
\end{remark}
 
\begin{remark}\label{rmk:Hol.boundedness-vertical-strips}
 As any closed vertical strip is the union of two closed vertical half-strips, a $\Hol_\infty(\Omega)$-family in $\sE$ is automatically bounded on every closed vertical strip contained in $\Omega$. 
\end{remark}
 
\begin{remark}\label{rmk:unif.composition-cont-lin} 
 If $u(z)$, ${z\in \Omega}$,  is a $\Hol_\infty(\Omega)$-family in $\sE$ and $\Phi:\sE \rightarrow \sE_1$ is a continuous linear map from $\sE$ to another locally convex
 or quasi-Banach space, then the composition $\Phi[u(z)]$, ${z\in \Omega}$,  is a $\Hol_\infty(\Omega)$-family in $\sE_1$. We denote by $\Hol_\infty(\Omega;\sE)$ the space of $\Hol_\infty(\Omega)$-families in $\sE$. If $\sE$ is locally convex, then $\Hol_\infty(\Omega;\sE)$ is a locally convex topological vector space with respect to the semi-norms~(\ref{eq:Hol.Holinf-seminorms}). 
\end{remark}

\begin{definition}
  We say that a family $\rho(z)(\xi)\in \Hol(\Omega;S^\bt(\T^n_\theta\times \R^n))$ of order $w(z)$ is a $\Hol_\infty(\Omega)$-\emph{family} if the following conditions are satisfied:
  \begin{enumerate}
\item[(i)]
For every closed vertical half-strip  $\Sigma\subseteq \Omega$, the subset $\{\Re w(z);\ z\in \Sigma\}$ is compact. 

\item[(ii)] $\rho(z)$, ${z\in \Omega}$,  is a $\Hol_\infty(\Omega)$-family in $C^\infty(\T^n_\theta\times\Rn)$.

\item[(iii)] We have $\rho(z)\sim\sum_{j\geq 0}\rho_j(z)$ with $\rho_j(z)\in S_{w(z)-j}(\T^n_\theta\times\Rn)$, in the sense that, for any $N\geq 0$, any  
closed vertical half-strip  $\Sigma\subseteq \Omega$, and any multi-orders $\alpha$, $\beta$, we have 
\begin{equation} \label{eq:Hol.hol-infty-family-classical-symbol-estimate}
\Big\| \delta^\alpha\partial_\xi^\beta \Big(\rho(z)-\sum_{j<N}\rho_j(z) \Big)(\xi) \Big\| \leq C_{\Sigma N\alpha\beta}|\xi|^{\Re{w(z)}-N-|\beta|}  \quad \text{for all}\ (z,\xi)\in \Sigma\times \U^n.
\end{equation}
\end{enumerate}
We denote by $\Hol_\infty(\Omega;S^\bt(\T^n_\theta\times \R^n))$ the class of $\Hol_\infty(\Omega)$-families in $S^\bt(\T^n_\theta\times \R^n)$ indexed by $\Omega$. 
\end{definition}

\begin{remark}
In the same way as in Remark~\ref{rem:Hol.symbol-uniqueness}, if $\rho(z)$ is any $\Hol_\infty(\Omega)$-family in $S^\bt(\T_\theta^n\times\Rn)$ of order $w(z)$ such that $P(z) = P_\rho(z)$, then $\rho_{\!{}_{P(z)}}-\rho(z)$, $z\in \Omega$, is a $\Hol_\infty(\Omega)$-family in $\bS^{-\infty}(\T^n_\theta\times \R^n)$. Thus, if $\rho(z)\sim\sum_{j\geq 0}\rho_j(z)$, where $\rho_j(z)\in S_{w(z)-j}(\T_\theta^n\times\Rn)$ and $\sim$ is meant in the sense of~(\ref{eq:Hol.hol-infty-family-classical-symbol-estimate}), then $\rho_{\!{}_{P(z)}}\sim\sum_{j\geq 0}\rho_j(z)$ in the sense of~(\ref{eq:Hol.hol-infty-family-classical-symbol-estimate}). 
\end{remark}

\begin{example}\label{ex:Uniform.symbol-Deltaz}
 Suppose that $\Omega=\C$ and let $\chi(\xi)\in C_c^\infty(\R^n)$ be as in Remark~\ref{rem:Hol.holfamily-homogeneous-cutoff-classical}. Set 
 $\rho(z)(\xi)=(1-\chi(\xi))|\xi|^z$, $\xi\in \R^n$, $z\in \C$. As explained in Example~\ref{ex:Hol.powers}, this defines a holomorphic family in $S^\bt(\T^n_\theta \times \R^n)$ of order $z$. This is not a $\Hol_\infty(\C)$-family, since the derivatives $\partial_\xi^\beta|\xi|^z$, $\beta\neq 0$, grow polynomially in $z$, but, for every $\alpha>0$, $e^{\pm i\alpha z}\rho(z)$, $z\in \bH_\pm$, is a $\Hol_\infty(\bH_\pm)$-family in $S^\bt(\T^n_\theta \times \R^n)$. 
\end{example}

\begin{remark}
 The compactness condition in (i) ensures that if $\pm\Re w(z)<m$ on $\Omega$ for some $m\in \R$, then, for every closed vertical half-strip  $\Sigma\subseteq \Omega$, we also have $\sup\{\pm \Re w(z); z\in \Sigma\}<m$. This condition is automatically satisfied in the main examples where $w(z)=wz+q$ with $w\in \R\setminus 0$ and $q\in \C$. 
\end{remark}

\begin{remark}\label{rmk:uniform.holinf-symbols-classical-standard}
Let $\rho(z)$, ${z\in \Omega}$,  be a $\Hol_\infty(\Omega)$-family in $S^\bt(\T^n_\theta \times \R^n)$ of order $w(z)$. In the same way as in Remark~\ref{rmk:Hol.holfamily-classical-standard}, if $\Re w(z)\leq m$ on $\Omega$, then $\rho(z)$, ${z\in \Omega}$,  is a   $\Hol_\infty(\Omega)$-family in $\bS^m(\T^n_\theta \times \R^n)$. Moreover, as in Remark~\ref{rmk:Hol.holfamily-bounded-symbols}, if $\Sigma$ is any closed vertical half-strip contained in $\Omega$, and we set $m_\Sigma=\sup\{ \Re w(z); z\in \Sigma\}$, then 
 $\rho(z)$, $z\in \Sigma$, is a bounded family in $\bS^{m_\Sigma}(\T^n_\theta\times \R^n)$. 
\end{remark}

\begin{definition}\label{def:Holinf}
We say that a family $P(z)\in \Hol(\Omega;\Psi^\bt(\T^n_\theta))$ of order $w(z)$ is a $\Hol_\infty(\Omega)$-\emph{family} if there is a family
 $\rho(z)(\xi)\in \Hol_\infty(\Omega;S^\bt(\T^n_\theta \times \R^n))$ of order $w(z)$  such that 
\begin{equation*}
 P(z)=P_{\rho}(z)\qquad \forall z\in \Omega. 
\end{equation*}
We denote by $\Hol_\infty(\Omega;\Psi^\bt(\T^n_\theta))$ the class of $\Hol_\infty(\Omega)$-families in $\Psi^\bt(\T^n_\theta)$ indexed by $\Omega$. 
\end{definition}

\begin{example}
 We saw in Example~\ref{ex:Hol.Delta-powers} that  $\Delta^z$, $z\in \C$,  is a holomorphic family in $\Psi^\bt(\T^n_\theta)$ of order $2z$ and has symbol $\rho(z)(\xi)=(1-\chi(\xi))|\xi|^{2z}$ 
 with  $\chi(\xi)\in C_c^\infty(\R^n)$  as in Remark~\ref{rem:Hol.holfamily-homogeneous-cutoff-classical}. As pointed out in Example~\ref{ex:Uniform.symbol-Deltaz}, for every $\alpha>0$, $e^{\pm i\alpha z}\rho(z)$, $z\in \bH_\pm$, is a $\Hol_\infty(\bH_\pm)$-family in $S^\bt(\T_\theta^n\times \R^n)$. Therefore, we see that $e^{\pm i\alpha z}\Delta^z$, $z\in \bH_\pm$, is a $\Hol_\infty(\bH_\pm)$-family
  in $\Psi^\bt(\T^n_\theta)$.  
\end{example}

All the properties of holomorphic families of symbols and \psidos\ stated in Section~\ref{sec:hol-PsiDOs} extend to $\Hol_\infty(\Omega)$-families \emph{verbatim}, or with only minor modifications. In particular, we have the following analogue of Lemma~\ref{lem:hol.local-global}. 

\begin{lemma}\label{lem:unif.local-global}
 Let $P(z)\in \Hol(\Omega;\Psi^\bt(\T^n_\theta))$ have order $w(z)$. Assume further that $\Omega$ is contained in an upper or a lower half-plane (e.g., $\Omega=\bH_\pm$).
 The following are equivalent:
\begin{enumerate}
\item[(i)] $\rho_{\!{}_{P(z)}}$, ${z\in \Omega}$,  is a $\Hol_\infty(\Omega)$-family in $S^\bt(\T^n_\theta\times \R^n)$.  

 \item[(ii)] $P(z)$, ${z\in \Omega}$,  is a $\Hol_\infty(\Omega)$-family in $\Psi^\bt(\T^n_\theta)$. 
 
 \item[(iii)] For every $z_0\in \Omega$, there is an open vertical half-strip $\Sigma_0\subseteq \Omega$ containing $z_0$ such that 
 $P(z)\in \Hol_\infty(\Sigma_0;\Psi^\bt(\T^n_\theta))$. 
\end{enumerate}
\end{lemma}

For future reference, we mention the following $\Hol_\infty$-version of Corollary~\ref{cor:Hol.composition-of-psidos}. 

\begin{proposition}\label{prop:unif.composition}
 Let $P(z)$, ${z\in \Omega}$,  and $Q(z)$, ${z\in \Omega}$,  be $\Hol_\infty(\Omega)$-families in $\Psi^\bt(\T^n_\theta)$ of respective orders $w(z)$ and $u(z)$. 
 Then   $P(z)Q(z)$, ${z\in \Omega}$,  is a $\Hol_\infty(\Omega)$-family in $\Psi^\bt(\T^n_\theta)$ of order $w(z)+u(z)$. 
\end{proposition}

We also have the following $\Hol_\infty$-versions of Proposition~\ref{prop:Hol.propertiesP(z)-W2s} and Proposition~\ref{prop:Hol.propertiesP(z)-sLp}. 

\begin{proposition}\label{prop:unif.propertiesP(z)-W2s} 
Let $P(z)$, ${z\in \Omega}$,  be a $\Hol_\infty(\Omega)$-family in $\Psi^\bt(\T^n_\theta)$ of order $w(z)$. 
\begin{enumerate}
\item $P(z)$, ${z\in \Omega}$,  is a $\Hol_\infty(\Omega)$-family in $\cL(C^\infty(\T^n_\theta))$ and (uniquely extends to) a $\Hol_\infty(\Omega)$-family in $\cL(\scD'(\T^n_\theta))$. 
 
 \item If $\Re w(z)\leq m$ on $\Omega$, then $P(z)$, ${z\in \Omega}$,  is a $\Hol_\infty(\Omega)$-family in $\cL(W_2^{s+m}(\T^n_\theta), W_2^s(\T^n_\theta))$ for every $s\in \R$.
  
 \item Let $\Sigma\subseteq \Omega$ be a closed vertical half-strip or strip, and set $w_\Sigma=\sup\{\Re w(z); \ z\in \Sigma\}$. Then $P(z)$, $z\in \Sigma$, is a strongly continuous bounded family in $\sL(W_2^{s+w_\Sigma}(\T^n_\theta), W_2^s(\T^n_\theta))$ for every $s\in \R$. 
\end{enumerate}
\end{proposition}
\begin{proof}
 By arguing along the same lines as the proof of Proposition~\ref{prop:Hol.propertiesP(z)-W2s}  we obtain the first two parts, as well as the third part in the case of closed vertical half-strips. The extension to closed vertical strips follows from Remark~\ref{rmk:Hol.boundedness-vertical-strips}. 
\end{proof}

\begin{proposition}\label{prop:unif.propertiesP(z)-sLp} 
Let $P(z)$, ${z\in \Omega}$,  be a $\Hol_\infty(\Omega)$-family in $\Psi^\bt(\T^n_\theta)$ of order $w(z)$. 
\begin{enumerate}
 \item If $\Re w(z)<m$ (resp., $\Re w(z)\leq m$) on $\Omega$ with $m<0$, then $P(z)$, $z\in \Omega$,  is a  $\Hol_\infty(\Omega)$-family in $\sL_{p}$ (resp., $\sL_{p,\infty}$) with $p=n|m|^{-1}$. 
 
 \item   Let $\Sigma\subseteq \Omega$ be a closed vertical half-strip or strip such that $w_\Sigma=\sup\{\Re w(z); \ z\in \Sigma\}<0$, and set $p=n|w_\Sigma|^{-1}$. The family $P(z)$, $z\in \Sigma$, is a bounded family in $\sL_{p,\infty}$, and hence is bounded in $\sL_r$ for all $r>p$. 
\end{enumerate}
\end{proposition}
\begin{proof}
 As with the proof of Proposition~\ref{prop:unif.propertiesP(z)-W2s}, arguing along the same lines as the proof of Proposition~\ref{prop:Hol.propertiesP(z)-sLp} yields the first part, as well as the second part for closed vertical half-strips. Using Remark~\ref{rmk:Hol.boundedness-vertical-strips} allows us to extend the result to closed vertical strips. 
\end{proof}

In what follows, we set  $\Omega_\pm =\Omega\cap\bH_\pm$. 

\begin{lemma}\label{lem:Unif.Holinf-rapid-decay}
 Let $u(z)\in \Hol(\Omega;\sE)$. Given any $\alpha \geq 0$, the following are equivalent:    
\begin{enumerate}
\item[(i)]  $e^{\pm i\alpha z}u(z)\in\Hol_\infty(\Omega_\pm;\sE)$.

\item[(ii)] $u(z)\in  e^{\alpha |\Im z|}L^\infty(\Sigma;\sE)$ for every  closed vertical half-strip $\Sigma \subseteq \Omega$. 
\end{enumerate}
Moreover, if (i)--(ii) hold, then $u(z)\in  e^{\alpha |\Im z|}L^\infty(\Sigma;\sE)$ for every  closed vertical strip $\Sigma \subseteq \Omega$. 
\end{lemma}
\begin{proof}
It is clear that $e^{\pm i\alpha z}u(z)\in\Hol(\Omega_\pm;\sE)$. If (ii) holds, then $e^{\pm i\alpha z}u(z)$ is uniformly bounded on every closed vertical half-strip $\Sigma \subseteq \Omega_\pm$, and so  $e^{\pm i\alpha z}u(z)\in\Hol_\infty(\Omega_\pm;\sE)$. Furthermore, it follows from Remark~\ref{rmk:Hol.boundedness-vertical-strips} that $u(z)$ is in  $e^{\alpha |\Im z|}L^\infty(\Sigma;\sE)$ for every  closed vertical strip $\Sigma \subseteq \Omega$. 
  
It remains to show that~(i) implies~(ii). Assume that $e^{\pm i\alpha z}u(z)\in\Hol_\infty(\Omega_\pm;\sE)$. Let $\Sigma \subseteq \Omega$ be a closed vertical half-strip. Set $\Sigma_\pm= \Sigma \cap \{\pm \Im z \geq 1\}$ and $K=\Sigma \cap \{|\Im z|\leq 1\}$. Here $\Sigma_\pm$ is either a closed vertical half-strip in $\Omega_\pm$  or a compact subset. Either way $e^{\pm i \alpha z}u(z)\in L^\infty(\Sigma_\pm;\sE)$, and hence $u(z)\in e^{ \alpha |\Im z|}L^\infty(\Sigma_\pm;\sE)$. Moreover, as $K$ is a compact subset of $\Omega$, we also have that  $u(z)\in e^{ \alpha |\Im z|}L^\infty(K;\sE)$. As $\Sigma=\Sigma_{-}\cup \Sigma_+\cup K$, it follows that $u(z)\in  e^{\alpha |\Im z|}L^\infty(\Sigma;\sE)$. This shows that~(i) implies~(ii).  
\end{proof}

\begin{remark}
 If $\sE$ is locally convex, then~(ii) means that, for every continuous semi-norm $\fp$ on $\sE$, there is $C_{\Sigma \fp}>0$ such that 
\begin{equation*}
\fp\big[ u(z)\big] \leq C_{\Sigma \fp} e^{\alpha |\Im z|} \qquad \forall z\in \Sigma.  
\end{equation*}
In case $\sE$ is a quasi-Banach space, we obtain a similar estimate in terms of the quasi-norm of $\sE$. Moreover, as in Remark~\ref{rmk:Hol.boundedness-vertical-strips}, such estimates hold on every closed vertical strip contained in $\Omega$. 
\end{remark}

Combining Proposition~\ref{prop:unif.propertiesP(z)-W2s} and Proposition~\ref{prop:unif.propertiesP(z)-sLp} with Lemma~\ref{lem:Unif.Holinf-rapid-decay}, we arrive at the following result. 

\begin{proposition}\label{prop:unif.exp-control} 
Let $P(z)$, ${z\in \Omega}$,  be a holomorphic family   in $\Psi^\bt(\T^n_\theta)$ of order $w(z)$ for which there is $\alpha\geq 0$ such that
\begin{equation*}
 e^{\pm i\alpha z}P(z)\in \Hol_{\infty}\big(\Omega_\pm;\Psi^\bt(\T^n_\theta)\big). 
\end{equation*}
 Let $\Sigma\subseteq \Omega$ be a closed vertical half-strip or strip, and set $w_\Sigma=\sup\{\Re w(z); \ z\in \Sigma\}$. 
\begin{enumerate}
 \item  $P(z) \in e^{\alpha |\Im z|} L^\infty(\Sigma;\sL(W_2^{s+w_\Sigma}(\T^n_\theta), W_2^{s}(\T^n_\theta)))$ for all $s\in \R$. 

\item If $w_\Sigma<0$, and we set $p=n|w_\Sigma|^{-1}$, then $P(z)$ is in $e^{\alpha |\Im z|} L^\infty\left(\Sigma;\sL_{p,\infty}\right)$, and hence in
$e^{\alpha |\Im z|} L^\infty\left(\Sigma;\sL_{r}\right)$ for all $r>p$. 
\end{enumerate}
\end{proposition}
 
\begin{remark}
 The above result is the main motivation for the definition of $\Hol_\infty$-families of \psidos. We will see later that complex powers of elliptic \psidos\ satisfy the assumptions of Proposition~\ref{prop:unif.exp-control}, and so this will immediately provide uniform exponential bounds along vertical strips for such families. 
 \end{remark}

\section{Complex Powers of Elliptic \psidos}\label{sec:powers} 
In this section, we construct the complex powers of positive-order elliptic \psidos\ on $\T^n_\theta$ as holomorphic families of \psidos. This will provide us with analogues for NC tori of the well-known results of Seeley~\cite{Se:PSPM67}. 

\subsection{Holomorphic families of \psidos\ out of parametric symbols}\label{subsec:powers.parametric-hol}
In what follows, we let $\Lambda$ be an open pseudo-cone of the form,
\begin{equation*}
 \Lambda= \Theta \cup D(0,r_0) \qquad \text{or} \qquad \Lambda= \Theta \cup [D(0,r_0)\setminus \{0\}] , 
\end{equation*}
where $r_0>0$ and $\Theta\subseteq \C^*$ is a non-empty open cone. In addition, we let $L_\phi=\{\arg \lambda=\phi\}$ be a ray contained in $\Theta$. 

Given rays $L_{\phi_j}=\{\arg \lambda =\phi_j\}$, $j=1,2$, contained in $\Theta$ with $\phi-2\pi\leq \phi_2<\phi_1\leq \phi$, we form the oriented contour,  
\begin{gather}\label{eq:powers.contourG1} 
\Gamma(\phi_1,\phi_2,r):= \Gamma(\phi_1,r)^- \cup C(\phi_1,\phi_2,r) \cup \Gamma(\phi_2,r)^+, \quad \text{where}\\ 
 \Gamma(\phi_j,r)^\pm=[r,\infty)e^{i\phi_j}, \qquad C(\phi_1,\phi_2,r) =\{re^{i\eta};\phi_1>\eta>\phi_2\}, \qquad 0<r<r_0. 
 \label{eq:powers.contourG2}
\end{gather}
We orient $\Gamma(\phi_1,\phi_2,r)$ so that $\Gamma(\phi_1,r)^-$ (resp., $\Gamma(\phi_2,r)^+$) is oriented toward (resp., away from) the origin and $C(\phi_1,\phi_2,r)$ is oriented clockwise. 
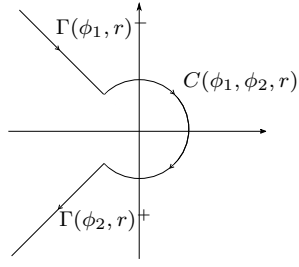
\begin{figure}[h]
\begin{minipage}{0.25\linewidth}
\centering
\begin{tikzpicture}[cpfig]\cpbox
 \cpaxes
 \draw[arrowat=0.45] (135:5.95)--(135:1.75);
 \draw[arrowsat={0.333}{0.667}] (135:1.75) arc[start angle=135,end angle=-135,radius=1.75];
 \draw[arrowat=0.49] (-135:1.75)--(-135:6.2);
 \node[anchor=base west,inner sep=0pt] at (1.628,1.472) {\footnotesize $C(\phi_1,\phi_2,r)$};
 \node[anchor=base west,inner sep=0pt] at (-2.882,3.337) {\footnotesize $\Gamma(\phi_1,r)^-$};
 \node[anchor=base west,inner sep=0pt] at (-2.766,-3.462) {\footnotesize $\Gamma(\phi_2,r)^+$};
\end{tikzpicture}
\end{minipage}
\caption{Contour $\Gamma(\phi_1,\phi_2,r)$.}
\end{figure}

We refer to~\cite[Appendix~A]{LP:Part1} for background on improper integrals $\int_{\Gamma} u(\lambda)d\lambda$ associated with continuous maps $u:\Gamma \rightarrow \sE$, where $\sE$ is a given locally convex space. Such an integral is called \emph{convergent} if the limits $\lim_{a\rightarrow \infty}\int_{re^{i\phi_j}}^{ae^{i\phi_j}}u(\lambda)d\lambda$, $j=1,2$,  
both exist in $\sE$. We then define
\begin{equation*}
 \int_{\Gamma} u(\lambda)d\lambda := -\lim_{a\rightarrow \infty}\int_{re^{i\phi_1}}^{ae^{i\phi_1}}u(\lambda)d\lambda + 
 \int_{C(\phi_1,\phi_2,r)} u(\lambda)d\lambda + \lim_{a\rightarrow \infty}\int_{re^{i\phi_2}}^{ae^{i\phi_2}}u(\lambda)d\lambda.  
\end{equation*}

\begin{proposition}[see~{\cite[Appendix~A]{LP:Part1}}]\label{prop:powers.improper-int} 
 Assume that the integral $\int_{\Gamma} u(\lambda)d\lambda$ converges.
\begin{enumerate}
 \item $ \int_{\Gamma} u(\lambda)d\lambda$ is the unique element of $\sE$ such that
\begin{equation*}
 \int_\Gamma \varphi\big[u(\lambda)\big] d\lambda = \varphi\bigg(  \int_{\Gamma} u(\lambda)d\lambda\bigg) \qquad \forall \varphi \in \sE'. 
\end{equation*}

\item Given any continuous linear map $\Phi:\sE\rightarrow \sE_1$, where  $\sE_1$ is some locally convex space,  the integral  
$\int_\Gamma \Phi\circ u(\lambda)d\lambda$ converges, and we have
\begin{equation*}
 \int_\Gamma \Phi\circ u(\lambda)d\lambda =  \Phi\bigg( \int_{\Gamma} u(\lambda)d\lambda\bigg) . 
\end{equation*}
\end{enumerate}
 \end{proposition}

We shall further say that a continuous (or even measurable) map $u:\Gamma \rightarrow \sE$ is \emph{absolutely integrable} if 
\begin{equation*}
 \int_\Gamma \mathfrak{p}\big[u(\lambda)\big] |d\lambda| <\infty \quad \text{for every continuous semi-norm $\mathfrak{p}$ on $\sE$}. 
\end{equation*}
For instance, any $\Hol^{d}(\Lambda)$-family in $\sE$ with $d<-1$ restricts to an absolutely integrable map on $\Gamma$. In case $\sE$ is a Banach space, absolute integrability simply means that $u(\lambda)$ is Bochner-integrable.

\begin{proposition}[see~{\cite[Appendix~A]{LP:Part1}}] 
Assume that $\sE$ is quasi-complete and $u:\Gamma \rightarrow \sE$ is continuous and absolutely integrable. 
\begin{enumerate}
 \item The integral $ \int_{\Gamma} u(\lambda)d\lambda$ is convergent. 
 
 \item  For every continuous semi-norm $\fp$ on $\sE$, we have
\begin{equation}
 \fp\bigg( \int_{\Gamma} u(\lambda)d\lambda\bigg) \leq  \int_{\Gamma}\fp\big[  u(\lambda)\big] |d\lambda|.
 \label{eq:Powers.int-semi-norms}  
\end{equation}
\end{enumerate}
\end{proposition}

\begin{remark}
 A locally convex space is called \emph{complete} if every Cauchy net is convergent. It is called \emph{quasi-complete} if every bounded Cauchy net is convergent. All the locally convex spaces under consideration in this section are complete. For instance, the completeness of $\scD'(\T^n_\theta)$ and $\sL(C^\infty(\T^n_\theta))$ follows from the fact that $C^\infty(\T^n_\theta)$ is a Fr\'echet--Montel space (see~\cite{Tr:AP67}). All the integrals under consideration in this and the following sections are integrals of absolutely integrable maps. 
\end{remark}

Given any $z\in \C$, we define the power function $\lambda \rightarrow \lambda_\phi^z$ by 
\begin{equation}\label{eq:Powers.Powers-scalars}
 \lambda_\phi^z= |\lambda|^z e^{iz \arg_{\phi} (\lambda)},
\end{equation}
where $\arg_\phi:\C^*\setminus L_\phi\rightarrow (\phi-2\pi,\phi)$ is the continuous determination of the argument associated with $\phi$. If $\lambda$ goes along 
$\Gamma(\phi,r)^-$ (resp., $\Gamma(\phi-2\pi,r)^{+}$), then we make the convention that $\lambda_\phi^z=|\lambda|^ze^{iz\phi}$ (resp., 
$\lambda_\phi^z=|\lambda|^ze^{iz(\phi-2\pi)}$). With this convention $(z,\lambda)\rightarrow \lambda_\phi^z$ is a well-defined continuous function on $\C\times \Gamma$ which is holomorphic with respect to $z$. 

In the following, we put $\Gamma=\Gamma(\phi_1,\phi_2,r)$, with $(\phi_1,\phi_2,r)$ as above. 

\begin{lemma}\label{lem:Powers.LCS} 
 Suppose that $\sE$ is a quasi-complete locally convex space, and let $v(\lambda)\in \Hol^{-1}(\Lambda;\sE)$. Then the integral
 \begin{equation}
 u(z)= \int_\Gamma \lambda_\phi^z v(\lambda)d\lambda, \qquad \Re z<0,
 \label{eq:powers.int-LCS} 
\end{equation}
converges in $\sE$, and this defines a family $u(z)\in \Hol(\Re z<0;\sE)$. 
\end{lemma}
\begin{proof}
 Let $a>0$. As $v(\lambda)\in \Hol^{-1}(\Lambda;\sE)$, given any compact set $K\subseteq \{\Re z<-a\}$ and any continuous semi-norm $\fp$ on $\sE$, there is $C_{K\fp}>0$ such that
 \begin{equation}
 \fp\big[\lambda_\phi^z v(\lambda)\big] \leq C_{K\fp} \left(1+|\lambda|\right)^{-1-a} \qquad \forall \lambda \in \Gamma, \quad \forall z\in K.
 \label{eq:powers.estimates-v}  
\end{equation}
In particular, if $\Re z <-a$, then $\lambda_\phi^zv(\lambda)$ is an absolutely integrable map from $\Gamma$ to $\sE$, and so the integral in~(\ref{eq:powers.int-LCS}) converges in $\sE$. Furthermore, for any $\varphi \in \sE'$, we have 
\begin{equation*}
 \varphi\left[u(z)\right] =  \int_\Gamma \lambda_\phi^z  \varphi\left[v(\lambda)\right]d\lambda. 
\end{equation*}
The estimates~(\ref{eq:powers.estimates-v}) ensure that $\lambda_\phi^z  \varphi[v(\lambda)]$ is a locally bounded holomorphic map from $\{\Re z <-a\}$ to $L^1(\Gamma)$. It then follows that $\varphi[u(z)]$ is a holomorphic function on $\{\Re z<-a\}$ for every $\varphi\in \sE'$. By Proposition~\ref{prop:Hol.weakly-strong}, this shows that $u(z)\in \Hol(\Re z <-a;\sE)$ for every $a>0$, and hence $u(z)\in \Hol(\Re z <0;\sE)$. 
\end{proof}

\begin{remark}\label{lem:Powers.contour-LCS}
 Let $\delta>0$ be such that the sectors $\{|\arg \lambda -\phi_j|\leq \delta\}$, $j=1,2$, are contained in $\Theta$. Let $r'\in (0,r_0)$ and let $\phi'_j$, $j=1,2$, be angles such that $|\phi'_j-\phi_j|\leq \delta$ and $\phi-2\pi\leq \phi_2'<\phi_1'\leq \phi$. It follows from~\cite[Proposition~A.7(3)]{LP:Part1} that in the formula~(\ref{eq:powers.int-LCS}) the value of $u(z)$ remains unchanged if we replace the contour $\Gamma=\Gamma(\phi_1,\phi_2,r)$ by  $\Gamma(\phi_1',\phi_2',r')$. More generally, if $d\in \R$ and $v(\lambda)\in \Hol^{d}(\Lambda;\sE)$, then the formula~(\ref{eq:powers.int-LCS}) makes sense for $\Re z<-1-d$ and defines a family in $\Hol(\Re z<-1-d;\sE)$. 
\end{remark}

\begin{lemma}\label{lem:Powers.st-symbols}
 Let $\sigma(\xi;\lambda)\in \stS^{m,-1}(\T^n_\theta\times \R^n\times \Lambda)$, $m\in \R$. Define 
\begin{equation*}
 \rho(z)(\xi):=\int_\Gamma \lambda_\phi^z \sigma(\xi;\lambda)d\lambda, \qquad \Re z<0,
\end{equation*}
where the integral converges in $ \stS^{m}(\T^n_\theta\times \R^n)$. Then $\rho(z)\in \Hol(\Re z<0; \stS^{m}(\T^n_\theta\times \R^n))$, and we have
\begin{equation}
 P_\rho(z) = \int_\Gamma \lambda_\phi^z P_\sigma(\lambda)d\lambda, \qquad \Re z<0,
  \label{eq:powers.int-P-st-symbol}  
\end{equation}
where the integral converges in $\sL(C^\infty(\T^n_\theta))$. 
\end{lemma}
\begin{proof}
The first part is an immediate consequence of Lemma~\ref{lem:Powers.LCS}. In particular, if $\Re z<0$, then $ \lambda_\phi^z \sigma(\xi;\lambda)$ is an absolutely integrable map from $\Gamma$ to $\stS^m(\T^n_\theta\times \R^n)$. As $\sigma \rightarrow P_\sigma$ is a continuous linear map from $\stS^m(\T^n_\theta\times \R^n)$ to 
$\sL(C^\infty(\T^n_\theta))$ (see~\cite[Proposition~5.4]{HLP:Part1}), the second part is a direct consequence of the second part of Proposition~\ref{prop:powers.improper-int}.
\end{proof}

\begin{remark}
 Given $s\in \R$, the map $\sigma \rightarrow P_\sigma$ is continuous from $\stS^m(\T^n_\theta \times \R^n)$ to $\sL(W_2^{s+m}(\T^n_\theta), W_2^s(\T^n_\theta))$. Thus, in the same way as in the proof of Lemma~\ref{lem:Powers.st-symbols}, it can be shown that the integral in~(\ref{eq:powers.int-P-st-symbol}) actually converges in $\sL(W_2^{s+m}(\T^n_\theta), W_2^s(\T^n_\theta))$ for all $s\in \R$. In particular, if $m=0$, then it converges in $\sL(L_2(\T^n_\theta))$. 
\end{remark}

Given any $c>0$, we define the open set $\Omega_c(\Theta)\subseteq (\R^n\setminus 0)\times \C$ as in~(\ref{eq:Parameter.parameter-set-for-homogeneous-symbols}). That is, 
\begin{equation}\label{eq:Powers.parameter-set-for-homogeneous-symbols}
 \Omega_c(\Theta)=\bigcup_{\xi\neq 0} \{\xi\}\times \Lambda\left(c|\xi|^w\right),
\end{equation}
where for $r>0$ we set
\begin{equation*}
 \Lambda(r):=\Theta \cup \{|\lambda|<r\}. 
\end{equation*}
Note that, by assumption, $\Gamma \subsubset \Lambda(r')\setminus\{0\}\subseteq \Lambda(r_0)\setminus\{0\} \subseteq \Lambda$ for all $r'\in(r,r_0]$. 

\begin{lemma}\label{lem:Powers.hom-symbols}
 Let $\sigma(\xi;\lambda)\in S^{-1}_{m}(\T^n_\theta\times \Omega_c(\Theta))$ with $c>r$. Define 
 \begin{equation}
 \rho(z)(\xi) = \int_{|\xi|^w\Gamma}  \lambda_\phi^z \sigma(\xi;\lambda)d\lambda, \qquad \xi\neq 0, \quad \Re z <0,
 \label{eq:powers.int-hom-symb}
\end{equation}
where the integral converges in $C^\infty(\T^n_\theta)$. Then $\rho(z)(\xi)$ is a holomorphic family in 
$S_\bt(\T^n_\theta \times \R^n)$ of degree $m+w(z+1)$. 
\end{lemma}
\begin{proof}
 By assumption $\sigma(\xi;\lambda)\in C^{\infty,\omega}(\T^n_\theta\times \Omega_c(\Theta))$. Thus, given any $\xi \neq 0$, the map $\lambda  \rightarrow \sigma(\xi;\lambda)$ is in $\Hol^{-1}(\Lambda(c|\xi|^w); C^\infty(\T^n_\theta))$. Moreover, as $r<c$, we have $|\xi|^w\Gamma =\Gamma(\phi_1,\phi_2, r|\xi|^w) \subseteq \Lambda(c|\xi|^w)$. Therefore, by Lemma~\ref{lem:Powers.LCS}, the formula~(\ref{eq:powers.int-hom-symb}) defines a family in $\Hol(\Re z<0; C^\infty(\T^n_\theta))$. 

Let $R>0$, and set $U_R=\{|\xi|>R\}$. If $\xi \in U_R$, then, by Remark~\ref{lem:Powers.contour-LCS}, in~(\ref{eq:powers.int-hom-symb}) we may replace the contour $|\xi|^w\Gamma$ by $R^w\Gamma$ without changing the value of the integral, i.e., 
\begin{equation*}
  \rho(z)(\xi) = \int_{R^w\Gamma}  \lambda_\phi^z \sigma(\xi;\lambda)d\lambda, \qquad |\xi|>R, \quad \Re z <0. 
\end{equation*}
 Note that, as $R<|\xi|$ and $r<c$, we have $R^w\Gamma =\Gamma(\phi_1,\phi_2,rR^w)\subseteq \Lambda(c|\xi|^w)$. In particular, $U_R\times \Lambda(cR^w)$ is an open subset of $\Omega_c(\Theta)$, and so $\sigma(\xi;\lambda)\in \Hol^{-1}(\Lambda(cR^w);C^\infty(\T^n_\theta \times U_R))$. Therefore, by Lemma~\ref{lem:Powers.LCS}, the above formula actually defines a family in $\Hol(\Re z<0; C^\infty(\T^n_\theta \times U_R))$ for all $R>0$, and hence we get a family in $\Hol(\Re z<0; C^\infty(\T^n_\theta\times (\R^n\setminus 0)))$. 
 
 It remains to check the homogeneity of $\rho(z)(\xi)$. Let $\xi \neq 0$ and $t>0$. We have
 \begin{equation*}
 \rho(z)(t\xi) = \int_{t^w|\xi|^w\Gamma}  \lambda_\phi^z \sigma(t\xi;\lambda)d\lambda=  t^{w(z+1)} \int_{|\xi|^w\Gamma}  \lambda_\phi^z \sigma(t\xi;t^w\lambda)d\lambda. 
\end{equation*}
As $\sigma(t\xi;t^w\lambda)=t^m \sigma(\xi;\lambda)$, we get
\begin{equation*}
  \rho(z)(t\xi) =  t^{m+w(z+1)}\int_{|\xi|^w\Gamma} \lambda_\phi^z \sigma(\xi;\lambda)d\lambda= t^{m+w(z+1)} \rho(z)(\xi). 
\end{equation*}
 This shows that  $\rho(z)(\xi) \in S_{m+w(z+1)}(\T^n_\theta \times \R^n)$. 
\end{proof}
In the same way as in Remark~\ref{lem:Powers.contour-LCS}, in the formula~(\ref{eq:powers.int-hom-symb}) the value of the integral is unaffected by replacing $\Gamma$ by any contour $\Gamma(\phi_1',\phi_2',r')$, $0<r'<c$, with $(\phi_1',\phi_2')$ as in Remark~\ref{lem:Powers.contour-LCS}. 

\begin{proposition}\label{prop:powers.param-psido-hol-psidos} 
Let $Q(\lambda)\in \Psi^{m,-1}(\T^n_\theta;\Lambda)$ have symbol $\sigma(\xi;\lambda)\sim \sum_{j\geq 0} \sigma_{m-j}(\xi;\lambda)$,  with $\sigma_{m-j}(\xi;\lambda)$ in $S_{m-j}^{-1}(\T^n_\theta\times \Omega_c(\Theta))$, $c>0$. We also assume that $\Gamma=\Gamma(\phi_1,\phi_2,r)$ with $0<r<\min(r_0,c)$.
\begin{enumerate}
 \item The integral
\begin{equation}
 P(z) = \int_{\Gamma} \lambda_\phi^z  Q(\lambda)d\lambda, \qquad \Re z <0, 
  \label{eq:powers.int-PsiDO} 
\end{equation}
converges in $\sL(C^\infty(\T^n_\theta))$, and this defines a holomorphic family in $\Psi^\bt(\T^n_\theta)$ of order ${m+w(z+1)}$. 
 
 \item  $P(z)$ has symbol $\rho(z)(\xi) \sim\sum_{j\geq 0}  \rho_j(z)(\xi)$, with $\rho_j(z)(\xi)\in S_{m-j+w(z+1)}(\T^n_\theta\times \R^n)$ given by
 \begin{equation}
 \rho_j(z)(\xi) = \int_{|\xi|^w\Gamma} \lambda_\phi^z  \sigma_{m-j}(\xi;\lambda)d\lambda, \qquad \xi\neq 0, \quad \Re z <0, 
 \label{eq:powers.int-classical-symb-hom} 
\end{equation}
where the integral converges in $C^\infty(\T^n_\theta)$. 
\end{enumerate}
\end{proposition}
\begin{proof}
 By assumption, $Q(\lambda)$ is in $\Psi^{m,-1}(\T^n_\theta;\Lambda)$, and so we know from Remark~\ref{rmk:Parametric.parametric-Psidos-families-Cinfty-sDprime} that $Q(\lambda)$ is a family in $\Hol^{-1}(\Lambda;\sL(C^\infty(\T^n_\theta)))$. Lemma~\ref{lem:Powers.LCS} then ensures that the integral in~(\ref{eq:powers.int-PsiDO}) converges in $\sL(C^\infty(\T^n_\theta))$ and defines a family in  $\Hol(\Re z<0; \sL(C^\infty(\T^n_\theta)))$. 
 We also know from Proposition~\ref{symbols:inclusion-classical-standard} that $\sigma(\xi;\lambda)\in \stS^{m+w,-1}(\T^n_\theta\times \R^n\times \Lambda)$. Thus, by Lemma~\ref{lem:Powers.st-symbols}, we define a family in $\Hol(\Re z<0;\stS^{m+w}(\T^n_\theta\times \R^n))$ by letting
 \begin{equation*}
  \rho(z)(\xi) =\int_{\Gamma} \lambda_\phi^z  \sigma(\xi;\lambda)d\lambda, \qquad \Re z <0, 
\end{equation*}
where the integral converges in $\stS^{m+w}(\T^n_\theta\times \R^n)$. Moreover, we have
\begin{equation*}
 P(z)=  \int_{\Gamma} \lambda_\phi^z  P_\sigma(\lambda)d\lambda=P_\rho(z), \qquad \Re z<0,
\end{equation*}
where the integral converges in $\sL(C^\infty(\T^n_\theta))$.

In addition, given any $j\geq 0$, Lemma~\ref{lem:Powers.hom-symbols} ensures that~(\ref{eq:powers.int-classical-symb-hom}) defines a holomorphic family in $S_\bt(\T^n_\theta\times \R^n)$ of degree $m-j+w(z+1)$ over the half-plane $\{\Re z<0\}$. In particular, we get holomorphic families in $C^\infty(\T^n_\theta\times (\R^n\setminus 0))$. 

 In order to complete the proof, we just need to show that $\rho(z)(\xi) \sim\sum_{j\geq 0}  \rho_j(z)(\xi)$ in the sense of~(\ref{eq:Hol.hol-family-classical-symbol-estimate}). Let $\chi(\xi)\in C^\infty_c(\R^n)$ be such that $\chi(\xi)=1$ for $|\xi|\leq (c^{-1}r_0)^{1/w}$. Lemma~\ref{lem:Parameter.homogeneous-symbol-estimate} and Remark~\ref{rmk:Parameter.classical-symbol-asymptotics-equivalent-conditions} ensure that $(1-\chi(\xi))\sigma_{m-j}(\xi;\lambda) \in \stS^{m+w-j,-1}(\T^n_\theta\times \R^n\times \Lambda)$ and $\sigma(\xi;\lambda) \sim \sum_{j\geq 0} (1-\chi(\xi))\sigma_{m-j}(\xi;\lambda)$ in the sense of~(\ref{eq:Parameter.Standard-asymptotic}). Thus, given any $N\geq 0$, there is $J_0\geq 1$ such that, for each $J\geq J_0$, we have
\begin{equation*}
s_{NJ}(\xi;\lambda):= \sigma(\xi;\lambda)- \sum_{j<J} \left(1-\chi(\xi)\right)\sigma_{m-j}(\xi;\lambda) \in \stS^{-N,-1}(\T^n_\theta \times \R^n \times \Lambda).
\end{equation*}
 Thus, for $\Re z<0$, we have 
 \begin{equation*}
 \rho(z)(\xi) = \sum_{j<J} \tilde{\rho}_j(z)(\xi) + r_{NJ}(z)(\xi), 
\end{equation*}
where we have set 
\begin{gather}
 \tilde{\rho}_j(z)(\xi): = \int_{\Gamma} \lambda_\phi^z  \left(1-\chi(\xi)\right)\sigma_{m-j}(\xi;\lambda)d\lambda, \nonumber\\
  r_{NJ}(z)(\xi):= \int_{\Gamma} \lambda_\phi^z  s_{NJ}(\xi;\lambda)d\lambda. 
 \label{eq:powers.tilde-rho-rNJ}
\end{gather}
 
 It follows from Lemma~\ref{lem:Powers.st-symbols} that $ r_{NJ}(z)(\xi)\in \Hol(\Re z<0; \stS^{-N}(\T^n_\theta \times \R^n))$. We claim that 
\begin{equation}
 \tilde{\rho}_j(z)(\xi)= \left(1-\chi(\xi)\right)\rho_j(z)(\xi). 
 \label{eq:powers.tilde-rho-rho}
\end{equation}
The equality holds trivially for $|\xi|\leq (c^{-1}r_0)^{1/w}$, since in this case both sides are equal to $0$. As explained in the proof of Lemma~\ref{lem:Powers.hom-symbols}, for $\xi\neq 0$, the map $\lambda \rightarrow \sigma_{m-j}(\xi;\lambda)$ is in $\Hol^{-1}(\Lambda(c|\xi|^w);C^\infty(\T^n_\theta))$. By assumption, $r<c$, and hence $r|\xi|^w<c|\xi|^w$. Furthermore, if $|\xi|> (c^{-1}r_0)^{1/w}$, then $c|\xi|^w>r_0>r$. Therefore, in this case, it follows from Remark~\ref{lem:Powers.contour-LCS} that, in the definition~(\ref{eq:powers.tilde-rho-rNJ}) of $\tilde{\rho}_j(z)(\xi)$, we may replace the contour $\Gamma=\Gamma(\phi_1,\phi_2,r)$ by $\Gamma(\phi_1,\phi_2, r|\xi|^w)=|\xi|^w\Gamma$ without changing the value of $\tilde{\rho}_j(z)(\xi)$. Thus,
\begin{equation*}
 \tilde{\rho}_j(z)(\xi)= \int_{|\xi|^w\Gamma} \left(1-\chi(\xi)\right)\lambda_\phi^z  \sigma_{m-j}(\xi;\lambda)d\lambda= \left(1-\chi(\xi)\right)\rho_j(z)(\xi).  
\end{equation*}
This shows that~(\ref{eq:powers.tilde-rho-rho}) holds for all $\xi\in \R^n$. 

All this shows that, for all $N\geq 0$, there is $J_0\geq 1$ such that 
\begin{equation*}
 \rho(z)(\xi) - \sum_{j<J}\left(1-\chi(\xi)\right)\rho_j(z)(\xi)\in \Hol\left(\Re z<0; \stS^{-N}(\T^n_\theta \times \R^n)\right)\qquad \forall J\geq J_0. 
\end{equation*}
 This implies that $\rho(z)(\xi) \sim \sum_{j\geq 0} (1-\chi(\xi))\rho_j(z)(\xi)$ in the sense of~(\ref{eq:Hol.asymptotic-standardNJ}), and so, by Lemma~\ref{lem:Hol.asymptotic-standard-NJ}, the asymptotic expansion holds in the sense of~(\ref{eq:Hol.asymptotic-standard}). 
 Changing the cut-off $\chi(\xi)$ changes each $(1-\chi(\xi))\rho_j(z)(\xi)$ only by a holomorphic family in $\stS^{-\infty}(\T^n_\theta\times\R^n)$, and so this expansion also holds with $\chi(\xi)$ as in Remark~\ref{rem:Hol.holfamily-homogeneous-cutoff-classical}. Lemma~\ref{lem:Hol.asymptotic-classical-standard}  then ensures that $\rho(z)(\xi) \sim\sum_{j\geq 0}  \rho_j(z)(\xi)$ in the sense of~(\ref{eq:Hol.hol-family-classical-symbol-estimate}). 
 \end{proof}

Assume now that $\phi=\pi$. As $\Theta$ is an open cone, this implies that there is $\alpha\in (0,\pi)$ such that
\begin{equation}\label{eq:unif.Theta-alpha}
 \Theta \supseteq \{ \alpha \leq \arg \lambda \leq 2\pi -\alpha\}. 
\end{equation}
To simplify notation, we drop the subscript from the notation of the powers~(\ref{eq:Powers.Powers-scalars}) associated with $\phi=\pi$ and simply denote them by $\lambda^z$, $z\in \C$. In addition, we let $\Gamma$ be any contour of the form $\Gamma(\pi,-\pi,r)$ as in~(\ref{eq:powers.contourG1})--(\ref{eq:powers.contourG2}) with $0<r<r_0$. 

We shall now explain that, under the above conditions, the holomorphic families of \psidos\ produced by Proposition~\ref{prop:powers.param-psido-hol-psidos} become $\Hol_\infty$-families on $\bQ_\pm$ after multiplication by $e^{\pm i\alpha z}$. This is a consequence of the following observation. 

\begin{lemma}\label{lem:Powers.LCS-boundedness} 
Let $\sE$ be a quasi-complete locally convex space, and let $v(\lambda)\in \Hol^{-1}(\Lambda;\sE)$. In the same way as in Lemma~\ref{lem:Powers.LCS}, define 
\begin{equation}
 u(z)= \int_\Gamma \lambda^z v(\lambda)d\lambda, \qquad \Re z<0. 
 \label{eq:unif.int-LCS} 
\end{equation}
Then $e^{\pm i\alpha z}u(z)\in\Hol_\infty(\bQ_\pm;\sE)$. 
\end{lemma}
\begin{proof}
We know from Lemma~\ref{lem:Powers.LCS} that $u(z)$, $\Re z<0$, is a holomorphic family in $\sE$. By assumption, the angular sector $\{\alpha \leq \arg \lambda \leq 2\pi-\alpha\}$ is contained in $\Theta$. Therefore, as pointed out in Remark~\ref{lem:Powers.contour-LCS}, in the 
formula~(\ref{eq:unif.int-LCS}) defining $u(z)$, we may replace the contour $\Gamma=\Gamma(\pi,-\pi,r)$ by the contour $\Gamma':=\Gamma(\alpha,-\alpha,r)$. That is, 
\begin{equation*}
 u(z)=\int_{\Gamma'} \lambda^z v(\lambda)d\lambda, \qquad \Re z<0, 
\end{equation*}
where the integral converges in $\sE$. Moreover, as mentioned in the proof of Lemma~\ref{lem:Powers.LCS}, 
given any continuous semi-norm $\fp$ on $\sE$, there is $C_{\fp}>0$ such that
\begin{equation}\label{eq:Unif.estimates-vlambda}
 \fp\big[ v(\lambda)\big] \leq C_\fp (1+|\lambda|)^{-1} \qquad \forall \lambda \in \Gamma'. 
\end{equation}

Let $\Sigma=\{b\leq \Re z\leq a\}$, $b\leq a<0$, be a closed vertical strip contained in the half-plane $\{\Re z<0\}$. If $\lambda\in \Gamma'$, and we set $\psi=\arg_\pi \lambda$, then $\lambda^z=|\lambda|^z e^{i\psi z}$. Moreover, as $\lambda\in \Gamma'$, we have $|\psi|\leq \alpha$ and $|\lambda|\geq r$. Thus, if $z\in \Sigma$, then 
\begin{equation*}
 |\lambda^z|= |\lambda|^{\Re z} e^{-\psi\Im z}\leq r^{\Re z} \big(r^{-1}|\lambda|\big)^{\Re z} e^{\alpha |\Im z|} 
 \leq r^{-a}\max(1, r^b) |\lambda|^a e^{\alpha |\Im z|}. 
 \end{equation*}
 Combining this with~(\ref{eq:Unif.estimates-vlambda}) and using~(\ref{eq:Powers.int-semi-norms}), we deduce that there is $C_{\Sigma\fp}>0$ such that
\begin{equation*}
 \fp\big[ u(z)\big]  \leq \int_{\Gamma'} |\lambda^z| \fp\big[v(\lambda)\big] |d \lambda| \leq C_{\Sigma\fp}e^{\alpha |\Im z|} \qquad \text{for all}\  z\in \Sigma. 
\end{equation*}
This implies that $e^{\pm i\alpha z}u(z)$ is bounded on any closed vertical half-strip $\Sigma_\pm \subseteq \bQ_\pm$, and hence 
$e^{\pm i\alpha z}u(z)\in \Hol_\infty(\bQ_\pm;\sE)$.  
\end{proof}

As a special case of Lemma~\ref{lem:Powers.LCS-boundedness}, we have the following refinement of Lemma~\ref{lem:Powers.st-symbols}. 

\begin{lemma}\label{lem:Powers.st-symbols-boundedness}
 Let $\sigma(\xi;\lambda)\in \stS^{m,-1}(\T^n_\theta\times \R^n\times \Lambda)$, $m\in \R$. As in Lemma~\ref{lem:Powers.st-symbols}, define
\begin{equation*}
 \rho(z)(\xi):=\int_\Gamma \lambda^z \sigma(\xi;\lambda)d\lambda, \qquad \Re z<0, \quad \xi \in \R^n. 
\end{equation*}
Then $e^{\pm i\alpha z}\rho(z)\in \Hol_\infty(\bQ_\pm;\stS^{m}(\T^n_\theta\times \R^n))$.
\end{lemma}

Given any $c>0$, we define the open set $\Omega_c(\Theta)\subseteq (\R^n\setminus 0)\times \C$ as in~(\ref{eq:Parameter.parameter-set-for-homogeneous-symbols}) and~(\ref{eq:Powers.parameter-set-for-homogeneous-symbols}). By using Lemma~\ref{lem:Powers.LCS-boundedness} and arguing as in the proof of Lemma~\ref{lem:Powers.hom-symbols}, we get the following refinement of that result.

\begin{lemma}\label{lem:Powers.hom-symbols-boundedness}
 Let $\sigma(\xi;\lambda)\in S^{-1}_{m}(\T^n_\theta\times \Omega_c(\Theta))$ with $c>r$. As in Lemma~\ref{lem:Powers.hom-symbols}, define
 \begin{equation*}
 \rho(z)(\xi) = \int_{|\xi|^w\Gamma}  \lambda^z \sigma(\xi;\lambda)d\lambda, \qquad \xi\neq 0, \quad \Re z <0. 
\end{equation*}
Then $e^{\pm i\alpha z}\rho(z)(\xi)$ is in $\Hol_\infty(\bQ_\pm;C^\infty(\T^n_\theta\times (\R^n\setminus 0)))$ and is homogeneous
of degree $m+w(z+1)$. 
\end{lemma}

Thanks to Lemma~\ref{lem:Powers.st-symbols-boundedness} and Lemma~\ref{lem:Powers.hom-symbols-boundedness}, we may argue as in the first part of the proof of Proposition~\ref{prop:powers.param-psido-hol-psidos} to get the following result. 

\begin{proposition}\label{prop:unif.param-psido-hol-infty-psidos} 
Let $Q(\lambda)\in \Psi^{m,-1}(\T^n_\theta;\Lambda)$ have symbol $\sigma(\xi;\lambda)\sim \sum_{j\geq 0} \sigma_{m-j}(\xi;\lambda)$,  with $\sigma_{m-j}(\xi;\lambda)$ in $S_{m-j}^{-1}(\T^n_\theta\times \Omega_c(\Theta))$, $c>0$. We also assume that $\Gamma=\Gamma(\pi,-\pi,r)$ with $0<r<\min(r_0,c)$. As in Proposition~\ref{prop:powers.param-psido-hol-psidos}, define 
\begin{equation}
 P(z) = \int_{\Gamma} \lambda^z  Q(\lambda)d\lambda, \qquad \Re z <0.  
\end{equation}
Then $e^{\pm i\alpha z}P(z)$, $z\in \bQ_\pm$, is a $\Hol_\infty(\bQ_\pm)$-family in $\Psi^\bt(\T^n_\theta)$ of order $m+w(z+1)$. 
\end{proposition}

\subsection{Complex powers of elliptic operators} 
As in Section~\ref{sec:Resolvent}, let $P\in \Psi^w(\T^n_\theta)$, $w>0$, be elliptic. We further assume that 
$\Theta(P)\neq   \emptyset$ and set $\Lambda=\Lambda(P)$, where $\Lambda(P)$ is defined as in~(\ref{eq:Resolvent.Agmon-pseudo-cone}). 

Recall that $\Theta(P)$ is the open cone in $\C^*$ that consists of all $\lambda \in \C^*$ such that $\rho_w(\xi)-\lambda$ is invertible for all $\xi\neq 0$. The cone $\check{\Theta}(P)$ is obtained from $\Theta(P)$ by deleting all the rays in $\Theta(P)$ that contain an eigenvalue of $P$. This is again an open cone in $\C^*$, and it is the conical part of $\Lambda$. We have $\Lambda \supseteq \check{\Theta}(P)\cup [D(0,r_0)\setminus \{0\}]$, where 
$r_0=\inf\{|\lambda|; \lambda \in \Sp(P), \ \lambda \neq 0\}$.

We further set
\begin{equation*}
 c=\inf_{|\xi|=1} \left\|\rho_w(\xi)^{-1}\right\|^{-1} \quad \text{and} \quad c'=\sup_{|\xi|=1} \left\|\rho_w(\xi)\right\|. 
\end{equation*}
The ellipticity of $P$ ensures that $c>0$. Note that, for every $\xi\neq 0$, the spectrum of $\rho_w(\xi)$ is contained in the annulus $\{\|\rho_w(\xi)^{-1}\|^{-1} \leq |\lambda| \leq \|\rho_w(\xi)\|\}$. As by homogeneity $\rho_w(\xi)=|\xi|^w\rho_w(|\xi|^{-1}\xi)$, we deduce that
\begin{equation}
 \Sp\left( \rho_w(\xi)\right) \subseteq \left\{c|\xi|^w\leq |\lambda|\leq c'|\xi|^w\right\}\qquad \forall \xi\in \R^n\setminus 0. 
 \label{eq:Powers.spectrum-rhow}
\end{equation}

Set $\Omega_c(P):=\Omega_c({\Theta}(P))$. By Theorem~\ref{thm:Resolvent.resolvent-is-psido-with-parameter}, the resolvent $(P-\lambda)^{-1}$ is in $\Psi^{-w,-1}(\T^n_\theta;\Lambda)$ and has symbol $\sigma(\xi; \lambda) \sim \sum_{j\geq 0} \sigma_{-w-j}(\xi;\lambda)$, with $\sigma_{-w-j}(\xi;\lambda)\in 
S_{-w-j}^{-1}(\T^n_\theta\times \Omega_c(P))$ given by~(\ref{eq:Resolvent.symbol-resolvent1})--(\ref{eq:Resolvent.symbol-resolvent2}). In particular, we have $\sigma_{-w}(\xi;\lambda)=(\rho_w(\xi)-\lambda)^{-1}$. 

If $\xi\in \R^n\setminus 0$, then  $\lambda\rightarrow \sigma_{-w}(\xi;\lambda)=(\rho_w(\xi)-\lambda)^{-1}$ is a holomorphic map from $\C\setminus \Sp(\rho_w(\xi))$ to $C^\infty(\T^n_\theta)$. Using~(\ref{eq:Resolvent.symbol-resolvent2}) and arguing by induction then shows that, for every $j\geq 0$, the map $\lambda \rightarrow \sigma_{-w-j}(\xi;\lambda)$ is a holomorphic map from $\C\setminus \Sp(\rho_w(\xi))$ to $C^\infty(\T^n_\theta)$. In fact, it can even be shown that $ \sigma_{-w-j}(\xi;\lambda)$ is a $C^\infty(\T^n_\theta)$-valued $C^{\infty,\omega}$-map on 
\begin{equation*}
 \hat{\Omega}^*(P):=\left\{(\xi,\lambda)\in (\R^n\setminus 0)\times \C; \ \lambda\not \in \Sp\left(\rho_w(\xi)\right)\right\}. 
\end{equation*}

In addition, we let $L_\phi=\{\arg \lambda =\phi\}$ be a ray contained in $\check{\Theta}(P)$, and we form the oriented contour $\Gamma=\Gamma(\phi, \phi-2\pi,r)$ as in~(\ref{eq:powers.contourG1})--(\ref{eq:powers.contourG2}) with $0<r<r_0$.
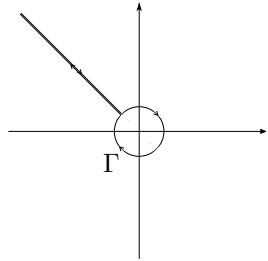
\begin{figure}[h]
\begin{minipage}{0.25\linewidth}
\centering
\begin{tikzpicture}[cpfig]\cpbox
 \cpaxes
 \draw[arrowat=0.55] ($(135:5.85)+(45:0.05)$)--($(135:0.85)+(45:0.05)$);
 \draw[arrowsat={0.25}{0.75}] (131.6:0.85) arc[start angle=131.6,end angle=-221.6,radius=0.85];
 \draw[arrowat=0.45] ($(135:0.85)+(-135:0.05)$)--($(135:5.85)+(-135:0.05)$);
 \node[anchor=base west,inner sep=0pt] at (-1.203,-1.528) {$\Gamma$};
\end{tikzpicture}
\end{minipage}
\caption{Contour $\Gamma=\Gamma(\phi,\phi-2\pi,r)$.}\label{Fig:contour2}
\end{figure}

As $(P-\lambda)^{-1}\in \Psi^{-w,-1}(\T^n_\theta;\Lambda)$, it follows from Proposition~\ref{prop:powers.param-psido-hol-psidos} that we define a holomorphic family in $\Psi^\bt(\T^n_\theta)$ of order $wz$ by letting
\begin{equation}
 P_\phi^z:= \frac{i}{2\pi}\int_\Gamma \lambda^z_\phi (P-\lambda)^{-1}d\lambda, \qquad \Re z<0,
 \label{eq:powers.int-definition} 
\end{equation}
where the integral converges in $\sL(C^\infty(\T^n_\theta))$. In particular, we obtain a holomorphic family of bounded operators on $L_2(\T^n_\theta)$ thanks to Proposition~\ref{prop:Hol.propertiesP(z)-W2s}. In fact, as $(P-\lambda)^{-1}\in \Hol^{-1}(\Lambda;\sL(L_2(\T^n_\theta)))$ by Proposition~\ref{prop:PsiDOs-parameter.classical-PsiDO-Sobolev-mapping-properties}, the above integral converges in $\sL(L_2(\T^n_\theta))$ as well. 

For $0<c_1<c_2$ and $0<\delta<\pi$, we denote by $S_\phi(c_1,c_2,\delta)$ the circular sector, 
\begin{equation*}
 S_\phi(c_1,c_2,\delta):=\left\{ c_1\leq |\lambda|\leq c_2\ \text{and}\ \phi-2\pi+\delta\leq \arg \lambda \leq \phi -\delta\right\}. 
\end{equation*}

\begin{figure}[h]
\begin{minipage}{0.25\linewidth}
\centering
\begin{tikzpicture}[cpfig]\cpbox
 \fill[black!6] (135:2.41)--(135:3.6) arc[start angle=135,end angle=-135,radius=3.6]
   --(-135:2.41) arc[start angle=-135,end angle=135,radius=2.41]--cycle;
 \draw[black!45,line width=0.2pt] (0,0)--(135:2.41) (0,0)--(-135:2.41);
 \draw[line width=0.4pt] (135:2.41)--(135:3.6) arc[start angle=135,end angle=-135,radius=3.6]
   --(-135:2.41) arc[start angle=-135,end angle=135,radius=2.41]--cycle;
 \cpaxes
\end{tikzpicture}
\end{minipage}
\caption{A sector $S_\phi(c_1,c_2,\delta)$.}
\end{figure}
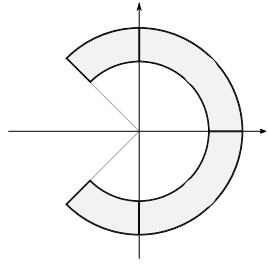

As $\Theta(P)$ is an open cone, there is $\delta>0$ such that  $\Theta(P)\supseteq \{|\arg\lambda -\phi|<\delta\}$. As $\Sp(\rho_w(\xi))$ is contained in $\C^*\setminus \Theta(P)$, it follows that $\Sp(\rho_w(\xi))$ is contained in the angular sector $\{\phi-2\pi+\delta\leq \arg\lambda \leq \phi-\delta\}$. Combining this with~(\ref{eq:Powers.spectrum-rhow}), we deduce that 
\begin{equation*}
 \Sp\left(\rho_w(\xi)\right) \subseteq |\xi|^w S_{\phi}(c,c',\delta) \qquad \forall \xi\in \R^n\setminus 0. 
\end{equation*}

In what follows, we shall call \emph{rectifiable closed contour} any finite union of Lipschitz Jordan curves. For instance, if $\gamma$ is the boundary of any circular sector $S_\phi(c_1,c_2,\delta')$ with $0<c_1<c<c'<c_2$ and $0<\delta'<\delta$, then $\gamma_\xi:=|\xi|^w\gamma$, $\xi \neq 0$, is a counter-clockwise oriented rectifiable closed contour in $\C^*\setminus L_\phi$ whose interior contains $\Sp(\rho_w(\xi))$. 

\begin{lemma}\label{lem:powers.powers-half-plane} 
The operator $P_\phi^z$, $\Re z<0$, has symbol $\rho(z)(\xi)\sim \sum_{j\geq 0} \rho_j(z)(\xi)$, with $\rho_j(z)(\xi)\in S_{wz-j}(\T^n_\theta\times \R^n)$  given by
\begin{equation}
 \rho_{j}(z)(\xi) :=\frac{i}{2\pi} \int_{\gamma_{\xi}} \lambda^{z}_\phi \sigma_{-w-j}(\xi; \lambda)d\lambda,\qquad \xi\neq 0, 
  \label{eq:Powers.homogeneous-symbols}
\end{equation}
where $\sigma_{-w-j}(\xi; \lambda)$ is given by~(\ref{eq:Resolvent.symbol-resolvent1})--(\ref{eq:Resolvent.symbol-resolvent2}) and $\gamma_{\xi}$ is a counter-clockwise oriented
rectifiable closed contour in $\C^*\setminus L_\phi $ whose interior contains $\Sp(\rho_w(\xi))$. In particular, the principal symbol of $P_\phi^{z}$ is $\rho_{w}(\xi)_\phi^{z}$.  
\end{lemma}

\begin{remark}
Given any $z\in \C$, the function $\lambda \rightarrow \lambda^z_\phi$ is holomorphic on $\C^*\setminus L_\phi$. Thus, for any $a\in C^\infty(\T^n_\theta)$ such that $\Sp(a)\subseteq \C^*\setminus L_\phi$, we may define $a^z_\phi$ by standard holomorphic functional calculus. This is automatically an element of $C^\infty(\T^n_\theta)$, since $C^\infty(\T^n_\theta)$ is closed under holomorphic functional calculus. This allows us to define 
 $\rho_{w}(\xi)_\phi^{z}$ as an element of $C^\infty(\T^n_\theta)$ for all $\xi\neq 0$. 
\end{remark}

\begin{proof}[Proof of Lemma~\ref{lem:powers.powers-half-plane}] 
 The only part that does not follow from Proposition~\ref{prop:powers.param-psido-hol-psidos} is the formula~(\ref{eq:Powers.homogeneous-symbols}). Note that if we take $r<\min(r_0,c)$, then, by Proposition~\ref{prop:powers.param-psido-hol-psidos}, we have that 
 $\rho(z)(\xi)\sim \sum_{j\geq 0} \rho_j(z)(\xi)$ in the sense of~(\ref{eq:Hol.hol-family-classical-symbol-estimate}), with 
\begin{equation}
 \rho_{j}(z)(\xi) :=\frac{i}{2\pi} \int_{|\xi|^w\Gamma} \lambda_\phi^{z} \sigma_{-w-j}(\xi; \lambda)d\lambda,\qquad \xi\neq 0. 
  \label{eq:Powers.homogeneous-symbols2}
\end{equation}

Note that $|\xi|^w\Gamma=|\xi|^w\Gamma(\phi,\phi-2\pi,r)$ is contained in $\C\setminus [|\xi|^wS_{\phi}(c,c',\delta)]$. As mentioned above, if $\xi\in \R^n\setminus 0$, then $\lambda \rightarrow \sigma_{-w-j}(\xi;\lambda)$ is a holomorphic map from $\C\setminus \Sp(\rho_w(\xi))$ to $C^\infty(\T^n_\theta)$. Therefore,
 in~(\ref{eq:Powers.homogeneous-symbols2}), we may replace the contour $|\xi|^w\Gamma$ by any counter-clockwise oriented
rectifiable closed contour $\gamma_{\xi}$ in $\C^*\setminus L_\phi $ whose interior contains $\Sp(\rho_w(\xi))$ (see, e.g.,~\cite[Proposition~A.7]{LP:Part1}). This gives~(\ref{eq:Powers.homogeneous-symbols}). In particular, for $j=0$, we get 
\begin{equation*}
 \rho_0(z)(\xi)= \frac{i}{2\pi} \int_{\gamma_\xi} \lambda_\phi^{z}  \left(\rho_w(\xi)- \lambda\right)^{-1}d\lambda= \rho_{w}(\xi)_\phi^{z}. 
\end{equation*}
\end{proof}

\begin{lemma}\label{lem:powers.properties-half-plane} 
 The following hold. 
 \begin{gather}
 P_\phi^{z_1+z_2} = P_\phi^{z_1}P_\phi^{z_2}, \qquad \Re z_j<0,
 \label{eq:powers.semigroup}\\
 \left(1-\Pi_0(P)\right)  P_\phi^z =  P_\phi^z\left(1-\Pi_0(P)\right)= P_\phi^z, \qquad \Re z<0,
 \label{eq:powers.compatibility-Riesz-proj}\\
 P_\phi^{-k}=P^{-k}, \qquad k\in \N. 
 \label{eq:powers.negative-powers1} 
\end{gather}
Here $P^{-k}$ is the partial inverse of $P^k$ as defined in~(\ref{eq:Spectral.partial-inverse-matrix}) and $\Pi_0(P)$ is the spectral projection~(\ref{eq:ST.Riesz-root}) onto the generalized nullspace $E_0(P)$. 
\end{lemma}
\begin{proof}
 The proof of~(\ref{eq:powers.semigroup}) goes along the same lines as the proof of~\cite[Proposition~10.1(a)]{Sh:Springer01}. In the invertible case, the proof of~(\ref{eq:powers.negative-powers1}) goes along the same lines as the proof of~\cite[Proposition~10.1(b)]{Sh:Springer01}. Moreover, if $P$ is invertible, then~(\ref{eq:powers.compatibility-Riesz-proj}) holds trivially, since in that case $\Pi_0(P)=0$. 
 
In the non-invertible case, we proceed as follows. We observe that~(\ref{eq:Resolvent.singularity}) implies that
\begin{equation*}
 (P-\lambda)^{-1}\Pi_0(P)=\Pi_0(P)(P-\lambda)^{-1}= -\sum_{1\leq j \leq \nu_0} \lambda^{-j} \Pi_0(P)P^{j-1}, \qquad \lambda \in \Lambda,
\end{equation*}
where $\nu_0$ is the nilpotency index of the nilpotent operator induced by $P$ on $E_0(P)$. Thus, 
\begin{align*}
 \Pi_0(P)P_\phi^z=P_\phi^z \Pi_0(P) & = \frac{i}{2\pi} \int_\Gamma   \lambda^z_\phi  \Pi_0(P)(P-\lambda)^{-1}d\lambda\\
 &= - \sum_{1\leq j \leq \nu_0} \bigg(\frac{i}{2\pi} \int_\Gamma \lambda^{z-j}_\phi d\lambda\bigg)\Pi_0(P)P^{j-1}=0.
\end{align*}
 This gives~(\ref{eq:powers.compatibility-Riesz-proj}). 
 
 As $P^{-k}=(P^{-1})^k$ and we have the semigroup property~(\ref{eq:powers.semigroup}), it is enough to prove~(\ref{eq:powers.negative-powers1}) for $k=1$. Thanks to~(\ref{eq:Spectral.PP-1-P-1P1}) and~(\ref{eq:powers.compatibility-Riesz-proj}), we have
\begin{equation}
 P_\phi^{-1}\left(PP^{-1}\right)=P_\phi^{-1}\left(1-\Pi_0(P)\right)=  P_\phi^{-1}. 
 \label{eq:powers.Pphi-partial-inverse}
\end{equation}
Moreover, as $\lambda_\phi^{-1}=\lambda^{-1}$, in the integral $\int_\Gamma \lambda_\phi^{-1}(P-\lambda)^{-1}d\lambda$, 
 we have the same integrand over the pieces $\Gamma(\phi,r)^{-}$ and $\Gamma(\phi-2\pi,r)^{+}$. These two pieces agree, but have opposite orientations, and so the integrals over these two pieces cancel each other. Note further that $C(\phi,\phi-2\pi,r)$ is the  circle $\{|\lambda|=r\}$ (with the point $re^{i\phi}$ deleted) endowed with the clockwise orientation. Thus,
\begin{equation*}
 P_\phi^{-1} = -\frac{i}{2\pi} \int_{|\lambda|=r} \lambda^{-1}(P-\lambda)^{-1}d\lambda. 
\end{equation*}
Using~(\ref{eq:ST.Riesz-root}) and the fact that $\lambda^{-1}(P-\lambda)^{-1}P=\lambda^{-1}+(P-\lambda)^{-1}$, we then get
\begin{equation*}
 P_\phi^{-1}P= \frac{1}{2i\pi} \int_{|\lambda|=r}\lambda^{-1}d\lambda +  \frac{1}{2i\pi} \int_{|\lambda|=r} (P-\lambda)^{-1}d\lambda =1-\Pi_0(P). 
\end{equation*}
 Therefore, we have
\begin{equation*}
 (P_\phi^{-1}P)P^{-1}= \left(1-\Pi_0(P)\right)P^{-1}=P^{-1}. 
\end{equation*}
 Comparing this with~(\ref{eq:powers.Pphi-partial-inverse}) then shows that $P_\phi^{-1}=P^{-1}$. 
\end{proof}

Thanks to Lemma~\ref{lem:powers.powers-half-plane} and Lemma~\ref{lem:powers.properties-half-plane}, we may extend the definition of the powers $P_\phi^z$ to all $z\in \C$. Namely, we directly define $P_\phi^z$ as the \psido\ by letting
\begin{equation}
 P_\phi^z:= P^kP_\phi^{z-k}\in \Psi^{wz}(\T^n_\theta), \qquad z\in \C, 
 \label{eq:powers.int-definition2}
\end{equation}
where $k$ is any non-negative integer~$>\Re z$. The above definition does not depend on $k$. Indeed, if $\ell$ is any integer~$\geq 1$ and $\Re z<k$, then, by using~(\ref{eq:powers.semigroup})--(\ref{eq:powers.negative-powers1}), we get
\begin{equation*}
 P^{k+\ell} P_{\phi}^{z-(k+\ell)}=P^k P^\ell P^{-\ell} P_{\phi}^{z-k} = P^k\left(1-\Pi_0(P)\right) P_{\phi}^{z-k} =P^kP_\phi^{z-k}. 
\end{equation*}

\begin{theorem} \label{thm:powers.complex-powers} 
 The following hold. 
\begin{enumerate}
 \item $P_\phi^z$, $z\in \C$, is a holomorphic family in $\Psi^{\bt}(\T^n_\theta)$ of order $wz$.

\item For $z\in \C$ and $j=0,1, \ldots$, the homogeneous symbol of degree $wz-j$ of $P_\phi^z$ is given by~(\ref{eq:Powers.homogeneous-symbols}). In particular, the principal symbol of $P_\phi^z$ is $\rho_w(\xi)_\phi^z$. 

\item We have the group properties, 
 \begin{gather}
 P_\phi^0=1-\Pi_0(P), \qquad P_\phi^{z_1+z_2}=P_\phi^{z_1}P_\phi^{z_2}, \quad z_i\in \C,
  \label{eq:powers.group-property}\\ 
 P_\phi^{-k}=P^{-k}, \qquad P_\phi^k=\left(1-\Pi_0(P)\right)P^k, \qquad k=1,2,\ldots  
   \label{eq:powers.integer-powers}
\end{gather}
\end{enumerate}
\end{theorem}
\begin{proof}
 Let $k\in \N_0$. It is clear that $P^k$, $\Re z<k$, is a holomorphic family in $\Psi^\bt(\T^n_\theta)$ of order $wk$. Lemma~\ref{lem:powers.powers-half-plane} ensures that $P_\phi^{z-k}$, $\Re z<k$, is a holomorphic family in $\Psi^\bt(\T^n_\theta)$ of order $wz-wk$. Therefore, it follows from~(\ref{eq:powers.int-definition2}) and Corollary~\ref{cor:Hol.composition-of-psidos} that 
 $P_\phi^{z}$, $\Re z<k$, is a holomorphic family in $\Psi^\bt(\T^n_\theta)$ of order $wz$. This is true for all $k\in \N_0$, and so, by using Lemma~\ref{lem:hol.local-global}, we see that $P_\phi^{z}$, $z\in \C$, is a holomorphic family in $\Psi^\bt(\T^n_\theta)$ of order $wz$ over the whole complex plane $\C$. This proves the first part. 
 
 For $j=0,1,\ldots$ and $\Re z<0$, the homogeneous symbol $\rho_j(z)(\xi)$ of degree $wz-j$ of $P_\phi^z$ is given by~(\ref{eq:Powers.homogeneous-symbols}). If $\xi \neq 0$, then
 $\rho_j(z)(\xi)\in\Hol(\C; C^\infty(\T^n_\theta))$. Moreover, for all $v$ in the topological dual $\scD'(\T_\theta^n)$, we have
\begin{equation*}
 \acou{v}{\rho_j(z)(\xi)} = \frac{i}{2\pi} \int_{\gamma_{\xi}} \lambda^{z}_\phi  \acou{v}{\sigma_{-w-j}(\xi; \lambda)}d\lambda, \qquad \Re z<0. 
\end{equation*}
The left-hand side is a holomorphic function on $\C$. The right-hand side is holomorphic on $\C$ as well, since the function $(z,\lambda)\rightarrow \lambda^{z}_\phi  \acou{v}{\sigma_{-w-j}(\xi; \lambda)}$ is continuous on $\C\times \gamma_\xi$ and holomorphic with respect to $z$. Thus, the above equality holds for all $z\in \C$ and all $v\in \scD'(\T^n_\theta)$. By Proposition~\ref{prop:Hol.integral}, this means that~(\ref{eq:Powers.homogeneous-symbols}) holds for all $z\in \C$. This gives the second part. 

It remains to prove the last part. As $P_\phi^z$, $z\in \C$, is a holomorphic family with values in the locally convex space $\sL(C^\infty(\T^n_\theta))$, the group property~(\ref{eq:powers.group-property}) follows from~(\ref{eq:powers.semigroup}) by analytic continuation (\emph{cf}.~Proposition~\ref{prop:Hol.unique-analytic-continuation}). Note also that, by definition, $P_\phi^0=PP_\phi^{-1}=PP^{-1}=1-\Pi_0(P)$. 

The first set of equalities in~(\ref{eq:powers.integer-powers}) repeats~(\ref{eq:powers.negative-powers1}) \emph{verbatim}. As for the second set, given any $k\in \N$, by using~(\ref{eq:Spectral.PP-1-P-1P1}) and~(\ref{eq:powers.negative-powers1}), we get
\begin{equation*}
 P_\phi^k=P^{k+1}P_\phi^{-1}=P^k P P^{-1}=P^k \left(1-\Pi_0(P)\right)= \left(1-\Pi_0(P)\right)P^k. 
\end{equation*}
This proves~(\ref{eq:powers.integer-powers}). 
 \end{proof}

\begin{remark}
 Let $\nu_0$ be the nilpotency index of the nilpotent operator induced by $P$ on the generalized nullspace $E_0(P)$. As $E_0(P)=\ker P^{\nu_0}$, if $k\geq \nu_0$, then $\Pi_0(P)P^k =0$, and so in this case~(\ref{eq:powers.integer-powers}) gives 
\begin{equation*}
 P_\phi^k =P^k \qquad \text{for}\ k=\nu_0,\nu_0+1, \ldots.
\end{equation*}
In particular, if $P$ is invertible, then $P^k_\phi=P^k$ for all $k\in \N_0$, and so $P_\phi^0=1$. 
\end{remark}

\begin{remark}
 It was suggested without proof in~\cite{FGK:MPAG17} that complex powers of elliptic \psidos\ on NC 2-tori give rise to holomorphic families of \psidos. However, the definition of holomorphic families of \psidos\ and the construction of complex powers in~\cite{FGK:MPAG17} are not fully clear. Therefore, Theorem~\ref{thm:powers.complex-powers} provides a precise formulation of that result.
\end{remark}

\begin{remark}\label{rmk:powers.normal}
 Three special cases of the above construction are worth recording explicitly. 
\begin{enumerate}
 \item[(a)] Suppose that $P$ is normal, i.e., it commutes with its formal adjoint $P^*$. In this case,  as $L_\phi$ and $\Sp(P)$ are disjoint, we may define $P_\phi^z$, $z\in \C$, by means of the Borel functional calculus for $P$ on $L_2(\T^n_\theta)$, with the convention that 
 $\lambda_\phi^z=0$ for $\lambda=0$. Alternatively, if $(e_\ell)_{\ell\geq 0}$ is an orthonormal eigenbasis with $Pe_\ell =\lambda_\ell e_\ell$, then
\begin{equation}\label{eq:powers.normal} 
 P_\phi^ze_\ell= \left\{
\begin{array}{cl} 
  (\lambda_\ell)^z_\phi e_\ell & \text{if $\lambda_\ell\neq 0$},\\
  0 & \text{if $\lambda_\ell= 0$}.
\end{array}\right. 
\end{equation}
 In this definition, $P_\phi^z$ is a bounded operator for $\Re z\leq 0$ and has domain $W^{s}_2(\T^n_\theta)$ with $s=w \Re z$ for $\Re z>0$. On these domains we recover the operator $P_\phi^z$ defined by~(\ref{eq:powers.int-definition}) and~(\ref{eq:powers.int-definition2}). 

 \item[(b)]\label{rmk:powers.positive} Suppose that $P$ is selfadjoint. If $P$ has a positive invertible principal symbol and has non-negative spectrum, then $\Theta(P)=\check{\Theta}(P)=\C\setminus [0,\infty)$. Moreover, in view of~(\ref{eq:powers.normal}), we obtain the same family of complex powers for any $\phi\in (0,2\pi)$, since in this case $\arg_\phi \lambda=0$ for all $\lambda \in (0,\infty)$.  We denote this family by $P^z$, $z\in \C$.  In particular, we recover~\cite[Theorem~7.3]{LP:Part1}, which asserts that in this case $P^z\in \Psi^{wz}(\T^n_\theta)$ for all $z\in \C$ and gives the formula~(\ref{eq:Powers.homogeneous-symbols}) for its homogeneous symbols.  

 \item[(c)]\label{rmk:powers.selfadjoint} Suppose instead that $P$ is selfadjoint with $\check{\Theta}(P)=\C\setminus \R$, i.e., $P$ has both positive and negative eigenvalues. In this case, we obtain  two families of complex powers $P_\uparrow^z$, $z\in \C$, and $P_\downarrow^z$, $z\in \C$, according to whether the ray $L_\phi$ is contained in the upper half-plane $\{\Im \lambda>0\}$  with $\phi\in (0,\pi)$ or is contained in the lower half-plane  $\{\Im \lambda<0\}$ with $\phi \in (\pi,2\pi)$. In fact, if $(e_\ell)_{\ell\geq 0}$ is any orthonormal eigenbasis with $Pe_\ell =\lambda_\ell e_\ell$, then 
\begin{equation}\label{eq:powers.selfadjoint} 
 P_{\uparrow\downarrow}^ze_\ell= \left\{
\begin{array}{cl} 
  \lambda_\ell^z e_\ell & \text{if $\lambda_\ell>0 $},\\
  0 & \text{if $\lambda_\ell= 0$},\\
   e^{\mp i\pi z} |\lambda_\ell|^z e_\ell & \text{if $\lambda_\ell<0 $}, 
\end{array}\right. 
\end{equation}
where the sign $\mp$ is $-$ (resp., $+$) when the arrow is up (resp., down). 
\end{enumerate}
\end{remark}

Combining Theorem~\ref{thm:powers.complex-powers} with Proposition~\ref{prop:Hol.propertiesP(z)-W2s} and Proposition~\ref{prop:Hol.propertiesP(z)-sLp} yields the following two results. 

\begin{corollary} \label{cor:powers.powers-W2s}
 Given any $m\in \R$, the following hold.
 \begin{enumerate}
 \item $P_\phi^z$, $\Re z<m$, is a holomorphic family in  $\sL(W_2^{s+wm}(\T^n_\theta), W_2^s(\T^n_\theta))$ for every $s\in \R$. 
 
 \item Let $K\subseteq \C$ be a compact set such that $\Re z\leq m$ on $K$. Then $P_\phi^z$, $z\in K$, is a bounded strongly continuous family in  $\sL(W_2^{s+wm}(\T^n_\theta), W_2^s(\T^n_\theta))$ for every $s\in \R$. 
\end{enumerate}
\end{corollary}

\begin{corollary} \label{cor:powers.powers-sLp}
 Let $m<0$, and set $p=n(w|m|)^{-1}$. The following hold.
\begin{enumerate}
 \item $P_\phi^z$, $\Re z<m$, is a holomorphic family in $\sL_p$, and so this is a holomorphic family in $\sL_{q,\infty}$ for all $q \geq p$. 
 
 \item Given any compact set $K\subseteq \{\Re z<m\}$, the family $P_\phi^z$, $z\in K$, is bounded in $\sL_p$, and hence is bounded in $\sL_{p,\infty}$. 

\item For $p=1$, i.e., $m=-w^{-1}n$, the family $P_\phi^z$, $\Re z<-w^{-1}n$, is a holomorphic family of trace-class operators, and, for every compact set $K\subseteq \{\Re z<-w^{-1}n\}$, the family $P_\phi^z$, $z\in K$, is bounded in $\sL_1$, and hence is bounded in $\sL_{1,\infty}$. 
\end{enumerate}
\end{corollary}

We also mention the following extension of Theorem~\ref{thm:powers.complex-powers}. 

\begin{proposition}\label{prop:Powers.APz}
 Let $A\in \Psi^q(\T^n_\theta)$, $q\in \C$. Then $AP_\phi^z$, $z\in \C$, is a holomorphic family in $\Psi^\bt(\T^n_\theta)$ of order $wz+q$. 
\end{proposition}
\begin{proof}
 We may regard $A$ as a (constant) holomorphic family in $\Psi^\bt(\T^n_\theta)$ of order $q$ over $\C$. The result then follows from
  Corollary~\ref{cor:Hol.composition-of-psidos} and Theorem~\ref{thm:powers.complex-powers}. 
\end{proof}

\begin{remark}\label{rmk:Powers.extension-APz}
 If $P$ is invertible, then $AP_\phi^z=A$ at $z=0$, so the family $AP_\phi^{z/w}$ is a holomorphic gauging for $A$  in the sense of Definition~\ref{def:Hol.gauging-PsiDOs}. In addition, Corollary~\ref{cor:powers.powers-W2s} and Corollary~\ref{cor:powers.powers-sLp} extend \emph{mutatis mutandis} to the families $AP_\phi^z$ as follows: 
 \begin{itemize}
 \item  For Corollary~\ref{cor:powers.powers-W2s} we just need to replace the exponents $s+wm$ by $s+wm+\Re q$.
 
 \item For Corollary~\ref{cor:powers.powers-sLp} the assumption $m<0$ has to be replaced by $wm+\Re q<0$, i.e., $m<-w^{-1}\Re q$, and then we have to set $p=n|wm+\Re q|^{-1}$. 
\end{itemize}
 In particular, $AP_\phi^z$, $\Re z<-w^{-1}(n+\Re q)$, is a holomorphic family of trace-class operators. 
\end{remark}

Suppose now that $P$ is any elliptic operator in $\Psi^w(\T^n_\theta)$. In this case, $P^*P$ is an operator in $\Psi^{2w}(\T^n_\theta)$ with principal symbol $\rho_w(\xi)^*\rho_w(\xi)$. This is a positive invertible symbol and $P^*P$ has non-negative spectrum, and so $\Theta(P^*P)=\check{\Theta}(P^*P)= \C\setminus [0,\infty)$. By~\cite[Theorem~7.5]{LP:Part1} (or by Theorem~\ref{thm:powers.complex-powers}) the absolute value $|P|=\sqrt{P^*P}$ is an operator in $\Psi^{w}(\T^n_\theta)$ whose principal symbol is $\sqrt{\rho_w(\xi)^*\rho_w(\xi)}=|\rho_w(\xi)|$. 
Here again, $|P|$ has non-negative spectrum, and  we have
\begin{equation*}
 \Theta(|P|)=\check{\Theta}(|P|)=\C\setminus [0,\infty). 
\end{equation*}
Therefore, by applying Theorem~\ref{thm:powers.complex-powers} to $|P|$, we arrive at the following statement. 

\begin{proposition}[compare~{\cite[Theorem~7.5]{LP:Part1}}] \label{prop:powers.absolute-value}
 Let $P$ be any elliptic operator in $\Psi^w(\T^n_\theta)$, $w>0$. 
 \begin{enumerate}
\item The complex powers $|P|^z$, $z\in \C$, give rise to a holomorphic family in 
 $\Psi^\bt(\T^n_\theta)$ of order $wz$.
 
 \item  For each $z\in \C$, the principal symbol of $|P|^z$ is equal to $|\rho_w(\xi)|^z$. 
\end{enumerate}
\end{proposition}

Let us illustrate the previous results in the special case of the Laplace--Beltrami operators of~\cite{HP:JGP20}. By a \emph{smooth density} on $\T_\theta^n$, we shall mean a positive invertible element of $C^\infty(\T^n_\theta)$. Any smooth density $\nu$ on $\T^n_\theta$ defines an equivalent inner product on $L_2(\T^n_\theta)$ by
\begin{equation}\label{eq:Powers.inner-product-density}
 \acoup{x}{y}_\nu:=(2\pi)^n\tau\big[y^*\nu x\big], \qquad x,y\in L_2(\T^n_\theta). 
\end{equation}
We denote by $L_2(\T^n_\theta;\nu)$ the corresponding Hilbert space. A unitary isomorphism from $L_2(\T^n_\theta;\nu)$ onto $L_2(\T^n_\theta)$ is simply given by the left multiplication by $(2\pi)^{n/2}\sqrt{\nu}$. Note that $\sqrt{\nu}\in C^\infty(\T^n_\theta)$, since $ C^\infty(\T^n_\theta)$ is closed under holomorphic functional calculus. 

In the terminology of~\cite{HP:JGP20, Ro:SIGMA13} a \emph{Riemannian metric} is given by a positive invertible matrix $g=(g_{ij})\in M_n(C^\infty(\T^n_\theta))$ such  that its entries $g_{ij}$ and the entries of its inverse $g^{-1}=(g^{ij})$ are selfadjoint elements of $C^\infty(\T^n_\theta)$. The \emph{Riemannian density} is defined by
\begin{equation*}
 \nu(g)= \sqrt{\det(g)}, \qquad \text{where}\ \det(g):=\exp\big[\Tr\left(\log (g)\right)\big]. 
\end{equation*}
Here $\log(g)\in M_n(C^\infty(\T^n_\theta))$ is defined by holomorphic functional calculus and $\Tr:M_n(C^\infty(\T^n_\theta))\rightarrow C^\infty(\T^n_\theta)$ is the usual matrix trace. Note that $\det(g)$ is a positive invertible element of $C^\infty(\T^n_\theta)$, and hence $\nu(g)$ is a smooth density on $\T^n_\theta$.  

The \emph{Laplace--Beltrami operator} $\Delta_{g,\nu}:C^\infty(\T^n_\theta)\rightarrow C^\infty(\T^n_\theta)$ is then defined by
\begin{equation}\label{eq:Powers.Laplace-Beltrami}
 \Delta_{g,\nu}u= \nu^{-1}\sum_{1\leq i,j\leq n} \delta_j\big(\sqrt{\nu}g^{ij}\sqrt{\nu} \delta_i(u)\big), \qquad u \in C^\infty(\T^n_\theta). 
\end{equation}
For $\nu=\nu(g)$ this operator is simply denoted $\Delta_g$. The operator  $\Delta_{g,\nu}$ is a second-order elliptic differential operator with symbol $\rho(\xi)=\rho_2(\xi)+\rho_1(\xi)$, where 
\begin{gather*}
 \rho_2(\xi)= \nu^{-\frac12} |\xi|_g^2  \nu^{\frac12}, \quad \text{with}\ |\xi|_g:=\bigg( \sum_{i,j} g^{ij}\xi_i\xi_j\bigg)^{\frac12},\\
 \rho_1(\xi)= \nu^{-1} \sum_{i,j} \delta_i\big(\sqrt{\nu}g^{ij}\sqrt{\nu}\big)\xi_j. 
\end{gather*}
Note that $|\xi|_g$ is a positive invertible element of $C^\infty(\T^n_\theta)$ for all $\xi\neq 0$ (see~\cite{HP:JGP20}). Thus, $\Sp(\rho_2(\xi))=\Sp(|\xi|_g^2)\subseteq (0,\infty)$, and hence $\Theta(\Delta_{g,\nu})=\C\setminus [0,\infty)$. 

In addition, $\Delta_{g,\nu}$ with domain $W_2^2(\T^n_\theta)$ is selfadjoint on $L_2(\T^n_\theta;\nu)$ and has non-negative spectrum  with  $\ker \Delta_{g,\nu}=\C\cdot 1$ (see~\cite{HP:JGP20}). It follows that $\check{\Theta}(\Delta_{g,\nu})=\C\setminus [0,\infty)$. In the same way as in Remark~\ref{rmk:powers.normal}, we may define the complex powers $\Delta_{g,\nu}^z$, $z\in \C$, by using the Borel functional calculus for $\Delta_{g,\nu}$ on $L_2(\T^n_\theta;\nu)$. As in Remark~\ref{rmk:powers.positive}, they agree with the complex powers $(\Delta_{g,\nu})_\phi^z$,  $z\in \C$, with $\phi \in (0,2\pi)$. Therefore, by using Theorem~\ref{thm:powers.complex-powers}, we obtain the following result. 

\begin{proposition}[compare~{\cite[Lemma~10.3]{Po:JMP20}}]\label{prop:Powers.Laplace-Beltrami1}
 The complex powers $\Delta_{g,\nu}^z$, $z\in \C$, form a holomorphic family in $\Psi^{\bt}(\T^n_\theta)$ of order $2z$. For every $z\in \C$, the principal symbol of  $\Delta_{g,\nu}^z$ is equal to $\nu^{-\frac12}|\xi|_{g}^{2z} \nu^{\frac12}$. 
\end{proposition}

As mentioned above, the left multiplication by $(2\pi)^{n/2}\sqrt{\nu}$ is a unitary isomorphism from $L_2(\T^n_\theta;\nu)$ onto $L_2(\T^n_\theta)$. Under this isomorphism, $\Delta_{g,\nu}$ corresponds to the operator $\hat{\Delta}_{g,\nu}:= \nu^{1/2} \Delta_{g,\nu} \nu^{-1/2}$. Namely,
\begin{equation}\label{eq:Powers.hat-Laplace-Beltrami}
 \hat{\Delta}_{g,\nu}u= \nu^{-\frac12}\sum_{1\leq i,j\leq n} \delta_j\big(\sqrt{\nu}g^{ij}\sqrt{\nu} \delta_i(\nu^{-\frac12}u)\big), \qquad u \in C^\infty(\T^n_\theta). 
\end{equation}
We obtain a second-order elliptic differential operator whose principal symbol is $|\xi|_g^2$. It is also a selfadjoint operator on $L_2(\T^n_\theta)$ with the same spectrum as $\Delta_{g,\nu}$. In particular, its spectrum is non-negative. Therefore, by using Theorem~\ref{thm:powers.complex-powers} and Remark~\ref{rmk:powers.positive}, we immediately arrive at the following statement. 

\begin{proposition}[compare~{\cite[Lemma~10.3]{Po:JMP20}}]\label{prop:Powers.Laplace-Beltrami2}
  The complex powers $\hat{\Delta}_{g,\nu}^z$, $z\in \C$, form a holomorphic family in $\Psi^{\bt}(\T^n_\theta)$ of order $2z$. For every $z\in \C$, the principal symbol of  $\hat{\Delta}_{g,\nu}^z$ is $|\xi|_{g}^{2z}$. 
\end{proposition}

\begin{example}
If $g$ is the flat metric $g_{ij}=\delta_{ij}$, then $\nu(g)=1$ and $\Delta_{g}=\hat{\Delta}_{g}=\Delta$, where  $\Delta=\delta_1^2+\cdots +\delta_n^2$ is  the flat Laplacian. In this case, we recover the fact that the complex powers  $\Delta^z$, $z\in \C$, form a holomorphic family of \psidos\ (\emph{cf}.~Example~\ref{ex:Hol.Delta-powers}). 
\end{example}

\begin{example}
 Suppose that $g$ is conformally flat, i.e., $g_{ij}=k^2\delta_{ij}$, where $k$ is a positive invertible element of $C^\infty(\T^n_\theta)$. This is the kind of metric considered by Connes--Tretkoff~\cite{CT:Baltimore11}, Connes--Moscovici~\cite{CM:JAMS14}, and various other authors 
 (see, e.g., \cite{FK:JNCG12,  FK:JNCG15, FGK:JNCG19, LM:GAFA16,  Liu:arXiv18a, Liu:arXiv18b}). In this case, $\nu(g)=k^n$, and we have
\begin{equation*}
 \Delta_g=k^{-2} \Delta + \sum_{1\leq i \leq n} k^{-n}  \delta_i(k^{n-2}) \delta_i. 
\end{equation*}
In particular, in dimension $n=2$ we get $\Delta_g=k^{-2}\Delta$ and $\hat{\Delta}_g=k^{-1} \Delta k^{-1}$. In any dimension, $\Delta_g$ and $\hat{\Delta}_g$ both have  $k^{-2}|\xi|^2$ as principal symbol. 
\end{example}

\subsection{Behavior of complex powers on vertical strips} 
Assume now that $\phi=\pi$. As in~(\ref{eq:unif.Theta-alpha}), this implies that there is $\alpha \in (0,\pi)$ such that
\begin{equation*}
 \check{\Theta}(P)\supseteq \{\alpha \leq \arg \lambda \leq 2\pi-\alpha\}. 
\end{equation*}
 We then define 
\begin{equation*}
 \alpha_0=\inf\left\{ \alpha \in \left(0,\pi\right); \ \check{\Theta}(P)\supseteq \{\alpha \leq \arg \lambda \leq 2\pi -\alpha\}\right\}. 
\end{equation*}

We define the complex powers $P^z:=P_\pi^z$, $z\in \C$, as in~(\ref{eq:powers.int-definition}) and~(\ref{eq:powers.int-definition2}) with $\phi=\pi$. We know from Theorem~\ref{thm:powers.complex-powers} that they form a holomorphic family in $\Psi^\bt(\T^n_\theta)$ of order $wz$. In fact, by using Lemma~\ref{lem:unif.local-global} and Proposition~\ref{prop:unif.param-psido-hol-infty-psidos}, and by arguing as in the first part of the proof of Theorem~\ref{thm:powers.complex-powers}, we obtain the following  refinement of this result. 

\begin{theorem}\label{thm:unif.complex-powers}
 The family $P^z$, $z\in \C$, is a holomorphic family in $\Psi^\bt(\T^n_\theta)$ of order $wz$ such that
 \begin{equation*}
 e^{\pm i\alpha z}P^z\in \Hol_\infty\big(\bH_\pm;\Psi^\bt(\T^n_\theta)\big) \qquad \forall \alpha \in (\alpha_0,\pi). 
\end{equation*}
  \end{theorem}

As an application of Theorem~\ref{thm:unif.complex-powers}, we obtain the following exponential bounds along vertical strips for the complex powers $P^z$, $z\in \C$. 

\begin{corollary} \label{cor:unif.powers-Ws} 
Let $\Sigma$ be any closed vertical strip in $\C$, and set $\sigma=\max\{\Re z; z\in \Sigma\}$. For every  $\alpha \in (\alpha_0,\pi)$, the following hold. 
\begin{enumerate}
 \item $P^z\in e^{\alpha |\Im z|}L^\infty\left(\Sigma; \sL\left(W_2^{s+w\sigma}(\T^n_\theta), W_2^{s}(\T^n_\theta)\right)\right)$ for all $s\in \R$. In particular, for $\Sigma =i\R$ we get that $P^{it}\in e^{\alpha |t|}   L^\infty\left(\R; \sL\left(W_2^{s}(\T^n_\theta)\right)\right)$ for all $s\in \R$. 

\item If $\sigma<0$ and $p:=n(w|\sigma|)^{-1}$, then $P^z$ is in $e^{\alpha |\Im z|}L^\infty\left(\Sigma; \sL_{p,\infty}\right)$, and hence is in  $e^{\alpha |\Im z|}L^\infty\left(\Sigma; \sL_{r}\right)$ for all $r>p$. 
\end{enumerate}
\end{corollary}
\begin{proof}
As $P^z$ has order $wz$, in the notation of Proposition~\ref{prop:unif.exp-control}, we have 
\begin{equation*}
 w_\Sigma:=\max\{\Re (wz); z\in \Sigma\}=w\sigma. 
\end{equation*}
The result then follows from Proposition~\ref{prop:unif.exp-control} with $w_\Sigma=w\sigma$. 
\end{proof}

\begin{remark}\label{rmk:Powers.examples-Holinf}
 If $\check{\Theta}(P)=\C\setminus [0,\infty)$, then $\alpha_0=0$,  and so Theorem~\ref{thm:unif.complex-powers} and Corollary~\ref{cor:unif.powers-Ws} hold for all $\alpha\in (0,\pi)$. In particular, this occurs in the following cases:
 \begin{itemize}
 \item $P=|Q|$ where $Q$ is any elliptic operator in $\Psi^w(\T^n_\theta)$, $w>0$. 
 
 \item The Laplace--Beltrami operator $\Delta_{g,\nu}$ and its unitary conjugate $\hat{\Delta}_{g,\nu}$ as defined in~(\ref{eq:Powers.Laplace-Beltrami}) and~(\ref{eq:Powers.hat-Laplace-Beltrami}), respectively.  
\end{itemize}
\end{remark}

In addition, in the same way as in the proof of Proposition~\ref{prop:Powers.APz}, by using Theorem~\ref{thm:unif.complex-powers} and Proposition~\ref{prop:unif.composition}, we get the following extension of Theorem~\ref{thm:unif.complex-powers}. 

\begin{proposition}\label{prop:Powers.APz-hol-infty}
 If $A\in \Psi^q(\T^n_\theta)$, $q\in \C$, then $AP^z$, $z\in \C$, is a holomorphic family in $\Psi^\bt(\T^n_\theta)$ of order $wz+q$ such that
  \begin{equation*}
 e^{\pm i\alpha z}AP^z\in \Hol_\infty\big(\bH_\pm;\Psi^\bt(\T^n_\theta)\big) \qquad \forall \alpha \in (\alpha_0,\pi). 
\end{equation*}
\end{proposition}

\begin{remark}\label{rmk:unif.powers-Ws-APz} 
 In the same way as in Remark~\ref{rmk:Powers.extension-APz}, Corollary~\ref{cor:unif.powers-Ws} extends \emph{mutatis mutandis} to the families $AP^z$ as follows: 
 \begin{itemize}
 \item For the first part we should replace the exponents $s+w\sigma$ by $s+w\sigma+\Re q$.
 
 \item For the second part, the assumption $\sigma<0$ has to be replaced by $w\sigma+\Re q<0$, i.e., $\sigma<-w^{-1}\Re q$, and then we need to set $p=n|w\sigma+\Re q|^{-1}$. 
\end{itemize}
 In particular, for every $\alpha\in(\alpha_0,\pi)$, $e^{\pm i\alpha z}AP^z\in\Hol_\infty(\Omega_\pm;\sL_1)$, where $\Omega=\{\Re z<-w^{-1}(n+\Re q)\}$. 
\end{remark}

In applications to the heat equation, we are actually interested in the family $\Gamma(z)P^{-z}$, $\Re z>0$, and more generally families  $\Gamma(z)AP^{-z}$, $\Re z>0$, 
with $A\in \Psi^q(\T^n_\theta)$, $q\in \R$ (see~\cite{LP:SG2}).  We have the following exponential decay results for these families.

\begin{theorem}\label{thm:unif.rapid-decay-zeta}
 Assume that $\alpha_0<\pi/2$, i.e., $\check{\Theta}(P)$ contains $\{\Re \lambda\leq 0\}\setminus\{0\}$. Let $\Sigma \subseteq \C\setminus \Z_{-}$ be a closed vertical half-strip or strip, and set  $\sigma=\min\{\Re z;\ z\in \Sigma\}$. Given any $A\in \Psi^q(\T^n_\theta)$, $q\in \R$, for all $\epsilon\in (0,\pi/2-\alpha_0)$, the following hold. 
 \begin{enumerate}
 \item $\Gamma(z)AP^{-z}\in e^{-\epsilon |\Im z|}   L^\infty\left(\Sigma; \sL\left(W_2^{s}(\T^n_\theta), W_2^{s+w\sigma-q}(\T^n_\theta)\right)\right)$ for all $s\in \R$. 

\item If $\sigma>w^{-1}q$ and $p=n(w\sigma -q)^{-1}$, then $\Gamma(z)AP^{-z}$ is in $e^{-\epsilon |\Im z|} L^\infty\left(\Sigma; \sL_{p,\infty}\right)$, and hence is in $e^{-\epsilon |\Im z|}    L^\infty\left(\Sigma; \sL_{r}\right)$ for all $r>p$. 

\item If $\sigma>w^{-1}(n+q)$, then $\Gamma(z)AP^{-z} \in e^{-\epsilon |\Im z|}   L^\infty(\Sigma; \sL_1)$, and so there is $C_{A\Sigma\epsilon}>0$ such that 
\begin{equation*}
  \big|\Gamma(z)\Tr\left[AP^{-z}\right]\big| \leq C_{A\Sigma \epsilon} e^{-\epsilon |\Im z|}\qquad \forall z\in \Sigma.
\end{equation*}
 \end{enumerate}
\end{theorem}
\begin{proof}
Let $\alpha \in (\alpha_0, \pi/2-\epsilon)$, and set $\alpha'=\alpha+\epsilon$. Note that $\alpha'\in (0,\pi/2)$.  
On the one hand, as mentioned in Example~\ref{ex:Uniform.Gamma-function}, the function 
$\Gamma(z)$ is holomorphic on $\C\setminus \Z_{-}$ and $e^{\mp i \alpha' z}\Gamma(z)$ is a $\Hol_\infty$-function on $\bH_\pm$. It follows that there exists $C_{\Sigma}>0$ such that
\begin{equation} \label{eq:unif.gamma-function-estimate}
 \left| \Gamma(z)\right| \leq C_{\Sigma} e^{-\alpha' |\Im z|} \qquad \forall z\in \Sigma. 
\end{equation}
 
 On the other hand, by Corollary~\ref{cor:unif.powers-Ws} and Remark~\ref{rmk:unif.powers-Ws-APz}, we know that 
 \begin{itemize}
 \item[(i)] $AP^{-z}$ is in $e^{\alpha |\Im z|}   L^\infty(\Sigma; \sL(W_2^{s-w\sigma +q}(\T^n_\theta), W_2^{s}(\T^n_\theta)))$ for all $s\in \R$. 
 
 \item[(ii)] If $\sigma>w^{-1}q$ and $p=n(w\sigma-q)^{-1}$, then $AP^{-z}$ is in $e^{\alpha |\Im z|}   L^\infty(\Sigma; \sL_{p,\infty})$. 
\end{itemize}
Combining~(i) with~(\ref{eq:unif.gamma-function-estimate}) and substituting $s+w\sigma -q$ for $s$ gives the first part. The second part follows by combining~(ii) with~(\ref{eq:unif.gamma-function-estimate}). The third part is a byproduct of the second part.  
\end{proof}

\section{Logarithms of Elliptic \psidos}\label{sec:log} 
In this section, we construct the logarithms of elliptic \psidos\ as the infinitesimal generators of their one-parameter groups of complex powers. 

Throughout this section, $P\in \Psi^w(\T^n_\theta)$, $w>0$, is as in Section~\ref{sec:Resolvent}, with the additional assumption that $\Theta(P)\neq \emptyset$. 

Let $L_\phi=\{\arg \lambda=\phi\}$ be a ray contained in $\check{\Theta}(P)$. The logarithm $\log_\phi: \C^*\setminus  L_\phi\rightarrow \C$ is defined by \begin{equation*}
 \log_\phi (\lambda) =  \log|\lambda|+ i\arg_\phi \lambda, \qquad \lambda\in \C^*\setminus L_\phi, 
 \end{equation*}
where, as above, $\arg_\phi:\C^*\setminus L_\phi\rightarrow (\phi-2\pi,\phi)$ is the continuous determination of the argument associated with the angle $\phi$. 
We get a holomorphic function on $\C^*\setminus L_\phi$. Moreover, in view of~(\ref{eq:Powers.Powers-scalars}), we have
\begin{equation*}
 \log_\phi (\lambda) =  \left.\frac{\dl}{\dl z}\right|_{z=0} \lambda_\phi^z \qquad \text{for all $\lambda \in \C^*\setminus L_\phi$}. 
\end{equation*}

By Theorem~\ref{thm:powers.complex-powers}, the complex powers $P_\phi^z$, $z\in \C$, defined by~(\ref{eq:powers.int-definition}) and~(\ref{eq:powers.int-definition2}) form a holomorphic family in $\Psi^\bt(\T^n_\theta)$ of order $wz$. In particular, we get a holomorphic family in $\sL(C^\infty(\T^n_\theta))$. 

\begin{definition}
 The \emph{logarithm} of $P$ associated with the angle $\phi$ is the operator in $\sL(C^\infty(\T^n_\theta))$ defined by
\begin{equation}
 \log_\phi(P) = \left.\frac{d}{dz}\right|_{z=0} P_\phi^z.
 \label{eq:Log.definition}
\end{equation}
\end{definition}

\begin{remark}
 Given any $s>0$, for $\Re z<sw^{-1}$, the real part of the order of $P_\phi^z$ is~$\leq s$, and so, by Proposition~\ref{prop:Hol.propertiesP(z)-W2s}, the restriction $P_\phi^z$, $\Re z<sw^{-1}$, is a holomorphic family in $\sL(W_2^s(\T^n_\theta),L_2(\T^n_\theta))$. Therefore, in~(\ref{eq:Log.definition}), the derivative is taken in  $\sL(W_2^s(\T^n_\theta),L_2(\T^n_\theta))$, and   so $\log_\phi(P)$ uniquely extends to an operator in $\sL(W_2^s(\T^n_\theta),L_2(\T^n_\theta))$ for all $s>0$. More generally, it extends to a bounded operator from $W_2^s(\T^n_\theta)$ to $W_2^t(\T^n_\theta)$ for all $s,t\in \R$, with $s>t$. 
\end{remark}

In what follows, we denote by $\Gamma$ any contour $\Gamma(\phi,\phi-2\pi,r)$ as in~(\ref{eq:powers.contourG1})--(\ref{eq:powers.contourG2}) with $0<r<r_0$, where $r_0=\inf\{|\lambda|; \lambda \in \Sp(P), \ \lambda \neq 0\}$ (see Figure~\ref{Fig:contour2}). We make the convention that $\log_\phi\lambda$ is equal to $\log|\lambda|+i\phi$ (resp., $\log|\lambda|+i\phi-i2\pi$) when $\lambda$ ranges over  $\Gamma(\phi,r)^{-}$ (resp., $\Gamma(\phi-2\pi,r)^{+}$). 

\begin{proposition}[compare~{\cite[Proposition~4.2]{GG:MS08}}]\label{prop:Log.properties} 
We have
\begin{equation}\label{eq:Log.int-formula}
 \log_\phi(P)= \frac{i}{2\pi} \int_{\Gamma} \log_\phi(\lambda)\lambda^{-1} P(P-\lambda)^{-1} d\lambda,
\end{equation}
where the integral converges in $\sL(C^\infty(\T^n_\theta))$. 
\end{proposition}
\begin{proof}
Let $\varphi\in \sL(C^\infty(\T^n_\theta))'$. We have
\begin{equation} \label{eq:Log.derivative-int}
 \acou{\varphi}{\log_\phi(P)}= \acou{\varphi}{ \left.\frac{d}{dz}\right|_{z=0} P_\phi^z}=  \left.\frac{d}{dz}\right|_{z=0} \acou{\varphi}{ P_\phi^z}. 
\end{equation}
As mentioned above, $(P-\lambda)^{-1}\in \Psi^{-w,-1}(\T^n_\theta;\Lambda)$, and so this is a $\Hol^{-1}(\Lambda)$-family in $\sL(C^\infty(\T^n_\theta))$. This ensures that, for $\Re z<1$, the maps $\lambda\rightarrow \lambda_\phi^{z-1}(P-\lambda)^{-1}$ and $\lambda \rightarrow  \lambda_\phi^{z-1}P(P-\lambda)^{-1}$ are absolutely integrable maps from $\Gamma$ to $\sL(C^\infty(\T^n_\theta))$, and so we have
\begin{equation*}
 P_\phi^z=PP_\phi^{z-1} = P\bigg(  \frac{i}{2\pi} \int_\Gamma \lambda_\phi^{z-1}(P-\lambda)^{-1} d\lambda\bigg)=
 \frac{i}{2\pi} \int_\Gamma \lambda_\phi^{z-1}P(P-\lambda)^{-1} d\lambda.  
\end{equation*}
Thus, 
\begin{equation*}
 \acou{\varphi}{ P_\phi^z} = \frac{i}{2\pi} \int_\Gamma \lambda_\phi^{z-1}\acou{\varphi}{P(P-\lambda)^{-1}} d\lambda. 
\end{equation*}
As $|\acou{\varphi}{P(P-\lambda)^{-1}}|\leq C|\lambda|^{-1}$ on $\Gamma$, for $|z|\leq \frac{1}{2}$, the function $\lambda \rightarrow \lambda_\phi^{z-1} \acou{\varphi}{P(P-\lambda)^{-1}}$ and its $z$-derivative are bounded by $C'(1+|\log|\lambda||)(1+|\lambda|)^{-3/2}$, which is integrable on $\Gamma$. Therefore, we may differentiate under the integral sign, and we get
\begin{align*}
  \left.\frac{d}{dz}\right|_{z=0} \acou{\varphi}{ P_\phi^z} & =  \frac{i}{2\pi} \int_\Gamma  \left.\frac{\dl}{\dl z}\right|_{z=0}\big\{ \lambda_\phi^{z-1} \acou{\varphi}{P(P-\lambda)^{-1}} \big\}d\lambda\\
  &= \frac{i}{2\pi} \int_{\Gamma} \lambda^{-1}\log_\phi\lambda \acou{\varphi}{P(P-\lambda)^{-1}} d\lambda. 
\end{align*}
Combining this with~(\ref{eq:Log.derivative-int}) shows that, for all $\varphi \in \sL(C^\infty(\T^n_\theta))'$, we have 
\begin{equation*}
  \acou{\varphi}{\log_\phi(P)}=  \frac{i}{2\pi} \int_{\Gamma} \lambda^{-1}\log_\phi\lambda \acou{\varphi}{P(P-\lambda)^{-1}} d\lambda. 
\end{equation*}
Note that the integral $ i(2\pi)^{-1} \int_{\Gamma} \log_\phi(\lambda)\lambda^{-1} P(P-\lambda)^{-1} d\lambda$ converges in $\sL(C^\infty(\T^n_\theta))$, since the map $\lambda \rightarrow  \lambda^{-1}\log_\phi\lambda P(P-\lambda)^{-1}$ is an absolutely integrable map from $\Gamma$ to $\sL(C^\infty(\T^n_\theta))$. 
Combining this with the first part of Proposition~\ref{prop:powers.improper-int} gives~(\ref{eq:Log.int-formula}). 
\end{proof}

\begin{remark}\label{rmk:Log.convergenceWs}
In~(\ref{eq:Log.int-formula}), the integral can also be seen as a Bochner integral in $\sL(W_2^s(\T^n_\theta),W^t_2(\T^n_\theta))$ for all $s,t\in \R$ with $s>t$ (compare~\cite[Proposition~4.2]{GG:MS08}). 
\end{remark}

The following result shows that $\log_\phi (P)$ is actually a \psido.  

\begin{theorem}\label{thm:Log.psido-ness} 
 The logarithm $\log_\phi (P)$ is a \psido\ with symbol, 
\begin{equation}
 \tilde{\omega}(\xi) \sim w\log|\xi| + \sum_{j\geq 0} \omega_{-j}(\xi), 
 \label{eq:Log.asymptotic-symbol}
\end{equation}
where $\sim$ is meant in the sense of~(\ref{eq:Symbols.classical-estimates-qualitative}) and $\omega_{-j}(\xi) \in S_{-j}(\T^n_\theta \times \R^n)$, $j\geq 0$,  is given by
\begin{gather}\label{eq:Log.symbols1} 
 \omega_0(\xi)= \log_\phi\big[ \rho_w(|\xi|^{-1}\xi)\big], \qquad \xi\neq 0,\\
 \omega_{-j}(\xi)= \frac{i}{2\pi} \int_{\gamma_\xi} \log_\phi (\lambda)\sigma_{-w-j}(\xi;\lambda)d \lambda, \qquad j\geq 1.
 \label{eq:Log.symbols2}  
\end{gather}
Here $\sigma_{-w-j}(\xi; \lambda)$ is given by~(\ref{eq:Resolvent.symbol-resolvent1})--(\ref{eq:Resolvent.symbol-resolvent2}) and $\gamma_{\xi}$ is any counter-clockwise oriented rectifiable closed contour in $\C^*\setminus L_\phi $ whose interior contains $\Sp(\rho_w(\xi))$.
\end{theorem}
\begin{proof}
 As $P_\phi^z$, $z\in \C$, is a holomorphic family in $\Psi^\bt(\T^n_\theta)$ of order $wz$, there is a holomorphic family  $\rho(z)$, $z\in \C$,  in $S^\bt(\T^n_\theta\times \R^n)$ of order $wz$ such that $P_\phi^z=P_\rho(z)$ for all $z\in \C$. Given any $a>0$, the restriction $\rho(z)$, $\Re z<a$, is a holomorphic family in $\stS^{m}(\T^n_\theta\times \R^n)$ with $m=wa$. As $\rho \rightarrow P_\rho$ is a continuous linear map from $\stS^m(\T^n_\theta\times \R^n)$ to $\sL(C^\infty(\T^n_\theta))$, we get 
\begin{equation*}
 \frac{d}{dz} P^z_\phi = \frac{d}{dz} P_\rho(z)= P_{\partial_z \rho(z)}, \qquad \Re z<a.
\end{equation*}
In particular, if we set $\tilde{\omega}(\xi):=\partial_z\rho(0)(\xi)$, then $\tilde{\omega}(\xi)\in \stS^m(\T^n_\theta\times \R^n)$, and we have
\begin{equation*}
 \log_\phi(P) = \left.\frac{d}{dz}\right|_{z=0} P_\phi^z =P_{\tilde{\omega}}. 
\end{equation*}
This shows that $\log_\phi(P) $ is a \psido. 

Bearing this in mind, Theorem~\ref{thm:powers.complex-powers} ensures that $\rho(z)(\xi) \sim \sum_{j\geq 0} \rho_j(z)(\xi)$ in the sense of~(\ref{eq:Hol.hol-family-classical-symbol-estimate}), with
\begin{equation}
 \rho_{j}(z)(\xi)= \frac{i}{2\pi}\int_{\gamma_\xi} \lambda_\phi^z  \sigma_{-w-j}(\xi;\lambda)d \lambda, \qquad \xi \neq 0,
 \label{eq:Log.symbols-powers}
\end{equation}
where $\sigma_{-w-j}(\xi; \lambda)$ is given by~(\ref{eq:Resolvent.symbol-resolvent1})--(\ref{eq:Resolvent.symbol-resolvent2}) and $\gamma_{\xi}$ is any counter-clockwise oriented rectifiable closed contour in $\C^*\setminus L_\phi $ whose interior contains $\Sp(\rho_w(\xi))$. In particular, $\rho_{j}(z)(\xi)$, $z\in \C$, is a holomorphic family in $C^\infty(\T^n_\theta\times (\R^n\setminus 0))$. Moreover, by arguing as in the proof of Proposition~\ref{prop:Log.properties}, for $\xi\neq 0$, we obtain
\begin{equation*}
 \left.\frac{\dl}{\dl z}\right|_{z=0} \rho_{j}(z)(\xi)=  
 \frac{i}{2\pi} \int_{\gamma_\xi} \left.\frac{\dl}{\dl z}\right|_{z=0}\big\{\lambda_\phi^z  \sigma_{-w-j}(\xi;\lambda)\big\}d \lambda
 = \frac{i}{2\pi} \int_{\gamma_\xi}\log_\phi(\lambda) \sigma_{-w-j}(\xi;\lambda)d \lambda.
\end{equation*}
In particular, for $j=0$, we have
\begin{equation*}
  \left.\frac{\dl}{\dl z}\right|_{z=0} \rho_{0}(z)(\xi)=\frac{i}{2\pi} \int_{\gamma_\xi} \log_\phi(\lambda)\big(\rho_w(\xi)-\lambda\big)^{-1}d \lambda = \log_\phi \big[\rho_w(\xi)\big]. 
\end{equation*}
It follows from all this that the formulas~(\ref{eq:Log.symbols1})--(\ref{eq:Log.symbols2}) define elements of $C^\infty(\T^n_\theta\times (\R^n\setminus 0))$. Note also that, for $j\geq 1$, we have 
\begin{equation}\label{eq:Log.symbols-derivatives1}
 \omega_{-j}(\xi)= \left.\frac{\dl}{\dl z}\right|_{z=0} \rho_{j}(z)(\xi), \qquad \xi \neq 0. 
\end{equation}
Moreover, we have
\begin{equation}\label{eq:Log.symbols-derivatives2}
 \left.\frac{\dl}{\dl z}\right|_{z=0} \rho_{0}(z)(\xi)= \log_\phi \big[|\xi|^w\rho_w(|\xi|^{-1}\xi)\big]=w\log |\xi| +\omega_0(\xi). 
\end{equation}

It is immediate from its definition that $\omega_0(\xi)$ is homogeneous of degree $0$ with respect to $\xi$, and so this is an element of $S_{0}(\T^n_\theta\times \R^n)$. Suppose that $j\geq 1$ and $\xi\neq 0$. Let $t>0$. As $\Sp(\rho_w(t\xi))=t^w\Sp(\rho_w(\xi))$, in the formulas~(\ref{eq:Log.symbols2}) for $\omega_{-j}(\xi)$ and  $\omega_{-j}(t\xi)$, we may choose the contours $\gamma_\xi$ and $\gamma_{t\xi}$ so that $\gamma_{t\xi}=t^w\gamma_\xi$. Thus,
\begin{align}
 \omega_{-j}(t\xi) = & \frac{i}{2\pi} \int_{t^w\gamma_\xi} \log_\phi (\lambda) \sigma_{-w-j}(t\xi;\lambda)d\lambda \nonumber\\
  = & t^w  \frac{i}{2\pi} \int_{\gamma_\xi} \log_\phi (t^w\lambda) \sigma_{-w-j}(t\xi;t^w\lambda)d\lambda \label{eq:Log.homogeneity-omegaj} \\
  = & wt^{-j}\log t  \frac{i}{2\pi} \int_{\gamma_\xi} \sigma_{-w-j}(\xi;\lambda)d\lambda 
 +  t^{-j} \frac{i}{2\pi} \int_{\gamma_\xi} \log_\phi (\lambda) \sigma_{-w-j}(\xi;\lambda)d\lambda
\nonumber\\
 =  &  wt^{-j}(\log t) \rho_j(0)(\xi) + t^{-j}\omega_{-j}(\xi),\nonumber
\end{align}
where we have used~(\ref{eq:Log.symbols-powers}) to get the last equality. We observe that $\rho_j(0)(\xi)$ is the symbol of degree $-j$ of $P_\phi^0=1-\Pi_0(P)$. As $\Pi_0(P)$ is a smoothing operator, the  symbol of $1-\Pi_0(P)$ is~$\sim 1$. In particular, all its homogeneous symbols of negative degree are 0, i.e.,  $\rho_{j}(0)(\xi)=0$ for $j\geq 1$. Combining this with~(\ref{eq:Log.homogeneity-omegaj}) shows that $\omega_{-j}(\xi)$ is homogeneous of degree~$-j$ for $j\geq 1$, and hence it lies in the symbol class $S_{-j}(\T^n_\theta \times \R^n)$. 

It remains to show that $\tilde{\omega}(\xi) \sim w\log|\xi|+ \sum_{j\geq 0} \omega_{-j}(\xi)$. In view of~(\ref{eq:Log.symbols-derivatives1})--(\ref{eq:Log.symbols-derivatives2}), this is equivalent to showing that
\begin{equation}
 \tilde{\omega}(\xi) \sim \sum_{ j \geq 0}\partial_z\rho_j(0)(\xi). 
 \label{eq:Log.asymptotic-symbol2}
\end{equation}
Let $a>0$ and let $\chi(\xi)\in C_c^\infty(\R^n)$ be such that $\chi(\xi)=1$ near $\xi=0$ and $\chi(\xi)=0$ for $|\xi|\geq 1$. Once again, the real part of the order of $\rho(z)(\xi)$ is~$\leq wa$ for $\Re z<a$. Therefore, by Lemma~\ref{lem:Hol.asymptotic-classical-standard}, the fact that $\rho(z)(\xi) \sim \sum_{j\geq 0} \rho_j(z)(\xi)$ in the sense of~(\ref{eq:Hol.hol-family-classical-symbol-estimate}) implies that $\rho(z)(\xi) \sim \sum_{j\geq 0} (1-\chi(\xi)) \rho_j(z)(\xi)$ in the sense of~(\ref{eq:Hol.asymptotic-standard}). Thus, by Remark~\ref{rmk:Hol.asymptotic-standard-stS}, we have
\begin{equation*}
 \rho(z)(\xi) -\sum_{j<N} \big(1-\chi(\xi)\big) \rho_j(z)(\xi) \in \Hol\big(\Re z<a;\stS^{wa-N}(\T^n_\theta\times \R^n)\big) \qquad \forall N\geq 1. 
\end{equation*}
Differentiating at $z=0$, we then get
\begin{equation*}
  \tilde{\omega}(\xi) -\sum_{j<N} \big(1-\chi(\xi)\big) \partial_z\rho_j(0)(\xi) \in \stS^{wa-N}(\T^n_\theta\times \R^n) \qquad \forall N\geq 1.
\end{equation*}
This gives~(\ref{eq:Log.asymptotic-symbol2}). 
\end{proof}
We refer to~\cite{Gr:MS12, Ok:Duke95} for versions of Theorem~\ref{thm:Log.psido-ness} for elliptic \psidos\ on closed manifolds. 

\begin{remark}
 We can get a more quantitative version of~(\ref{eq:Log.asymptotic-symbol}). For all $N\geq 0$ and multi-orders $\alpha$ and $\beta$ there is $C_{N\alpha\beta}>0$ such that, for all $\xi\in \R^n$ with $|\xi|\geq 1$, we have
 \begin{equation*}
 \Big\| \delta^\alpha \partial_\xi^\beta \big( \tilde{\omega}(\xi)-w\log|\xi| - \sum_{j<N} \omega_{-j}(\xi) \big)\Big\| \leq C_{N\alpha\beta} | \xi |^{-N-| \beta |} .
\end{equation*}
This implies the following estimates for the symbol $\tilde{\omega}(\xi)$. First, there is $C>0$ such that 
\begin{equation*}
 \big\| \tilde{\omega}(\xi)\big\|\leq C\big(1+ \log(1+ |\xi|)\big) \qquad \forall \xi \in \R^n. 
\end{equation*}
Second, for all multi-orders $\alpha$ and $\beta$ with $(\alpha,\beta)\neq (0,0)$, there is $C_{\alpha\beta}>0$ such that 
\begin{equation*}
 \big\| \delta^\alpha \partial_\xi^\beta \tilde{\omega}(\xi)\big\|\leq C_{\alpha\beta} \big(1+ |\xi|\big)^{-|\beta|} \qquad \forall \xi \in \R^n. 
\end{equation*}
\end{remark}

\begin{remark}
 Suppose that $P$ has a positive principal symbol and is a non-negative selfadjoint operator with respect to an inner product of the form~(\ref{eq:Powers.inner-product-density}). In this case, $\Theta(P)=\check{\Theta}(P)=\C\setminus [0,\infty)$ and we can define $\log (P)$ by using the Borel functional calculus for $P$ (with the convention that $\log (P)=0$ on $\ker P$). Equivalently, if $(e_{\ell})_{\ell\geq 0}$ is any orthonormal eigenbasis for $P$, then 
\begin{equation*}
 \log(P) e_\ell  = \left\{ 
\begin{array}{cl}
(\log \lambda_\ell) e_{\ell} & \text{if $\lambda_\ell>0$},\\
0 & \text{if $\lambda_\ell=0$}. 
\end{array}\right.
\end{equation*}
This agrees with the operator $\log_\phi (P)$ defined in~(\ref{eq:Log.definition}) with any angle $\phi\in (0,2\pi)$. 
\end{remark}

\begin{example}\label{ex:log.log-Delta}
 Let $\Delta=\delta_1^2+\cdots + \delta_n^2$ be the flat Laplacian on $\T^n_\theta$. For all $k \in \Z^n$, we have
 \begin{equation*}
 \log(\Delta) U^k  = \left\{ 
\begin{array}{cl} 
2(\log |k|)U^k  & \text{if $k\neq0$},\\
0 & \text{if $k=0$}. 
\end{array}\right.
\end{equation*}
Therefore, we see that $\log(\Delta)=2P_\eta$, where $\eta(\xi)=(1-\chi(\xi))\log|\xi|$ with  $\chi(\xi)\in C_c^\infty(\R^n)$ such that $\chi(\xi)=1$ near $\xi=0$ and $\chi(\xi)=0$ for $|\xi|\geq 1$. 
\end{example}

The above example allows us to reformulate Theorem~\ref{thm:Log.psido-ness} as follows. 

\begin{corollary}\label{cor:Log.logD-PsiDO}
We may write
\begin{equation*}
 \log_\phi(P)=\frac{1}{2}w \log(\Delta) + Q_\phi,
\end{equation*}
where $Q_\phi\in \Psi^0(\T^n_\theta)$ has symbol $\omega(\xi)\sim \sum_{j\geq 0} \omega_{-j}(\xi)$ with $\omega_{-j}(\xi)$ given by~(\ref{eq:Log.symbols1})--(\ref{eq:Log.symbols2}). 
\end{corollary}
\begin{proof}
 Set $Q_\phi= \log_\phi(P)-\frac{1}{2}w \log(\Delta)$. Let $\chi(\xi)\in C_c^\infty(\R^n)$ be such that $\chi(\xi)=1$ near $\xi=0$ and $\chi(\xi)=0$ for $|\xi|\geq 1$, and set $\eta(\xi)=(1-\chi(\xi))\log|\xi|$. As mentioned above, $\log(\Delta)=2P_\eta$. Theorem~\ref{thm:Log.psido-ness} asserts that $\log_\phi(P)$ has a symbol such that 
 $\tilde{\omega}(\xi)\sim w\log|\xi|+ \sum_{j\geq 0} \omega_{-j}(\xi)$ with $\omega_{-j}(\xi)$ given by~(\ref{eq:Log.symbols1})--(\ref{eq:Log.symbols2}). Thus, if we set 
 $\omega(\xi)=\tilde{\omega}(\xi)-w \eta(\xi)$, then $Q_\phi=P_{\omega}$ and $\omega(\xi)\sim  \sum_{j\geq 0} \omega_{-j}(\xi)$. Thus, $\omega(\xi)\in S^0(\T^n_\theta\times \R^n)$, and hence $Q_\phi\in \Psi^0(\T^n_\theta)$. This gives the result. 
\end{proof}

\begin{remark}\label{rmk:Log.domain}
 It follows from Corollary~\ref{cor:Log.logD-PsiDO} that $\log_\phi(P)$ is closable as an operator on $L_2(\T_\theta^n)$, and its closure has the same domain as $\log(\Delta)$, where 
\begin{equation*}
 \dom\big( \log(\Delta)\big)=\left\{u =\sum u_kU^k \in L_2(\T^n_\theta); \ \sum_{k\neq 0} (\log |k|)^2|u_k|^2<\infty\right\}. 
\end{equation*}
\end{remark}

It follows from Theorem~\ref{thm:Log.psido-ness} that $\log_\phi(P)$ is a (standard) \psido\ of order $\leq a$ for any $a>0$. Therefore, if $\Re z<-w^{-1}n$, then $\log_\phi(P)P_\phi^z$ is a  standard \psido\ of order~$<-n$, and hence it is trace-class. In fact, we have the following result. 

\begin{proposition} \label{prop:Log.trace-of-Pz-derivative}
 We have
 \begin{equation*}
 \frac{d}{dz} \Tr\big[P_\phi^z\big]= \Tr\big[\log_\phi(P)P_\phi^z\big], \qquad \Re z<-w^{-1}n.  
\end{equation*}
\end{proposition}
\begin{proof}
 Let $z_0\in \C$, $\Re z_0<-w^{-1}n$. Set $m=w\Re z_0$, and let $a\in (0,w^{-1}|m+n|)$.  In addition, let 
  $\rho(z)$, $z\in \C$,  be a holomorphic family in $S^\bt(\T^n_\theta\times \R^n)$ of order $wz$ such that $P_\phi^z=P_{\rho}(z)$, $z\in \C$. Thus, 
\begin{equation*}
 P_\phi^{z+z_0}=P_\phi^zP_\phi^{z_0}=P_{\rho(z)} P_{\rho(z_0)}=P_{\rho(z)\sharp \rho(z_0)}. 
\end{equation*}
 Moreover, as shown in the proof of Theorem~\ref{thm:Log.psido-ness}, we have $\log_\phi(P)=P_{\partial_z\rho(0)}$. 
 
 Note that $\rho \rightarrow \rho\sharp \rho(z_0)$ is a continuous linear map from $\bS^{wa}(\T^n_\theta\times \R^n)$ to $\bS^{wa+m}(\T^n_\theta\times \R^n)$. As $wa+m<-n$, we get a continuous linear map $\rho \rightarrow P_{\rho\sharp \rho(z_0)}$ from $\bS^{wa}(\T^n_\theta\times \R^n)$ to $\sL_1$, and we get a continuous linear form $\rho \rightarrow \Tr[P_{\rho\sharp \rho(z_0)}]$ on $\bS^{wa}(\T^n_\theta\times \R^n)$. 
 As $\rho(z)$, $\Re z<a$, is a holomorphic family in $\bS^{wa}(\T^n_\theta\times \R^n)$, it follows that $\Tr[P_{\rho(z)\sharp \rho(z_0)}]$, $\Re z<a$, is a holomorphic function, and we have
\begin{equation*}
\frac{d}{dz} \Tr\big[ P_{\rho(z)\sharp \rho(z_0)}\big] \Big|_{z=0} = \Tr\big[ P_{\partial_z\rho(0)\sharp \rho(z_0)}\big]. 
\end{equation*}
As $P_{\rho(z)\sharp \rho(z_0)}=P_\phi^{z+z_0}$ and $P_{\partial_z\rho(0)\sharp \rho(z_0)}=P_{\partial_z\rho(0)}P_{\rho(z_0)}=\log_\phi(P)P_\phi^{z_0}$, we get
\begin{equation*}
 \frac{d}{dz} \Tr\big[ P_\phi^{z+z_0}\big] \Big|_{z=0} = \Tr\big[ \log_\phi(P)P_\phi^{z_0}\big]. 
\end{equation*}
This proves the result. 
\end{proof}

\section{Sectorial Projections} \label{sec:sectorial}
In this section, we construct the sectorial projections of elliptic \psidos\ and use them to study how their complex powers and logarithms depend on the ray used to define them.  

Throughout this section, $P\in \Psi^w(\T^n_\theta)$, $w>0$, is as in Section~\ref{sec:Resolvent}.

Suppose first that $P$ is selfadjoint with $\check{\Theta}(P)=\C\setminus \R$. As in Remark~\ref{rmk:powers.selfadjoint}, we then obtain two families of complex powers $P_{\uparrow\downarrow}^z$, $z\in \C$, and, by Proposition~\ref{prop:powers.absolute-value}, the powers $|P|^{z}$ form a holomorphic family in $\Psi^\bt(\T^n_\theta)$. In particular, $|P|^{-1}\in \Psi^{-w}(\T^n_\theta)$ has principal symbol $|\rho_w(\xi)|^{-1}$, so the sign operator $\op{Sign}(P)=P|P|^{-1}$ lies in $\Psi^0(\T^n_\theta)$ with principal symbol $\op{Sign}(\rho_w(\xi))$. Hence the orthogonal projections $\Pi_{\pm}(P)=\frac12\big(1-\Pi_0(P)\pm \op{Sign}(P)\big)$ onto the positive and negative eigenspaces of $P$ are operators in $\Psi^{0}(\T^n_\theta)$ with principal symbols $\Pi_{\pm}(\rho_w(\xi))$. Using~(\ref{eq:powers.selfadjoint}), we then obtain
\begin{equation}\label{eq:powers.asymmetry-selfadjoint}
 P_{\uparrow}^z - P_{\downarrow}^z = \big( 1- e^{2i\pi z}\big)\Pi_{-}(P)P_{\uparrow}^z, \qquad z\in \C,
\end{equation}
so the asymmetry between the two families of complex powers of $P$ is encoded by the \psido\ projection $\Pi_{-}(P)$.

We shall now extend the above considerations to the non-selfadjoint case. Assume that $\Theta(P)\neq \emptyset$, and let $L_{\phi}=\{\arg \lambda =\phi\}$ and $L_{\phi'}=\{\arg \lambda =\phi'\}$ be rays contained in $\check{\Theta}(P)$ with $\phi<\phi'\leq\phi+2\pi$. In what follows, for $\phi_1<\phi_2\leq \phi_1+2\pi$, we denote by 
 $S_{\phi_1,\phi_2}$ the angular sector $\{\phi_1<\arg \lambda<\phi_2\}$, and we set 
 \begin{equation}\label{eq:sectorial.sectorial-generalized-space}
 E_{\phi_1,\phi_2}(P):= \!\!\!\!\!{ \bigplus^{.}_{\lambda \in \Sp(P)\cap S_{\phi_1,\phi_2}}}  \!\!\!\!\! E_\lambda(P),  
\end{equation}
where $\dotplus$ denotes the algebraic direct sum. For instance, if $\phi_2=\phi_1+2\pi$, then $S_{\phi_1,\phi_2}$ is just $\C^*\setminus L_{\phi_1}$, and $E_{\phi_1,\phi_2}(P)$ is the direct sum of the generalized eigenspaces $E_\lambda(P)$ associated with non-zero eigenvalues of $P$. 

We define the complex powers $P_\phi^z$ and $P_{\phi'}^z$, $z\in \C$, as in Section~\ref{sec:powers}. We seek an analogue of~(\ref{eq:powers.asymmetry-selfadjoint}) for the difference $P_{\phi}^z-P_{\phi'}^z$. 

As $(P-\lambda)^{-1}\in \Psi^{-w,-1}(\T^n_\theta;\Lambda(P))$ and we can regard  $P$ as an element of $\Psi^{w,0}(\T^n_\theta;\Lambda(P))$, we see that $P(P-\lambda)^{-1}$ is contained in $\Psi^{0,-1}(\T^n_\theta;\Lambda(P))$. Part~(2) of Proposition~\ref{prop:PsiDOs-parameter.classical-PsiDO-Sobolev-mapping-properties} then ensures that, for any $s\in (0,w]$, this is a  $\Hol^{-s/w}(\Lambda(P))$-family in $\sL(W_2^s(\T^n_\theta), L_2(\T^n_\theta))$. This enables us to define 
\begin{equation}\label{def:sectorial.sectorial-projection}
 \Pi_{\phi,\phi'}(P):= \frac{i}{2\pi} \int_{\Gamma'} \lambda^{-1}P(P-\lambda)^{-1}d\lambda, 
\end{equation}
where the integral converges absolutely in $\sL(W_2^s(\T^n_\theta), L_2(\T^n_\theta))$, and 
 $\Gamma'$ is any contour of the form $\Gamma(\phi',\phi,r)$ as in~(\ref{eq:powers.contourG1})--(\ref{eq:powers.contourG2}) with $0<r<r_0$, where $r_0=\inf\{|\lambda|; \lambda \in \Sp(P), \ \lambda \neq 0\}$. 
\begin{figure}[h]
\begin{minipage}{0.25\linewidth}
\centering
\begin{tikzpicture}[cpfig]\cpbox
 \cpaxes
 \draw[arrowat=0.5] (135:6.1)--(135:0.85);
 \draw[arrowsat={0.37}{0.63}] (135:0.85) arc[start angle=135,end angle=-135,radius=0.85];
 \draw[arrowat=0.43] (-135:0.85)--(-135:6.7);
 \node[anchor=base west,inner sep=0pt] at (0.883,0.778) {$\Gamma'$};
\end{tikzpicture}
\end{minipage}
\caption{Contour $\Gamma'=\Gamma(\phi',\phi,r)$.}\label{fig:sectorial-contour}
\end{figure}
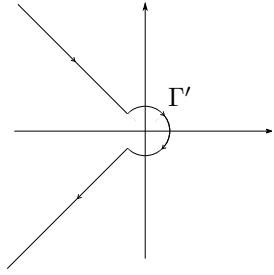

We may regard $ \Pi_{\phi,\phi'}(P)$ as an unbounded operator on $L_2(\T^n_\theta)$ whose domain is 
\begin{equation*}
 \bigcap_{s>0} W_2^s(\T^n_\theta). 
\end{equation*}

In fact, given any real numbers $t>s$,  Proposition~\ref{prop:PsiDOs-parameter.classical-PsiDO-Sobolev-mapping-properties} further asserts that $P(P-\lambda)^{-1}$ is a 
$\Hol^{-\alpha}(\Lambda(P))$-family in $\sL(W_2^t(\T^n_\theta), W_2^s(\T^n_\theta))$ with $\alpha=\min(1,w^{-1}(t-s))$ (if $t-s>w$, we use the continuous inclusion of $W_2^t(\T^n_\theta)$ into $W_2^{s+w}(\T^n_\theta)$), and so the integral~(\ref{def:sectorial.sectorial-projection}) converges in $\sL(W_2^t(\T^n_\theta), W_2^s(\T^n_\theta))$. Given $0<s<t$, we may then regard $ \Pi_{\phi,\phi'}(P)$ both as an operator in $\sL(W_2^t(\T^n_\theta), W_2^s(\T^n_\theta))$  and in $\sL(W_2^s(\T^n_\theta), L_2(\T^n_\theta))$. This allows us to define the composition 
$ \Pi_{\phi,\phi'}(P)\circ \Pi_{\phi,\phi'}(P)$ as an operator in $\sL(W_2^t(\T^n_\theta), L_2(\T^n_\theta))$. By arguing along the same lines as the proof of~\cite[Proposition~3.2]{Po:IJM06}, it can then be shown that 
\begin{equation}\label{eq:sectorial.idempotence}
 \Pi_{\phi,\phi'}(P)\circ \Pi_{\phi,\phi'}(P) = \Pi_{\phi,\phi'}(P). 
\end{equation}
The proof of~\cite[Proposition~3.2]{Po:IJM06} further shows that if $\lambda\in \Sp(P)$, then 
\begin{equation*}
 \Pi_{\phi,\phi'}(P)\Pi_\lambda(P)= \left\{ 
 \begin{array}{cl}
 \Pi_\lambda(P) & \text{if} \ \lambda \in S_{\phi,\phi'},\\
 0 & \text{if} \ \lambda \not\in S_{\phi,\phi'}.
\end{array}\right.
\end{equation*}
 It follows that 
  \begin{equation}\label{eq:sectorial.range-kernel-algebraic}
 \ran \Pi_{\phi,\phi'}(P) \supseteq E_{\phi,\phi'}(P)  \qquad \text{and} \qquad \ker \Pi_{\phi,\phi'}(P) \supseteq E_0(P) \dotplus E_{\phi',\phi+2\pi}(P),   
 \end{equation}
where $E_{\phi,\phi'}(P)$ and $E_{\phi',\phi+2\pi}(P)$ are defined as in~(\ref{eq:sectorial.sectorial-generalized-space}). 

Furthermore, as $P(P-\lambda)^{-1}$ is in $\Psi^{0,-1}(\T^n_\theta;\Lambda(P))$, we also know from Remark~\ref{rmk:Parametric.parametric-Psidos-families-Cinfty-sDprime} that $P(P-\lambda)^{-1}$ is a $\Hol^{-1}(\Lambda(P))$-family in 
$\sL(C^\infty(\T^n_\theta))$. This ensures that the integral~(\ref{def:sectorial.sectorial-projection}) converges in $\sL(C^\infty(\T^n_\theta))$, and so $\Pi_{\phi,\phi'}(P)$ maps continuously $C^\infty(\T^n_\theta)$ to itself. We actually have the following result. 

\begin{proposition}\label{prop:powers.symbol-sect-proj}
The following hold. 
\begin{enumerate}
 \item $\Pi_{\phi,\phi'}(P)$ is a \psido\ of order $0$. 
 
\item It has symbol $\pi(\xi)\sim \sum_{j\geq 0} \pi_{-j}(\xi)$, with $\pi_{-j}(\xi)\in S_{-j}(\T^n_\theta\times \R^n)$ given by
\begin{equation}
 \pi_{-j}(\xi):=\frac{i}{2\pi} \int_{\gamma_{\xi,\phi,\phi'}} \sigma_{-w-j}(\xi; \lambda)d\lambda,\qquad \xi\neq 0, 
  \label{eq:Powers.homogeneous-symbols-sectorial}
\end{equation}
where $ \sigma_{-w-j}(\xi; \lambda)$ is given by~(\ref{eq:Resolvent.symbol-resolvent1})--(\ref{eq:Resolvent.symbol-resolvent2}) and $\gamma_{\xi,\phi,\phi'}$ is any counter-clockwise oriented rectifiable closed contour in $\C^*\setminus L_\phi$ enclosing $\Sp(\rho_w(\xi))\cap S_{\phi,\phi'}$ but no other point of $\Sp(\rho_w(\xi))$. In particular, the principal symbol of $\Pi_{\phi,\phi'}(P)$ is equal to $\Pi_{\phi,\phi'}(\rho_w(\xi))$. 

\item $\Pi_{\phi,\phi'}(P)$ is a smoothing operator if and only if $S_{\phi,\phi'}\subseteq \Theta(P)$. 
\end{enumerate}
\end{proposition}

\begin{remark}\label{rmk:sectorial.principal-symbol}
 If $a\in C^\infty(\T^n_\theta)$ has its spectrum contained in $\C^*\setminus [L_{\phi}\cup L_{\phi'}]$, then the characteristic function $\car_{S_{\phi,\phi'}}$ is holomorphic near $\Sp(a)$, and so we may define the projection $\Pi_{\phi,\phi'}(a):=\car_{S_{\phi,\phi'}}(a)$ by holomorphic functional calculus. This yields an idempotent element of $C^\infty(\T^n_\theta)$, since $C^\infty(\T^n_\theta)$ is closed under holomorphic functional calculus. 
\end{remark}

\begin{proof}[Proof of Proposition~\ref{prop:powers.symbol-sect-proj}] 
 As $P(P-\lambda)^{-1} \in \Psi^{0,-1}(\T^n_\theta;\Lambda(P))$, Proposition~\ref{prop:powers.param-psido-hol-psidos} ensures that we define a holomorphic family in $\Psi^{\bt}(\T^n_\theta)$ by letting
\begin{equation*}
 \tilde{P}_{\phi,\phi'}(z):= \frac{i}{2 \pi} \int_{\Gamma'} \lambda^z_\phi P(P-\lambda)^{-1}d\lambda, \qquad \Re z<0, 
\end{equation*}
where the integral converges in $\sL(C^\infty(\T^n_\theta))$. Here $\tilde{P}_{\phi,\phi'}(z)$ has order $w(z+1)$ and, by definition, $\Pi_{\phi,\phi'}(P)=\tilde{P}_{\phi,\phi'}(-1)$. In particular, this shows that $\Pi_{\phi,\phi'}(P) \in \Psi^0(\T^n_\theta)$. Moreover, if we let $\tilde{\pi}(z)(\xi)\sim \sum_{j\geq 0} \tilde{\pi}_j(z)(\xi)$ be the symbol of $\tilde{P}_{\phi,\phi'}(z)$, then 
$ \tilde{\pi}_j(z)(\xi)\in \Hol(\Re z<0; C^\infty(\T^n_\theta))$ for $\xi\neq 0$. 

Similarly, as $(P-\lambda)^{-1} \in \Psi^{-w,-1}(\T^n_\theta;\Lambda(P))$, we also define a holomorphic family in $\Psi^{\bt}(\T^n_\theta)$ by letting
\begin{equation*}
 P_{\phi,\phi'}(z):= \frac{i}{2\pi} \int_{\Gamma'} \lambda^z_\phi (P-\lambda)^{-1}d\lambda, \qquad \Re z<0.  
\end{equation*}
Moreover, along the same lines as the proof of Lemma~\ref{lem:powers.powers-half-plane}, it can be shown that $P_{\phi,\phi'}(z)$ has symbol $\pi(z)\sim \sum_{j\geq 0} \pi_j(z)(\xi)$, with $\pi_j(z)(\xi)\in S_{wz-j}(\T^n_\theta\times \R^n)$  given by
\begin{equation}
 \pi_{j}(z)(\xi) :=\frac{i}{2\pi} \int_{\gamma_{\xi,\phi,\phi'}} \lambda^{z}_\phi \sigma_{-w-j}(\xi; \lambda)d\lambda,\qquad \xi\neq 0, 
  \label{eq:Powers.homogeneous-symbols-sectorial2}
\end{equation}
where $\sigma_{-w-j}(\xi; \lambda)$ is as above and $\gamma_{\xi,\phi,\phi'}$ is any counter-clockwise oriented rectifiable closed contour in $\C^*\setminus L_\phi$ enclosing $\Sp(\rho_w(\xi))\cap S_{\phi,\phi'}$ but no other point of $\Sp(\rho_w(\xi))$.

Thanks to the equality $P(P-\lambda)^{-1}=1+\lambda(P-\lambda)^{-1}$, for $\Re z<-1$, we get
\begin{equation*}
 \tilde{P}_{\phi,\phi'}(z) =\frac{i}{2 \pi} \int_{\Gamma'} \lambda^z_\phi d\lambda +\frac{i}{2 \pi} \int_{\Gamma'} \lambda^{z+1}_\phi (P-\lambda)^{-1}d\lambda = P_{\phi,\phi'}(z+1). 
\end{equation*}
Combining this with~(\ref{eq:Powers.homogeneous-symbols-sectorial2}), we then deduce that, for $j\geq 0$ and $\xi\neq 0$, we have
\begin{equation*}
 \tilde{\pi}_j(z)(\xi) =  \pi_{j}(z+1)(\xi)= \frac{i}{2\pi} \int_{\gamma_{\xi,\phi,\phi'}} \lambda^{z+1}_\phi \sigma_{-w-j}(\xi; \lambda)d\lambda,\qquad \Re z<-1.
\end{equation*}
We observe that the map $z\rightarrow \tilde{\pi}_j(z)(\xi)$ is holomorphic on the half-plane $\{\Re z<0\}$. Thus, by using the same kind of analytic continuation arguments as in the proof of Theorem~\ref{thm:powers.complex-powers}, we get
\begin{equation*}
 \tilde{\pi}_j(z)(\xi) =\frac{i}{2 \pi} \int_{\gamma_{\xi,\phi,\phi'}} \lambda^{z+1}_\phi \sigma_{-w-j}(\xi; \lambda)d\lambda\quad \text{for $\Re z<0$}.
\end{equation*}
For $z=-1$, this gives the formula~(\ref{eq:Powers.homogeneous-symbols-sectorial})  for the homogeneous symbols of $\Pi_{\phi,\phi'}(P)$, since, in this case, we may allow $\gamma_{\xi,\phi,\phi'}$ to intersect $L_\phi$. In particular, for $j=0$, we get 
\begin{equation*}
 \pi_0(\xi)= \frac{i}{2\pi} \int_{\gamma_{\xi,\phi,\phi'}} \left(\rho_w(\xi)- \lambda\right)^{-1}d\lambda=\car_{S_{\phi,\phi'}}\left(\rho_w(\xi)\right)= \Pi_{\phi,\phi'}\left(\rho_w(\xi)\right). 
 \end{equation*}
 
It remains to prove the last part. If $S_{\phi,\phi'}\subseteq \Theta(P)$, then $\Sp(\rho_w(\xi))\cap S_{\phi,\phi'}=\emptyset$ for all $\xi \neq 0$, so $\lambda \rightarrow \sigma_{-w-j}(\xi;\lambda)$ is holomorphic inside $\gamma_{\xi,\phi,\phi'}$. Hence the integral in~(\ref{eq:Powers.homogeneous-symbols-sectorial}) vanishes and $\pi_{-j}(\xi)=0$ for all $j\geq 0$. Thus, $\Pi_{\phi,\phi'}(P)$ is a smoothing operator.

Conversely, if $\Pi_{\phi,\phi'}(P)$ is a smoothing operator, then $\Pi_{\phi,\phi'}(\rho_w(\xi))=\pi_0(\xi)=0$ for all $\xi \neq 0$. Recall that $\Pi_{\phi,\phi'}(\rho_w(\xi))=\car_{S_{\phi,\phi'}}(\rho_w(\xi))$, where $\car_{S_{\phi,\phi'}}(\rho_w(\xi))$ is defined by using the holomorphic functional calculus for $\rho_w(\xi)$. Therefore, by the spectral mapping theorem, we have
\begin{equation*}
 \car_{S_{\phi,\phi'}} \left(\Sp\left(\rho_w(\xi)\right)\right)= \Sp\big(\car_{S_{\phi,\phi'}}(\rho_w(\xi))\big)=\Sp(0)=\{0\}.  
\end{equation*}
 This implies that $\Sp(\rho_w(\xi))\cap S_{\phi,\phi'} =\emptyset$ for all $\xi \neq 0$, and hence $S_{\phi,\phi'}\subseteq \Theta(P)$. This completes the proof. 
 \end{proof}

Combining Proposition~\ref{prop:powers.symbol-sect-proj} with~(\ref{eq:sectorial.idempotence}) and~(\ref{eq:sectorial.range-kernel-algebraic}), we obtain the following. 

\begin{corollary}
 The operator $\Pi_{\phi,\phi'}(P)$ is a bounded projection on $L_2(\T^n_\theta)$ such that 
  \begin{equation}\label{eq:powers.range-kernel-spec-proj}
 \ran \Pi_{\phi,\phi'}(P) \supseteq \overline{ E_{\phi,\phi'}(P)}  \qquad \text{and} \qquad \ker \Pi_{\phi,\phi'}(P) \supseteq E_0(P) \dotplus \overline{ E_{\phi',\phi+2\pi}(P)}. 
 \end{equation}
\end{corollary}

The above corollary motivates the following definition. 

\begin{definition}
 The operator $\Pi_{\phi,\phi'}(P)$ is called the \emph{sectorial projection} onto the sector $S_{\phi,\phi'}$. 
\end{definition}

 \begin{remark}
Sectorial projections were introduced by Burak~\cite{Bu:Pisa70} for elliptic differential operators associated with regular elliptic boundary problems. Later, they were also considered by Wodzicki~\cite{Wo:IM82, Wo:PhD, Wo:Weyl} in his work on spectral asymmetry of zeta functions of elliptic \psidos\ on closed manifolds. We refer to~\cite{BCLZ:JPDOA12, GG:MS08, Gr:MS12, Po:IJM06, Wo:SS85} for further accounts of sectorial projections. 
\end{remark}

\begin{remark}
 For classical elliptic \psidos\ of positive order on closed manifolds, it has been argued by some authors that the approach of Seeley~\cite{Se:PSPM67} can be applied \emph{verbatim} to show that sectorial projections are \psidos. This seems to be an oversight (see~\cite{BCLZ:JPDOA12}). Grubb~\cite{Gr:MS12} overcame this by expressing the sectorial projection as a difference of logarithms (see Remark~\ref{rmk:sectorial.log-sectorial-PsiDO} below). We observe that the approach of this paper extends \emph{mutatis mutandis} to \psidos\ on ordinary closed manifolds. This simplifies all previous approaches to sectorial projections and provides a unified treatment of complex powers and sectorial projections of elliptic \psidos. 
\end{remark}

\begin{remark}[see~{\cite[Proposition~3.4]{Po:IJM06}}]\label{rmk:Sectorial.transitivity}
 If $L_{\phi''}=\{\arg \lambda =\phi''\}$ is a ray contained in $\check{\Theta}(P)$ such that $\phi < \phi' <\phi'' \leq \phi +2\pi$, then  it can be shown 
 that $ \Pi_{\phi,\phi'}(P)$ and $ \Pi_{\phi',\phi''}(P)$ are disjoint projections, and we have
\begin{equation}\label{eq:powers.sect-proj-additivity}
 \Pi_{\phi,\phi'}(P) + \Pi_{\phi',\phi''}(P)=\Pi_{\phi,\phi''}(P). 
\end{equation}
Moreover, if we take $\phi''=\phi+2\pi$, then
\begin{equation*}
 \Pi_{\phi,\phi'}(P) + \Pi_{\phi',\phi+2\pi}(P)=\Pi_{\phi,\phi+2\pi}(P)=1-\Pi_{0}(P). 
\end{equation*}
\end{remark}

\begin{remark}
 If $\Theta(P)\supseteq S_{\phi,\phi'}$, then $\overline{S_{\phi,\phi'}}\setminus 0\subseteq \Theta(P)$, and so $\Sp(P)\cap S_{\phi,\phi'}$ is finite by Theorem~\ref{thm:Resolvent.P-has-discrete-spectrum-resolvent-estimate}. By arguing as in the proof of~\cite[Proposition~A.2]{Po:IJM06}, it can then be shown that
 \begin{equation}\label{eq:powers.sect-proj-finite}
 \Pi_{\phi,\phi'}(P) = \!\!\!\!\! \sum_{\lambda \in \Sp(P)\cap S_{\phi,\phi'}} \!\!\!\!\! \Pi_\lambda(P). 
\end{equation}
\end{remark}

\begin{remark}\label{rmk:Sectorial.normal}
 If $P$ is normal, then $P$ admits an orthonormal eigenbasis. In this case, $\Pi_{\phi,\phi'}(P)$ is the orthogonal projection onto 
\begin{equation*}
 \bigoplus_{\lambda \in \Sp(P)\cap S_{\phi,\phi'}} \!\!\! \!\!\!  \ker(P-\lambda). 
\end{equation*}
In particular, if $P$ is selfadjoint, and we take $\phi\in (0,\pi)$ and $\phi'\in (\pi,2\pi)$, then $\Pi_{\phi,\phi'}(P)$ is just the orthogonal projection $\Pi_{-}(P)$  onto the negative eigenspace of $P$. 
\end{remark}

The inclusions in~(\ref{eq:powers.range-kernel-spec-proj}) are equalities if $P$ has a complete system of generalized eigenvectors in $L_2(\T^n_\theta)$, i.e., $\dotplus_{\lambda \in \Sp(P)}E_\lambda(P)$ is dense in $L_2(\T^n_\theta)$ (see~\cite[Proposition~A.2]{Po:IJM06}). This holds whenever $P$ is normal (see Remark~\ref{rmk:Sectorial.normal}), but not in general (see~\cite{AM:ZAA89, Se:CPDE86} for counterexamples on closed manifolds).

The next statements provide conditions ensuring that the inclusions in~(\ref{eq:powers.range-kernel-spec-proj}) are equalities.

\begin{proposition}\label{prop:powers.range-sect-proj}
 Suppose there are rays $L_{\phi_0}, \ldots, L_{\phi_N}$ in $\Theta(P)$ with $\phi_0=\phi$, $\phi_{N}=\phi'$ and $\phi_{j-1}<\phi_j<\phi_{j-1}+n^{-1}\pi w$. Then the range of $\Pi_{\phi,\phi'}(P)$ is equal to $\overline{E_{\phi,\phi'}(P)}$. 
\end{proposition}
\begin{proof}
 Let us first assume that $N=1$, i.e., $\phi<\phi'<\phi+n^{-1}\pi w$. Set $\sE_{\phi,\phi'}(P)=\ran \Pi_{\phi,\phi'}(P)$. This is a closed subspace of $L_2(\T^n_\theta)$ which is invariant under $P$, and so $P$ induces on $\sE_{\phi,\phi'}(P)$ a closed operator $P_{\phi,\phi'}$ with domain $\sE_{\phi,\phi'}(P)\cap W_2^w(\T^n_\theta)$. Moreover, as $E_{\phi,\phi'}(P)\subseteq \sE_{\phi,\phi'}(P)$ and $\sE_{\phi,\phi'}(P)\cap (E_0(P)\dotplus E_{\phi',\phi+2\pi}(P))=\{0\}$, we see that $\Sp(P_{\phi,\phi'})=\Sp(P)\cap S_{\phi,\phi'}$ and $E_\lambda(P_{\phi,\phi'})=E_\lambda(P)$ for all $\lambda \in \Sp(P)\cap S_{\phi,\phi'}$. Thus, $E_{\phi,\phi'}(P)$ is the subspace spanned by the generalized eigenvectors of $P_{\phi,\phi'}$. 
 
 As the partial inverse $P^{-1}$ is a \psido\ of order~$-w$, it is contained in the weak Schatten class $\sL_{p,\infty}$ with $p=nw^{-1}$ (\emph{cf}.~Proposition~\ref{prop:PsiDOs.Schatten}). As $P^{-1}$ and $(P_{\phi,\phi'})^{-1}$ agree on $\sE_{\phi,\phi'}(P)$, it follows that the resolvent of $P_{\phi,\phi'}$ is contained in $\sL_{p,\infty}$, and hence is in every Schatten class $\sL_{q}$ with $q>p$. As the spectrum of $P_{\phi,\phi'}$ is contained in the sector $S_{\phi,\phi'}$, which has aperture~$<n^{-1}\pi w=\pi p^{-1}$, it then follows from a criterion of Dunford--Schwartz~\cite[Corollary~XI.9.31]{DS:IP63} that the generalized eigenvectors of $P_{\phi,\phi'}$ are complete in $\sE_{\phi,\phi'}(P)$. In other words, the closure of $E_{\phi,\phi'}(P)$ is equal to $\ran \Pi_{\phi,\phi'}(P)$. This proves the result for $N=1$. 
 
 Suppose that $N\geq 2$. Here $L_{\phi_1}, \ldots, L_{\phi_{N-1}}$ are rays contained in $\Theta(P)$. As $\Theta(P)$ is an open cone,  if $\delta>0$ is small enough, then the angular sectors $\{|\arg \lambda-\phi_j|\leq \delta\}$, $j=1,\ldots, N-1$, are contained in $\Theta(P)$, and so, by  Theorem~\ref{thm:Resolvent.P-has-discrete-spectrum-resolvent-estimate}, they contain at most finitely many eigenvalues of $P$. 
 Therefore, we may assume that the rays $L_{\phi_1}, \ldots, L_{\phi_{N-1}}$ are actually contained in $\check{\Theta}(P)$. In this case, by using~(\ref{eq:powers.sect-proj-additivity}) and arguing by induction, we see that $\Pi_{\phi_0,\phi_1}(P), \ldots, \Pi_{\phi_{N-1},\phi_N}(P)$ are mutually disjoint projections such that $\Pi_{\phi_0,\phi_1}(P)+ \cdots +\Pi_{\phi_{N-1},\phi_N}(P)=\Pi_{\phi,\phi'}(P)$. Thus, 
\begin{equation*}
 \ran \Pi_{\phi,\phi'}(P) = \ran \Pi_{\phi_0,\phi_1}(P)\dotplus \cdots \dotplus\ran \Pi_{\phi_{N-1},\phi_N}(P).
\end{equation*}
 The $N=1$ case shows that $ E_{\phi_{j-1},\phi_j}(P)$ is dense in $\ran \Pi_{\phi_{j-1},\phi_j}(P)$ for $j=1,\ldots, N$. It then follows that $E_{\phi,\phi'}(P)$ is dense in $ \ran \Pi_{\phi,\phi'}(P)$. This proves the result. 
\end{proof}

Specializing Proposition~\ref{prop:powers.range-sect-proj} to $\phi'=\phi+2\pi$ gives the following result. 

\begin{corollary} \label{cor:sectorial.plane-division-eigenvectors-completeness}
 Assume we can find rays $L_{\phi_0}, \ldots, L_{\phi_N}$ in $\Theta(P)$ that divide $\C^*$ into angular sectors of aperture~$<n^{-1}\pi w$. Then the 
generalized eigenvectors of $P$ are complete in $L_2(\T^n_\theta)$. 
\end{corollary}

\begin{remark}
 We refer to~\cite{Ag:CPAM62, Ag:Springer94, Bu:Pisa68, Po:IJM06} for versions of Proposition~\ref{prop:powers.range-sect-proj} and Corollary~\ref{cor:sectorial.plane-division-eigenvectors-completeness} for elliptic \psidos\ on closed manifolds. 
\end{remark}

Finally, we have the following asymmetry formulas for the complex powers and logarithms.

\begin{theorem} \label{thm:powers.asymmetry}
We have 
\begin{equation}
 P_\phi^z-P_{\phi'}^z = \left(1-e^{2i\pi z}\right) \Pi_{\phi,\phi'}(P)P_\phi^z \qquad \forall z\in \C. 
 \label{eq:powers.asymmetry} 
\end{equation}
In addition, we have
\begin{equation}\label{eq:Log.difference} 
 \log_{\phi'}(P) -  \log_{\phi}(P)= (2i\pi) \Pi_{\phi,\phi'}(P).  
\end{equation}
\end{theorem}
\begin{proof}
 The proof of~(\ref{eq:powers.asymmetry}) is similar to the proof of~\cite[Proposition~4.1]{Po:IJM06}. We only highlight the main steps. If $\Re z<0$, then, 
 by arguing along the same lines, we get the following two formulas,
\begin{gather*}
  P_\phi^z-P_{\phi'}^z =  \left(e^{2i\pi z}-1\right)  \frac{i}{2\pi} \int_{\Gamma'} \lambda^z_\phi (P-\lambda)^{-1}d\lambda, \quad \text{and}
  \\ 
 \Pi_{\phi,\phi'}(P)P_\phi^z= - \frac{i}{2\pi} \int_{\Gamma'} \lambda^z_\phi (P-\lambda)^{-1}d\lambda. 
\end{gather*}
This gives~(\ref{eq:powers.asymmetry}) for $\Re z<0$. As both sides of~(\ref{eq:powers.asymmetry}) are holomorphic families in the locally convex space $\sL(C^\infty(\T^n_\theta))$, the equality holds for all $z\in \C$ (\emph{cf}.~Proposition~\ref{prop:Hol.unique-analytic-continuation}).

It remains to prove~(\ref{eq:Log.difference}). By using~(\ref{eq:powers.asymmetry}) and the fact that $P_\phi^0=1-\Pi_0(P)$, we get
\begin{equation}
 \log_{\phi'}(P)-\log_{\phi}(P)= \left.\frac{d}{dz}\right|_{z=0} \bigg\{ (e^{2i\pi z}-1) \Pi_{\phi,\phi'}(P)P_\phi^z\bigg\} = (2i\pi)  \Pi_{\phi,\phi'}(P)\big(1-\Pi_0(P)\big). 
 \label{eq:Log.difference2} 
\end{equation}
Recall from Remark~\ref{rmk:Sectorial.transitivity} that $\Pi_{\phi,\phi'}(P)$ and $\Pi_{\phi',\phi+2\pi}(P)$ are disjoint projections such that
$\Pi_{\phi,\phi'}(P)+\Pi_{\phi',\phi+2\pi}(P)=1-\Pi_0(P)$, and so $\Pi_{\phi,\phi'}(P)(1-\Pi_0(P)) =\Pi_{\phi,\phi'}(P)$. Combining this with~(\ref{eq:Log.difference2}) gives~(\ref{eq:Log.difference}). 
\end{proof}

\begin{remark}
 For elliptic \psidos\ on closed manifolds, the asymmetry formula~(\ref{eq:powers.asymmetry}) goes back to Wodzicki~\cite{Wo:IM82, Wo:PhD, Wo:Weyl}, and the formula~(\ref{eq:Log.difference}) goes back to Okikiolu~\cite{Ok:Duke95} and Gaarde--Grubb~\cite{GG:MS08}.
\end{remark}

\begin{remark}\label{rmk:sectorial.log-sectorial-PsiDO}
Combining~(\ref{eq:Log.difference}) with Theorem~\ref{thm:Log.psido-ness} allows us to recover the fact that $\Pi_{\phi,\phi'}(P)$ is a zeroth-order \psido\ (see also~\cite{Gr:MS12}). 
\end{remark}

\appendix
\section{Quasi-Banach Spaces and Schatten Classes}\label{sec:quasi+Schatten}
In this appendix, for the reader's convenience, we recall basic facts about quasi-Banach spaces and Schatten classes and their weak versions. We refer to~\cite{GK:AMS69, LSZ:Book, Si:AMS05} and the references therein for further background on these topics.

\subsection{Quasi-Banach spaces} 
\begin{definition}
A \emph{quasi-norm} on a vector space $E$ is a function $\|\cdot\|: E\rightarrow [0,\infty)$ with the following properties:
\begin{enumerate}
 \item[(i)] $\| \lambda x\|=|\lambda| \|x\|$ for all $x\in E$ and $\lambda \in \C$.

 \item[(ii)] $\|0\|=0$ and $\|x\|>0$ for $x\neq 0$.

  \item[(iii)] There exists $C>0$ such that
 \begin{equation*}
 \|x+y\|\leq C\left(\|x\|+\|y\|\right) \qquad \forall x,y\in E.
\end{equation*}
\end{enumerate}
\end{definition}

In other words, a quasi-norm is like a norm where the usual triangle inequality is relaxed to the quasi-triangle inequality. In particular, any norm is a quasi-norm. As with norms, any quasi-norm $\|\cdot\|$ on a vector space $E$ defines a topology in which a basis of the neighborhood system for the origin consists of the balls
\begin{equation*}
 B(0,\delta)=\left\{x\in E; \ \|x\|\leq \delta\right\}, \qquad \delta>0.
\end{equation*}
Note that this topology need not be locally convex if $\|\cdot\|$ is not a norm. Moreover, every point admits a countable neighborhood basis. 

A subset $B$ of a quasi-normed space $(E, \|\cdot\|)$ is bounded if and only if the quasi-norm is bounded on $B$. Moreover, if $(E_1, \|\cdot \|_1)$ is another quasi-normed space, then a linear map $T:E\rightarrow E_1$ is continuous if and only if there is $C>0$ such that 
\begin{equation*}
 \|Tx\|_1 \leq C \|x\| \qquad \forall x \in E. 
\end{equation*}
 
\begin{definition}
 A \emph{quasi-Banach space} is a quasi-normed space that is complete, i.e., every Cauchy sequence is convergent.
\end{definition}

\subsection{Schatten classes and their weak versions}
Let $\sH$ be a (separable) Hilbert space. We denote by $\sL(\sH)$ the algebra of bounded linear operators on $\sH$ with norm $\|\cdot\|$. We also denote by $\sK$ the (closed) ideal of compact operators on $\sH$.

Given any operator $T\in \sK$, we let $(\mu_j(T))_{j\geq 0}$ be its sequence of \emph{singular values}, i.e., $\mu_j(T)$ is the $(j+1)$-th eigenvalue of the absolute value $|T|=\sqrt{T^*T}$, where the eigenvalues are counted with multiplicity and arranged in non-increasing order. By the \emph{min-max principle}, we have
\begin{equation*}
 \mu_j(T)=\min \left\{\|T_{|E^\perp}\|;\ \dim E=j\right\}.
\end{equation*}
In addition, we have
\begin{equation*}
  \mu_j(T)=\dist(T,\sR_j), \qquad \text{where}\ \sR_j=\{R\in \sL(\sH);\ \rk R\leq j\}.
\end{equation*}

We have the following properties of singular values:
\begin{gather*}
 \mu_j(T)=\mu_j(T^*)=\mu_j(|T|),\\
 \left|\mu_j(T)-\mu_j(S)\right|\leq \|S-T\|,\\
\mu_j(ATB)\leq \|A\| \mu_j(T) \|B\| \qquad \forall A, B\in \sL(\sH).
\end{gather*}

\begin{definition}
 A \emph{quasi-Banach ideal} is a two-sided ideal $\sJ$ of $\sL(\sH)$ with a complete quasi-norm $\|\cdot\|_{\sJ}$ such that
 \begin{equation*}
 \|ATB\|_{\sJ} \leq \|A\| \|T\|_{\sJ} \|B\| \qquad \text{for all $T\in \sJ$ and $A,B\in \sL(\sH)$}.
\end{equation*}
If $\|\cdot\|_\sJ$ is a norm, then we say that $\sJ$ is a \emph{Banach ideal}.
\end{definition}

If $p>0$, then the \emph{Schatten class}
\begin{equation*}
 \sL_p:= \big\{ T\in \sK; \ \sum \mu_j(T)^p<\infty\big\},
\end{equation*}
is a quasi-Banach ideal with quasi-norm
\begin{equation*}
 \|T\|_p:=  \big(\sum \mu_j(T)^p\big)^{1/p}, \qquad T\in \sL_p.
\end{equation*}
This is actually a Banach ideal for $p\geq 1$.

The \emph{weak Schatten class}
\begin{equation*}
 \sL_{p,\infty}:= \big\{ T\in \sK; \ \mu_j(T)=\op{O}(j^{-\frac{1}{p}})\big\},
\end{equation*}
is a quasi-Banach ideal with quasi-norm
\begin{equation*}
 \|T\|_{p,\infty}:= \sup_{j\geq 0}(j+1)^{\frac 1p}\mu_j(T), \qquad T\in \sL_{p,\infty}.
\end{equation*}
For $p>1$, this quasi-norm is equivalent to the norm
\begin{equation*}
 \|T\|_{p,\infty}':= \sup_{N\geq 1} N^{-1+\frac{1}{p}} \sum_{j<N}\mu_j(T),  \qquad T\in \sL_{p,\infty}.
\end{equation*}
Therefore, in this case, $\sL_{p,\infty}$ is a Banach ideal with respect to $\|\cdot\|_{p,\infty}'$. In particular, its topology is locally convex. 

Finally, we have the (strict) continuous inclusions:
\begin{equation*}
 \sL_q \subsetneq \sL_{p,\infty}  \subsetneq \sL_r , \qquad q\leq p <r. 
\end{equation*}

\end{document}